%% file: obg2.tex
\documentclass[12pt]{article}

\title{Gluing pseudoholomorphic curves along branched covered
 cylinders II}
 \author{Michael Hutchings and Clifford Henry Taubes}

\usepackage{amssymb}
\usepackage{latexsym}
\usepackage{amsmath}
\usepackage{amsthm}
\usepackage{amscd}

\numberwithin{equation}{section}

\newcommand{\mc}[1]{{\mathcal #1}}

\newtheorem{theorem}{Theorem}[section]
\newtheorem{proposition}[theorem]{Proposition}
\newtheorem{corollary}[theorem]{Corollary}
\newtheorem{lemma}[theorem]{Lemma}
\newtheorem{lemma-definition}[theorem]{Lemma-Definition}

\theoremstyle{definition}
\newtheorem{definition}[theorem]{Definition}
\newtheorem{remark}[theorem]{Remark}

\newcommand{\floor}[1]{\left\lfloor #1 \right\rfloor}
\newcommand{\ceil}[1]{\left\lceil #1 \right\rceil}

\newcommand{\eqdef}{\;{:=}\;}
\newcommand{\fedqe}{\;{=:}\;}

\renewcommand{\frak}{\mathfrak}

\newcommand{\C}{{\mathbb C}}

\newcommand{\R}{{\mathbb R}}

\newcommand{\Z}{{\mathbb Z}}
\newcommand{\op}{\operatorname}
\newcommand{\dbar}{\overline{\partial}}
\newcommand{\zbar}{\overline{z}}
\newcommand{\wbar}{\overline{w}}
\newcommand{\ubar}{\overline{u}}

\newcommand{\Hom}{\op{Hom}}
\newcommand{\Ker}{\op{Ker}}
\newcommand{\Coker}{\op{Coker}}

\newcommand{\tensor}{\otimes}

\newcommand{\vu}{\nu}
\newcommand{\Nbar}{\overline{N}}
\newcommand{\vbar}{\overline{v}}
\newcommand{\morrey}{{1/2}}
\newcommand{\morreyovertwo}{{1/4}}
\newcommand{\morreyovertwominusone}{{-3/4}}

\newcommand{\bpm}{\begin{pmatrix}}
\newcommand{\epm}{\end{pmatrix}}

\begin{document}

\maketitle

\begin{abstract}
This paper and its prequel (``Part I'') prove a generalization of the
usual gluing theorem for two index $1$ pseudoholomorphic curves $U_+$
and $U_-$ in the symplectization of a contact 3-manifold.  We assume
that for each embedded Reeb orbit $\gamma$, the total multiplicity of
the negative ends of $U_+$ at covers of $\gamma$ agrees with the total
multiplicity of the positive ends of $U_-$ at covers of $\gamma$.
However, unlike in the usual gluing story, here the individual
multiplicities are allowed to differ.  In this situation, one can
often glue $U_+$ and $U_-$ to an index $2$ curve by inserting genus
zero branched covers of $\R$-invariant cylinders between them.  This
paper shows that the signed count of such gluings equals a signed
count of zeroes of a certain section of an obstruction bundle over the
moduli space of branched covers of the cylinder.  Part I obtained a
combinatorial formula for the latter count and, assuming the result of
the present paper, deduced that the differential $\partial$ in
embedded contact homology satisfies $\partial^2=0$.  The present paper
completes all of the analysis that was needed in Part I.  The gluing
technique explained here is in principle applicable to more gluing
problems. We also prove some lemmas concerning the generic behavior of
pseudoholomorphic curves in symplectizations, which may be of
independent interest.
\end{abstract}

\section{Introduction}
\label{sec:intro}

\input{obg-intro2}

\section{Asymptotics of $J$-holomorphic curves}
\label{sec:decay}

\input{obg-decay}

\section{Generic behavior of asymptotic eigenfunctions}
\label{sec:generic}

\input{obg-generic}

\section{Genericity of immersion singularities}
\label{sec:immersed}

\input{obg-immersed}

\section{The gluing construction}
\label{sec:gluing}

\input{obg-gluing}

\section{Properties of the obstruction section}
\label{sec:properties}

\input{obg-properties}

\section{Bijectivity of the gluing map}
\label{sec:map}

\input{obg-map}

\section{Deforming to the linearized section}
\label{sec:deform}

\input{obg-deform}

\section{Coherent orientations}
\label{sec:coherent}

\input{obg-coherent}

\section{Counting ends of the index 2 moduli space}
\label{sec:signs}

\input{obg-signs}

\end{document}

%% file: obg-intro2.tex
This paper is a sequel to \cite{obg1}, which we refer to here as
``Part I''; references to Part I are given in the form I.$*$.  We
assume some familiarity with \S{I.1}, \S{I.2}, and \S{I.3}, although
we will attempt to review the essentials of what is needed here.  We
will not use any of \S{I.4}, \S{I.5}, or \S{I.7}, and we will only
rarely use
\S{I.6}.

\subsection{Geometric setup}

This paper studies pseudoholomorphic curves in the symplectization of
a contact 3-manifold.  The setup for this is as follows:  Let $Y$ be a closed
oriented 3-manifold, and let $\lambda$ be a contact $1$-form on $Y$.
Let ${\mathbf R}$ denote the Reeb vector field associated to
$\lambda$, and assume that all Reeb orbits are nondegenerate.  Let $J$
be an admissible almost complex structure on $\R\times Y$.
``Admissible'' here means that $J$ is $\R$-invariant; $J$ sends the
$\R$ direction, denoted by $\partial_s$, to ${\mathbf R}$; and $J$
sends the contact plane field $\xi=\Ker(\lambda)$ to itself, rotating
positively with respect to $d\lambda$.  If
$\alpha=(\alpha_1,\ldots,\alpha_k)$ and
$\beta=(\beta_1,\ldots,\beta_l)$ are ordered lists of Reeb orbits,
possibly repeated or multiply covered, then $\mc{M}^J(\alpha,\beta)$
denotes the moduli space of $J$-holomorphic curves $u:C\to\R\times Y$
with ordered and asymptotically marked positive ends at
$\alpha_1,\ldots,\alpha_k$, ordered and asymptotically marked negative
ends at $\beta_1,\ldots,\beta_l$, and no other ends.  The precise
definitions of the above notions are reviewed in \S{I.1.1}.  We assume that
the domain $C$ is a punctured compact Riemann surface, which may be
disconnected, and whose components may have any genus.  We also assume
that $J$ is generic so that all non-multiply-covered $J$-holomorphic
curves are unobstructed \cite{dragnev}.

We want to glue together two curves $U_+\in\mc{M}^J(\alpha_+,\beta_+)$
and $U_-\in\mc{M}^J(\beta_-,\alpha_-)$ that consitute a ``gluing
pair'' in the sense of Definition~{I.1.9}.  This means the following:
First, $U_+$ and $U_-$ have Fredholm index $1$, cf.\ \S{I.1.1}.
Second, $U_+$ and $U_-$ are immersed, and not multiply covered, except
that they may contain unbranched covers of $\R$-invariant cylinders.
Third, for each embedded Reeb orbit $\gamma$, the total covering
multiplicity of Reeb orbits covering $\gamma$ in the list $\beta_+$ is
the same as the corresponding total for $\beta_-$.  For example, $U_+$
could have two negative ends at $\gamma$, while $U_-$ could have a
single positive end at the double cover of $\gamma$, which we denote
by $\gamma^2$.  Note that in the usual gluing story, one would assume
that the lists $\beta_+$ and $\beta_-$ are identical; the weakening of
this condition above is the novelty of the present paper.  Finally,
when $U_+$ or $U_-$ contain covers of $\R$-invariant cylinders over
elliptic Reeb orbits, there is a fourth condition in
Definition~{I.1.9}, called ``partition minimality'', concerning the
multiplicities of these covers.  We will not review the meaning of
this condition now, because it will not come to the fore in this paper
until
\S\ref{sec:SDR}.

Given a gluing pair $(U_+,U_-)$ as above, we want to compute a signed
count of ends of the index 2 part of the moduli space
$\mc{M}^J(\alpha_+,\alpha_-)/\R$ that are ``close to breaking'' into
$U_+$ and $U_-$ together with some index zero branched covers of
$\R$-invariant cylinders between them.  This count is made precise in
Definition~{I.1.12}, reviewed here in Definition~\ref{def:count}, and
denoted by $\#G(U_+,U_-)\in\Z$.  To define this count one needs to
assume that $J$ is generic, so that all moduli spaces of
non-multiply-covered $J$-holomorphic curves are smooth and have the
expected dimension.  To determine the signs one also needs to fix
``coherent orientations'' of the relevant moduli spaces; our conventions for
doing so are specified in \S\ref{sec:orienting}.

Note that the definition of $\#G(U_+,U_-)$ does not count boundary
points of the compactification of the moduli space
$\mc{M}^J(\alpha_+,\alpha_-)/\R$ as in \cite{behwz,cm} (for all we know
there could be infinitely many such boundary points), but rather
counts boundary points of a truncation of the moduli space.
Consequently, this definition is insensitive to the number of levels
of branched covers that may appear in the limit of a sequence of
curves in $\mc{M}^J(\alpha_+,\alpha_-)/\R$.

\subsection{Statement of the main result}
\label{sec:SMR}

Before continuing, we make two additional assumptions on our gluing
pair $(U_+,U_-)$.  The first is:
\begin{equation}
\label{eqn:i}
\begin{array}{l}
\text{All negative ends of $U_+$, and all positive ends of $U_-$, are at}\\
\text{covers of a single embedded elliptic Reeb orbit $\alpha$.}
\end{array}
\end{equation}
Here the statement that $\alpha$ is ``elliptic'' means that the
linearized return map of the Reeb flow around $\alpha$ has eigenvalues
on the unit circle, and thus is conjugate to a rotation by angle
$2\pi\theta$ for some $\theta\in\R$.  Our standing assumption that all
Reeb orbits are nondegenerate implies that $\theta$ is irrational.

To state the second assumption on $(U_+,U_-)$, let
$a_1,\ldots,a_{N_+}$ denote the multiplicities of the negative ends of
$U_+$ (this means that $U_+$ has negative ends at the covers
$\alpha^{a_1},\ldots,\alpha^{a_{N_+}}$), and likewise let
$a_{-1},\ldots,a_{-N_-}$ denote the multiplicities of the positive
ends of $U_-$.  The second assumption is:
\begin{equation}
\label{eqn:ii}
\sum_{i=1}^{N_+}\ceil{a_i\theta} - \sum_{i=-1}^{-N_-}\floor{a_i\theta}
= 1.
\end{equation}

To see the significance of this assumption, let
\begin{equation}
\label{eqn:specifyM}
\mc{M} \eqdef
\mc{M}(a_1,\ldots,a_{N_+} \mid a_{-1},\ldots,a_{-N_-})
\end{equation}
denote the moduli space of connected genus zero branched covers of
$\R\times S^1$ from Definition I.2.1.  Recall that a branched cover in
$\mc{M}$ has positive ends indexed by $1,\ldots,N_+$, and negative
ends indexed by $-1,\ldots,-N_-$, such that the end indexed by $i$ has
multiplicity $a_i$; and all ends are asymptotically marked.  We use a
parametrization of $\alpha$ to identify elements of $\mc{M}$ with
branched covers of the $J$-holomorphic cylinder $\R\times\alpha$ in
$\R\times Y$.  As explained in \S{I.1.2}, a branched cover of
$\R\times\alpha$ with positive ends of multiplicities
$a_1,\ldots,a_{N_+}$ and negative ends of multiplicities
$a_{-1},\ldots,a_{-N_-}$ has Fredholm index zero if and only if it
consists of $\kappa_\theta$ genus zero components, where
$\kappa_\theta$ denotes the left hand side of \eqref{eqn:ii}.  Hence
the assumption \eqref{eqn:ii} implies that index zero branched covers
of $\R\times\alpha$ with ends as above correspond to elements of
$\mc{M}$.

We can now explain the idea of the gluing construction.  Fix
$R>>r>>0$.  Let $\mc{M}_R$ denote the set of branched covers in
$\mc{M}$ such that all ramification points have $|s|\le R$, where $s$
denotes the $\R$ coordinate on $\R\times S^1$.  Given a branched cover
$\pi:\Sigma\to\R\times S^1$ in $\mc{M}_R$, we can form a ``preglued''
curve by using appropriate cutoff functions to patch the negative ends
of the $s\mapsto s+R+r$ translate of $U_+$ to the positive ends of
$\Sigma$, and the positive ends of the $s\mapsto s-R-r$ translate of
$U_-$ to the negative ends of $\Sigma$.  Now try to perturb the
preglued curve to a $J$-holomorphic curve, where near the ramification
points of the branched cover we only perturb in directions normal to
$\R\times\alpha$.  For a given branched cover $\Sigma$, we can obtain
a (unique) $J$-holomorphic curve this way if and only if
$\frak{s}(\Sigma)=0$, where $\frak{s}$ is a certain section of the
obstruction bundle $\mc{O}\to\mc{M}_R$.  Here the fiber of $\mc{O}$
over a branched cover $\Sigma$ consists of the (dual of the) cokernel of an
associated linear deformation operator $D_\Sigma$; see \S{I.2.3} for
details.  Note that the rank of the obstruction bundle equals the
dimension of $\mc{M}$.  In this way the count of gluings
$\#G(U_+,U_-)$ that we are after is related to a count of zeroes of
the obstruction section $\frak{s}$.

The section $\frak{s}$ is difficult to understand directly, because it
arises in a somewhat indirect way out of the analysis in
\S\ref{sec:gluing}.  Fortunately there is a more tractable
section, the ``linearized section'' $\frak{s}_0$, which has the same
count of zeroes.  The linearized section $\frak{s}_0$ is defined
explicitly in \S\ref{sec:DLS} in terms of the collection of
``asymptotic eigenfunctions'' $\gamma$ associated to the negative ends
of $U_+$ and the positive ends of $U_-$.  (For the definition of the
asymptotic eigenfunction associated to an end of a $J$-holomorphic
curve, see \S\ref{sec:AF}.)  As explained in
\S{I.3.2} (using results from \S\ref{sec:SDR}), the signed count of
zeroes of $\frak{s}_0$ over $\mc{M}_R$, denoted by
$\#\frak{s}_0^{-1}(0)\in\Z$, is well-defined as long as $\gamma$ is
``admissible'' in the sense of Definition~{I.3.2}.  We will prove in
\S\ref{sec:generic} that this admissibility condition holds if $J$ is
generic.  We also showed in
\S{I.3.2} that the count $\#\frak{s}_0^{-1}(0)$ for admissible
$\gamma$ does not depend on $\gamma$, but only on the multiplicities
of the $\R$-invariant and non-$\R$-invariant negative ends of $U_+$
and positive ends of $U_-$ at covers of $\alpha$.  Let
$\#\frak{s}_0^{-1}(0)$ denote this count for admissible $\gamma$. The
main result of this paper can then be stated as follows.  (This
appears in Part I as Theorem~{I.3.6}.)

\begin{theorem}
\label{thm:main}
Fix coherent orientations, let $J$ be a generic admissible almost
complex structure on $\R\times Y$, and let $(U_+,U_-)$ be a gluing
pair satisfying assumptions \eqref{eqn:i} and \eqref{eqn:ii} above.
Then
\[
\#G(U_+,U_-) = \epsilon(U_+)\epsilon(U_-)\#\frak{s}_0^{-1}(0).
\]
\end{theorem}

Here $\epsilon(U_+),\epsilon(U_-)\in\{\pm1\}$ denote the signs associated
to $U_+$ and $U_-$ by the system of coherent orientations; see \S{I.1.1} and
\S\ref{sec:orienting}.

As explained in \S{I.1.8}, there is a straightforward generalization
of this story in which the assumptions \eqref{eqn:i} and
\eqref{eqn:ii} are dropped.  This requires gluing in disconnected
branched covers.  We will omit the details of this generalization, as
 it does not involve any new analysis and differs only in the amount
 of notation.

Recall from Part I that the significance of Theorem~\ref{thm:main} is
as follows.  In Proposition~{I.5.1}, we obtained a combinatorial
formula for the count $\#\frak{s}_0^{-1}(0)$.  Combining this with
Theorem~\ref{thm:main} (and its generalization for disconnected
branched covers) proves the main result of this pair of papers, namely
Theorem~{I.1.13}, which gives a combinatorial formula for
$\#G(U_+,U_-)$.  An important application of Theorem~{I.1.13} is given in
\S{I.7}, which deduces that the differential $\partial$ in embedded
contact homology (see \cite{t3}) satisfies $\partial^2=0$.

Essentially the same argument shows that $\partial^2=0$ in the
periodic Floer homology of mapping tori.  In fact, our gluing theorem
generalizes easily to stable Hamiltonian structures, as defined in
\cite{cm,siefring}, of which contact structures and mapping tori are
special cases.  The starting point for the analysis in the present
paper is a nice local coordinate system around a Reeb orbit, and that
exists just as well in this more general setting.

\subsection{Guide to the paper}

This paper divides roughly into three parts.

The first part, consisting of \S\ref{sec:decay}--\S\ref{sec:immersed},
does not yet address the gluing problem, but rather proves some
general results on $J$-holomorphic curves in $\R\times Y$ which we
will use in the gluing story and which might be of independent
interest.  In \S\ref{sec:decay} we describe the asymptotic behavior of
ends of $J$-holomorphic curves and define their asymptotic
eigenfunctions.  Although asymptotic results of this sort have
appeared previously in \cite{hwz1,mora,siefring}, we will find it
useful to review the asymptotics in a particular way in order to
prepare for the subsequent analysis.  In \S\ref{sec:generic} we prove
that for generic $J$, if $u$ is an index $1$, connected,
non-multiply-covered $J$-holomorphic curve, then the asymptotic
eigenfunctions describing the ends of $u$ are all nonzero; and
moreover, whenever two ends of $u$ have the same ``asymptotic
eigenvalue'', the corresponding asymptotic eigenfunctions are
geometrically distinct.  This is exactly what is needed to show that
the collection of asymptotic eigenfunctions $\gamma$ determined by a
gluing pair $(U_+,U_-)$ as in \S\ref{sec:SMR} is admissible.  In
\S\ref{sec:immersed} we prove that if $J$ is generic, then all
non-multiply-covered $J$-holomorphic curves of index $\le 2$ are
immersed.  Although this is probably not really necessary, it will
simplify the analysis in the rest of the paper by allowing us to
consider only immersed curves (except of course for the branched
covers of cylinders that we are gluing in).

The second part of the paper, consisting of
\S\ref{sec:gluing}--\S\ref{sec:deform}, explains the details
of gluing.  In \S\ref{sec:gluing} we show how to glue $U_+$ and $U_-$
along a branched cover of the cylinder, wherever the obstruction
section $\frak{s}$ vanishes.  Note that the section $\frak{s}$ in
\S\ref{sec:gluing} is not defined over $\mc{M}_R$ as in
\S\ref{sec:SMR}, but rather over a slightly different domain, because we
are not yet modding out by the $\R$ action on moduli spaces of
$J$-holomorphic curves.  In \S\ref{sec:properties} we prove various
technical properties of the obstruction section $\frak{s}$, including
its continuity.  In \S\ref{sec:map} we show that the gluing
construction describes all curves that are ``close to breaking'' into
$U_+$ and $U_-$ along a branched cover of the cylinder, in the precise
sense of Definition~{I.1.10}, which is reviewed in
Definition~\ref{def:CTB}.  To count ends of the index $2$ part of the moduli
space $\mc{M}^J(\alpha_+,\alpha_-)/\R$, we now want to count zeroes of
$\frak{s}$ over a relevant slice of the quotient of its domain by the
$\R$ action.  This slice is identified with $\mc{M}_R$, minus a fringe
region where $\frak{s}$ has no zeroes.  In \S\ref{sec:deform} we
prove that if the collection of asymptotic eigenfunctions $\gamma$
determined by $(U_+,U_-)$ is admissible, then whenever $R>>r>>0$, the
signed count of zeroes of $\frak{s}$ on $\mc{M}_R$ is the same as that
of the linearized section $\frak{s}_0$:
\begin{equation}
\label{eqn:nsns0}
\#\frak{s}^{-1}(0) = \#\frak{s}_0^{-1}(0).
\end{equation}
For the precise statement see Corollary~\ref{cor:ecd} and
Remark~\ref{rem:cf}.  In the proof of \eqref{eqn:nsns0}, the
admissibility condition on $\gamma$ is needed to ensure that no zeroes
of the section cross the boundary of $\mc{M}_R$ as we deform
$\frak{s}$ to $\frak{s}_0$.

As described previously, the count $\#G(U_+,U_-)$ of relevant ends of
the index $2$ part of the moduli space
$\mc{M}^J(\alpha_+,\alpha_-)/\R$ can be identified with a count of
zeroes of $\frak{s}$ on $\mc{M}_R$.  However equation
\eqref{eqn:nsns0} does not yet prove Theorem~\ref{thm:main}, because
the signs with which the zeroes of $\frak{s}$ are counted in
\eqref{eqn:nsns0} are determined by canonical orientations of the
obstruction bundle and of the moduli space of branched covers, and
might not agree with the signs (coming from the coherent orientations)
with which the ends of $\mc{M}^J(\alpha_+,\alpha_-)/\R$ are counted.
This brings us to the third and last part of the paper, which is a
detailed discussion of signs, occupying
\S\ref{sec:coherent} and \S\ref{sec:signs}.  To go from equation
\eqref{eqn:nsns0} to Theorem~\ref{thm:main}, in \S\ref{sec:signs}
we will prove Theorem~\ref{thm:count}, asserting that if $R>>r>>0$
then
\begin{equation}
\label{eqn:Gees}
\#G(U_+,U_-) = \epsilon(U_+)\epsilon(U_-)\#\frak{s}^{-1}(0).
\end{equation}
To prove \eqref{eqn:Gees}, it turns out that for generic $R$, each
side of the equation is a signed count of points in the same finite
set, so we just need to compare the signs.  To set up this comparison,
we need to rework the theory of coherent orientations from scratch,
which is what we do in
\S\ref{sec:coherent}.

That completes the outline of the paper.  Before plunging into the
details, let us briefly indicate the basic idea of the gluing
analysis.  (This is adapted from a technique pioneered by Donaldson in
the context of four-dimensional gauge theory; see
\cite[Ch.\ 7]{donkron}.)  Suppose we want to glue together some curves
$u_1,\ldots,u_n$ in some configuration.  For simplicity suppose that
each $u_i$ is immersed with domain $C_i$.  To start, we can use
appropriate cutoff functions to form a preglued curve $C_0$.  Now if
$\psi_i$ is a section of the normal bundle to $C_i$ for each $i$, then
we can deform $C_0$ in the direction
$\beta_1\psi_1+\cdots+\beta_n\psi_n$, where $\beta_i$ is a cutoff
function supported over the part of $C_0$ coming from $C_i$.  The
deformed pregluing will be pseudoholomorphic if and only if an
equation of the form
\[
\beta_1\Theta_1(\psi_1,\ldots,\psi_n) + \cdots +
\beta_n\Theta_n(\psi_1,\ldots,\psi_n) = 0
\]
holds.  Here $\Theta_i$ is defined on all of $C_i$, and has the form
\[
\Theta_i = D_i\psi_i + \cdots
\]
where $D_i$ is the linear deformation operator associated to $C_i$, and
the remaining terms are mostly nonlinear and involve the $\psi_j$'s
for those $j$ such that $C_i$ is adjacent to $C_j$ in the gluing
configuration.  If one sets this up properly, then the contraction
mapping theorem in a suitable Banach space finds a unique $n$-tuple
$(\psi_1,\ldots,\psi_n)$ such that $\psi_i \perp \Ker(D_i)$ and
\[
\Theta_i(\psi_1,\ldots,\psi_n) \in \Coker(D_i)
\]
for each $i$.  If the $n$-tuple $(u_1,\ldots,u_n)$ varies over some
moduli space, then these elements of $\Coker(D_i)$ define an
obstruction section over this moduli space, and we will obtain a
pseudoholomorphic curve wherever this obstruction section vanishes.
Further analysis shows that this construction identifies the zero set
of the obstruction section with the set of all gluings in an
appropriate sense.  Finally, the main contribution to $\Theta_i$,
other than $D_i\psi_i$, arises from the failure of the original
preglued curve to be pseudoholomorphic, which is essentially
determined by the asymptotic behavior of the $u_i$'s; and this is what
we use to define the ``linearized section''.  We expect that this
technique can be applied to additional gluing problems.

%% file: obg-decay.tex
Let $J$ be an admissible almost complex structure on $\R\times Y$, and
let $u:C\to\R\times Y$ be a $J$-holomorphic curve.  In this section we
prove an asymptotic formula, stated in Proposition~\ref{prop:AF}
below, for the behavior of $u$ on each end of $C$.  Similar asymptotic
formulas have been previously established in
\cite{hwz1,mora,siefring}.  However we will find it useful to go over
the asymptotics in a particular way in order to prepare for the
subsequent analysis.

\subsection{Neighborhoods of $\R$-invariant cylinders}
\label{sec:NRIC}

We begin by writing down equations for $J$-holomorphic curves near a
Reeb orbit.  These equations will be used throughout the paper.

Let ${\mathbf R}$ denote the Reeb vector field on $Y$.  Let $\alpha$
be an embedded Reeb orbit; by rescaling the $s$ and $t$
coordinates on $\R\times Y$, we may assume that $\alpha$ has period
$2\pi$. Fix a parametrization $\alpha:S^1\to Y$ with
$\alpha'(t)={\mathbf R}$.  Recall that the admissible almost complex
structure $J$ sends the $\R$ direction $\partial_s$ to the Reeb vector
field ${\mathbf R}$, so $J(\partial_s) = \alpha'(t)$ on
$\R\times\alpha$.

We begin the analysis by choosing coordinates on a tubular
neighborhood of $\R\times\alpha$ in $\R\times Y$, via an
``exponential map'' $e$ with certain nice properties.

\begin{lemma}
\label{lem:EM}
For each embedded Reeb orbit $\alpha:S^1\to Y$, there exists a disc
$D\subset\C$ containing the origin and an embedding
\[
e:\R\times S^1\times D \longrightarrow \R\times Y
\]
with the following properties:
\begin{itemize}
\item
$e(s,t,0)=(s,\alpha(t))$.
\item
The derivative of $e$  at $(s,t,0)$
sends $T_0D=\C$ to the contact plane $\xi_{\alpha(t)}$.
\item
$e$ commutes with translations of the $\R$ coordinate.
\item
The restriction of $e$ to each disc $\{s\}\times\{t\}\times D$ is
$J$-holomorphic.
\end{itemize}
\end{lemma}

\begin{proof}
This can be proved similarly to \cite[Lem.\ 5.4]{swgr}.
\end{proof}

We denote the coordinates on $\R\times S^1\times D$ by
$(s,t,w)$, and write $z\eqdef s+it$ and  $w\fedqe x+iy$.

Recall that a ``positive end of $u$ at $\alpha$'' is an end of $u$
whose constant $s$ slices converge as $s\to +\infty$ to $\alpha$.  By
positivity of intersections with the $J$-holomorphic discs
$e(\{s\}\times\{t\}\times D)$, such an end pulls back via $e$ to the
graph of a smooth map
\[
\eta: [s_0,\infty)\times S^1 \to D
\]
with $\lim_{s\to\infty}\eta(s,t)=0$.

We now write down an equation for the end described by $\eta$ to be
 $J$-holomorphic.  The conditions in Lemma~\ref{lem:EM} imply that in
 the image of $e$,
\begin{equation}
\label{eqn:t10ry}
T^{1,0}(\R\times Y) = \op{span}(dz - a\,d\zbar,dw + b\,d\zbar),
\end{equation}
where $a$ and $b$ are smooth functions of $t$ and $w$ which vanish
where $w=0$.  It follows from \eqref{eqn:t10ry} that the graph of
$\eta$ is $J$-holomorphic if and only if
\begin{equation}
\label{eqn:hol1}
\frac{\partial\eta}{\partial\zbar} +
a\frac{\partial\eta}{\partial z} + b = 0.
\end{equation}
To see this, note that the tangent space to the graph of $\eta$ is the
kernel of the $\C$-valued $1$-form $dw-d\eta$.  The latter can be
written as a linear combination of the forms on the right side of
\eqref{eqn:t10ry}, plus $d\zbar$ times the left side of
\eqref{eqn:hol1}.  Thus \eqref{eqn:hol1} holds if and only if the
tangent space to the graph of $\eta$ is $J$-invariant.

Equation \eqref{eqn:hol1} can be rewritten as
\begin{equation}
\label{eqn:hol2}
\frac{\partial\eta}{\partial\zbar} + \nu \eta + \mu
\overline{\eta} + {r}_0
+ {r}_1\frac{\partial\eta}{\partial z} = 0,
\end{equation}
where $\nu$ and $\mu$ denote the functions of $t$ given by the derivatives of
$b$ at $w=0$ with respect to $w$ and $\overline{w}$ respectively,
while ${r}_0$ and ${r}_1$ are smooth functions of $t$ and $w$.  Since
$a$ and $b$ both vanish where $w=0$, the nonlinear terms in
\eqref{eqn:hol2} are bounded by
\begin{equation}
\label{eqn:rbounds}
|{r}_0(t,w)| \le c|w|^2, \quad\quad |{r}_1(t,w)| \le c|w|,
\end{equation}
where $c$ denotes a constant which does not depend on $t$ or $w$.

The inequalities \eqref{eqn:rbounds} lead to the following elliptic
estimate for solutions to the equation \eqref{eqn:hol2}, which will be
used frequently below.  Given $z\in \R\times S^1$, let $B(z,1)$ denote
the ball of radius $1$ centered at $z$.  In the lemma that follows,
$\nabla$ is used to denote the $\C$-valued $1$-form of first
derivatives along $\R\times S^1$.  Meanwhile, $\nabla^k$ denotes the
associated $\C$-valued tensor of $k^{th}$ derivatives.  Elsewhere in
this paper, $\nabla$ will denote the covariant derivative on the
indicated section of whatever vector bundle is under consideration,
and $\nabla^k$ for $k\ge 1$ the associated tensor of $k^{th}$ order
covariant derivatives.

\begin{lemma}
\label{lem:ER}
Given functions $r_0$ and $r_1$ satisfying \eqref{eqn:rbounds}, there
exists a positive constant $\varepsilon$, and for each nonnegative
integer $k$ a constant $c_k$, such that the following holds.  Let
$\eta$ be a solution to \eqref{eqn:hol2} on $[R_--1,R_++1]\times S^1$
with $|\eta|\le \varepsilon$.  (We allow $R_\pm=\pm\infty$.)  Then for
each $z\in [R_-,R_+]\times S^1$, we have
\begin{equation}
\label{eqn:ER}
\left|\nabla^k\eta(z)\right|^2
  \le c_k^2\int_{B(z,1)}|\eta|^2.
\end{equation}
\end{lemma}

\begin{proof}
Choose $\varepsilon$ small enough that if $|\eta|<\varepsilon$ then
$|r_1|<1/2$, so that the derivative term in \eqref{eqn:hol2} is
elliptic.  The lemma then follows by a standard bootstrapping argument
e.g.\ using \cite[Thm.\ 5.5.3]{morrey}.
\end{proof}

It proves useful to further rewrite equation \eqref{eqn:hol2} in terms
of the asymptotic operator $L\eqdef L_\alpha$ associated to the Reeb
orbit $\alpha$.  Our convention is to define
\[
L:C^\infty(S^1,\alpha^*\xi)\longrightarrow
C^\infty(S^1,\alpha^*\xi)
\]
by
\[
L \eqdef J \nabla^{\mathbf R}_t,
\]
where $\nabla^{\mathbf R}$ denotes the symplectic connection on
$\alpha^*\xi$ defined by the linearized Reeb flow.  (In the
literature, the operator $L$ is often defined with the opposite sign.)
Recall from \S{I.2.2} that since the connection $\nabla^{\mathbf R}$
is symplectic, the operator $L$ is self-adjoint; and since the Reeb
orbit $\alpha$ is assumed nondegenerate, the spectrum of $L$ does not
contain $0$.

If we use the coordinate $w$ to trivialize the bundle $\alpha^*\xi$,
then it follows from the admissibility condition $J\partial_s={\mathbf
R}$ and equation \eqref{eqn:t10ry} that
\begin{equation}
\label{eqn:L}
L\eta = i\partial_t\eta + 2(\nu \eta + \mu \overline{\eta}).
\end{equation}
Hence equation \eqref{eqn:hol2} can be rewritten as
\begin{equation}
\label{eqn:hol3}
\partial_s\eta + L\eta + \frak{r} = 0,
\end{equation}
where $\frak{r}$ is shorthand for $2(r_0 + r_1\partial\eta/\partial
z)$.

The above discussion generalizes to describe a positive end of $u$ at
the $m$-fold cover $\alpha^m$ of $\alpha$, where $m$ is a positive
integer.  Let $\widetilde{S}^1\eqdef
\R/2\pi m\Z$ denote the $m$-fold cover of $S^1$.
A positive end of $u$ at $\alpha^m$ is then described by a function
$\eta:[s_0,\infty)\times \widetilde{S^1}\to D$ satisfying the
modification of equation \eqref{eqn:hol2}, in which the functions
$\nu$, $\mu$, $r_0$, and $r_1$ on $S^1\times D$ are replaced by their
pullbacks to $\widetilde{S^1}\times D$.  The latter equation can also
be written as
\begin{equation}
\label{eqn:Lm}
\partial_s\eta + L_m\eta + \frak{r} = 0,
\end{equation}
where $L_m$ denotes the asymptotic operator associated to $\alpha^m$,
acting on sections of the bundle $(\alpha^m)^*\xi$ over
$\widetilde{S^1}$.  Note that a solution $\eta$ to \eqref{eqn:Lm} also
satisfies the elliptic estimate \eqref{eqn:ER}.

\subsection{Initial decay estimates}

We now derive a decay estimate for solutions to the equation
\eqref{eqn:hol3}.  Let $E_+$ and $E_-$ respectively denote the
smallest positive and largest negative eigenvalues of the asymptotic
operator $L$.  Also, let $\Pi_+$ and $\Pi_-$ respectively denote the
$L^2(S^1;\R^2)$ projections to the direct sums of the positive and
negative eigenspaces of $L$.

\begin{lemma}
\label{lem:IDE}
There exist a positive constant $\varepsilon_0$ and constants $c_k$
for each nonnegative integer $k$ with the following property.  Let
$\eta$ be a solution to equation \eqref{eqn:hol3} defined on
$[R_--1,R_++1]\times S^1$; we allow $R_\pm = \pm\infty$.  Assume that
$|\eta| \le \varepsilon$ where $\varepsilon<\varepsilon_0$.  Then for
$s\in[R_-+2,R_+-2]$, we have
\begin{equation}
\label{eqn:decay6}
\left|\nabla^k\eta\right| \le c_k\varepsilon\left[e^{-E_+(s-R_-)} +
 e^{-E_-(s-R_+)}\right].
\end{equation}
\end{lemma}

\begin{proof}
Below, $c$ denotes a constant which is independent of $\varepsilon$,
$\eta$, and $R_\pm$, but which may change from line to line.

To start, each solution $\eta$ to the nonlinear equation
\eqref{eqn:hol3} satisfies an associated linear equation, depending on
$\eta$, of the form
\begin{equation}
\label{eqn:hol4}
\frac{\partial\eta}{\partial s} + L\eta + \frak{r}_0\cdot\eta +
\frak{r}_1\cdot\nabla\eta=0.
\end{equation}
Here $\frak{r}_0$ and $\frak{r}_1$ are smooth, $\R$-linear bundle maps
with norm bounded by $c\cdot\varepsilon$.  Using \eqref{eqn:L} to
express the derivative $\partial/\partial t$ in terms of $L$, we can
rewrite equation \eqref{eqn:hol4} as
\begin{equation}
\label{eqn:hol5}
\frac{\partial\eta}{\partial s} + L\eta + \frak{b}_0\cdot\eta +
\frak{b}_1\cdot L\eta = 0.
\end{equation}
Here $\frak{b}_0$ and $\frak{b}_1$ are also $\R$-linear bundle maps
with norm bounded by some constant $c$ times $\varepsilon$.

To analyze \eqref{eqn:hol5}, for $s\in[R_-,R_+]$ define $f_\pm(s)$ to
be one half of the square of the $L^2$ norm of $|L|^{1/2}\Pi_\pm\eta$
on the circle $\{s\}\times S^1$.  Assume that $\varepsilon$ is much
smaller than $|E_\pm|$.  Applying the projection $\Pi_+$ to
\eqref{eqn:hol5}, and taking the $L^2(S^1;\R^2)$ inner product with
$L\Pi_+\eta$ at $s$, gives a differential inequality
\[
\frac{\partial}{\partial s}f_+ + \langle L\Pi_+\eta, L\Pi_+\eta\rangle
\le c\varepsilon\big(\langle L\Pi_+\eta,L\Pi_+\eta\rangle +
\langle L\Pi_-\eta,L\Pi_-\eta\rangle\big).
\]
Likewise, applying $\Pi_-$ to \eqref{eqn:hol5} and taking the inner product
with $L\Pi_-\eta$ gives
\[
- \frac{\partial}{\partial s}f_- + \langle L\Pi_-\eta,L\Pi_-\eta\rangle
\le c\varepsilon\big(\langle L\Pi_+\eta,L\Pi_+\eta\rangle +
\langle L\Pi_-\eta,L\Pi_-\eta\rangle\big).
\]
Adding $\varepsilon_*\eqdef c\varepsilon/(1-c\varepsilon)$ times the
second inequality to the first gives
\[
\frac{d}{ds}\left(f_+-\varepsilon_*f_-\right) +
(1-\varepsilon_*)\langle L\Pi_+\eta,L\Pi_+\eta\rangle \le 0.
\]
This last inequality imples that
\[
\frac{d}{ds}\left(f_+-\varepsilon_*f_-\right) +
2(1-\varepsilon_*)E_+f_+ \le 0,
\]
and thus
\begin{equation}
\label{eqn:decay1}
\frac{d}{ds}\left(f_+-\varepsilon_*f_-\right) +
2(1-\varepsilon_*)E_+(f_+-\varepsilon_*f_-) \le 0.
\end{equation}

Now suppose that $R_-$ and $R_+$ are finite.  Integrating
\eqref{eqn:decay1} gives
\[
(f_+-\varepsilon_*f_-)(s) \le
e^{-2\vu_+(s-R_-)}(f_+-\varepsilon_*f_-)(R_-)
\]
for $s\ge R_-$, where $\vu_+\eqdef (1-\varepsilon_*)E_+$.  A similar
sequence of manipulations finds
\[
(f_--\varepsilon_*f_+)(s) \le
e^{-2\nu_-(s-R_+)}(f_- - \varepsilon_*f_+)(R_+)
\]
for $s\le R_+$, where $\vu_-\eqdef (1-\varepsilon_*)E_-$.  The
preceding two equations can be combined to obtain
\begin{equation}
\label{eqn:decay3}
\begin{split}
f_+(s) & \le (1-\varepsilon_*^2)^{-1}\left[f_+(R_-)e^{-2\nu_+(s-R_-)}
+ \varepsilon_* f_-(R_+)e^{-2\nu_-(s-R_+)}\right],\\
f_-(s) & \le (1-\varepsilon_*^2)^{-1}\left[f_-(R_+)e^{-2\nu_-(s-R_+)}
+ \varepsilon_* f_+(R_-)e^{-2\nu_+(s-R_-)}\right]
\end{split}
\end{equation}
for $s\in[R_-,R_+]$.  Since $|\eta|\le\varepsilon$, Lemma~\ref{lem:ER}
implies that if $\varepsilon$ is chosen sufficiently small, then
$|f_+(R_-)|,|f_-(R_+)|<c\varepsilon$.  Hence adding the equations
\eqref{eqn:decay3} shows that if $s\in[R_-,R_+]$, then on
$\{s\}\times S^1$,
\[
\||L|^{1/2}\eta\|_2 \le c\varepsilon\left[e^{-\vu_+(s-R_-)} +
 e^{-\nu_-(s-R_+)}\right].
\]
It follows that the $L^2$ norm of $\eta$ over a ball of radius $1$ in
$[R_-,R_+]\times S^1$ has a bound of the same form.
Lemma~\ref{lem:ER} then gives a pointwise bound
\begin{equation}
\label{eqn:decay4}
|\nabla^k\eta| \le c_k\varepsilon\left[e^{-\vu_+(s-R_-)} +
 e^{-\nu_-(s-R_+)}\right]
\end{equation}
for $s\in[R_-+1,R_+-1]$.

We now refine the estimate \eqref{eqn:decay4} by feeding it back into
the previous calculation. To do so, recall from \eqref{eqn:rbounds} that
the terms $\frak{r}_0$ and $\frak{r}_1$ that appear in
\eqref{eqn:hol4} are bounded by $c|\eta|$. Using the bound
\eqref{eqn:decay4} on the latter, we can repeat the calculations that
led to
\eqref{eqn:decay3}, replacing the constant $\varepsilon$ by the
function
\[
\widehat{\varepsilon}(s)
\eqdef c_0\varepsilon\left[e^{-\vu_+(s-R_-)} +
e^{-\nu_-(s-R_+)}\right].  
\]
If $\varepsilon$ is small, this procedure allows \eqref{eqn:decay3} to
be replaced by
\begin{equation}
\label{eqn:decay5}
\begin{split}
f_+(s) & \le (1+c\varepsilon)\left[f_+(R_-)e^{-2E_+(s-R_-)}
+ c\varepsilon f_-(R_+)e^{-2E_-(s-R_+)}\right],\\
f_-(s) & \le (1+c\varepsilon)\left[f_-(R_+)e^{-2E_-(s-R_+)}
+ c\varepsilon f_+(R_-)e^{-2E_+(s-R_-)}\right]
\end{split}
\end{equation}
for $s\in[R_-+1,R_+-1]$.  The argument that gave \eqref{eqn:decay4}
now gives the desired estimate \eqref{eqn:decay6} for
$s\in[R_-+2,R_+-2]$.  Taking the limit shows that \eqref{eqn:decay6}
also holds when $R_-=-\infty$ or $R_+=+\infty$.
\end{proof}

\subsection{Asymptotic formula}
\label{sec:AF}

Fix an embedded Reeb orbit $\alpha$ and a positive integer $m$.  We
now prove an asymptotic formula for the behavior of a positive end of
a $J$-holomorphic curve $u$ at $\alpha^m$.  To state the result,
recall that $\widetilde{S^1}\eqdef \R/2\pi m\Z$, and let
$\pi:\widetilde{S^1}\to S^1$ denote the projection.  Also, let $E_m^+$
and $E_m^-$ respectively denote the smallest positive and largest
negative eigenvalues of $L_m$.

\begin{proposition}
\label{prop:AF}
There exist constants $c, \kappa > 0$ such that
the following holds.  Let $\mc{E}$ be a positive end of a
$J$-holomorphic curve $u$ at $\alpha^m$.  Then there is a real
number $s_0$ and a function
$\eta:[s_0,\infty)\times\widetilde{S^1}\to\R^2$ such that:
\begin{description}
\item{(a)}
The end $\mc{E}$ of $u$ is
described by the map
\[
\begin{split}
[s_0,\infty)\times\widetilde{S^1} & \longrightarrow \R\times Y,\\
(s,t) & \longmapsto
e(s,\pi(t),\eta(s,t)).
\end{split}
\]
\item{(b)}
There is a (possibly zero) eigenfunction $\gamma$
of $L_m$ with eigenvalue $E_m^+$ such that
\begin{equation}
\label{eqn:PAF}
\left|\eta(s,t)
 - e^{-E_m^+ s}\gamma(t)\right|
 \le c e^{-(E_m^+ +\kappa)(s-s_0)}.
\end{equation}
\end{description}
\end{proposition}

An analogous result holds for a negative end of $u$ at $\alpha^m$,
with an analogous proof.  Such an end is described by a function
$\eta:(-\infty,s_0]\to\R^2$; and there is a (possibly zero)
eigenfunction $\gamma$ of $L_m$ with eigenvalue $E_m^-$ such that
\begin{equation}
\label{eqn:NAF}
\left|\eta(s,t)
 - e^{-E_m^- s}\gamma(t)\right|
 \le c e^{-(E_m^- -\kappa)(s+s_0)}.
\end{equation}

\begin{definition}
\label{def:AE}
If $\mc{E}$ is a positive or negative end of $u$, then the {\em
asymptotic eigenfunction\/} of the end $\mc{E}$ is the eigenfunction
$\gamma$ of $L_m$ with eigenvalue $E_m^+$ or $E_m^-$ appearing in
\eqref{eqn:PAF} or \eqref{eqn:NAF} respectively.  Note that the
estimates \eqref{eqn:PAF} and \eqref{eqn:NAF} imply that $\gamma$ is unique.
\end{definition}

We now prove Proposition~\ref{prop:AF}.  We already know part (a) from
\S\ref{sec:NRIC}.  To prove part (b), by translating the $s$
coordinate we can arrange that $|\eta|\le \varepsilon$ on
$[-2,\infty)\times
\widetilde{S^1}$.
Moreover, the analysis does not depend in any essential way on $m$,
and so we may assume that $m=1$. The following lemma then implies part
(b), with $s_0=0$.

\begin{lemma}
\label{lem:AF}
There exist constants $c,\kappa,\varepsilon_0>0$ such that the
following holds.
Let $\eta$ be a solution to
\eqref{eqn:hol3} on $[-2,\infty)\times {S}^1$ with
$|\eta|\le\varepsilon$ where $\varepsilon<\varepsilon_0$.  Then there
is a (possibly zero) eigenfunction $\gamma$ of $L$ with eigenvalue
$E_+$ such that for $s\ge 0$,
\begin{equation}
\label{eqn:AF}
\left|\eta(s,t)
 - e^{-E_+ s}\gamma(t)\right|
 \le c\varepsilon e^{-(E_+ +\kappa)s}.
\end{equation}
\end{lemma}

\begin{proof}
Suppose first that the eigenspace of $L$ with eigenvalue $E_+$ is
one-dimensional.  Let $\gamma_+$ be a corresponding normalized
eigenfunction.  Let $\Pi_{1+}$ denote the projection to the sum of the
eigenspaces with eigenvalue greater than $E_+$.  We can then decompose
$\eta$ as
\begin{equation}
\label{eqn:eta3}
\eta = \eta_- + a_+(s)\gamma_+ + \eta_{1+},
\end{equation}
where $\eta_-\eqdef\Pi_-\eta$ and $\eta_{1+}\eqdef \Pi_{1+}\eta$.  We
now individually analyze the terms on the right hand side of \eqref{eqn:eta3}.

Taking the $L^2(S^1;\R^2)$ inner product of equation \eqref{eqn:hol3}
with $\gamma_+$ gives the following differential equation for $a_+(s)$:
\[
\frac{d}{ds}a_+ + E_+ a_+ + \langle\gamma_+,\frak{r}\rangle = 0.
\]
This equation can be integrated to give
\begin{equation}
\label{eqn:asym1}
a_+(s) = \widehat{a}e^{-E_+s} + \int_s^\infty
e^{-E_+(s-\tau)}\langle\gamma_+,\frak{r}\rangle|_\tau d\tau,
\end{equation}
where $\widehat{a}$ is the constant
\begin{equation}
\label{eqn:ahat}
\widehat{a} \eqdef a_+(0) - \int_0^\infty
e^{E_+\tau}\langle\gamma_+,\frak{r}\rangle|_\tau d\tau.
\end{equation}
In \eqref{eqn:asym1} and \eqref{eqn:ahat}, the
integral of $e^{E_+\tau}\langle\gamma_+,\frak{r}\rangle|_\tau$ on the
half-line $[0,\infty)$ is absolutely convergent since $\frak{r}$ is
quadratic in $\eta$.  Indeed, by Lemma~\ref{lem:IDE},
\begin{equation}
\label{eqn:decay7}
\int_s^\infty e^{E_+\tau}|\langle\gamma_+,\frak{r}\rangle|_\tau|d\tau
\le c\varepsilon^2 e^{-E_+s}.
\end{equation}
Here and below, $c$ denotes a constant which does not depend on
$\eta$, but which may change from one appearance to the next.
Combining \eqref{eqn:decay7} with \eqref{eqn:asym1} gives
\begin{equation}
\label{eqn:asym2}
\left|a_+(s) - \widehat{a}e^{-E_+s}\right| \le ce^{-2E_+s}.
\end{equation}

We now bound the size of $\eta_-$.  Let $f_-(s)$ denote the $L^2$ norm
of $\eta_-$ on the circle $\{s\}\times S^1$.  Taking the
$L^2(S^1;\R^2)$ inner product of equation \eqref{eqn:hol3} with
$\eta_-$ shows that
\begin{equation}
\label{eqn:decay8}
\frac{d}{ds}f_- - |E_-|f_- \ge -\|\Pi_-\frak{r}\|_2,
\end{equation}
where $\|\cdot\|_2$ denotes the $L^2$ norm on the circle $\{s\}\times
S^1$.  Integrating this last equation and estimating as 
in \eqref{eqn:decay7} gives
\begin{equation}
\label{eqn:decay9}
f_-(s) \le e^{|E_-|s}\int_s^\infty
e^{-|E_-|\tau}\|\Pi_-\frak{r}\|_2|_\tau d\tau
\le
c \varepsilon^2 e^{-2E_+s}.
\end{equation}
By a standard Sobolev inequality, to obtain a pointwise bound on
$\eta_-$, it is enough to bound the $L^2_1$ norm of $\eta_-$, for
which purpose it suffices to bound the $L^2$ norm over the circle of
$L\eta_-$.  We use $g_-(s)$ to denote the latter function of $s$.  To
obtain a bound on $g_-(s)$, apply $L$ to equation \eqref{eqn:hol3} and
take the $L^2$ inner product with $L\eta_-$ to obtain
\[
\frac{d}{ds}g_- - |E_-|g_- \ge -\|\Pi_-L\frak{r}\|_2.
\]
It follows as in \eqref{eqn:decay9} that $g_-(s)\le
c\varepsilon^2 e^{-2E_+s}$.  Consequently, we also have
\[
|\eta_-(s)| \le
c\varepsilon^2 e^{-2E_+s}.
\]

To bound $|\eta_{1+}|$, first introduce $f_{1+}(s)$ to denote the
$L^2$ norm of $\eta_{1+}$ on $\{s\}\times S^1$.  Steps that are
analogous to those leading to \eqref{eqn:decay8} find that $f_{1+}(s)$
obeys the inequality
\begin{equation}
\label{eqn:decay10}
\frac{d}{ds}f_{1+} +E_{1+}f_{1+} \le \|\Pi_{1+}\frak{r}\|_2,
\end{equation}
where $E_{1+}$ denotes the second smallest positive eigenvalue of $L$.
This last equation integrates to give
\[
f_{1+}(s) \le e^{-E_{1+}s}\left[f_+(0) + \int_0^s
e^{E_{1+}\tau}\|\Pi_{1+}\frak{r}\|_2|_\tau d\tau\right].
\]
Bounding the integral here using the estimate
$\|\Pi_{1+}\frak{r}\|_2\le c\varepsilon^2 e^{-2E_+s}$, we obtain
\begin{equation}
\label{eqn:decay11}
f_{1+}(s) \le c \varepsilon\left[(1+\varepsilon)e^{-E_{1+}s} +
\varepsilon e^{-2E_+s}\right].
\end{equation}
Meanwhile, the $L^2$ norm on $\{s\}\times S^1$ of $L\eta_{1+}$ obeys a
differential inequality which is identical to \eqref{eqn:decay10} but
for the replacement of $\frak{r}$ with $L\frak{r}$.  Hence this $L^2$
norm is bounded by a constant multiple of
the right hand side of \eqref{eqn:decay11}.  It follows that the same
holds for $|\eta_{1+}|$.

Putting together the above analysis of the terms in \eqref{eqn:eta3},
we conclude that if $\kappa\eqdef\min\{2E_+,E_{1+}\}-E_+$, then
\[
|\eta - \widehat{a} e^{-E_+s}\gamma_+| \le
 c\varepsilon e^{-(E_++\kappa)s}.
\]
This proves \eqref{eqn:AF} when the $E_+$ eigenspace of $L$ has
dimension $1$.

In the general case where the $E_+$ eigenspace is possibly degenerate,
let $\Pi_1$ denote the projection onto the $E_+$ eigenspace.  Then the
above argument proves
\eqref{eqn:AF} with
\begin{equation}
\label{eqn:gammahat}
\gamma = 
\Pi_1\eta|_{s=0} - \int_0^\infty e^{E_+\tau}\Pi_1\frak{r}|_\tau d\tau.
\qedhere
\end{equation}
\end{proof}

The integral formula \eqref{eqn:ahat} above for the asymptotic
eigenfunction will play an important role in \S\ref{sec:generic}.

%% file: obg-generic.tex
We have seen in Proposition~\ref{prop:AF} that the asymptotic behavior
of an end of a $J$-holomorphic curve at a Reeb orbit $\alpha$ is
controlled by an ``asymptotic eigenfunction'' $\gamma$ of the
asymptotic operator $L_\alpha$.  In this section we show that if the
admissible almost complex structure $J$ on $\R\times Y$ is generic,
and if $u$ is an index $1$, connected, non-multiply-covered
$J$-holomorphic curve, then the asymptotic eigenfunctions associated
to the ends of $u$ are all nonzero; and moreover, whenever two ends of
$u$ at covers of the same Reeb orbit have the same ``asymptotic
eigenvalue'', the corresponding asympotic eigenfunctions are
geometrically distinct.  Below, the statement that ``generic''
admissible almost complex structures have a given property means that
the space of admissible almost complex structures (with the $C^\infty$
Frechet space topology) contains a Baire set whose elements have the
desired property.

\subsection{Nondegenerate ends for generic $J$}
\label{sec:G1}

\begin{definition}
An end of a $J$-holomorphic curve at a Reeb orbit $\alpha$ is {\em
degenerate\/} if the corresponding asymptotic eigenfunction $\gamma$
of the asymptotic operator $L_{\alpha}$ is zero.
\end{definition}

\begin{proposition}
\label{prop:G1}
If the admissible almost complex structure $J$ on $\R\times Y$ is
generic, then no index $1$, connected, non-multiply-covered
$J$-holomorphic curve has a degenerate end.
\end{proposition}

\begin{remark}
A nice way to prove this, from \cite{wendl}, is to generalize the fact
that moduli spaces of non-multiply-covered $J$-holomorphic curves are
smooth and have dimension equal to the index of the deformation
operator for generic $J$, to consider $J$-holomorphic curves with
asymptotic exponential weight constraints on the ends.  If one changes
the weight associated to an end so that it crosses an eigenvalue of
the corresponding asymptotic operator, then the index of the
deformation operator changes.  We will use a different approach here,
in order to set up the proof of Proposition~\ref{prop:G2} below.
\end{remark}

\begin{proof}[Proof of Proposition~\ref{prop:G1}.]
  Let $\mc{J}$ denote the Frechet space of admissible almost complex
  structures.  For each positive integer $n$, let $\mc{C}_n$ denote
  the space of pairs $(J,C)$ such that $J\in\mc{J}$ and $C$ is a
  $J$-holomorphic, connected, non-multiply covered index $1$ curve
  with the following four properties: First, $C$ has at most $n$ ends,
  each of which is at a (possibly multiply covered) Reeb orbit of
  symplectic action at most $n$.  Second, $C$ is ``not close to
  breaking'' in the sense that if $C'\subset C$ is a connected subset
  with area greater than or equal to $n$ and distance less than or
  equal to $1/n$ from an $\R$-invariant cylinder, then $C'$ lies in an
  end of $C$.  Third, $C$ is ``not close to a multiple cover'' in the
  sense that if $C'$ is another $J$-holomorphic curve such that every
  point in $C'$ has distance less than $1/n$ from a point in $C$ and
  vice-versa, then the energy of $C'$ is at least $2/3$ that of $C$.
  Fourth, $C$ is ``not close to a nodal curve'' in the sense that $C$
  does not contain a simple closed curve of length less than $1/n$
  that separates $C$ into two noncompact components.  Note that if $C$
  is any index $1$, connected, non-multiply covered $J$-holomorphic
  curve, then $(J,C)\in\mc{C}_n$ for $n$ sufficiently large.  Now let
  $\frak{p}_n:\mc{C}_n \to \mc{J}$ denote the projection.  Then
  standard Gromov compactness arguments prove that
  $\frak{p}_n^{-1}(J)/\R$ is compact for each $J\in\mc{J}$.

Next, let $\mc{J}_n\subset\mc{J}$ denote the set of $J\in\mc{J}$ such
that no curve in $\frak{p}_n^{-1}(V)$ is obstructed.  By the
aforementioned compactness, if $J\in\mc{J}_n$ then
$\frak{p}_n^{-1}(J)/\R$ is finite.  A straightforward limit argument
then proves that $\mc{J}\setminus\mc{J}_n$ is closed, so
$\mc{J}_n$ is open in $\mc{J}$.  By well-known arguments, cf.\
\cite[Ch. 3]{mcds}, $\mc{J}_n$ is
also dense in $\mc{J}$.

Now let $\widehat{\mc{J}}_n\subset\mc{J}_n$ denote the set of $J\in\mc{J}_n$
such that no curve in $\frak{p}_n^{-1}(J)$ has a degenerate end.
Since the asymptotic eigenfunctions depend continuously on $C$, it
follows that $\widehat{\mc{J}}_n$ is also open in $\mc{J}$.  Finally, we will
prove:

\begin{lemma}
\label{lem:dense}
$\widehat{\mc{J}}_n$ is dense in $\mc{J}_n$.
\end{lemma}

Granted this lemma, we conclude that $\widehat{\mc{J}}_\infty \eqdef
\bigcap_n\widehat{\mc{J}}_n$ is a Baire subset of $\mc{J}$.  This proves
Proposition~\ref{prop:G1}, because by definition, every
$J\in\widehat{\mc{J}}_\infty$ obeys the condition stated in
Proposition~\ref{prop:G1}.
\end{proof}

\begin{proof}[Proof of Lemma~\ref{lem:dense}.]
Given $l>>0$, let $\mc{J}_n^l$ and $\widehat{\mc{J}}_n^l$ denote the
analogues of $\mc{J}_n$ and $\widehat{\mc{J}}_n$ defined using $C^l$
almost complex structures.  Define $\frak{p}_n: \mc{C}_n^l\to\mc{J}^l$
as above, using class $C^{l}$ pseudoholomorphic curves.  It is
enough to show that $\widehat{\mc{J}}_n^l$ is dense in $\mc{J}_n^l$
for all $l>>0$, cf.\ \cite[\S3]{mcds}.

Fix $l>>0$ and $J\in\mc{J}_n^l$.  Let ${\mathbf O}$ denote the set of
embedded Reeb orbits $\alpha$ for which a curve
$C\in\frak{p}_n^{-1}(J)$ has an end at a cover of $\alpha$.  Note that
the set ${\mathbf O}$ is finite, because $\frak{p}_n^{-1}(J)/\R$ is finite.
Fix a small $\delta>0$, and let $\mc{U}\subset\mc{J}_n^l$ be a small,
contractible neighborhood of $J$ in the space of almost complex
structures $J'\in\mc{J}_n^l$ that agree with $J$ within distance
$\delta$ of $\R\times\alpha$ for each $\alpha\in{\mathbf O}$.

Now fix $C\in\frak{p}_n^{-1}(J)$, and fix a positive end $\mc{E}$ of
$C$ at a Reeb orbit $\alpha$.  In what follows, we shall assume that
$\alpha$ is embedded; the argument for the general case differs only
in the notation.  Let $\mc{B}_{\mc{E}}$ denote the $E_+$ eigenspace of
$L$. Let $\mc{C}$ denote the universal moduli space consisting of
pairs $(J',C')$ such that $J'\in\mc{U}$ and $C'$ is a $J'$-holomorphic
curve that is a deformation of $C$.  There is an obvious projection
$\pi:\mc{C}\to\mc{U}$.  Since $\mc{U}\subset\mc{J}_n^l$, it follows that
$\pi^{-1}(J')$ is $1$-dimensional for each $J'$, and consists of the
$\R$-translates of a single unobstructed curve.  For each pair
$(J',C')\in\mc{C}$, the end $\mc{E}$ of $C$ determines an end of $C'$
at $\alpha$, which we also denote by $\mc{E}$.  Because $J=J'$ along
$\R\times\alpha$, the asymptotic operator $L\eqdef L_{\alpha}$ is the
same for $J$ and $J'$, so the end $\mc{E}\subset C'$ determines an
asymptotic eigenfunction $\gamma(C')\in \mc{B}_{\mc{E}}$.

Fix a smooth section $\psi:\mc{U}\to\mc{C}$ with $\psi(J)=(C,J)$.  We
then have a function $\gamma\circ\psi: \mc{U}\to \mc{B}_{\mc{E}}$; let
$\mc{Z}\subset\mc{U}$ denote the zero locus of $\gamma\circ\psi$.  The
zero locus $\mc{Z}$ does not depend on the choice of $\psi$, because
for any $(J',C')\in\mc{U}$, translating $C'$ upward by $R$ multiplies
$\gamma(C')$ by $e^{E_+R}$.

Below, we will prove:

\begin{lemma}
\label{lem:Zsub}
If $\delta>0$ above is sufficiently small, then $\mc{Z}$ is a
submanifold of $\mc{U}$, with $\op{codim}(\mc{Z})=\dim(\mc{B}_{\mc{E}})$.
\end{lemma}

An analogous statement also holds if $\mc{E}$ is a negative end of
$C$, with an analogous proof.

Granted Lemma~\ref{lem:Zsub} and its negative end version, we now
complete the proof of Lemma~\ref{lem:dense}.  Since there are finitely
many curves in $\frak{p}_n^{-1}(J)/\R$ and each has finitely many
ends, if $\delta>0$ is sufficiently small then we can apply
Lemma~\ref{lem:Zsub} and its negative end version a finite number of
times to obtain a finite set of codimension $1$ or $2$ submanifolds in
$\mc{U}$ whose complement consists of almost complex structures in
$\widehat{\mc{J}}_n^l$.  As $J$ is in $\mc{U}$, this proves that there
are points in $\widehat{\mc{J}}_n^l$ that lie in any given neighborhood of $J$.
Thus, $\widehat{\mc{J}}_n^l$ is dense in $\mc{J}_n^l$.
\end{proof}

We now prepare for the proof of Lemma~\ref{lem:Zsub}.  It proves
convenient to fix a normalized eigenfunction $\gamma_+\in
\mc{B}_{\mc{E}}$, and write $\gamma\circ\psi =
\widehat{a}_{\mc{E}}\gamma_+$, where $\widehat{a}_{\mc{E}}$ is a
function on $\mc{U}$ with values in $\R$ or $\C$, depending on whether
$\mc{B}_{\mc{E}}$ is $1$- or $2$-dimensional.

\begin{lemma}
\label{lem:asmooth}
The function $\widehat{a}_{\mc{E}}$ on $\mc{U}$ is smooth.
\end{lemma}

\begin{proof}
To simplify notation, assume that $\dim(\mc{B}_{\mc{E}})=1$, so that
$\widehat{a}_{\mc{E}}:\mc{U}\to\ \R$.  Introduce $S\eqdef
[0,\infty)\times S^1 \subset \R\times Y$, viewed as part of the
cylinder $\R\times\alpha$.  We also need weighted versions of the
Sobolev spaces $L^2_{k=0,1,2}(S;\R^2)$.  The norm for the weighted
version of $L^2_k$ assigns to a smooth, compactly supported function
$\eta$ on $S$ the square root of
\[
\int_S e^{E_+s} \sum_{0\le j\le k} \left|\nabla^j\eta\right|^2 \,
ds\,dt.
\]
We denote the corresponding weighted Sobolev space by $L^2_{k+}$.  Now
consider the equation \eqref{eqn:hol3} determined by the almost
complex structure $J$ near $\alpha$.

\begin{lemma}
\label{lem:CC}
The space $\mc{M}$ of $L^2_2$ solutions to \eqref{eqn:hol3} on
$S$ with small $L^2_2$ norm is a smooth
manifold.  Moreover, each element in $\mc{M}$ is in $L^2_{2+}$, and
this inclusion defines a smooth map from $\mc{M}$ into $L^2_{2+}$.  Finally,
 there exists a ball ${B}\subset \Pi_+
L^2_{3/2}(S^1;\R^2)$ about the origin, and a smooth embedding
$\frak{t}:{B}\to\mc{M}$ as a coordinate chart about $\eta=0$, such
that $\Pi_+\frak{t}(\lambda)|_{s=0}=\lambda$ for all
$\lambda\in{B}$.
\end{lemma}

Granted this lemma, we now complete the proof of
Lemma~\ref{lem:asmooth}.  Equation \eqref{eqn:ahat} associates to each
$\eta\in\mc{M}$ the number $\widehat{a}$, and so defines a function on
$\mc{M}$.  This function is smooth on $\frak{t}({B})$.  Indeed, the
term $a_+(0)$ that appears in \eqref{eqn:ahat} is a bounded linear
function on ${B}$.  Meanwhile, the integral term in \eqref{eqn:ahat}
is the pullback to $\mc{M}$ of a smooth function on $L^2_{2+}$, and so
Lemma~\ref{lem:CC} guarantees that it too defines a smooth function on
$\frak{t}({B})$.  With $\widehat{a}$ understood, let $J'\in\mc{U}$.
Since any complex structure on $\mc{U}$ agrees with $J$ within
distance $\delta$ of $\alpha$, it follows that there exists $s_0>0$,
and a neighborhood $\mc{U}'$ of $J'$ in $\mc{U}$, with the following
property: If $J''\in\mc{U}'$, then the end $\mc{E}$ of the $s\to
s-s_0$ translate of $\psi(J'')$ restricts to $S$ as an element of the
space $\frak{t}({B})$.  Denote this element of $\frak{t}({B})$ by
$\psi(J'')_0$.  Equation \eqref{eqn:asym2} implies that
\begin{equation}
\label{eqn:ahat+}
\widehat{a}_{\mc{E}}(J'') = e^{E_+ s_0} \cdot \widehat{a}(\psi(J'')_0).
\end{equation}
Thus the function $\widehat{a}_{\mc{E}}$ on $\mc{U}'$ is a constant times the
pullback of $\widehat{a}$ via a smooth map from $\mc{U}'$ to
$\frak{t}({B})$.
\end{proof}

\begin{proof}[Proof of Lemma~\ref{lem:CC}.]
Let $B_0$ be a small radius ball around the origin in
$L^2_{2+}(S;\R^2)$.  Define a smooth map
\[
F: B_0 \longrightarrow \Pi_+ L^2_{3/2}(S^1;\R^2) \times L^2_{1+}(S;\R^2)
\]
by the rule
\[
\eta \longmapsto \left(\Pi_+\eta|_{s=0}, \partial_s\eta + L\eta
+ \frak{r}\right),
\]
where $\frak{r}\eqdef 2(r_0+r_1\partial\eta/\partial z)$ as before.
Note that the map $F$ is well defined if $B_0$ has sufficiently small
radius, so that elements of $B_0$ have sufficiently small pointwise
norm.

We claim that the differential of $F$ at the origin in $B_0$ is an
isomorphism
\[
dF|_{0}:L^2_{2+}(S;\R^2)\stackrel{\simeq}{\longrightarrow} \Pi_+
L^2_{3/2}(S^1;\R^2) \times L^2_{1+}(S;\R^2).
\]
This claim implies the lemma by the implicit
function theorem.

To prove the claim, note that
\[
dF|_{0}(\eta) = \left(\Pi_+\eta|_{s=0}, \partial_s\eta +
L\eta\right).
\]
To show that $dF|_{0}$ is an isomorphism, we first need to show that
given $f\in \Pi_+ L^2_{3/2}(S^1;\R^2)$ and $g\in L^2_{1+}(S;\R^2)$,
there exists a unique $\eta\in L^2_{2+}(S;\R^2)$ with
$\Pi_+\eta|_{s=0}=f$ and $\partial_s\eta + L\eta = g$.  To prove the
latter statement, note that any solution to these equations can be
written as $\eta =
\sum_{\gamma}a_\gamma(s)\gamma$, where the sum is over an orthonormal
basis of eigenfunctions for $L$.  Similarly write $g=\sum_\gamma
g_\gamma(s)\gamma$ and $f=\sum_\gamma f_\gamma \gamma$.  Then the
$a_\gamma$'s must be given as follows:  If $\gamma$ is an
eigenfunction with eigenvalue $E$, then
\begin{equation}
\label{eqn:agamma}
a_\gamma(s)= 
\left\{
\begin{array}{cl}
f_\gamma e^{-Es} +
\int_0^s e^{-E(s-\tau)}g_\gamma(\tau)d\tau, & E>0, \\
- \int_s^\infty e^{|E|(s-\tau)}g_\gamma(\tau)d\tau, & E<0.
\end{array}
\right.
\end{equation}

We must now verify that the function $\eta$ defined by
\eqref{eqn:agamma} satisfies
\begin{equation}
\label{eqn:needc}
\|\eta\|_{L^2_{2+}} \le c\left(\|f\|_{L^2_{3/2}} +
\|g\|_{L^2_{1+}}\right)
\end{equation}
for some $\eta$-independent constant $c$.
For this purpose, observe that the norm on
$L^2_{k+}(S;\R^2)$ is equivalent to the norm defined by
\[
\|\eta\|_{L^2_{k+}}^2 \eqdef \int_S e^{E_+s} \sum_{i+j=
k}\left|\partial_s^iL^j\eta\right|^2
\,ds\,dt.
\]
Also, the norm on $L^2_{3/2}(S^1;\R^2)$ can be defined by
\[
\|f\|_{L^2_{3/2}}^2 \eqdef \int_{S^1} \langle
f,|L|^3f\rangle \, dt.
\]
Granted this, one need only establish
\eqref{eqn:needc} when
$g=g_\gamma(s)\gamma$ and $f=f_\gamma \gamma$, where $\gamma$ is an
eigenfunction of $L$, and $c$ does not depend on $\gamma$.  This is
straightforward using the preceding equations in the case where $E\neq
E_+$.  In the case $E=E_+$, it is also necessary to use the fact that
\[
\int_0^\infty|\partial_s a|^2 e^{E_+s}\, ds \ge
\frac{E_+^2}{4}\int_0^\infty |a|^2 e^{E_+s}\, ds
\]
for any given function $a$ of $s$ that has limit zero as $s\to
\infty$.  This last inequality is proved by writing $E_+ e^{E_+s} =
\frac{d}{ds} e^{E_+s}$ in the right hand integral above, and
integrating by parts.
\end{proof}

\begin{proof}[Proof of Lemma~\ref{lem:Zsub}.]  By
Lemma~\ref{lem:asmooth} and the implicit function theorem, it is
enough to assume that $\widehat{a}_{\mc{E}}(J)=0$, and construct $j\in
T\mc{U}|_J$ with $\nabla_j\widehat{a}_{\mc{E}}\neq 0$.  We will
proceed in five steps.

Before starting, by translating $s$ we can assume that the end
$\mc{E}$ of $C$ is described by a map $\eta$ defined on
$[-2,\infty)\times S^1$ such that for each $s_0\ge 0$, the $s\to
s-s_0$ translate of $\eta$ restricts to $S$ as an element of
$\frak{t}({B})$.  Also, we will assume that $C$ is immersed; the
general case can be handled by introducing more notation, or by
appealing to the results of \S\ref{sec:immersed} below.

\medskip
{\em Step 1.\/} The differential of the section $\psi:\mc{U}\to\mc{C}$
at $J$ defines a linear map $d\psi_J$, from the tangent space
$T\mc{U}|_J$, to the space of sections of the normal bundle $N_C\to
C$.  In this step we derive a useful formula for $d\psi_J$.

If $C'$ is any immersed surface in $\R\times Y$, then $C'$ is
$J$-holomorphic if and only if $\dbar_J(C')=0$, where $\dbar_J(C') :
TC'\to N_{C'}$ is the bundle map defined by
\[
\dbar_J(C') \eqdef \Pi_{N_{C'}}\circ J.
\]
The linearization of $\dbar_J$ at our $J$-holomorphic curve $C$
defines a real linear operator, mapping sections of the normal bundle
$N_C$ to sections of the complex line bundle
\begin{equation}
\label{eqn:TCN}
\Hom^{0,1}(TC,N_C)
= T^{0,1}C\tensor_\C N_C.
\end{equation}
Define $D_C$ to be $-i/2$ times this operator.  On the $s\ge 0$ part
of the end $\mc{E}$ of $C$, the operator $D_C$ has the following form:
If we identify the normal bundle $N_{C}$ with $\C$ via the coordinate
$w$, and if we trivialize $T^{0,1}C$ using $d\overline{z}$, then
$D_{C}$ sends a function $v\in C^\infty(S,\C)$ to
\begin{equation}
\label{eqn:DCv}
D_{C}v = \frac{\partial v}{\partial\zbar} + \nu v + \mu\overline{v} +
\frak{r}_{0*}\cdot v +
\frak{r}_{1*}\frac{\partial v}{\partial z}.
\end{equation}
Here $\frak{r}_{0*}$ is an $\R$-linear bundle map, and $\frak{r}_{1*}$ is
a complex-valued function on $S$, satisfying
\[
|\frak{r}_{0*}| + |\frak{r}_{1*}| \le c|\eta|.
\]

Since $C$ has index $1$ and is a smooth point in its moduli space, the
operator
\begin{equation}
\label{eqn:DCO}
D_{C}:L^2_1(C,N_{C}) \longrightarrow L^2(C,T^{0,1}C\tensor_\C N_{C})
\end{equation}
is surjective and has a $1$-dimensional kernel.  Let $D_C^{-1}$
denote the unique right inverse of \eqref{eqn:DCO} that maps to the orthogonal
complement of $\Ker(D_C)$.

The differential $d\psi_J$ can now be described in terms of the
operator $D_C^{-1}$ as follows.  Let $j$ be a tangent vector at $J$ to
$\mc{U}$.  Let $j_C\in\Hom^{0,1}(TC,N_C)$ denote the $(0,1)$ bundle
map given by $i/2$ times the composition
\begin{equation}
\label{eqn:jC}
TC \stackrel{j}{\longrightarrow} T(\R\times Y)|_C
\stackrel{\Pi_{N_C}}{\longrightarrow} N_C.
\end{equation}
Let $\{J_\tau\}$ be a smooth family of almost complex structures
parametrized by a neighborhood of $0$ in $\R$ with $J_0=J$ and
$\frac{d}{d\tau}|_{\tau=0}J_\tau=j$.  Write
$\psi(J_\tau)=(J_\tau,C_\tau)$.  Then differentiating the equation
$\dbar_{J_\tau}(C_\tau)=0$ at $\tau=0$ shows that
\[
(-2i) j_C + (2i) D_C(d\psi_J(j)) = 0.
\]
Thus, the section $d\psi_J(j)$ of $C$'s normal bundle is given by
\begin{equation}
\label{eqn:dpsi}
d\psi_J(j) = D_C^{-1}\left(j_C\right) + w_j,
\end{equation}
where $w_j\in\Ker(D_C)$ depends linearly on $j$.

{\em Step 2.\/} We now choose a tangent vector $j\in T\mc{J}|_J$.  Our $j$
will vanish identically on some neighborhood of $\R\times\alpha'$ for
all Reeb orbits $\alpha'$ in ${\mathbf O}$, so that $j\in T\mc{U}|_J$ if
$\delta$ is chosen sufficiently small.  (See the beginning of the
proof of Lemma~\ref{lem:dense}.) To finish the proof of
Lemma~\ref{lem:Zsub}, we will later show that
$\nabla_j\widehat{a}_{\mc{E}}\neq 0$.

To prepare for the choice of $j$, we need to consider the locus where
the projection of $C$ to $Y$ is not an embedding.  First, let
$\Lambda\subset C$ denote the set of points where $C$ intersects the
$s\mapsto s-s_0$ translate of $C$ for some $s_0\neq 0$ (including
$s_0=\pm\infty$, i.e.\ points where the projection of $C$ to $Y$
intersects one of the Reeb orbits at the ends of $C$).  For any given
$s_0$, the set of such intersections is discrete (in fact finite, see
\cite{siefring}).  Thus $\Lambda$ is a closed codimension $1$
subvariety of $C$.  Next, let $\mc{T}\subset C$ denote the set of
points where $C$ is tangent to $\xi$ or to
$\op{span}(\partial_s,{\mathbf R})$.  The set $\mc{T}$ is also
discrete (in fact finite).

Next, recall that an admissible almost complex structure $J$ is
required to be $\R$-invariant and to send $\partial_s\mapsto {\mathbf
R}$ and $\xi\to\xi$.  Thus a tangent vector $j\in T\mc{J}|_J$ is
equivalent to a $(0,1)$ bundle map $\xi\to\xi$ over $Y$.  As a
consequence, any $(0,1)$ bundle map $f:TC\to N_C$ can be realized as
$j_C$ for some $j\in T\mc{J}|_J$, provided that $f=0$ in a
neighborhood of $\Lambda\cup\mc{T}$.

We now specify such a bundle map $f$.  First of all, $f$ will be
supported in the $s> 1$ part of the end $\mc{E}$ of $C$.  On this end,
under the identifications in \eqref{eqn:DCv}, a $(0,1)$ bundle map
$f:TC\to N_C$ is equivalent to a complex function $g$ on
$[-2,\infty)\times S^1$. To specify $g$, fix $r>1$ large and $\rho>0$
small.  Recall that $\gamma_+$ denotes the chosen normalized
eigenfunction of $L$ with eigenvalue $E_+$.  Let
$\chi:[-2,\infty)\times S^1\to[0,1]$ be a function which vanishes
where $s\notin[r,r+1]$ or the distance to $\Lambda\cup\mc{T}$ is less
than $\rho/2$, and which is $1$ where $s\in[r+\rho,r+1-\rho]$ and the
distance to $\Lambda\cup\mc{T}$ is greater than $\rho$.  Now define
\[
g(s,t) \eqdef \chi(s,t)\gamma_+(t).
\]
This completes the specification of $f$.  Finally, choose $j$ such that
$f=j_C$.

{\em Step 3.\/}
We now calculate $\nabla_j\widehat{a}_{\mc{E}}$. By equation
\eqref{eqn:dpsi},
\begin{equation}
\label{eqn:UAZ}
\nabla_j\widehat{a}_{\mc{E}} = 
(d\widehat{a}_{\mc{E}})_J D_C^{-1}\left(j_C \right).
\end{equation}
To see why the $w_j$ term in \eqref{eqn:dpsi} is irrelevant, observe
that the latter is in $\Ker(D_C)$, and hence it is a multiple of the
tangent vector at $C$ to the $1$-parameter family of pseudoholomorphic
curves given by by translating $C$ along the $\R$ factor of $\R\times
Y$.  Since we are assuming that $\widehat{a}_{\mc{E}}(C)=0$, these
translations have no effect on $\widehat{a}_{\mc{E}}$.

To evaluate the right hand side of \eqref{eqn:UAZ}, note that the
restriction of $D_C^{-1}(j_C)$ to the end $\mc{E}$ appears as a
function $v:[-2,\infty)\times S^1\to\C$.  Now fix $s_0\ge r+1$, and
let $\eta_0$ and $v_0$ denote the restrictions to $S$ of the $s\to
s-s_0$ translates of $\eta$ and $v$ respectively.  Since $j$ is
supported where $r\le s \le r+1$, it follows using equation
\eqref{eqn:dpsi} that $v_0\in T\mc{M}|_{\eta_0}$.  So by
\eqref{eqn:UAZ} and \eqref{eqn:ahat+},
\[
\nabla_j\widehat{a}_{\mc{E}} = e^{E_+s_0} (d\widehat{a})|_{\eta_0} (v_0).
\]
We conclude from this and equation \eqref{eqn:ahat} that for any
$s_0\ge r+1$,
\begin{gather}
\nonumber
\nabla_j\widehat{a}_{\mc{E}} = e^{E_+ s_0} v_+(s_0)
\quad\quad\quad\quad\quad\quad\quad\quad\quad\quad\quad\quad
\quad\quad\quad\quad\quad
\\
\label{eqn:da+}
  \quad\quad
- \int_{s_0}^\infty
e^{E_+{s}} \left\langle\gamma_+,
\frak{r}_0\cdot v + \frak{r}_1\cdot\nabla v +
d\frak{r}_0(v) \cdot \eta +
d\frak{r}_1(v) \cdot\nabla \eta
\right\rangle|_{s} d{s}.
\end{gather}
Here $v_+(s_0)$ denotes the $L^2$ inner
product on $S^1$ between $\gamma_+$ and $v(s_0,\cdot)$, while $\frak{r}_0$
and $\frak{r_1}$ are the $w$-dependent bundle maps from
\eqref{eqn:hol4}.

To prove that the differential \eqref{eqn:da+} is nonzero, we will
show that for suitable $r$ and $s_0$, the $v_+$ term in
\eqref{eqn:da+} is much larger than the integral.

{\em Step 4.\/} 
We now establish an upper bound on the integral in \eqref{eqn:da+}.

By definition, $v$ obeys the equation
\begin{equation}
\label{eqn:vg}
\frac{1}{2}\left(\frac{\partial v}{\partial s} + Lv\right) +
\frak{r}_{0*}\cdot v + \frak{r}_{1*}\cdot\frac{\partial v}{\partial z}
= g.
\end{equation}
Since the operator $D_C$ in \eqref{eqn:DCO} is bounded and Fredholm,
there is a constant $c$ such that
$\|\lambda\|_2^2+\|\nabla\lambda\|_2^2 \le c^2\|D_C\lambda\|_2^2$ for
all $\lambda$ that are in the domain of $D_C$ and orthogonal to the
kernel of $D_C$.  Applying this to $\lambda\eqdef D_C^{-1}(j_C)$, we
deduce that
\begin{equation}
\label{eqn:vbound}
\|v\|_2^2 + \|\nabla v\|_2^2 \le c^2\|g\|_2^2.
\end{equation}

Now observe that \eqref{eqn:vg} is a homogeneous equation for $v$
where $s\ge r+1$.  Elliptic estimates as in Lemma~\ref{lem:ER} then
give pointwise bounds for the derivatives of $v$ on the $s=r+2$ circle
in terms of the $L^2$ norm of $v$, which by \eqref{eqn:vbound} is
bounded by $c\|g\|_2$.  In particular, the $L^2$ norm of $|L|^{1/2}v$
over the $s=r+2$ circle is bounded by an $r$ and $\rho$ independent
constant times the $L^2$ norm of $g$.  The analysis leading to
\eqref{eqn:decay6} then gives
\begin{equation}
\label{eqn:PVB}
|\nabla^k v| \le c_k e^{-E_+(s-r)}\|g\|_2
\end{equation}
where $s\ge r+3$.

Recall from Lemma~\ref{lem:IDE} that $|\eta|$ and $|\nabla\eta|$
are bounded by $c\varepsilon e^{-E_+s}$ for $s\ge 0$.  Combining this
with \eqref{eqn:PVB}, we conclude that if $s_0\ge r+3$, then the
integral in \eqref{eqn:da+} is bounded by
\begin{equation}
\label{eqn:IB}
\left|\int_{s_0}^\infty(\cdots)d{s}\right| \le c\varepsilon
e^{-E_+(s_0-r)}
\|g\|_2.
\end{equation}

{\em Step 5.\/}  We now establish a lower bound on the term
$e^{E_+s_0}v_+(s_0)$ in \eqref{eqn:da+}.

Suppose that $r$ is large.  The analysis leading to \eqref{eqn:decay6}
on the cylinder where $0\le s \le r$ can be employed to prove that
when $s=r/2$,
\begin{equation}
\label{eqn:vkb}
|\nabla^kv|\le c_k e^{-E_+r/2}\|g\|_2.
\end{equation}
Therefore
\begin{equation}
\label{eqn:vrb}
\left|v_+(r/2)\right| \le c e^{-E_+r/2} \|g\|_2.
\end{equation}

Keeping this in mind, define
$g_+(s)\eqdef\left\langle\gamma_+,g\right\rangle\big|_s$.
Take the $L^2(S^1;\R^2)$ inner product of
\eqref{eqn:vg} with $\gamma_+$ to obtain the differential equation
\[
\frac{1}{2}\left(\frac{d v_+}{d s} + E_+v_+\right) +
\left\langle\gamma_+,\frak{r}_{0*}\cdot v +
\frak{r}_{1*}\cdot\partial v\right\rangle\big|_s
= g_+.
\]
Integrating this from $s=r/2$ to $s=s_0$ gives
\begin{gather*}
e^{E_+s_0}v_+(s_0) = e^{E_+r/2} v_+(r/2)
\quad\quad\quad\quad\quad\quad\quad\quad\quad\quad\quad\\
\quad\quad\quad\quad\quad\quad\quad
+ 2\int_{r/2}^{s_0} e^{E_+s}
\left[
g_+ - \left\langle\gamma_+,\frak{r}_{0*}\cdot v +
\frak{r}_{1*}\cdot\partial v\right\rangle\big|_s
\right]ds.
\end{gather*}
Using \eqref{eqn:vrb}, the bounds on $\eta$ and its derivatives by
$c\varepsilon e^{-E_+s}$, and the bounds on $v$ and its
derivatives in \eqref{eqn:PVB}, we deduce that
\begin{equation}
\label{eqn:v+b}
e^{E_+s_0}v_+(s_0) \ge 2\int_{r/2}^{s_0}e^{E_+s}g_+(s)ds -
c(s_0-r/2+1)\|g\|_2.
\end{equation}
By the definition of $g$, we have
\begin{equation}
\label{eqn:BDG}
\begin{split}
\lim_{\rho\to 0}\int_{r/2}^{s_0}e^{E_+s}g_+(s)ds & =
\int_{r}^{r+1}
e^{E_+s}ds,\\
\lim_{\rho\to 0}\|g\|_2 & = 2\pi. 
\end{split}
\end{equation}
Combing
\eqref{eqn:da+}, \eqref{eqn:IB},
\eqref{eqn:v+b}, and \eqref{eqn:BDG}, we conclude that for any $r$, we
can choose $\rho$ sufficiently small that for any $s_0$,
\begin{equation}
\label{eqn:da+b}
\nabla_j \widehat{a}_{\mc{E}}
 \ge c_1 e^{E_+r} -
c_2(s_0-r/2+1) - c_3 e^{-E_+(s_0-r)}.
\end{equation}
If $r$ is sufficiently large and $s_0=2r$, then the first term on the
right hand side of \eqref{eqn:da+b} is much larger than the other two
terms, so $\nabla_j \widehat{a}_{\mc{E}} \neq 0$.  This completes the
proof of Lemma~\ref{lem:Zsub}, and with it Proposition~\ref{prop:G1}.
\end{proof}

\subsection{Nonoverlapping pairs of ends for generic $J$}
\label{sec:G2}

We now prove a genericity result for the asymptotic eigenfunctions
associated to pairs of ends of a $J$-holomorphic curve.  To state the
result, recall that if $\alpha$ is an embedded Reeb orbit, if $m_1$
and $m_2$ are positive integers, and if the smallest positive
eigenvalues of $L_{\alpha^{m_1}}$ and $L_{\alpha^{m_2}}$ agree, then
the corresponding eigenspaces are pulled back from the smallest
positive eigenspace of $L_{\alpha^m}$, where $m$ denotes the greatest
common divisor of $m_1$ and $m_2$.  This is explained in
\S{I.3.1} for elliptic Reeb orbits, and is even easier for
hyperbolic Reeb orbits.  Note also that $\Z/m$ acts on the eigenspace
of $L_{\alpha^{m}}$ with smallest positive eigenvalue, via pullback
from its action on $\alpha^m$ by deck transformations for the covering
map to $\alpha$.

\begin{definition}
Let $C$ be a $J$-holomorphic curve with all ends nondegenerate.  A
pair of positive ends $\mc{E}_1, \mc{E}_2$ of $C$ is {\em
overlapping\/} if the following conditions hold:
\begin{itemize}
\item
There is a single embedded Reeb orbit $\alpha$, and positive integers
$m_1$ and $m_2$, such that $\mc{E}_i$ is a positive end at
$\alpha^{m_i}$.
\item
The smallest positive eigenvalues of $L_{\alpha^{m_1}}$ and
$L_{\alpha^{m_2}}$ agree.
\item
Let $m$ denote the greatest common divisor of $m_1$ and $m_2$, and let
$\gamma_i$ denote the eigenfunction of $L_{\alpha^{m}}$ whose pullback to
$\alpha^{m_i}$ is the asymptotic eigenfunction associated to
the end $\mc{E}_i$.  Then there exists $g\in\Z/m$ such that
$g\cdot\gamma_1=\gamma_2$.
\end{itemize}
\end{definition}

An overlapping pair of negative ends is defined analogously.

\begin{proposition}
\label{prop:G2}
If the admissible almost complex structure $J$ on $\R\times Y$ is
generic, then no index $1$, connected, non-multiply-covered
$J$-holomorphic curve has an overlapping pair of ends.
\end{proposition}

\begin{proof}
Fix $l>>0$, and reintroduce the notation $\frak{p}_n$ and
$\widehat{\mc{J}}_n^l$ from the beginning of the proof of
Proposition~\ref{prop:G1}.  Let
$\widetilde{\mc{J}}_n^l\subset\widehat{\mc{J}}_n^l$ denote the set of
$J\in\widehat{\mc{J}}_n^l$ such that no curve $C\in\frak{p}_n^{-1}(J)$
has an overlapping pair of ends.  Then as in the proof of
Proposition~\ref{prop:G1}, it is enough to show that
$\widetilde{\mc{J}}_n$ is dense in $\widehat{\mc{J}}_n$.

To prove that $\widetilde{\mc{J}}_n$ is dense in $\widehat{\mc{J}}_n$,
fix $J\in\widehat{\mc{J}}_n$, and let ${\mathbf O}$, $\delta$, and
$\mc{U}$ be defined as in the beginning of the proof of
Lemma~\ref{lem:dense}.

Next, fix $C\in \frak{p}_n^{-1}(J)$, with ordered positive ends at
(possibly multiply covered or repeated) Reeb orbits
$\alpha_1,\ldots,\alpha_{N_+}$, and ordered negative ends at Reeb
orbits $\alpha_{-1},\ldots,\alpha_{-N_-}$.  For $i=1,\ldots,N_+$ let
$E_i$ denote the smallest positive eigenvalue of the asymptotic
operator $L_{\alpha_i}$; and for $i=-1,\ldots,-N_-$ let $E_i$ denote
the largest negative eigenvalue of $L_{\alpha_i}$.  Let $\mc{B}_i$
denote the $E_i$ eigenspace of $L_{\alpha_i}$; then the ends of $C$
determine asymptotic eigenfunctions $\gamma_i\in
\mc{B}_i\setminus\{0\}$.

Applying the translation $s\mapsto s+R$ to $C$ acts on $\gamma_i$ as
\begin{equation}
\label{eqn:AOE}
\gamma_i\longmapsto
e^{E_iR}\gamma_i.
\end{equation}
To keep track of this, let $I$ denote the index set
$\{1,\ldots,N_+\}\cup\{-1,\ldots,-N_-\}$, and let
$\mc{B}\eqdef\bigoplus_{i\in I}\mc{B}_i$.  Define ${\mathbb P}$ to be the
set of tuples $(\gamma_i)_{i\in I}\in\mc{B}$ with all
components $\gamma_i$ nonzero, modulo the equivalence
relation
\begin{equation}
\label{eqn:EER}
(\gamma_i)_{i\in I}
\sim
(e^{E_iR}\gamma_i)_{i\in I}
\end{equation}
for all $R\in\R$.  Note that ${\mathbb P}$ is a smooth manifold.  The
asymptotic eigenfunctions of the ends of $C$ define an element
$p(C)\in{\mathbb P}$ which is invariant under translation of $C$.
Furthermore, $C$ has all pairs of ends nonoverlapping if and only if
$p(C)\in{\mathbb P}\setminus Z$, where $Z$ is a finite union of
codimension $1$ and $2$ submanifolds of ${\mathbb P}$.

As in the proof of Lemma~\ref{lem:dense}, let $\pi:\mc{C}\to\mc{U}$
denote the universal moduli space, and let $\psi:\mc{U}\to\mc{C}$ be a
smooth section with $\psi(J)=C$.  The asymptotic eigenfunctions
determine a smooth map $p:\mc{U}\to{\mathbb P}$ which does not depend
on $\psi$.

We will momentarily prove:

\begin{lemma}
\label{lem:G2}
If $\delta$ is sufficiently small, then the map $p:\mc{U}\to{\mathbb
P}$ is a submersion.
\end{lemma}

It follows from this lemma and the implicit function theorem that
$p^{-1}(Z)$ is a finite union of codimension $1$ and $2$ submanifolds
of $\mc{U}$.  This completes the proof of Proposition~\ref{prop:G2}
(in the same way that Lemma~\ref{lem:Zsub} completes the proof of
Proposition~\ref{prop:G1}).
\end{proof}

\begin{proof}[Proof of Lemma~\ref{lem:G2}.]
We will assume that $C$ is immersed, and also that each eigenspace
$\mc{B}_i$ is one dimensional.  The argument for the general case
differs only in the amount of notation.

Let $j_k\in T\mc{U}|_J$ denote the tangent vector that is constructed
in the proof of Lemma~\ref{lem:Zsub} for the $k^{th}$ end.  Here we
assume that the constant $r$ used is very large and the constant
$\rho$ used is very small, and that these constants are the same for
each $k$.  (The construction given was for positive ends; there is a
negative end construction that is completely analogous.)  To prove the
lemma, we will show that the set of tangent vectors
$\{dp_J(j_k)\}_{k\in I}$ spans $T{\mathbb P}|_{p(C)}$.

On the $i^{th}$ end we introduce, as in the proof of
Proposition~\ref{prop:G1}, a normalized eigenfunction for $\mc{B}_i$,
in order to view the assignment of $\gamma_i$ to the elements in
$\mc{U}$ as defining a function $\widehat{a}_{\mc{E}_i}:\mc{U}\to\R$.
For each pair $(i,k)\in I\times I$, we now have the following analogue
of equation \eqref{eqn:UAZ}:
\begin{equation}
\label{eqn:UAZ2}
\nabla_{j_k}\widehat{a}_{\mc{E}_i} =
(d\widehat{a}_{\mc{E}_i})|_J \left(D_C^{-1}(j_{kC}) + w_{j_k}\right).
\end{equation}
Here $D_C^{-1}(j_{kC})$ and $w_{j_k}$ are the $j_k$ versions of the
expression on the right hand side of \eqref{eqn:dpsi}.

We now introduce
\[
\widehat{a}_{ik} \eqdef 
(d\widehat{a}_{\mc{E}_i})|_J \left(D_C^{-1}(j_{kC})\right),
\]
and we claim that $(dp)|_J(j_k)$ depends only on $\widehat{a}_{ik}$.
Indeed, this follows from the fact that each $w_{j_k}$ is tangent to
the orbit through $C$ of the $\R$-action that translates curves along
the $\R$ factor in $\R\times Y$.  Since $p$ is invariant under this
action, the $w_{j_k}$ term in \eqref{eqn:UAZ2} contributes nothing to
$dp|_J(j_k)$.  Therefore $dp|_J(j_k)$ is the projection to $T{\mathbb
P}|_{p(C)}$ of the vector $(\widehat{a}_{ik})_{i\in I}\in\mc{B}$.

So, to prove that the set of tangent vectors $\{dp_J(j_k)\}_{k\in I}$
spans $T{\mathbb P}|_{p(C)}$, it is enough to show that the matrix
$(\widehat{a}_{ik})_{i,k\in I}$ has nonzero determinant.  For this
purpose it is enough to show that the diagonal entries are much larger
than the other entries.  Let $i$ and $k$ be distinct ends, and to
simplify notation assume that they are positive.  By equation
\eqref{eqn:da+b}, if $r$ is sufficiently large and $\rho$ is
sufficiently small, then
\begin{equation}
\label{eqn:akkbound}
\widehat{a}_{kk}
 \ge ce^{E_k r}.
\end{equation}
Next let
$g_k$ denote the function $g:[-2,\infty)\times S^1\to\C$ used to construct
$j_k$.  We will show that
\begin{equation}
\label{eqn:aikbound}
|\widehat{a}_{ik}| \le c e^{-E_kr/2}\|g_k\|_2.
\end{equation}
By \eqref{eqn:BDG}, this is much smaller than \eqref{eqn:akkbound} when
$r$ is large and $\rho$ is small.

To prove \eqref{eqn:aikbound}, let $v_k\eqdef D_C^{-1}(j_k)_C$.  Let
$\beta:C\to[0,1]$ be a smooth function which equals $1$ off of the $s
\ge r/2$ part of the $k^{th}$ end, and which on the $k^{th}$ end is a
function of $s$ such that $|\beta'|\le 2$, and $\beta=0$ for $s\ge
r/2+1$.  Define $v_k'\eqdef \beta v_k$.  By definition, $D_Cv_k'$ is
nonzero only on the part of the $k^{th}$ end where $r/2<s<r/2+1$.
Using the bound \eqref{eqn:vkb}, it follows that $\|D_Cv_k'\|_2$ is
bounded by the right side of \eqref{eqn:aikbound}.  In particular, the
$L^2$ norm of $v_k$, over the complement of the $s\ge r/2$ part of the
$k^{th}$ end, is bounded by the right side of
\eqref{eqn:aikbound}.

Using standard elliptic estimates, we deduce pointwise bounds on
$|v_k|$ and $|\nabla v_k|$ on the $s\ge -1$ part of the $i^{th}$ end.
It follows as in \eqref{eqn:decay6} that on the $s\ge 0$ part of the
$i^{th}$ end,
\[
|v_k|, |\nabla v_k| \le c e^{-E_kr/2}\|g_k\|_2 e^{-E_is}.
\]
A virtual repeat of the arguments that lead to \eqref{eqn:IB} now
proves \eqref{eqn:aikbound}.
\end{proof}

%% file: obg-immersed.tex
This section is devoted to proving:

\begin{theorem}
\label{thm:immersed}
If the admissible almost complex structure $J$ on $\R\times Y$ is
generic, then all non-multiply-covered $J$-holomorphic curves of index
$\le 2$ are immersed.
\end{theorem}

\begin{proof} The proof has nine steps.

{\em Step 1.\/}
We begin by setting up the deformation theory for $J$-holomorphic
curves in $\R\times Y$ that are not necessarily immersed.

Fix an admissible almost complex structure $J$ on $\R\times Y$.  Let
$\varphi:(C,j)\to\R\times Y$ be a $J$-holomorphic curve, with ends
asymptotic to Reeb orbits as usual. Let $\mc{O}$ denote the set of
Reeb orbits at which $C$ has ends.  Recall that the domain $C$ is a
punctured compact Riemann surface, and $j$ denotes the complex
structure on $C$; thus the equation for $\varphi$ to be
$J$-holomorphic can be written as
\begin{equation}
\label{eqn:Jhol}
\frac{1}{2}\left(d\varphi + J \circ d\varphi \circ j\right) = 0.
\end{equation}
The derivative of this equation with respect to deformations of
$\varphi$ defines a real linear Fredholm operator
\begin{equation}
\label{eqn:ddfo}
D_\varphi : L^2_{1+}(C;\varphi^*TX) \longrightarrow
L^2_{0+}(C;T^{0,1}C\tensor_\C \varphi^* TX),
\end{equation}
where $X\eqdef \R\times Y$.  Here the almost complex structure $J$ is
used to regard $\varphi^*TX$ as a rank $2$ complex vector bundle over
$C$.  Also $L^2_{k+}$ denotes the weighted $L^2_k$ space, using a
weight which on the ends of $C$ is equal to $e^{\varepsilon |s|}$,
where $\varepsilon>0$ is small with respect to the Reeb orbits in the
finite set $\mc{O}$.

Observe that the diagram
\begin{equation}
\label{eqn:dbarcd}
\begin{CD}
L^2_{1+}(C;TC) @>{\dbar}>> L^2_{0+}(C;T^{0,1}C\tensor_\C TC) \\
@VV{d\varphi}V @VV{1\tensor d\varphi}V \\
L^2_{1+}(C;\varphi^*TX) @>{D_\varphi}>> L^2_{0+}(C;T^{0,1}C\tensor_\C
\varphi^*TX).
\end{CD}
\end{equation}
commutes.  In particular, this is why weighted Sobolev spaces are
needed to make the operator \eqref{eqn:ddfo} Fredholm.  Indeed, the
weights are present so as to deal with the fact that the operator
$\dbar$ on a cylinder is not Fredholm as a map from $L^2_1$ to $L^2$,
because the operator $i\partial_t$ has zero modes on the circle.  Note
that the operator $\dbar$ in \eqref{eqn:dbarcd} has zero kernel,
except when $C$ is a plane in which case the kernel has dimension $1$
over $\C$.

Let $B$ be the image of a smooth embedding of a ball into the space of
complex structures on $C$ with the following two properties: First,
$j\in B$, and each $j'\in B$ agrees with $j$ outside of a compact
subset of $C$.  Second, $T_jB$ projects isomorphically to the cokernel
of the operator $\dbar$ in \eqref{eqn:dbarcd}.  This condition makes
sense because a tangent vector $\dot{j}\in T_jB$ defines a smooth,
compactly supported bundle endomorphism of $TC$ satisfying
$j\dot{j}+\dot{j}j=0$.

Differentation of equation \eqref{eqn:Jhol} with respect to
deformations of $\varphi$ and $j$ gives rise to a Fredholm operator
\begin{equation}
\label{eqn:grfo}
\widetilde{D}_{\varphi}: T_jB \oplus L^2_{1+}(C;\varphi^*TX)
\longrightarrow L^2_{0+} (C;T^{0,1}C\tensor_\C \varphi^*TX)
\end{equation}
defined by
\[
\widetilde{D}_{\varphi}(\dot{j},\dot{\varphi}) \eqdef
D_\varphi\dot{\varphi} + \frac{1}{2} J \circ d\varphi \circ \dot{j}.
\]
When $C$ is a plane, we implicitly use \eqref{eqn:dbarcd} to regard
$\widetilde{D}_{\varphi}$ instead as a Fredholm operator on $T_jB
\oplus L^2_{1+}(C;\varphi^*TX) / d\varphi(\Ker(\dbar))$.  The curve
$\varphi$ is said to be ``unobstructed'' if the operator
$\widetilde{D}_{\varphi}$ is surjective.
Standard arguments show that if $J$ is generic, then all
non-multiply-covered $J$-holomorphic curves are unobstructed in this
sense.

{\em Step 2.\/} We now explain why if $\varphi$ is unobstructed in the
above sense, then the kernel of $\widetilde{D}_{\varphi}$ is naturally
identified with the tangent space at $\varphi$ to the moduli space of
$J$-holomorphic curves where $J$ is fixed.

To define the correspondence, let
$(\dot{j},\dot{\varphi})\in\Ker(\widetilde{D}_{\varphi})$.  Since
$\widetilde{D}_{\varphi}$ is surjective, the implicit function theorem
can be used in the usual way to find a smooth family of
$J$-holomorphic maps $\varphi_t:(C,j_t)\to\R\times Y$ parametrized by
$t$ in a neighborhood of $0$, with $(j_0,\varphi_0)=(j,\varphi)$,
satisfying the following two properties: First,
$(dj_t/dt)|_{t=0}=\dot{j}$.  Second, $\varphi_t=\exp_\varphi(v_t)$
where $\exp:TX\to X$ is the exponential map determined by some
$\R$-invariant metric on $X$, while $v_t$ is a smooth $L^2_{1+}$
section of $\varphi^*TX$, and $(dv_t/dt)|_{t=0}=\dot{\varphi}$.  Since
$\dot{\varphi}$ is in $L^2_{1+}$, the considerations in
\S\ref{sec:decay} show that there is a constant $c>0$ depending only
on the Reeb orbits in the finite set $\mc{O}$, such that
$|\dot{\varphi}|\le e^{-c|s|}$ for large $|s|$.  Consequently each
curve $\varphi_t:(C,j_t)\to\R\times Y$ is still in the moduli space
$\mc{M}^J$.  There is then a well-defined map
$\Phi:\Ker(\widetilde{D}_{\varphi})\to T_{\varphi}\mc{M}^J$ sending
$(\dot{j},\dot{\varphi})$ to the derivative of the family of
holomorphic curves $\varphi_t:(C,j_t)\to\R\times Y$ with respect to
$t$ at $t=0$.

To show that $\Phi$ is surjective, consider a smooth family of
holomorphic curves $\varphi_t:(C,j_t)\to\R\times Y$ parametrized by
$t$ in a neighborhood of $0$ with $(j_0,\varphi_0)=(j,\varphi)$.  By
reparametrization of the holomorphic curves $\varphi_t$, we can
arrange that each $j_t$ agrees with $j$ outside of a compact subset of
$C$.  This is because the pairs $(C,j_t)$ are punctured compact
Riemann surfaces, and the complex structures on the corresponding
closed surfaces are locally diffeomorphic near the punctures.  Now
define $\dot{j} \eqdef (dj_t/dt)|_{t=0}$ and $\dot{\varphi}\eqdef
(d\varphi_t/dt)|_{t=0}$.  These are smooth sections over $C$ of
$T^{0,1}C\tensor_{\C}TC$ and $T^{0,1}C\tensor_{\C}\varphi^*TX$
respectively.  Differentiation of the equation $J\circ d\varphi_t =
d\varphi_t\circ j_t$ with respect to $t$ at $t=0$ shows that
\begin{equation}
\label{eqn:ts1}
D_{\varphi}\dot{\varphi} + \frac{1}{2}J\circ d\varphi \circ \dot{j}=0.
\end{equation}
By construction $\dot{j}$ is compactly supported. To describe the
asymptotic behavior of $\dot{\varphi}$, note that on the ends of $C$,
the map $\varphi$ is an immersion by \cite{siefring}, so we have a
spitting $\varphi^*TX=TC\oplus N$ where $N$ denotes the normal bundle
to $C$, and the splitting is defined using an $\R$-invariant metric on
$X$ which is preserved by $J$.  By a further reparametrization of the
holomorphic curves $\varphi_t$, we can arrange that the $TC$ component
of $\dot{\varphi}$ is $0$ on the ends of $C$.  Meanwhile, the analysis
from \S\ref{sec:decay} shows that the $N$ component of $\dot{\varphi}$
is bounded from above by $e^{-c|s|}$ on the ends.  So $\dot{j}$ and
$\dot{\varphi}$ are in $L^2_{1+}$, provided that the constant
$\varepsilon>0$ used to define the spaces $L^2_{k+}$ is chosen smaller
than $c$.  Since the projection of $T_jB$ to the cokernel of the
operator $\dbar$ in \eqref{eqn:dbarcd} is surjective, there is a pair
$(b,v)\in T_jB \oplus L^2_{1+}(C;TC)$ such that
\begin{equation}
\label{eqn:ts2}
\dbar v + b = \dot{j}.
\end{equation}
Since the diagram \eqref{eqn:dbarcd} commutes, it follows from
\eqref{eqn:ts1} and \eqref{eqn:ts2} that
\[
\left(b, \dot{\varphi} + \frac{1}{2}j\dbar v\right) \in
\Ker(\widetilde{D}_{\varphi}).
\]
And $\Phi$ sends the above to the tangent vector
$(d(j_t,\varphi_t)/dt)|_{t=0}$.

To show that $\Phi$ is injective, let
$(\dot{j},\dot{\varphi})\in\Ker(\widetilde{D}_\varphi)$, and suppose
that the corresponding tangent vector to the moduli space is zero.
This means that there is a smooth one-parameter family of holomorphic
curves $\varphi_t:(C,j_t)\to\R\times Y$ parametrized by $t$ in a
neighborhood of $0$ with $(j_0,\varphi_0)=(j,\varphi)$ and
$(dj_t/dt)|_{t=0}=\dot{j}$ and $(d\varphi_t/dt)|_{t=0}=\dot{\varphi}$,
such that the holomorphic curves $\varphi_t$ are all equivalent to
each other: that is, there is a one-parameter family of
diffeomorphisms $\psi_t:C\to C$ with $\varphi_t=\varphi\circ\psi_t$
and $d\psi_t\circ j_t = j\circ\psi_t$.  Define $\eta \eqdef
(d\psi_t/dt)|_{t=0}$; this is a smooth, $L^2_{1+}$ section of $TC$
satisfying $\dot{\varphi}=d\varphi(\eta)$.  Since the diagram
\eqref{eqn:dbarcd} commutes, and since $J\circ d\varphi = d\varphi
\circ j$, the equation $\widetilde{D}_\varphi(\dot{j},\dot{\varphi})$
becomes
\[
d\varphi\circ\left(\dbar\eta + \frac{1}{2} j \circ \dot{j}\right) = 0.
\]
Therefore
\[
\dbar(2j\eta) = \dot{j}
\]
wherever $d\varphi\neq 0$; and by continuity this equation holds on
all of $C$.  Since the projection of $TB$ to the cokernel of $\dbar$
is injective, it follows that $\dot{j}=0$ and $\dbar\eta=0$.  Thus
$(\dot{j},\dot{\varphi})$ is equivalent to $0$.

{\em Step 3.\/} We now set up the proof of Theorem~\ref{thm:immersed}.

Fix an admissible $J$, and assume that $J$ is generic so that all
non-multiply-covered $J$-holomorphic curves are unobstructed.  Let
$\varphi:(C,j)\to\R\times Y$ be a non-multiply-covered $J$-holomorphic
curve of index $\le 2$.  Assume that $\varphi$ is not an immersion; in
particular this implies that it is not $\R$-invariant.  Define $B$ as
in Step 1.  Fix $\delta>0$ small and $l>>2$, and let $\mc{U}$ be a
small neighborhood of $J$ in the space of $C^l$ admissible almost
complex structures on $\R\times Y$ that agree with $J$ where the
distance is less than $\delta$ to any of the Reeb orbits corresponding
to the ends of $C$.  There is a smooth, universal family
$\mc{C}\to\mc{U}$ whose fiber over $J'\in\mc{U}$ consists of pairs
$(j',\varphi')$, where $j'\in B$ and $\varphi':(C,j')\to \R\times Y$
is a $J'$-holomorphic map near $\varphi$, which is the composition of
the exponential map with an $L^2_{2+}$ section of $\varphi^*TX$.  When
$C$ is a plane, we require the latter section to be $L^2$ orthogonal
to $d\varphi(\Ker(\dbar))$.  We can choose the neighborhood $\mc{U}$
to be small enough so that the fiber of $\mc{C}$ over each
$J'\in\mc{U}$ consists of unobstructed curves, and in particular is a
manifold of dimension $\le 2$.  Note that $\R$ acts freely on
$\mc{C}$, by composing $\varphi'$ with translations of $\R\times Y$;
and the projection to $\mc{U}$ is invariant under this action.

Let $u_0\in C$ be a point where $\varphi$ is not an immersion; since
$\varphi$ is $J$-holomorphic, this means that the differential
$d\varphi_{u_0}=0$.  Let $D\subset C$ be a small disc containing
$u_0$, such that $\varphi$ is an embedding on the closure of $D$ minus
the origin.  Let $\mc{Z}\subset\mc{C}\times D$ denote the locus of
points $(J',j',\varphi',u)\in\mc{C}\times D$ such that
$d\varphi'_{u}=0$.  Note that the $\R$ action on $\mc{C}$, crossed
with the identity on $D$, sends $\mc{Z}$ to itself.

To prove Theorem~\ref{thm:immersed}, we will show that if
$\delta$ is sufficiently small, then $\mc{Z}$ is a codimension 4
submanifold of $\mc{C}\times D$.  Granted this, then $\mc{Z}/\R$ is a
codimension 4 submanifold of $(\mc{C}/\R)\times D$.  Since the fibers
of the projection $(\mc{C}/\R)\times D\to\mc{U}$ have dimension at
most 3, the usual Sard-Smale argument can be used to deduce that there
is an open dense subset of $\mc{U}$ whose inverse image in
$\mc{C}\times D$ is disjoint from $\mc{Z}$.  As in the beginning of
the proof of Proposition~\ref{prop:G1}, it then follows that there is
a Baire set of admissible almost complex structures satisfying the
condition in Theorem~\ref{thm:immersed}.

\medskip

{\em Step 4.\/} To prepare for the proof that $\mc{Z}$ is a
codimension 4 submanifold of $\mc{C}\times D$, we now choose
convenient local coordinates $(s,t,x,y)$ on a neighborhood of
$\varphi(u_0)$ in $\R\times Y$.  These coordinates will be defined for
$|t|, |x|, |y|$ small, and will have the following properties:
\begin{description}
\item{(1)}
 $\varphi(u_0)$ corresponds to $(0,0,0,0)$.
\item{(2)}
$\partial_s$ is the derivative of the $\R$ action on $\R\times Y$, and
  $J\partial_s=\partial_t$.
\item{(3)}
$\partial_x$ and $\partial_y$ are tangent to the contact plane field
  at $(x,y)=(0,0)$.
\item{(4)}
$J\partial_x=\partial_y$
  at $t=0$.
\end{description}

To find such coordinates, first note that there exists a
$J$-holomorphic embedding $\imath:\Delta\subset\R\times Y$, where
$\Delta$ is a neighborhood of the origin in $\C$, such that
$\imath(0)=\varphi(u_0)$, and $d\imath_0$ maps $T_0\C$ to the contact
plane field at $\varphi(u_0)$.  Write the holomorphic coordinate on
$\Delta$ as $w=x+iy$.  The desired coordinates are now described by a
map
\[
\Phi:\R\times (-\delta_0,\delta_0) \times \Delta \longrightarrow \R\times
Y,
\]
where $\Phi(s,t,w)$ is obtained by starting with the point
$\imath(w)\in\R\times Y$, translating the $\R$ coordinate by $s$, and
then flowing along the Reeb vector field ${\mathbf R}$ for time $t$.
If $\delta_0>0$ is small, then $\Phi$ is an embedding.

These coordinates satisfy property (1), and the first part of property
(2), by construction.  The second part of property (2) holds because
the admissible almost complex structure $J$ is required to satisfy
$J\partial_s={\mathbf R}$.  Property (3) holds for $t=0$ by
construction, and for general $t$ because the Lie derivatives with
respect to ${\mathbf R}$ of $\partial_x$, $\partial_y$, and the
contact form all vanish.  Property (4) holds for $s=0$ by
construction, and for general $s$ because $J$ is $\R$-invariant.

\medskip

{\em Step 5.\/} We now write down equations for $\mc{Z}$ to be a
codimension 4 submanifold of $\mc{C}\times D$.

To start, in the coordinate chart from Step 4, if we write $z\eqdef
s+it$, then a basis for the $J$ version of $T^{1,0}(\R\times Y)$ is
given by
\[
e_0 = dz +
\gamma\,d\wbar,
\quad\quad 
e_1 = dw + \sigma\, d\wbar,
\]
where $\gamma$ and $\sigma$ vanish at $t=0$.  Likewise, if
$J'\in\mc{U}$, then because $J'\partial_s=\partial_t$, a basis for the
$J'$ version of $T^{1,0}(\R\times Y)$ is given by
\begin{equation}
\label{eqn:gpsp}
e_0' = dz + \gamma'\,d\wbar,
\quad\quad
e_1' = dw + \sigma'\,d\wbar
\end{equation}
where $\gamma'$ and $\sigma'$ no longer necessarily vanish at $t=0$.

Next, let us choose the set $B$ of complex structures on $C$ from Step
1 so that all $j'\in B$ agree with $j$ on the disc $D\subset C$.
Also, choose $D$ sufficiently small so that $\varphi(D)$ is contained
in the coordinate chart from Step 4.  Fix a holomorphic local
coordinate $u$ on $D$, with $u=0$ corresponding to $u_0$.  Now let
$(J',j',\varphi')\in\mc{C}$.  Then the equations for $\varphi'$ to be
$J'$-holomorphic on $D$ can be written as follows.  On $D$, in our
local coordinates, write $\varphi'(u)=(z'(u),w'(u))$.  Then $\varphi'$
is $J'$-holomorphic on $D$ if and only if it pulls back $e_0'$ and
$e_1'$ to multiples of $du$, which means that
\begin{equation}
\label{eqn:varphi'}
\dbar z' + 
\gamma' \dbar\overline{w'} = 0, \quad
  \quad
\dbar w' + \sigma'\dbar\overline{w'}=0.
\end{equation}
Here $\gamma'$ and $\sigma'$ are shorthand for $(\varphi')^*\gamma'$
and $(\varphi')^*\sigma'$ respectively.

It follows using \eqref{eqn:varphi'} that $\mc{Z}$ is the zero locus
of the smooth function $f:\mc{C}\times D \to \C^2$ defined by
\begin{equation}
\label{eqn:f}
f(J',j',\varphi',u) \eqdef \left(
\partial_u z' + 
\gamma'{\partial_u\overline{w'}},\,
{\partial_u w'} + \sigma'{\partial_u
  \overline{w'}}
\right)\big|_u.
\end{equation}
Thus to prove that $\mc{Z}$ is a codimension 4 submanifold of
$\mc{C}\times D$ near our given point $(J,j,\varphi,0)\in\mc{Z}$, by
the implicit function theorem it is enough to show that the
differential
\begin{equation}
\label{eqn:df}
df_{(J,j,\varphi,0)}:T_{(J,j,\varphi,0)}(\mc{C}\times D)
\longrightarrow \C^2
\end{equation}
is surjective.

\medskip

{\em Step 6.\/} We now give an explicit description of the tangent
space $T_{(J,j,\varphi)}\mc{C}$ and the differential \eqref{eqn:df},
in preparation for showing that \eqref{eqn:df} is surjective.

In the notation from Step 1, a tangent vector in
$T_{(J,j,\varphi)}\mc{C}$ is equivalent to a triple
$(\dot{J},\dot{j},\dot{\varphi})$, where:
\begin{description}
\item
{(i)} $\dot{J}$ is a $C^l$ bundle endomorphism of $TX$ which satisfies
$J\dot{J}+\dot{J}J=0$, annihilates $\partial_s$ and ${\mathbf R}$, and
maps the contact plane field $\xi$ to itself.
\item{(ii)} $\dot{j}$ is a smooth, compactly supported bundle
  endomorphism of $TC$ satisfying $j\dot{j}+\dot{j}j=0$.
\item{(iii)} $\dot{\varphi}$ is an $L^2_{1+}$ section of $\varphi^*TX$
over $C$ satisfying
\begin{equation}
\label{eqn:satisfying}
D_\varphi \dot{\varphi} + \frac{1}{2}\left(
\dot{J}\circ d\varphi \circ j + J \circ d\varphi \circ
\dot{j}\right)=0.
\end{equation}
When $C$ is a plane, we declare two such sections $\dot{\varphi}$ to
be equivalent if their difference is in $d\varphi(\Ker(\dbar))$.
\item{(iv)}
$\dot{j}\in T_jB$.
\item{(v)} $\dot{J}$ is invariant under $\R$ translation and vanishes
within distance $\delta$ of the Reeb orbits in $\mc{O}$.
\end{description}
Of course condition (ii) follows from condition (iv); but our
construction later will obtain condition (ii) before obtaining
condition (iv).

In the local coordinates from Step 4, these data appear as follows.
By condition (i) above, $\dot{J}$ determines a pair of functions
$(\dot{\gamma},\dot{\sigma})$, which describe the respective changes
in $\gamma$ and $\sigma$.  Likewise, $\dot{\varphi}$ determines a pair
of functions $(\dot{z},\dot{w})$ on $D$.  With this notation, equation
\eqref{eqn:satisfying} on $D$ is equivalent to the equations obtained
by differentiating \eqref{eqn:varphi'}, which have the form
\begin{equation}
\label{eqn:tfs}
\begin{split}
\dbar\dot{z} + 
\gamma\dbar\overline{\dot{w}} +
e_\gamma(\dot{z},\dot{w})
&= - \dot{\gamma}\dbar\wbar,\\
\dbar\dot{w} + 
\sigma\dbar\overline{\dot{w}} +
e_\sigma(\dot{z},\dot{w})
&= - \dot{\sigma}\dbar\wbar.
\end{split}
\end{equation}
Here $e_\gamma$ and $e_\sigma$ are linear functions of $\dot{z}$ and
$\dot{w}$, arising from the derivatives of $\gamma$ and $\sigma$; and
$(z,w)$ denotes the pair of functions determined by $\varphi$.

In the above notation, it follows from \eqref{eqn:f}, with the help of
\eqref{eqn:varphi'}, that the differential \eqref{eqn:df} is given as
follows: If $\dot{u}\in T_0 D$, then
\begin{equation}
\label{eqn:dfe}
df_{(J,j,\varphi,0)}(\dot{J},\dot{j},\dot{\varphi},\dot{u}) =
\left(\partial_u\dot{z},\,
\partial_u\dot{w}\right)\big|_{u=0}
+
(\partial_u\partial_u z,\partial_u\partial_u w)\big|_{u=0} du(\dot{u}).
\end{equation}
To prove that the differential \eqref{eqn:df} is surjective, we will
show that for any $\eta=(\eta_z,\eta_w)\in\C^2$, there exists
$(\dot{J},\dot{j},\dot{\varphi})\in T_{(J,j,\varphi)}\mc{C}$ and
$\dot{u}\in T_0D$ such that
\begin{equation}
\label{eqn:wanteta}
\left(\partial_u\dot{z},\, \partial_u\dot{w}\right)\big|_{u=0} +
(\partial_u\partial_u z,\partial_u\partial_u w)\big|_{u=0} du(\dot{u})
= (\eta_z,\eta_w).
\end{equation}

\medskip

{\em Step 7.\/} We now begin the construction of
$(\dot{J},\dot{j},\dot{\varphi})\in T_{(J,j,\varphi)}\mc{C}$ and
$\dot{u}\in T_0D$ satisfying \eqref{eqn:wanteta}.

\begin{lemma}
\label{lem:eta}
There exist constants $\rho_0>0$ and $c$ such that if $D$ has radius
$\rho < \rho_0$, then given $\eta=(\eta_z,\eta_w)\in\C^2$, there exist
smooth functions $\dot{z},\dot{w}:D\to \C$ such that:
\begin{itemize}
\item
The equations \eqref{eqn:tfs} hold with $\dot{\gamma}=\dot{\sigma}=0$.
\item
Equation \eqref{eqn:wanteta} holds with $\dot{u}=0$.
\end{itemize}
\end{lemma}

\begin{proof}
We will first solve \eqref{eqn:tfs} with
$\dot{\gamma}=\dot{\sigma}=0$, and then explain how to solve
\eqref{eqn:wanteta} with $\dot{u}=0$ as well.

Write $\zeta\eqdef(\dot{z},\dot{w})$ and $\eta\eqdef(\eta_z,\eta_w)$.
The disc of radius $\rho$ can be identified with a disc of radius $1$
so that on the disc of radius $1$, the equations \eqref{eqn:tfs} with
$\dot{\gamma}=\dot{\sigma}=0$ have the form
\begin{equation}
\label{eqn:rescaled}
\dbar_u\zeta + \Theta \zeta = 0,
\end{equation}
where $\Theta$ is $\R$-linear and satisfies $|\Theta|,|d\Theta|<c\rho$.  It is
enough to solve this equation on the disc of radius $1/2$, for which
purpose we can assume that $\Theta$ vanishes outside of the disc of
radius $3/4$ and has derivative bounded by $c\rho$.  Now write
$\zeta=\zeta_0+\Delta$, where $\zeta_0\eqdef \eta u$; then equation
\eqref{eqn:rescaled} becomes
\begin{equation}
\label{eqn:Delta}
\dbar_u\Delta + \Theta\Delta + \Theta\zeta_0 = 0.
\end{equation}
If $\rho>0$ is sufficiently small, then the contraction mapping
theorem finds a unique continuous function $\Delta$ on $D$ satisfying
\begin{equation}
\label{eqn:ucd}
\Delta(u) = \frac{1}{\pi} \int_{|v|\le 1}\frac{1}{u-v}\left(\Theta\Delta
+ \Theta\zeta_0\right) |_{v} \, d^2v,
\end{equation}
and this $\Delta$ will satisfy \eqref{eqn:Delta}.  Moreover, it
follows from \eqref{eqn:ucd} that
\[
|\Delta(u)| \le c\rho(\sup|\Delta| + |\eta|),
\]
where $|\eta| \eqdef \sqrt{|\eta_z|^2 + |\eta_w|^2}$.  Thus
\begin{equation}
\label{eqn:DeltaBound}
|\Delta(u)| \le c\rho|\eta|(1-c\rho)^{-1}.
\end{equation}
Differentiating \eqref{eqn:ucd}, and using \eqref{eqn:DeltaBound} and
our assumptions on $\Theta$, gives a similar bound on
$|\partial_u\Delta|$.  As a result, if $\rho>0$ is
sufficiently small then
\[
\bigg|\partial_u\zeta|_{u=0} - \eta\bigg| \le \frac{1}{2}|\eta|.
\]
Since the set
\[
\left\{\partial_u\zeta|_{u=0} \mid \mbox{$\zeta$ solves
  \eqref{eqn:rescaled}}\right\}
\]
is a real linear subspace of $\C^2$, it follows that this set is all
of $\C^2$.  This proves the lemma.
\end{proof}

Now fix a smooth function $\beta:\C\to[0,1]$ with $\beta(u)=0$ for
$|u|\ge 1$ and $\beta(u)=1$ for $|u|\le 1/2$.  Given $\rho>0$, define
$\beta_\rho:\C\to[0,1]$ by $\beta_\rho(u)\eqdef \beta(\rho^{-1}u)$.

Next, given $\eta=(\eta_z,\eta_w)\in\C^2$, let $\rho>0$ be small, and
let $(\dot{z}_\eta,\dot{w}_\eta)$ denote the pair of functions
provided by Lemma~\ref{lem:eta}.  Take
\[
\dot{z} \eqdef \beta_\rho\dot{z}_\eta, \quad\quad \dot{w} \eqdef
\beta_\rho\dot{w}_\eta,
\]
and extend these to a section $\dot{\varphi}$ of $\varphi^*TX$ over
$C$ by declaring $\dot{\varphi}$ to be zero on the complement of $D$.
Then over the disc of radius $\rho/2$, equation \eqref{eqn:satisfying}
holds with $\dot{J}=\dot{j}=0$.

Let $D_\rho$ denote the portion of $D$ where the radius is between
$\rho/2$ and $\rho$.  If $\rho$ is sufficiently small, then the
restriction of $\varphi$ to $D_\rho$ does not hit the finite set of
points in $C$ where $\varphi(C)$ is tangent to $\xi$ or to the span of
vectors $\partial_s$ and ${\mathbf R}$. It then follows as in the
proof of Lemma~\ref{lem:Zsub}, Step 2, that there exist $\dot{J}$ and
$\dot{j}$ such that:
\begin{itemize}
\item
The quadruple
$(\dot{J},\dot{j},\dot{\varphi},0)$ satisfies conditions (i)--(iii)
above, as well as equation \eqref{eqn:wanteta}.
\item
$\dot{j}$
and the restriction of $\dot{J}$ to $C$ are supported in $D_\rho$.
\end{itemize}
It proves convenient later to also choose $\rho$ sufficiently small so
that:
\begin{itemize}
\item
The restriction of $\varphi$ to $D_\rho$ does not hit the finite set
of points in $C$ where $\varphi(C)$ intersects the Reeb orbits in
$\mc{O}$.
\end{itemize}

\medskip

{\em Step 8.\/} We now modify the quadruple
$(\dot{J},\dot{j},\dot{\varphi},0)$ so as to also satisfy condition
(iv), namely that $\dot{j}\in T_jB$, while still satisfying conditions
(i)--(iii) and equation \eqref{eqn:wanteta}.

The idea here is that one could regard $f$ as a function defined on a
larger space, where the complex structure on $C$ is not required to be
in $B$, such that $f$ is invariant under an appropriate equivalence
relation; then by moving along an appropriate slice in this larger
space we can obtain $\dot{j}\in T_jB$ without changing $df$.  This
works concretely as follows.  By definition, $T_jB$ projects
isomorphically onto the cokernel of the operator
\[
\dbar: L^2_{1+}(C;TC) \longrightarrow L^2_{0+}(C;T^{0,1}C\tensor_\C TC).
\]
Hence there exists a tangent vector $\dot{j}_0\in T_jB$, and a section
$\zeta\in L^2_{1+}(C;TC)$, such that
\begin{equation}
\label{eqn:jdz}
2j\circ\dbar\zeta + \dot{j}_0 = \dot{j}
\end{equation}
as bundle automorphisms of $TC$.  Let $\dot{\varphi}_0 \eqdef
d\varphi\circ\zeta$.  To achieve condition (iv), replace the quadruple
$(\dot{J},\dot{j},\dot{\varphi},0)$ with
\begin{equation}
\label{eqn:newquad}
\left(\dot{J},\dot{j}_0,
\dot{\varphi}+ \dot{\varphi}_0,-\zeta(0)\right).
\end{equation}
We need to check that the new quadruple \eqref{eqn:newquad} still satisfies
equations \eqref{eqn:satisfying} and \eqref{eqn:wanteta}.

To verify equation \eqref{eqn:satisfying}, use the commutativity of
the diagram \eqref{eqn:dbarcd}, the fact that $\varphi$ is
$J$-holomorphic, and equation \eqref{eqn:jdz}, to find that
\[
\begin{split}
D_\varphi(\dot{\varphi}_0) &=  d\varphi \circ \dbar\zeta \\
&= - J \circ d\varphi \circ j\dbar\zeta \\
&= J \circ d\varphi \circ
\frac{1}{2}(\dot{j}_0 - \dot{j}).
\end{split}
\]
It follows from this that the new quadruple still satisfies equation
\eqref{eqn:satisfying}.

To prove that the new quadruple \eqref{eqn:newquad} still satisfies
equation \eqref{eqn:wanteta}, write $\dot{\varphi}_0 =
(\dot{z}_0,\dot{w}_0)$ on $D$.  We need to show that
\begin{equation}
\label{eqn:stilleta}
(\partial_u\dot{z}_0, \partial_u\dot{w}_0)\big|_{u=0} =
(\partial_u\partial_uz, \partial_u\partial_u w)\big|_{u=0} du(\zeta(0)).
\end{equation}
It follows from the definition of $\dot{\varphi}_0$ and the
equations \eqref{eqn:varphi'} that
\[
(\dot{z}_0,\dot{w}_0) = (\partial_u z,\partial_u w)du(\zeta) -
(\gamma\partial_{\ubar}\wbar, \sigma\partial_{\ubar}\wbar) d\ubar(\zeta).
\]
Since $\partial_uz$, $\partial_uw$, $\gamma$, and $\sigma$ all vanish
at $u=0$, equation \eqref{eqn:stilleta} follows.

\medskip

{\em Step 9.\/}  The previous steps constructed a quadruple
$(\dot{J},\dot{j},\dot{\varphi},\dot{u})$ satisfying conditions
(i)--(iv) and equation \eqref{eqn:wanteta};  we now modify this
quadruple so as to also satisfy condition (v).

Recall from Step 7 that the restriction of $\dot{J}$ to $\varphi(C)$,
call it $\dot{J}_C$, is supported inside $D_\rho$.  Only $\dot{J}_C$
enters into equation \eqref{eqn:satisfying}; so we just need to modify
$\dot{J}_C$ (while changing the other data
$\dot{j},\dot{\varphi},\dot{u}$ as appropriate) so that $\dot{J}_C$
has an extension over $X$ satisfying conditions (i) and (v).

For this purpose, let $\Lambda\subset D$ denote the set of points
$u\in D$ such that $\varphi(u)$ intersects the $s\mapsto s - s_0$
translate of $\varphi$ for some
$s_0\in[-\infty,+\infty]\setminus\{0\}$.  As in the proof of
Lemma~\ref{lem:Zsub} Step 2, $\Lambda$ is a closed codimension $1$
subvariety of $D$.  Let $\varepsilon>0$ be small, and let
$\chi:D\to[0,1]$ be a smooth function which is $1$ where the distance
to $\Lambda$ is $\ge 2\varepsilon$ and $0$ where the distance to
$\Lambda$ is $\le \varepsilon$.  We now replace $\dot{J}_C$ by
$\chi\dot{J}_C$.  Note that $\chi\dot{J}_C$ still satisfies
condition (i) on $D$, because condition (i) is a system of homogeneous
linear equations for $\dot{J}$.  Furthermore, if $\delta>0$ is
sufficiently small then $\chi\dot{J}_C$ has an extension over $X$
satisfying (i) and (v); pick such an extension and call it $\dot{J}'$.

We now modify $\dot{j}$ and $\dot{\varphi}$ to restore equation
\eqref{eqn:satisfying}.  Since the operator $\widetilde{D}_\varphi$ is
surjective, there exists $\dot{j}_1\in T_jB$, and an $L^2_{1+}$ section
$\dot{\varphi}_1$ of $\varphi^*TX$, such that
\begin{equation}
\label{eqn:te1}
D_\varphi\dot{\varphi}_1 + \frac{1}{2}J \circ d\varphi \circ \dot{j}_1
= (\chi-1)\dot{J}_C.
\end{equation}
Moreover, these can be chosen so that
\begin{equation}
\label{eqn:te2}
\|\dot{\varphi}_1\|_{L^2_1}
\le
c\|(\chi-1)\dot{J}_C\|_{L^2},
\end{equation}
where $c$ is a constant which does not depend on $\varepsilon$.

It follows from \eqref{eqn:te1} that if we define $\dot{j}' \eqdef
\dot{j} + \dot{j}_1$ and $\dot{\varphi}' \eqdef \dot{\varphi} +
\dot{\varphi}_1$, then the triple $(\dot{J}',\dot{j}',\dot{\varphi}')$
satisfies equation \eqref{eqn:satisfying}, and hence all of the
conditions (i)--(v).

The quadruple $(\dot{J}',\dot{j}',\dot{\varphi}',\dot{u})$ might not
satisfy equation \eqref{eqn:wanteta}.  Rather, if we define
\[
(\eta_z',\eta_w') \eqdef
df_{(J,j,\varphi,0)}(\dot{J}',\dot{j}',\dot{\varphi}',\dot{u}),
\]
and if we write $\dot{\varphi}_1=(\dot{z}_1,\dot{w}_1)$ on $D$, then
it follows from equation \eqref{eqn:dfe} that
\begin{equation}
\label{eqn:mas}
(\eta_z',\eta_w') - (\eta_z,\eta_w) =
(\partial_u\dot{z}_1,\partial_u\dot{w}_1)\big|_{u=0}.
\end{equation}
To handle this discrepancy, note that by taking $\varepsilon$
sufficiently small, we can make the support of $(\chi-1)\dot{J}_C$ have
arbitrarily small measure, and so by \eqref{eqn:te2} we can make
$\dot{\varphi}_1$ have arbitrarily small $L^2_1$ norm.  It then
follows by elliptic regularity, as in Lemma~\ref{lem:ER}, that we can
make the expression in \eqref{eqn:mas} have arbitrarily small norm.

We conclude from the above discussion that for any nonzero
$(\eta_z,\eta_w)\in \C^2$, we can find
$(\dot{J},\dot{j},\dot{\varphi},\dot{u})\in
T_{(J,j,\varphi,0)}(\mc{C}\times D)$ such that
\[
\left| df_{(J,j,\varphi,0)} (\dot{J},\dot{j},\dot{\varphi},\dot{u}) -
(\eta_z,\eta_w) \right | \le \frac{1}{2} \left|(\eta_z,\eta_w)\right|.
\]
Since $df_{(J,j,\varphi,0)}$ is linear, it follows that it is
surjective.  This completes the proof of Theorem~\ref{thm:immersed}.
\end{proof}

%% file: obg-gluing.tex
As in \S\ref{sec:SMR}, fix a generic $J$ such that all
non-multiply-covered $J$-holomorphic curves are unobstructed, and let
$(U_+,U_-)$ be a gluing pair as in Definition~{I.1.9} satisfying the
additional assumptions \eqref{eqn:i} and \eqref{eqn:ii}.  Let $\mc{M}$
denote the moduli space of branched covers of $\R\times S^1$ as
specified in \eqref{eqn:specifyM}.

In this section we explain a construction for gluing $U_+$ and $U_-$
to a $J$-holomorphic curve by patching an element of $\mc{M}$ between
them.  This procedure finds such a gluing for each zero of a certain
section $\frak{s}$ of the obstruction bundle $\mc{O}$ from \S{I.2.3}
over a certain open subset of $\R^2\times \mc{M}$.  As a result, we
obtain a ``gluing map'' $G$ from $\frak{s}^{-1}(0)$ to the moduli
space of $J$-holomorphic curves $\mc{M}^J(\alpha_+,\alpha_-)$.

\subsection{Preliminaries}
\label{sec:glP}

\paragraph{(i)}
It follows from the definition of gluing pair that $U_\pm$ consists of
an immersed, non-multiply-covered, index $1$ component $u_\pm$,
together with a union $v_\pm$ of unbranched covers of $\R$-invariant
cylinders.  Index the negative ends of $U_+$ such that the negative
ends of $u_+$ are indexed by $1,\ldots,\overline{N}_+$, while the
negative ends of $v_+$ are indexed by $\overline{N}_++1,\ldots,N_+$.
Similarly, index the positive ends of $U_-$ such that the positive
ends of $u_-$ are indexed by $-1,\ldots,-\overline{N}_-$, while the
positive ends of $v_-$ are indexed by $-\overline{N}_--1,\ldots,-N_-$.
(In fact Lemma I.3.7 implies that $\overline{N}_+ \ge N_+-1$ and
$\overline{N}_-\ge N_--1$, but we will not need this.)

\paragraph{(ii)}
Fix an ``exponential map'' $e:\R\times S^1\times D \to \R\times Y$ for
$\alpha$ as given by Lemma~\ref{lem:EM}.  This $e$ defines coordinates
$(s,t,w)$ on a tubular neighborhood of $\R\times\alpha$ in $\R\times
Y$.  Fix $\delta_0>0$ sufficiently small so that $D$ contains the disc
of radius $4\delta_0$.  By translating $U_+$ upward, we may assume
that for $i=1,\ldots,\overline{N}_+$, the $i^{th}$ negative end of
$u_+$ is described in these coordinates by a map
\begin{equation}
\label{eqn:etai+}
\begin{split}
(-\infty,0]\times \widetilde{S^1} & \longrightarrow \R\times
    S^1\times\C,\\
(s,\tau) & \longmapsto (s,t,\eta_i(s,\tau)),
\end{split}
\end{equation}
where $\widetilde{S^1}$ denotes the $a_i$-fold cover of $S^1$; $t$
denotes the projection of $\tau$; and $|\eta_i|<\delta_0$.  Likewise,
by translating $U_-$ downward, we may assume that for
$i=-1,\ldots,-\overline{N}_-$, the $i^{th}$ positive end of $u_-$ is
described by a map
\begin{equation}
\label{eqn:etaj-}
\begin{split}
[0,\infty) \times \widetilde{S^1} & \longrightarrow \R\times
  S^1\times\C,\\
(s,\tau) & \longmapsto (s,t,\eta_i(s,\tau)),
\end{split}
\end{equation}
where $\widetilde{S^1}$ denotes the $a_i$-fold cover of $S^1$, and
$|\eta_i|<\delta_0$.

Next, as in Lemma~\ref{lem:EM}, choose an ``exponential map'' $e_-$,
from a small radius disc bundle in the normal bundle of $u_-$ to
$\R\times Y$, with the following properties.  First, $e_-$ is an immersion;
on the zero section $e_-$ agrees with $u_-$; and on each fiber disc
$u_-$ is a $J$-holomorphic embedding.  The constant $\delta_0$ above
should be chosen sufficiently small so that $e_-$ is defined on the
radius $\delta_0$ disc bundle.  In addition, on the positive ends of
$u_-$ we require $e_-$ to be compatible with $e$ in the
following sense: For $i= -1,\ldots,-\overline{N}_-$, in the notation of
\eqref{eqn:etaj-}, use $(s,\tau)$ as coordinates on the $i^{th}$
positive end of $u_-$, and use the coordinate $w$ to trivialize the
normal bundle to the $i^{th}$ positive end of $u_-$.  Then the
compatibility requirement is that if $|v|<\delta_0$ then
\begin{equation}
\label{eqn:ECC}
e_-((s,\tau),v) = e(s,t,\eta_i(s,\tau)+v).
\end{equation}
Choose an analogous exponential map $e_+$ from the radius $\delta_0$
disc bundle in the normal bundle of $u_+$ to $\R\times Y$.

\paragraph{(iii)}
Given a branched cover $\pi:\Sigma\to\R\times S^1$ in $\mc{M}$, let
$\Lambda\subset\Sigma$ denote the union of the components of the level
sets of $\pi^*s$ on $\Sigma$ that contain ramification points.  For
$i=1,\ldots,\overline{N}_+$, the $i^{th}$ positive end of $\Sigma$
corresponds to a component of $\Sigma\setminus\Lambda$, which the
asymptotic marking identifies with $(s_i-1,\infty)\times\R/2\pi a_i\Z$
for some real number $s_i$.  Likewise, for
$i=-1,\ldots,-\overline{N}_-$, the $i^{th}$ negative end of $\Sigma$
corresponds to a component of $\Sigma\setminus\Lambda$, which the
asymptotic marking identifies with $(-\infty,s_i+1)\times\R/2\pi
a_i\Z$.  Let $s_+\eqdef\max_{i>0}\{s_i\}$ and $s_-\eqdef\min_{i<0}\{s_i\}$.
Note that $s_i$, $s_+$, and $s_-$ define functions on
$\mc{M}$ which are continuous but not smooth.  It proves convenient
later to replace these functions by smooth functions which have
$C^0$-distance less than $1/2$ from the original functions.  We denote
these smoothings by the same symbols.

\paragraph{(iv)}
The gluing construction requires fixing two constants $0<h<1$ and
$r>h^{-1}$ which enter into the definitions of the relevant cutoff
functions.  The gluing construction will work for any $0<h<1$, as long
as $r$ is sufficiently large with respect to $h$ in a sense to be
explained below.  (In \S\ref{sec:deform} we will choose $h$ small in
order to obtain good estimates on the nonlinear part of the
obstruction section whose zero set characterizes the possible
gluings.)

Throughout this section, the letter `$c$' denotes a constant
which depends only on $U_+$ and $U_-$, and whose value may change from
one appearance to the next.

\subsection{Pregluing}
\label{sec:pregluing}

With $r$ and $h$ fixed, the ``gluing parameters'' consist of a
branched cover $\pi:\Sigma\to\R\times S^1$ in $\mc{M}$, together with
real numbers $T_+,T_-\ge 5r$.  Given $T_+$, $T_-$, and $\Sigma$, we now
define the ``$(T_+,T_-)$-pregluing'' of $U_+$ and $U_-$ along
$\Sigma$.  This will be a map $u_*:C_*\to\R\times Y$.

To define the domain $C_*$, let $\Sigma'\subset\Sigma$ be obtained
from $\Sigma$ by removing the $s>s_+ + T_+$ portion of the positive
ends indexed by $1,\ldots,\overline{N}_+$ and the $s<s_- - T_-$
portion of the negative ends indexed by $-1,\ldots,-\overline{N}_-$.
Let $u_{+T}$ denote the $s\mapsto s + s_+ + T_+$ translate of $u_+$,
and let $u_{+T}'$ denote the $s\ge s_++T_+$ portion of $u_{+T}$.  Let
$u_{-T}$ denote the $s\mapsto s + s_- - T_-$ translate of $u_-$, and
let $u_{-T}'$ denote the $s \le s_--T_-$ portion of $u_{-T}$.

Let $C_{\pm T}'$ denote the domain of $u_{\pm T}'$.  The domain $C_*$
is the quotient of $C_{+T}'\sqcup \Sigma' \sqcup C_{-T}'$ obtained by
identifying the $i^{th}$ negative boundary circle of $C_{+T}'$ with
the $i^{th}$ positive boundary circle of $\Sigma'$ for
$i=1,\ldots,\overline{N}_+$, and the $i^{th}$ positive boundary circle
of $C_{-T}'$ with the $i^{th}$ negative boundary circle of $\Sigma'$
for $i=-1,\ldots,-\overline{N}_-$.  The identification maps are well
defined, because the asymptotic markings of the ends of $\Sigma$ and
$u_\pm$ fix an identification of each such boundary circle with
$\R/2\pi m\Z$, where $m$ is the covering multiplicity of the
associated end.

For $i=1,\ldots,\overline{N}_+$, let $\Sigma_i\subset \Sigma'$ denote the
cylinder consisting of the $s_i \le s\le s_+ + T_+$ part of the
$i^{th}$ positive end of $\Sigma$.  As above, the cylinder
$\Sigma_i$ can also be naturally identified with the $s_i \le s\le
s_+ + T_+$ portion of the $i^{th}$ negative end of $u_{+T}$.  For
$i=-1,\ldots,-\overline{N}_-$, let $\Sigma_i\subset
\Sigma'$ denote the cylinder consisting of the $s_- - T_- \le s \le s_i$
part of the $i^{th}$ negative end of $\Sigma$.  This can also be
identified with the corresponding portion of the $i^{th}$ positive end
of $u_{-T}$.

Fix a smooth function $\beta:\R\to[0,1]$ which is non-decreasing,
equal to $0$ on $(-\infty,0]$, and equal to $1$ on $[1,\infty)$.
Define a function $\beta_+:C_*\to[0,1]$ as follows.  The function
$\beta_+$ equals $1$ on all of $C_{+T}'$ and $0$ on all of $C_{-T}'$.
On the cylinder $\Sigma_i\subset \Sigma'$ for $i>0$, define $\beta_+\eqdef
\beta((s-s_i-hr)/(hr))$.  On the rest of $\Sigma'$ define
$\beta_+\eqdef 0$.  Similarly, define $\beta_-:C_*\to[0,1]$ to equal
$1$ on all of $C_{-T}'$, to equal $0$ on all of $C_{+T}'$, to equal
$\beta((-s+s_i - hr)/(hr))$ on $\Sigma_i$ for $i<0$, and to equal $0$
on the rest of $\Sigma'$.

The map $u_*$ is defined as follows.  On $C_{\pm T}'$, the map $u_*$ agrees
with the map $u_{\pm T}$. On $\Sigma'$, off of the cylinders
$\Sigma_i$, the map $u_*$ agrees with the composition
\begin{equation}
\label{eqn:composition}
\Sigma\stackrel{\pi}{\longrightarrow} \R\times S^1
\stackrel{\op{id}\times\alpha}{\longrightarrow} \R\times Y.
\end{equation}
On $\Sigma_i$ for $i>0$, with the notation as in \eqref{eqn:etai+},
define
\[
\eta_{i,T}(s,\tau) \eqdef
\eta_i(s-(s_++T_+),\tau)
\]
and
\begin{equation}
\label{eqn:beta+}
u_*(s,\tau) \eqdef (s,t,\beta_+(s,\tau)
\eta_{i,T}(s,\tau) ).
\end{equation}
When $s\le s_i+hr$ this agrees with the composition
\eqref{eqn:composition}, and when $s\ge s_i+2hr$ this agrees with the
$s\mapsto s + s_+ + T_+$ translate of the $i^{th}$ negative end of $u_+$.
Likewise, on $\Sigma_i$ for $i<0$, with the notation as in \eqref{eqn:etaj-},
define
\[
\eta_{i,T}(s,\tau) \eqdef \eta_i(s-(s_--T_-),\tau)
\]
and
\begin{equation}
\label{eqn:beta-}
u_*(s,\tau) \eqdef (s,t,\beta_-(s,\tau) \eta_{i,T}(s,\tau)
).
\end{equation}
When $s\ge s_i - hr$ this agrees with the composition
\eqref{eqn:composition}, and when $s\le s_i - 2hr$ this agrees with
the $s\mapsto s + s_- - T_-$ translate of the $i^{th}$ positive end of $u_-$.

\subsection{Deforming the pregluing}
\label{sec:defpre}

Let $\psi_+$ be a section of the normal bundle of $u_{+T}$, let
$\psi_-$ be a section of the normal bundle of $u_{-T}$, and let
$\psi_\Sigma$ be a complex function on $\Sigma$.  Assuming that
$\psi_\pm$ and $\psi_\Sigma$ have pointwise norm less than $\delta_0$,
we now explain how to use the data $(\psi_-,\psi_\Sigma,\psi_+)$ to
define a deformation of the map $u_*$.

The coordinate $w$ on a neighborhood of $\R\times\alpha$ trivializes
the normal bundles to the positive ends of $u_-$ and the negative ends
of $u_+$ near $\R\times\alpha$.  Hence the normal bundles to $C_{-T}'$
and $C_{+T}'$ and the trivial complex line bundle over $\Sigma'$ fit
together to define a complex line bundle $E_*$ over $C_*$.  The
exponential maps $e_-$, $e$, and $e_+$ fit into a map $e_*$ from a
small radius disc bundle in $E_*$ to $\R\times Y$ defined as follows.
Over $C_{\pm T}'$, the map $e_*$ is defined on the radius $\delta_0$
disc bundle and agrees with the appropriate translate of $e_\pm$.  For
$x\in\Sigma'$, the map $e_*$ is defined on the radius $2\delta_0$ disc
bundle as follows: If $u_*(x)=(s,t,w)$, then
$e_*(x,v)\eqdef (s,t,w+v)$.

Next define a function $\beta_\Sigma:C_*\to[0,1]$ as follows.  The
function $\beta_\Sigma$ is identically zero on $C_{+T}'$ and
$C_{-T}'$.  On the cylinders $\Sigma_i$ for $i>0$, define
\[
\beta_\Sigma(s,\tau) \eqdef \beta((-s+s_++T_+-r)/r).
\]
This is $1$ where $s\le s_++T_+-2r$ and $0$ where $s\ge s_++T_+-r$.
On the cylinders $\Sigma_i$ for $i<0$, define
\[
\beta_\Sigma(s,\tau) \eqdef\beta((s-s_-+T_--r)/r).
\]
This is $1$ where $s\ge s_--T_-+2r$ and $0$ where $s\le s_--T_-+r$.
On the rest of $\Sigma'$, define $\beta_\Sigma\eqdef 1$.

Finally, the deformation of $u_*$ is defined to be the map
\begin{equation}
\label{eqn:deformation}
\begin{split}
C_* & \longrightarrow \R\times Y,\\
x & \longmapsto e_*(x,\beta_-\psi_{-} + \beta_\Sigma\psi_\Sigma +
\beta_+\psi_{+}).
\end{split}
\end{equation}
This is well-defined, because under the above identifications,
$\psi_\Sigma$ defines a section of $E_*$ over the support of
$\beta_\Sigma$, while $\psi_\pm$ defines a section of $E_*$ over the
support of $\beta_\pm$.  If $\psi_\pm$ and $\psi_\Sigma$ are smooth,
then the map \eqref{eqn:deformation} is an immersion, except possibly
at the ramification points in $\Sigma$.

\subsection{Equation for the deformation to be $J$-holomorphic}
\label{sec:eqn7}

We now write an equation for the map \eqref{eqn:deformation} to be
$J$-holomorphic, for some complex structure on $C_*$.  This
equation will have the form
\begin{equation}
\label{eqn:7}
\beta_-\Theta_-(\psi_-,\psi_\Sigma) +
\beta_\Sigma\Theta_\Sigma(\psi_-,\psi_\Sigma,\psi_+) +
\beta_+\Theta_+(\psi_\Sigma,\psi_+)=0,
\end{equation}
where $\Theta_\pm$ is defined on all of $u_{\pm T}$, while
$\Theta_\Sigma$ is defined on all of $\Sigma$.

We begin by describing the schematic form of the $\Theta$'s.  By way
of preparation, let $C$ denote $C_{\pm T}$ or $\Sigma$, and let $E$
denote respectively the normal bundle to $C_{\pm T}$ or the trivial
complex line bundle over $\Sigma$.

\begin{definition}
\label{def:quadratic}
Let us call a $(T_\pm,\Sigma)$-dependent function
\[
F:C^\infty(C;E) \longrightarrow C^\infty(C;E\tensor_{\C} T^{0,1}C)
\]
``type 1 quadratic'' if it can be written in the form
\begin{equation}
\label{eqn:t1q}
F(\psi) = P(\psi) + Q(\psi) \cdot \nabla\psi,
\end{equation}
where $P$ and $Q$ are (nonlinear) bundle maps with uniformly bounded
derivatives to any given order in the fiber direction, obeying
$|P(\psi)| < c |\psi|^2$ and $|Q(\psi)| < c|\psi|$.  Let us call a
$(T_\pm,\Sigma)$-dependent function
\[
Z: C^\infty(C;E) \times C^\infty(C;E) \longrightarrow
C^\infty(C,E\tensor_{\C} T^{0,1}C)
\]
``type 2 quadratic'' if it can be written as
\[
Z(\psi_1,\psi_2) = a(\psi_1,\psi_2) + b_1(\psi_1,\psi_2) \cdot
\nabla\psi_2 + b_2(\psi_1,\psi_2) \cdot \nabla\psi_1,
\]
where $a$, $b_1$, $b_2$ are (nonlinear) bundle maps with uniformly
bounded derivatives to any given order in the fiber direction, obeying
$|a(\psi_1,\psi_2)| < c|\psi_1||\psi_2|$, $|b_1(\psi_1,\psi_2)| <
c|\psi_1|$, and $|b_2(\psi_1,\psi_2)| < c|\psi_2|$.
\end{definition}

Now let $\psi_\pm$ be a section of the normal bundle of $u_{\pm T}$
with $|\psi_\pm|<\delta_0$.  Then as in \S\ref{sec:NRIC}, the
composition of $\psi_\pm$ with the exponential map $e_\pm$ defines a
$J$-holomorphic map $C_{\pm T}\to \R\times Y$, for some complex
structure on $C_{\pm T}$, if and only if $D_\pm\psi_\pm +
F_\pm(\psi_\pm)=0$, where $D_\pm$ denotes the usual linear deformation
operator, while $F_\pm$ is type 1 quadratic.  If $\psi_\Sigma$ is a
complex function on $\Sigma$ with $|\psi_\Sigma|<4\delta_0$, then the
map $\Sigma\to\R\times Y$ sending $x\mapsto e(\pi(x),\psi_\Sigma(x))$
is $J$-holomorphic, for some complex structure on $\Sigma$, if and
only if it satisfies an equation of the form
$D_\Sigma\psi_\Sigma + F_\Sigma(\psi_\Sigma)=0$.  Here, as in
\S\ref{sec:NRIC}, the operator $D_\Sigma$ has the form
\begin{equation}
\label{eqn:operatorD}
D_\Sigma\psi =
\overline{\partial}\psi + (\nu\psi + \mu
\overline{\psi})\tensor d\zbar.
\end{equation}
(For more about the operator
$D_\Sigma$, see
\S{I.2.3}.)  Meanwhile, $F_\Sigma$ is type 1 quadratic, except near
the ramification points in $\Sigma$ (see \eqref{eqn:GNC} below).

With the preceding understood, $\Theta_-$ has the form
\begin{equation}
\label{eqn:Theta-}
\Theta_- = D_-\psi_- + F_-'(\psi_-) +
\frac{\beta_-}{2}
\frac{\partial\beta_\Sigma}{\partial s} \big((d\overline{z}
+ \frak{m}) \psi_\Sigma + \frak{z}_-(\psi_-,\psi_\Sigma)\big) +
\frak{q}_-\cdot\psi_- + \frak{q}_-'\cdot\nabla\psi_-.
\end{equation}
Here $F_-'$ is type 1 quadratic, and differs from $F_-$ only in the
part of each cylinder $\Sigma_i$ for $i<0$ where $s_i-2hr < s < s_i-hr$.
Next, $\frak{m}$ is a bundle map on each cylinder $\Sigma_i$ for $i<0$
satisfying $|\frak{m}|<c|\eta_{-T}|$, where $\eta_{-T} \eqdef
\eta_{i,T}$.  The function $\frak{z}_-$ is type 2 quadratic.
 Finally, $\frak{q}_{-}$ and $\frak{q}_{-}'$ are bundle maps supported
in the cylinders $\Sigma_i$ for $i<0$ where $s_i - 2hr < s < s_i - hr$,
which satisfy $|\frak{q}_{-}|,|\frak{q}_{-}'| < c|\eta_{-T}|$.

Likewise, $\Theta_+$ has the form
\begin{equation}
\label{eqn:Theta+}
\Theta_+ = D_+\psi_+ + F_+'(\psi_+) +
\frac{\beta_+}{2}
\frac{\partial\beta_\Sigma}{\partial s} \big((d\overline{z}
+ \frak{m}) \psi_\Sigma + \frak{z}_+(\psi_+,\psi_\Sigma)\big) +
\frak{q}_{+}\cdot\psi_+ + \frak{q}_{+}'\cdot\nabla\psi_+.
\end{equation}
The terms in \eqref{eqn:Theta+} satisfy the obvious analogues of the
conditions on the terms in \eqref{eqn:Theta-}.

Finally, $\Theta_\Sigma$ has the form
\begin{equation}
\begin{split}
\label{eqn:ThetaSigma}
\Theta_\Sigma = \; & D_\Sigma\psi_\Sigma + F_\Sigma'(\psi_\Sigma) +
\frak{q}_0\cdot\psi_\Sigma + \frak{q}_0'\cdot\nabla\psi_\Sigma\\
&  +
\frak{p}_-(\eta_{-T})
+ \beta_-\frak{z}_{0-}(\psi_-,\psi_\Sigma) + 
\frac{1}{2}\frac{\partial\beta_-}{\partial s}\big((\eta_{-T} +
\psi_-)d\overline{z} +
\frak{z}_{0-}'(\psi_-,\psi_\Sigma)\big) \\
& 
+
\frak{p}_+(\eta_{+T})
+ \beta_+\frak{z}_{0+}(\psi_+,\psi_\Sigma) + 
\frac{1}{2}
\frac{\partial\beta_+}{\partial s}\big((\eta_{+T} +
\psi_+)d\overline{z} +
\frak{z}_{0+}'(\psi_+,\psi_\Sigma)\big).
\end{split}
\end{equation}
Here $F_\Sigma'$ is type 1 quadratic (except near the ramification
points), and differs from $F_\Sigma$ only in the cylinders $\Sigma_i$
for $i>0$ where $s_++T_+-2r<s<s_++T_+-r$, and in the cylinders
$\Sigma_i$ for $i<0$ where $s_--T_-+r < s < s_--T_-+2r$.  The terms
$\frak{p}_-$, and $\frak{p}_+$ are type 1 quadratic; $\frak{p}_+$ is
supported in the cylinders $\Sigma_i$ for $i>0$ where $s_i+hr < s <
s_i+2hr$; and $\frak{p}_-$ is supported in the cylinders $\Sigma_i$
for $i<0$ where $s_i-2hr<s<s_i-hr$.  The $\frak{q}$'s can be written
as $\frak{q}_0=\frak{q}_{0-} +
\frak{q}_{0+}$ and $\frak{q}_0'=\frak{q}_{0-}' + \frak{q}_{0+}'$, where
$\frak{q}_{0+}$ and $\frak{q}_{0+}'$ are supported in the cylinders
$\Sigma_i$ for $i>0$ where $s_i+hr < s < s_++T_+-r+1$ and satisfy
$|\frak{q}_{0+}|,|\frak{q}_{0+}'| <c|\eta_{+ T}|$.  Likewise,
$\frak{q}_{0-}$ and $\frak{q}_{0-}'$ are supported in the cylinders
$\Sigma_i$ for $i<0$ where $s_--T_-+r-1 < s < s_i-hr$ and satisfy
$|\frak{q}_{0-}|,|\frak{q}_{0-}'| <c|\eta_{-T}|$.  The functions
$\frak{z}_{0\pm}$ and $\frak{z}_{0\pm}'$ are supported in the
cylinders $\Sigma_i$ for $\pm i>0$, and are type 2 quadratic.

We formulate the above as a lemma:

\begin{lemma}
There exist functions $\Theta_-$, $\Theta_+$, and $\Theta_\Sigma$, of
the form \eqref{eqn:Theta-}, \eqref{eqn:Theta+}, and
\eqref{eqn:ThetaSigma} respectively, such that the map
\eqref{eqn:deformation} is $J$-holomorphic for some complex structure
on $C_*$ if and only if equation \eqref{eqn:7} holds.
\end{lemma}

\begin{proof}
On $C_{\pm T}'$, equation \eqref{eqn:7} reads $D_\pm\psi_\pm +
F_\pm(\psi_\pm)=0$; and on $\Sigma'$, off of the cylinders
$\Sigma_i$, equation \eqref{eqn:7} reads
$D_\Sigma\psi_\Sigma + F_\Sigma(\psi_\Sigma)=0$.  Hence we need only
consider the cylinders $\Sigma_i$, w.l.o.g.\ with $i<0$.  Here, we need to
show that the various terms in $\Theta_-$ and $\Theta_\Sigma$ can be
chosen so that
\begin{equation}
\label{eqn:HWN}
D_\Sigma(\beta_-(\eta_{-T} + \psi_-) + \beta_\Sigma\psi_\Sigma) +
F_\Sigma(\beta_-(\eta_{-T} + \psi_-) + \beta_\Sigma\psi_\Sigma)
= \beta_-\Theta_- + \beta_\Sigma\Theta_\Sigma.
\end{equation}

To start, it follows from \eqref{eqn:operatorD} that
\begin{equation}
\label{eqn:DSigma}
\begin{split}
D_\Sigma(\beta_-(\eta_{-T} + \psi_-)+\beta_\Sigma\psi_\Sigma) = &
\beta_-\left(\frac{\beta_-}{2}\frac{\partial\beta_\Sigma}{\partial
s}\psi_\Sigma d\overline{z} +
D_\Sigma(\eta_{-T} + \psi_-)\right)\\
&
+
\beta_\Sigma\left(
\frac{1}{2}
\frac{\partial\beta_-}{\partial s}(\eta_{-T}+\psi_-)d\overline{z}
+ D_\Sigma\psi_\Sigma\right).
\end{split}
\end{equation}
Here we have inserted some extra factors of $\beta_-$ and
$\beta_\Sigma$, using the fact that $\beta_-=1$ on the support of
$\partial_s\beta_\Sigma$, and $\beta_\Sigma=1$ on the support of
$\partial_s\beta_-$.  Next,
\begin{equation}
\label{eqn:DVCR}
D_-\psi_- + F_-(\psi_-) = D_\Sigma(\eta_{-T}+\psi_-) +
F_\Sigma(\eta_{-T}+\psi_-),
\end{equation}
because by \eqref{eqn:ECC}, the two sides of
\eqref{eqn:DVCR} measure the failure of the same immersed surface to be
$J$-holomorphic.  By \eqref{eqn:DSigma} and \eqref{eqn:DVCR}, the
equation \eqref{eqn:HWN} that we need to prove reduces to
\begin{equation}
\label{eqn:HWN2}
\begin{split}
& F_\Sigma(\beta_-(\eta_{-T} + \psi_-) + \beta_\Sigma\psi_\Sigma) -
\beta_-F_\Sigma(\eta_{-T} + \psi_-) - \beta_\Sigma F_\Sigma(\psi_\Sigma)\\
& = \beta_-\left(
 F_-''(\psi_-) +
\frac{\beta_-}{2}\frac{\partial\beta_\Sigma}{\partial s}
\big(\frak{m} \psi_\Sigma + \frak{z}_-(\psi_-,\psi_\Sigma)\big) +
\frak{q}_-\cdot\psi_- + \frak{q}_-'\cdot\nabla\psi_-
\right)\\
& \quad\quad\quad + \beta_\Sigma\bigg(
F_\Sigma''(\psi_\Sigma) +
\frak{q}_0\cdot\psi_\Sigma +
\frak{q}_0'\cdot\nabla\psi_\Sigma\\
& \quad\quad\quad\quad\quad\quad\quad+
\frak{p}_-(\eta_{-T})
+ \beta_-\frak{z}_{0-}(\psi_-,\psi_\Sigma) +
\frac{1}{2}\frac{\partial\beta_-}{\partial s}\frak{z}_{0-}'(\psi_-,\psi_\Sigma)
\bigg).
\end{split}
\end{equation}
Here $F_-''\eqdef F_-'-F_-$ and $F_\Sigma''\eqdef F_\Sigma'-F_\Sigma$
are type 1 quadratic, and supported
in our cylinder $\Sigma_i$ where $s_i-2hr<s<s_i-hr$ and
$s_--T_-+r<s<s_--T_-+2r$ respectively.

To prepare for the proof of \eqref{eqn:HWN2}, first note that since
$F\eqdef F_\Sigma$ is type 1 quadratic, it follows that
\begin{equation}
\label{eqn:F1}
F(\psi_1+\psi_2) = F(\psi_1) + F(\psi_2) + F_1(\psi_1,\psi_2),
\end{equation}
where $F_1$ is type 2 quadratic.  This last condition
implies that
\begin{equation}
\label{eqn:F2}
F_1(\psi_1+\psi_2,\psi_3) = F_1(\psi_1,\psi_3) + F_1(\psi_2,\psi_3) +
F_2(\psi_1,\psi_2,\psi_3),
\end{equation}
where $F_2$ has the form
\begin{equation}
\nonumber
F_2(\psi_1,\psi_2,\psi_3) = a(\psi_1,\psi_2,\psi_3) + \sum_{i=1}^3
b_i(\psi_1,\psi_2,\psi_3)\cdot\nabla\psi_i
\end{equation}
with $|a(\psi_1,\psi_2,\psi_3)| < c|\psi_1||\psi_2||\psi_3|$,
$b_1(\psi_1,\psi_2,\psi_3)| < c |\psi_2||\psi_3|$, and so forth.

To prove \eqref{eqn:HWN2}, use \eqref{eqn:F1} and \eqref{eqn:F2} to
expand the left hand side as
\begin{equation}
\label{eqn:F}
\begin{split}
&F(\beta_-(\eta_{-T}+\psi_-) + \beta_\Sigma\psi_\Sigma) -
\beta_-F(\eta_{-T}+\psi_-) - \beta_\Sigma F(\psi_\Sigma)\\
&\quad\quad= F_1(\beta_-\eta_{-T},\beta_\Sigma\psi_\Sigma) 
\\
& \quad\quad\quad\quad + F_1(\beta_-\psi_-,\beta_\Sigma\psi_\Sigma) +
F_2(\beta_-\eta_{-T},\beta_-\psi_-,\beta_\Sigma\psi_\Sigma)
 \\
& \quad\quad\quad\quad\quad\quad + \big(F_1(\beta_-\eta_{-T},\beta_-\psi_-) -
\beta_-F_1(\eta_{-T},\psi_-)\big) 
\\
& \quad\quad\quad\quad\quad\quad\quad\quad
 + \big(F(\beta_-\eta_{-T}) - \beta_-F(\eta_{-T})\big)\\
& \quad\quad\quad\quad\quad\quad\quad\quad\quad\quad
+ \big(F(\beta_-\psi_-) -
\beta_-F(\psi_-)\big) \\
& \quad\quad\quad\quad\quad\quad\quad\quad\quad\quad\quad\quad
+ \big(F(\beta_\Sigma\psi_\Sigma) -
\beta_\Sigma F(\psi_\Sigma)\big).
\end{split}
\end{equation}
On the right side of \eqref{eqn:F}, the first line gives the
$\frak{m}$, $\frak{q}_0$, and $\frak{q}_0'$ terms on the
right side of
\eqref{eqn:HWN2}.  The second line gives the $\frak{z}_-$, $\frak{z}_0$,
and $\frak{z}_0'$ terms.  The third line gives the $\frak{q}_-$ and
$\frak{q}_-'$ terms.  The fourth line gives the $\frak{p}_-$ term, the
fifth line gives the $F_-''$ term, and the last line gives the
$F_\Sigma''$ term.  Here we have used the fact that $|\nabla\eta_{-T}| <
c|\eta_{-T}|$, which follows from the results in
\S\ref{sec:decay}.  Also, we are assuming that $r>rh>1$
 so that $|\partial_s\beta_-|, |\partial_s\beta_\Sigma| < c$.
\end{proof}

To complete the picture of equation \eqref{eqn:7}, we now describe the
behavior of the term $F_\Sigma'(\psi_\Sigma)$ in
\eqref{eqn:ThetaSigma} near a ramification point of the branched cover
$\pi:\Sigma\to\R\times S^1$.  Recall that
$F_\Sigma'(\psi_\Sigma)=F_\Sigma(\psi_\Sigma)$ near the ramification
points.  Also recall our coordinates $(z=s+it,w)$ on a neighborhood of
$\R\times\alpha$ in $\R\times Y$.  On the complement of the
ramification points, equations
\eqref{eqn:hol1} and \eqref{eqn:hol2} imply that
\begin{equation}
\label{eqn:GNC0}
F_\Sigma(\psi_\Sigma) =
a(t,\psi_\Sigma)\frac{\partial\psi_\Sigma}{\partial z}d\zbar +
P(t,\psi_\Sigma),
\end{equation}
where $a(t,w)$ and $P(t,w)$ are smooth functions of their arguments
which vanish where $w=0$.  Near a ramification point, choose a local
holomorphic coordinate $u$ on $\Sigma$ such that $\pi^*z=z_0+u^{q+1}$
with $q$ a positive integer.  It follows from \eqref{eqn:GNC0} that
near the ramification point,
\begin{equation}
\label{eqn:GNC}
F'_\Sigma(\psi_\Sigma) = F_\Sigma(\psi_\Sigma) =
\left(\frac{\ubar}{u}\right)^q a(t,\psi_\Sigma)
\frac{\partial\psi_\Sigma}{\partial u}d\ubar + P(t,\psi_\Sigma).
\end{equation}
In particular, $F'_\Sigma(\psi_\Sigma)$ is generally not continuous at
the ramification points, even when $\psi_\Sigma$ is smooth.

\subsection{Banach space setup}
\label{sec:BSS}

We now select appropriate Banach spaces to use in solving equation
\eqref{eqn:7}.

Let $C$ denote $C_\pm$ or $\Sigma$, and let $E$
denote respectively the normal bundle to $C_\pm$ or the trivial
complex line bundle on $\Sigma$.  Let $\mc{H}_0(C)$ denote the Banach
space obtained by completing the space of compactly supported sections
of $E\tensor T^{0,1}C$ using the norm $\|\cdot\|$ defined by
\[
\|\eta\| \eqdef \bigg(\int_C|\eta|^2\bigg)^{1/2} + \bigg(\sup_{x\in
C}\sup_{\rho\in(0,1]}\rho^{-\morrey}\int_{\text{dist}(x,\cdot)<\rho}
|\eta|^2\bigg)^{1/2}.
\]
Here we have chosen a metric on each $\Sigma\in\mc{M}$ as in \S I.2.3,
and we use the metric on $C_\pm$ induced by its immersion $u_\pm$ into
$\R\times Y$.  (Note that for our purposes, the exponent of $\rho$
above could be replaced by $-v$ for any $v\in(0,1)$.  We will fix
$v=\morrey$ for definiteness.)

Next, let $\mc{H}_1(C)$ denote the completion of the space of
compactly supported sections of $E$ using the norm $\|\cdot\|_*$
defined by
\[
\|\eta\|_* \eqdef \|\nabla\eta\| + \|\eta\|.
\]

\begin{lemma}
\label{lem:HC}
The tautological map $\psi \mapsto
\psi$ defines a bounded map from $\mc{H}_1(C)$ into the Banach space of
sections of $E$ that are Holder continuous with exponent
$\morreyovertwo$, and decay to zero on the ends of $C$.
\end{lemma}

\begin{proof}
This is a consequence of \cite[Thm.\ 3.5.2]{morrey}, together with the
fact that $C$ has bounded geometry.
\end{proof}

Now let $D_C:C^\infty(E) \to C^\infty(E\tensor T^{0,1}C)$ denote the
deformation operator $D_\pm$ when $C=C_\pm$, or the operator
$D_\Sigma$ when $C=\Sigma$.  This extends as a bounded operator from
$L^2_1(E)$ to $L^2(E\tensor T^{0,1}C)$, and also as a bounded operator
from $\mc{H}_1(C)$ to $\mc{H}_0(C)$.

\begin{lemma}
\label{lem:fredholm}
There is a positive constant $\gamma_C$, such that
\begin{equation}
\label{eqn:gammaC}
\|D_C\eta\| \ge \gamma_C\|\eta\|_*
\end{equation}
for all $\eta$ in $\mc{H}_1(C)$ that are $L^2$-orthogonal to the
kernel of $D_C$.
\end{lemma}

\begin{proof}
Our assumption that all Reeb orbits are nondegenerate guarantees that
the operator $D_C$ is Fredholm.  Since $D_C$ has closed range, there
exists $\gamma>0$ such that $\|D_C\eta\|_{2} \ge
\gamma\|\eta\|_{L^2_1}$ whenever $\eta\in L^2_1$ is orthogonal to the
kernel of $D_C$.  The lemma follows from this and \cite[Thm.\
5.4.1]{morrey}.
\end{proof}

Recall from Lemma I.2.15(b) that the operator $D_\Sigma$ has trivial
kernel.  Lemma~\ref{lem:fredholm} then finds a positive constant
$\gamma_\Sigma$ for each branched cover $\Sigma\in\mc{M}$, such that
$\|D_\Sigma\eta\|
\ge \gamma_\Sigma\|\eta\|_*$ for all $\eta\in \mc{H}_1(\Sigma)$.  We will
need a positive lower bound on $\{\gamma_\Sigma\}$ as $\Sigma$ varies
over all of $\mc{M}$, where the multiplicities
$(a_1,\ldots,a_{N_+}\mid a_{-1},\ldots,a_{-N_-})$ entering into the
definition of $\mc{M}$ are fixed.

\begin{lemma}
\label{lem:SB}
There exists a $\Sigma$-independent constant $\gamma>0$ such that for
any $\Sigma\in\mc{M}$ and $\eta\in
\mc{H}_1(\Sigma)$, we have $\|D_\Sigma\eta\| \ge \gamma\|\eta\|_*$.
\end{lemma}

\begin{proof}
First observe that it is enough to find $\gamma>0$ with
$\|D_\Sigma\eta\|_{2} \ge \gamma\|\eta\|_{L^2_1}$ for all
$\Sigma\in\mc{M}$ and $\eta\in L^2_1(\pi^*N)$.  This follows from the
proof of \cite[Thm.\ 5.4.1]{morrey}, because that argument uses only
the local geometry in discs of radius 1, and the local geometry is
uniformly controlled over all branched covers.

Now suppose that there does not exist $\gamma>0$ such that
$\|D_\Sigma\eta\|_{2} \ge \gamma\|\eta\|_{L^2_1}$ for all $\Sigma$ and
$\eta$.  Then we can find a sequence of branched covers
$\{\Sigma_k\}_{k=1,2,\ldots}$, and for each $\Sigma_k$ an element
$\eta_k$ of the domain of the corresponding operator $D_{\Sigma_k}$, such that
$\|D_{\Sigma_k}\eta_k\|_{2}\le 1/k$ and $\|\eta_k\|_{2} +
\|\nabla\eta_k\|_{2}=1$.  For any $\Sigma$, since $D_{\Sigma}$ is a
first order elliptic operator, there is a constant $b$ such that
$\|\nabla\eta\|_{2} \le b\left(\|D_{\Sigma}\eta\|_{2} +
\|\eta\|_{2}\right)$ for all $\eta\in L^2_1(\pi^*N)$.  The constant
$b$ can be chosen independently of $\Sigma$ because it depends only on
the local geometry of the branched cover.  Thanks to the existence of
$b$, we can choose a new sequence $\{(\Sigma_k,\eta_k)\}$ such that
$\|D_{\Sigma_k}\eta_k\|_2\le 1/k$ and $\|\eta_k\|_2 = 1$.

By Lemma~{I.2.28}, we can pass to a subsequence so that the sequence of
branched covers $\{\Sigma_k\}$ converges in $\mc{M}/\R$, in the sense of
Definition I.2.27, to a tree $T$ together with a branched cover
$\Sigma_{*j}$ for each internal vertex $j$ of $T$.  By a standard
compactness argument using a priori elliptic estimates, we can pass to
a further subsequence so that for each $j$, the sequence $\eta_k$,
after suitable translations of the $s$ coordinate, converges to a
function $\eta_{*j}$ on $\Sigma_{*j}$.  The function $\eta_{*j}$ is in
the kernel of $D_{\Sigma_{*j}}$, but we know that the latter operator
has trivial kernel, so $\eta_{*j}=0$ for each $j$.  We conclude that
when $k$ is large, all but a small amount of the $L^2$ norm of
$\eta_k$ comes from subcylinders in $\Sigma_k$ that project to long
cylinders in $\R\times S^1$ and that are far away from any
ramification points.

To get a contradiction from this, note that if $\lambda$ is compactly
supported on a cylinder in any $\Sigma\in\mc{M}$ which projects to a
cylinder of the form $(a,b)\times S^1$ in $\R\times S^1$, then there
is a constant $\gamma'>0$, depending only on $\mc{M}$, such that
$\|D_\Sigma\lambda\|\ge \gamma'\left(\|\nabla\lambda\|_2 +
\|\lambda\|_2\right)$.  This follows by expanding $\lambda$ in terms
of eigenfunctions of the operator $L_m$, where $m$ denotes the
covering multiplicity of the cylinder.  Granted this bound,
multiplication of $\eta_k$ by suitable cutoff functions shows that
long tubes as above cannot account for most of its $L^2$ norm.
\end{proof}

\subsection{Solving for $\psi_-$ and $\psi_+$ in terms of $\psi_\Sigma$}
\label{sec:solve-}

Our strategy for solving equation \eqref{eqn:7} is to solve the three
equations
\begin{align}
\label{eqn:7-}
\Theta_-(\psi_-,\psi_\Sigma) &= 0\quad\quad\mbox{on all of $u_{-T}$,}\\
\label{eqn:7+}
\Theta_+(\psi_\Sigma,\psi_+) &= 0\quad\quad\mbox{on all of $u_{+T}$,}\\
\label{eqn:7Sigma}
\Theta_\Sigma(\psi_-,\psi_\Sigma,\psi_+) &= 0 \quad\quad
 \mbox{on all of $\Sigma$.}
\end{align}
More precisely, let $\mc{H}_\Sigma\eqdef \mc{H}_1(\Sigma)$, and let
$\mc{H}_\pm$ denote the orthogonal complement of $\Ker(D_\pm)$ in
$\mc{H}_1(u_{\pm T})$; we will solve the above equations for
$\psi_\pm\in\mc{H}_\pm$ and $\psi_\Sigma\in\mc{H}_\Sigma$.

Given $\psi_\Sigma$, we now explain how to solve equations
\eqref{eqn:7-} and \eqref{eqn:7+} for $\psi_-$ and $\psi_+$
respectively. Later, we will plug the results into \eqref{eqn:7Sigma}
in order to view \eqref{eqn:7Sigma} as an equation for $\psi_\Sigma$
alone.

To prepare for subsequent estimates, for $i=1,\ldots,\Nbar_+$ define
$\lambda_i$ to be the largest negative eigenvalue of the asymptotic
operator $L_{a_i}$, and define $\lambda_+ \eqdef
\min\{|\lambda_i|\}_{i=1,\ldots,\Nbar_+}$.  Likewise, for
$i=-1,\ldots,-\Nbar_-$ let $\lambda_i$ denote the smallest positive
eigenvalue of $L_{a_i}$, and define $\lambda_- \eqdef
\min\{\lambda_i\}_{i=-1,\ldots,-\Nbar_-}$.

\begin{proposition}
\label{prop:CMT}
Fix $h\in(0,1)$.  There exist constants $r_0>h^{-1}$ and
$\varepsilon,c>0$ such that the following holds: Fix $r>r_0$ and
$T_-,T_+ \ge 5r$.  Fix $\Sigma\in\mc{M}$, and let $\mc{B}_\Sigma$ denote
the ball of radius $\varepsilon$ in $\mc{H}_\Sigma$.  Then:
\begin{description}
\item{(a)} There exist maps $\psi_-$ and $\psi_+$, from
$\mc{B}_\Sigma$ into the radius $\varepsilon$ balls in $\mc{H}_-$ and
$\mc{H}_+$ respectively, such that $\psi_-=\psi_-(\psi_\Sigma)$ solves
\eqref{eqn:7-} and $\psi_+=\psi_+(\psi_\Sigma)$ solves \eqref{eqn:7+}.
\item{(b)}
$\psi_\pm$, when identified with a section of the
normal bundle to the untranslated curve $u_\pm$, varies
smoothly as $(T_-,T_+)$ and $\psi_\Sigma$ are varied.
\item{(c)}
$\|\psi_\pm(\psi_\Sigma)\|_* \le c r^{-1}\|\psi_\Sigma\|_*$.
\item{(d)}
The derivative of $\psi_\pm$ at a point $\psi_\Sigma\in \mc{B}_\Sigma$
defines a bounded linear functional $\mc{D}:\mc{H}_\Sigma \to
\mc{H}_\pm$ obeying
\[
\|\mc{D}\eta\|_* \le c r^{-1}\|\eta\|_*.
\]
\end{description}
\end{proposition}

\begin{proof}
(a) We will explain how to solve equation \eqref{eqn:7-} for $\psi_-$;
 an analogous procedure solves equation \eqref{eqn:7+} for $\psi_+$.
 To start, use \eqref{eqn:Theta-} to rewrite \eqref{eqn:7-} as
\begin{equation}
\label{eqn:7-'}
D_-\psi_- + \mc{F}_0(\psi_\Sigma) + \mc{F}_1(\psi_-,\psi_\Sigma) = 0,
\end{equation}
where
\begin{equation}
\label{eqn:F01}
\begin{split}
\mc{F}_0(\psi_\Sigma) & \eqdef
\frac{\beta_-}{2}
\frac{\partial\beta_\Sigma}{\partial s} (d\overline{z}
+ \frak{m}) \psi_\Sigma,\\
\mc{F}_1(\psi_-,\psi_\Sigma) & \eqdef
F_-'(\psi_-) +
\frac{\beta_-}{2}
\frac{\partial\beta_\Sigma}{\partial s} \frak{z}_-(\psi_-,\psi_\Sigma) +
\frak{q}_-\cdot\psi_- + \frak{q}_-'\cdot\nabla\psi_-.
\end{split}
\end{equation}
By virtue of Lemma~\ref{lem:HC}, there exists $\varepsilon>0$ such
that if $\psi_\Sigma\in \mc{H}_\Sigma$ satisfies
$\|\psi_\Sigma\|_*<\varepsilon$ then $|\psi_\Sigma|<\delta_0$; and if
$\psi_-\in \mc{H}_1(u_{-T})$ satisfies
$\|\psi_-\|_*<\varepsilon$ then $|\psi_-|<\delta_0$.  So if
$\|\psi_\Sigma\|_*<\varepsilon$, then $\mc{F}_0(\psi_\Sigma)\in
\mc{H}_0(u_{-T})$ is defined,
and $\mc{F}_1(\cdot,\psi_\Sigma)$ defines a smooth map from the radius
$\varepsilon$ ball in $\mc{H}_1(u_{-T})$ to $\mc{H}_0(u_{-T})$.

To solve \eqref{eqn:7-'}, we will apply the contraction mapping
theorem to a map $\mc{I}$ defined as follows.
Lemma~\ref{lem:fredholm} implies that $D_-$ has a bounded inverse
$D_-^{-1}:\mc{H}_0(u_{-T})\to\mc{H}_-$.  Consequently, for fixed
$\psi_\Sigma$ with $\|\psi_\Sigma\|_* <
\varepsilon$, the assigment
\begin{equation}
\label{eqn:CM}
\psi_- \longmapsto \mc{I}(\psi_-) \eqdef -D_-^{-1}(\mc{F}_0(\psi_\Sigma) +
\mc{F}_1(\psi_-,\psi_\Sigma))
\end{equation}
defines a smooth map from the radius $\varepsilon$ ball in
$\mc{H}_-$ to $\mc{H}_-$.

{\em Claim:\/} If $r$ and $T_-$ are sufficiently large and
$\varepsilon>0$ is sufficiently small, then the map $\mc{I}$ sends the
radius $\varepsilon$ ball in $\mc{H}_-$ to itself as a contraction
mapping with
\begin{equation}
\label{eqn:ICM}
\|\mc{I}(\psi_-) - \mc{I}(\psi_-')\|_* \le
\frac{1}{2} \|\psi_- - \psi_-'\|_*.
\end{equation}

{\em Proof of claim:\/}
It follows from the definition of $\beta_\Sigma$ that $|\partial_s
\beta_\Sigma| < c r^{-1}$, and so by \eqref{eqn:F01} we have
\[
\|\mc{F}_0(\psi_\Sigma)\|
\le c r^{-1} \|\psi_\Sigma\|_*.
\]
By Lemma~\ref{lem:HC} and the fact that $F_-'$ is type 1 quadratic, we
have
\[
\|F_-'(\psi_-)\|\le c
\|\psi_-\|_*^2.
\]
By Lemma~\ref{lem:HC} and the fact that $\frak{z}_-$ is type 2
quadratic, we have
\[
\|\frak{z}_-(\psi_-,\psi_\Sigma)\| \le
c\|\psi_-\|_*\|\psi_\Sigma\|_*.
\]
By the decay estimates on the ends of $u_-$ from
\S\ref{sec:decay}, we have
\[
\begin{split}
\|\frak{q}_-\cdot\psi_- + \frak{q}_-'\cdot\nabla\psi_-\| &\le c
\sum_{i=-1}^{-\Nbar_-}\exp(-\lambda_i(s_i - s_- + T_--2rh))\|\psi_-\|_* \\
& \le c\exp(-\lambda_-T_-/2)\|\psi_-\|_*.
\end{split}
\]
Since $D_-^{-1}$ is a bounded operator, the above estimates imply that
\begin{equation}
\label{eqn:IEstimate}
\|\mc{I}(\psi_-)\|_* \le c\left(\|\psi_-\|_*^2
+ r^{-1}(1+\|\psi_-\|_*)\|\psi_\Sigma\|_* + \exp(-\lambda_-
T_-/2)\|\psi_-\|_*.\right).
\end{equation}
If $\varepsilon$ is sufficiently small and if $r$ and $T_-$ are
sufficiently large, then the right hand side of
\eqref{eqn:IEstimate} is less than $\varepsilon$ whenever
$\|\psi_-\|_*,\|\psi_\Sigma\|_* < \varepsilon$.

We now prove the contraction property.  Since $F_-'$ is type 1
quadratic,
\[
\|F_-'(\psi_-) - F_-'(\psi_-')\| \le c\left(\|\psi_-\|_* +
\|\psi_-\|_*^2\right) \|\psi_- - \psi_-'\|_*.
\]
Since $\frak{z}_-$ is type 2 quadratic,
\begin{equation}
\label{eqn:z-CM}
\|\frak{z}_-(\psi_-,\psi_\Sigma) - \frak{z}_-(\psi_-',\psi_\Sigma)\|
\le c\left(\|\psi_\Sigma\|_* + \|\psi_-\|_*\|\psi_\Sigma\|_*
\right)\|\psi_- - \psi_-'\|_*.
\end{equation}
Thus for $\psi_-\neq\psi_-'$,
\begin{equation}
\label{eqn:contraction}
\begin{split}
\frac{\|\mc{I}(\psi_-) -
\mc{I}(\psi_-')\|}{\|\psi_- - \psi_-'\|_*}
\le &
c\big(
\|\psi_-\|_* +
\|\psi_-\|_*^2
+ r^{-1}\|\psi_\Sigma\|_*(1+\|\psi_-\|_*)\\
& \quad+ \exp(-\lambda_-T_-/2)
\big).
\end{split}
\end{equation}
If $\varepsilon$ is sufficiently small and if
$r$ and $T_-$ are sufficiently large, then the right hand side of
\eqref{eqn:contraction} is less than $\frac{1}{2}$ whenever $\|\psi_-\|_*,
\|\psi_\Sigma\|_* < \varepsilon$.

This completes the proof of the claim.  Part (a) of the proposition
now follows from the contraction mapping theorem.

(b) Smoothness of the map $\psi_-$ follows from smoothness of the maps
$\mc{F}_0$ and $\mc{F}_1$ used to define the contraction mapping
\eqref{eqn:CM}.

(c) By the estimate \eqref{eqn:IEstimate}, a fixed
point $\psi_-$ of \eqref{eqn:CM} satisfies
\begin{equation}
\label{eqn:psi-Estimate}
\|\psi_-\|_* \le cr^{-1}\|\psi_\Sigma\|_* + c\|\psi_-\|_*\left(
\|\psi_-\|_* + r^{-1}\|\psi_\Sigma\|_* + \exp(-\lambda_-T_-/2)\right).
\end{equation}
Recall that $\|\psi_\Sigma\|_*,\|\psi_-\|_*<\varepsilon$.  So if
$\varepsilon$ is sufficiently small and if $r$ and $T_-$ are
sufficiently large, then the sum in parentheses on the right hand side
of \eqref{eqn:psi-Estimate} is less than $c^{-1}/2$.

(d) Regard the right hand side of \eqref{eqn:CM} as a
function of both $\psi_-$ and $\psi_\Sigma$, and let $\mc{D}_-$
and $\mc{D}_\Sigma$ denote its derivatives with respect to
$\psi_-$ and $\psi_\Sigma$.  Then the derivative of $\psi_-$ as a
function of $\psi_\Sigma$ is given by
\begin{equation}
\label{eqn:mcD}
\mc{D} = (1-\mc{D}_-)^{-1}\mc{D}_\Sigma.
\end{equation}
By \eqref{eqn:ICM}, if $\varepsilon$ is sufficiently small and if $r$
and $T_-$ are sufficiently large, then the operator $\mc{D}_-$ has
norm less than $1/2$.  On the other hand, by the analogue of
\eqref{eqn:z-CM} in which the roles of $\psi_-$ and $\psi_\Sigma$ are
switched, we have
\[
\|\mc{D}_\Sigma\eta\|_* \le cr^{-1} \left(1 + \|\psi_-\|_* +
\|\psi_\Sigma\|_* \|\psi_-\|_*\right)\|\eta\|_*.
\]
Putting these estimates into \eqref{eqn:mcD} completes the proof.
\end{proof}

\subsection{Solving for $\psi_\Sigma$}

Let $h$, $r$, $T_-$, $T_+$ be as in Proposition~\ref{prop:CMT}.  Fix
$\Sigma\in\mc{M}$.  We now solve equation
\eqref{eqn:7Sigma} for $\psi_\Sigma\in \mc{B}_\Sigma$.
To start, write equation \eqref{eqn:7Sigma} as
\begin{equation}
\label{eqn:7SF}
D_\Sigma\psi_\Sigma + \mc{F}_{\Sigma}(\psi_\Sigma) = 0,
\end{equation}
where $\mc{F}_{\Sigma}(\psi_\Sigma)$ denotes the sum of the terms other than
$D_\Sigma\psi_\Sigma$ on the right hand side of
\eqref{eqn:ThetaSigma}.  Here Proposition~\ref{prop:CMT} is used to view
$\psi_-$ and $\psi_+$ as functions of $\psi_\Sigma$.

Equation \eqref{eqn:7SF} cannot be treated in the same way as equation
\eqref{eqn:7-'}, because the operator $D_\Sigma$ has a nontrivial
cokernel.  To deal with this issue, introduce the $L^2$-orthogonal
projection $\Pi$ from $L^2(\pi^*N\tensor T^{0,1}\Sigma)$ onto
$\Ker(D_\Sigma^*)$.  Equation
\eqref{eqn:7SF} is then equivalent to the two equations
\begin{gather}
\label{eqn:7SF1}
D_\Sigma\psi_\Sigma + (1-\Pi)\mc{F}_{\Sigma}(\psi_\Sigma) = 0,\\
\label{eqn:7SF2}
\Pi\mc{F}_{\Sigma}(\psi_\Sigma)=0.
\end{gather}
We now solve the first of these two equations.

\begin{proposition}
\label{prop:TRE}
Fix $h\in(0,1)$.  There exist constants $r_0>h^{-1}$ and
$\varepsilon>0$ such that when $r>r_0$ and $T_+,T_-\ge 5r$, the following
is true.  Fix $\Sigma\in\mc{M}$, and let $\mc{B}_\Sigma$ denote the
ball of radius $\varepsilon$ in $\mc{H}_\Sigma$.  Then:
\begin{description}
\item{(a)}
There exists a unique $\psi_\Sigma\in \mc{B}_\Sigma$ satisfying equation
\eqref{eqn:7SF1}.
\item{(b)}
This $\psi_\Sigma$ satisfies
\[
\begin{split}
\|\psi_\Sigma\|_* & < 
c \left(\sum_{i=1}^{\Nbar_+}\exp(-|\lambda_i|(s_+-s_i + T_+ -
2rh))\right. \\ &
\quad\;\; +
\left.\sum_{i=-1}^{-\Nbar_-} \exp(-\lambda_i(s_i - s_- + T_- - 2rh))\right)\\
& < c(\exp(-\lambda_- T_-/2) + \exp(-\lambda_+ T_+/2)).
\end{split}
\]
\item{(c)}
This $\psi_\Sigma$
defines a Lipschitz section of $\pi^*N$ which
is smooth except possibly at the ramification points of $\pi$.
\end{description}
\end{proposition}

\begin{proof}
To prove part (a), we apply the contraction mapping theorem to the map
 $\mc{I}:\mc{B}_\Sigma\to\mc{H}_\Sigma$ defined as follows.  Recall that
 the kernel of $D_\Sigma$ is trivial.  Thus it makes sense to define
\begin{equation}
\label{eqn:FCM}
\mc{I}(\psi_\Sigma) \eqdef
-D_\Sigma^{-1}(1-\Pi)\mc{F}_{\Sigma}(\psi_\Sigma),
\end{equation}
under the assumptions of Proposition~\ref{prop:CMT}.

To get estimates on $\mc{I}$, first recall from Lemma~\ref{lem:SB} that
there is a $\Sigma$-independent upper bound on the operator norm of
$D_\Sigma^{-1}$.  Next, given $\psi_\Sigma\in\mc{B}_\Sigma$, we claim
that on the $s\ge s_--T_-+2r$ part of the positive ends of $u_{-T}$,
the corresponding section $\psi_-$ satisfies
\begin{equation}
\label{eqn:pointwise-}
|\psi_-| + |\nabla\psi_-| \le c_- r^{-1} \|\psi_\Sigma\|_*
\exp(-\lambda_-(s-(s_--T_-+2r))),
\end{equation}
where $c_-$ depends only on $u_-$.  This follows from decay estimates as in
Lemma~\ref{lem:IDE}, together with Proposition~\ref{prop:CMT}(c).
Likewise, there is a constant $c_+,$ depending only on $u_+$, such that
on the $s\le s_+ + T_+ - 2r$ part of the negative ends of $u_{+T}$, we have
\begin{equation}
\label{eqn:pointwise+}
|\psi_+| + |\nabla\psi_+| \le c_+ r^{-1} \|\psi_\Sigma\|_*
 \exp(-\lambda_+((s_++T_+-2r) - s)).
\end{equation}
Estimating the individual terms in $\mc{I}$ as in the proof of
Proposition~\ref{prop:CMT}, and using Proposition~\ref{prop:CMT}(c)
together with \eqref{eqn:pointwise-} and \eqref{eqn:pointwise+}, we
find that
\begin{equation}
\label{eqn:FSEE}
\begin{split}
\|\mc{I}(\psi_\Sigma)\|_* & 
\le c\left(\|\psi_\Sigma\|_*^2 +
e^{-\lambda r}\|\psi_\Sigma\|_* \right.\\
&\quad\quad\quad\left.  + 
\sum_{i=1}^{\Nbar_+} e^{-|\lambda_i|(s_+-s_i + T_+ - 2rh)} 
+ \sum_{i=-1}^{-\Nbar_-} e^{-\lambda_i(s_i - s_- + T_-
- 2rh)}\right)\\ & \le c\left(\|\psi_\Sigma\|_*^2 + e^{-\lambda
r}\|\psi_\Sigma\|_* + e^{-\lambda_- T_-/2} + e^{-\lambda_+
T_+/2}\right)
\end{split}
\end{equation}
where $\lambda\eqdef\min\{\lambda_-,\lambda_+\}$.

It follows from \eqref{eqn:FSEE} that if $\varepsilon$ is sufficiently
small, if $r$ is sufficiently large, and if $T_-$ and $T_+$ are
sufficiently large with respect to $\varepsilon$, then $\mc{I}$ maps
$\mc{B}_\Sigma$ to itself.  Now if $\psi_\Sigma$ and $\psi_\Sigma'$
are distinct elements of $\mc{B}_\Sigma$, then using
Proposition~\ref{prop:CMT}(c),(d), assuming that $\varepsilon<1$, we
find that there is a constant $c$ with
\[
\frac{\|\mc{I}(\psi_\Sigma) - \mc{I}(\psi_\Sigma')\|_*}{\|\psi_\Sigma
- \psi_\Sigma'\|_*} \le c\left(\|\psi_\Sigma\|_* + r^{-1}\right).
\]
So $\mc{I}$ is a contraction mapping on ${B}_\Sigma$ provided that
$\varepsilon$ is sufficiently small and $r$ is sufficiently large.
Then $\mc{I}$ has a unique fixed point in ${B}_\Sigma$, which by
definition satisfies \eqref{eqn:7SF1}.

Part (b) follows from \eqref{eqn:FSEE} provided that $\varepsilon$ is
sufficiently small and $r$ is sufficiently large.

The proof of part (c) is deferred to \S\ref{sec:lipschitz}.
\end{proof}

\subsection{The obstruction section and the gluing map}
\label{sec:OSGM}

We now put the results of this section together.
Fix $h\in(0,1)$;  let $r_0,\varepsilon$ be as in
Proposition~\ref{prop:TRE}, and fix $r>r_0$.

\begin{definition}
Given $T_-,T_+\ge 5r$ and $\Sigma\in\mc{M}$, define the
``$(T_-,T_+)$-gluing along $\Sigma$'', denoted by $u(T_-,T_+,\Sigma)$,
to be the deformed pregluing \eqref{eqn:deformation}, where
$\psi_\Sigma$ is given by Proposition~\ref{prop:TRE} and $\psi_\pm$
are given by Proposition~\ref{prop:CMT}.
\end{definition}

Let $\mc{O}\to\times_2[5r,\infty)\times \mc{M}$ denote the
pullback of the obstruction bundle from
\S{I.2.3}.  This means that the fiber over $(T_-,T_+,\Sigma)$ is
\[
\mc{O}_{(T_-,T_+,\Sigma)} = \Hom(\Coker(D_\Sigma),\R).
\]

\begin{definition}
Define a section $\frak{s}:\times_2[5r,\infty)\times\mc{M}\to\mc{O}$
as follows: If $\sigma\in\Coker(D_\Sigma)$, then
\begin{equation}
\label{eqn:defs}
\frak{s}(T_-,T_+,\Sigma)(\sigma) \eqdef
\left\langle\sigma,
\mc{F}_{\Sigma}(\psi_\Sigma)
\right\rangle,
\end{equation}
where $\psi_\Sigma$ is the solution to \eqref{eqn:7SF1} given by
Proposition~\ref{prop:TRE}.
\end{definition}

Note that under the identification
$\Hom(\Coker(D_\Sigma),\R)\simeq\Coker(D_\Sigma)$ given by the inner
product, we have
\[
\frak{s}(T_-,T_+,\Sigma) =
\Pi\mc{F}_{\Sigma}(\psi_\Sigma).
\]
Thus by \eqref{eqn:7SF2}, $u(T_-,T_+,\Sigma)$ is $J$-holomorphic if
and only if $\frak{s}(T_-,T_+,\Sigma)=0$.

Recall that $\alpha_+$ denotes the list of Reeb orbits corresponding to the
positive ends of $U_+$, and $\alpha_-$ denotes the list of Reeb
orbits corresponding to the negative ends of $U_-$.  We now have a
well-defined ``gluing map''
\begin{equation}
\label{eqn:gluingMap}
G: \frak{s}^{-1}(0) \longrightarrow \mc{M}^J(\alpha_+,\alpha_-),
\end{equation}
sending $(T_-,T_+,\Sigma)\in\frak{s}^{-1}(0)$ to $u(T_-,T_+,\Sigma)$.

The next two sections sections establish important properties of the
obstruction section $\frak{s}$ and the gluing map $G$.

%% file: obg-properties.tex
Continue with the gluing setup from \S\ref{sec:glP}.  In this section
we prove that the obstruction section $\frak{s}$ defined in
\S\ref{sec:OSGM} is continuous.  We also show that the
restriction of $\frak{s}$ to the set of triples $(T_-,T_+,\Sigma)$,
such that the branched cover $\Sigma$ has only simple ramification
points, is smooth.  Finally, we show that if $J$ is generic, then all
zeroes of $\frak{s}$ occur in the latter set.

\subsection{Proof that $\psi_\Sigma$ is Lipschitz}
\label{sec:lipschitz}

We begin with the previously deferred:

\begin{proof}[Proof of Proposition~\ref{prop:TRE}(c).]
 Smoothness of $\psi_\Sigma$ off of the ramification points follows
by standard elliptic bootstrapping using \eqref{eqn:7SF1}.

To prove that $\psi_\Sigma$ is Lipschitz near a given ramification
point, identify a neighborhood of the ramification point in $\Sigma$
with a disc of radius $2\rho>0$ in $\C$, via a holomorphic local
coordinate $u$ for which $\pi^*z = z_0 + u^{q+1}$.  On this disc, equation
\eqref{eqn:7SF1} asserts that
\[
D_\Sigma\psi_\Sigma + F_\Sigma(\psi_\Sigma) = \eta,
\]
where $\eta$ is some element of $\Ker(D_\Sigma^*)$, and therefore
smooth.  It follows using \eqref{eqn:GNC} that $\psi_\Sigma$ obeys an
equation of the form
\[
\frac{\partial\psi_\Sigma}{\partial \overline{u}} +
\left(\frac{\overline{u}}{u}\right)^q a(u,\psi_\Sigma)\frac{\partial
\psi_\Sigma}{\partial u} + g(u,\psi_\Sigma) = 0,
\]
where $a(u,w)$ is the function depicted in \eqref{eqn:GNC}, and
$g(u,w)$ is a smooth function of its arguments.

To simplify the above equation, let $a_0\eqdef a(0,\psi_\Sigma(0))$,
and introduce a Lipschitz change of coordinates to
\begin{equation}
\label{eqn:LCC}
{v} \eqdef u\left(1 -
a_0\left(\frac{\overline{u}}{u}\right)^{q+1}\right)^{1/(q+1)}.
\end{equation}
This change of coordinates is invertible if $|a_0|$ is small, which we
can arrange by taking $\varepsilon$ small in
Proposition~\ref{prop:TRE}.

We now prove that $\psi_\Sigma$ is Lipschitz at ${v}=0$.  To do so,
define a function $\lambda$ of ${v}$, defined on the disc of radius
$\rho$, by writing
\[
\psi_\Sigma(u) = \lambda({v}(u)) + \psi_\Sigma(0) -
g(0,\psi_\Sigma(0))\overline{u}.
\]
Application of the chain rule finds that $\lambda$ obeys an
equation of the form
\begin{equation}
\label{eqn:ACR}
\frac{\partial\lambda}{\partial\overline{{v}}} +
\frak{f}_1\frac{\partial \lambda}{\partial {v}} +
\frak{f}_0=0.
\end{equation}
Here each of $\frak{f}_1$ and $\frak{f}_0$ can be written as a product
$\frak{m}({v})\cdot\frak{b}({v},\lambda({v}))$, where $|\frak{m}|\le
c$, while $\frak{b}$ is a smooth function of two variables with
$\frak{b}(0,0)=0$, whose derivatives have bounds depending only on
$u_+$, $u_-$, and $\eta$.

As a consequence of \eqref{eqn:ACR}, for $|{v}|<\rho$ the function
$\lambda$ can be written as
\begin{equation}
\label{eqn:LHF}
\lambda({v}) = \frac{1}{\pi}\int_{|x|<\rho} \frac{1}{{v}-x}
\left(\frak{f}_1
\frac{\partial\lambda}{\partial {v}} + \frak{f}_0\right)\bigg|_{x}
+ \lambda_0({v}),
\end{equation}
where $\lambda_0$ is a holomorphic function with
$|d\lambda_0({v})|<c\rho^{-1}$ for $|{v}|<\rho/2$.

Recall from Lemma~\ref{lem:HC} that $\psi_\Sigma$ is a Holder
continuous function of $u$.  Since $\lambda(0)=0$, it follows that
there exists $\sigma'\in(0,1]$ such that $|\lambda({v})| \le
c_{\sigma'}|{v}|^{\sigma'}$ where $c_{\sigma'}$ is a constant.  Let
$\sigma$ denote the supremum of the set of such $\sigma'$.  Then for
$\sigma'\in(0,\sigma)$, we know that $|\frak{f}_1|,
|\frak{f}_0| \le cc_{\sigma'}|{v}|^{\sigma'}$.

To bound $|\lambda({v})|$, fix $\epsilon\in(0,\sigma)$.  Since
$\lambda(0)=0$, we can subtract the instances of equation
\eqref{eqn:LHF} for $\lambda({v})$ and $\lambda(0)$ to find that if
$|{v}|<\rho/2$, then
\begin{equation}
\label{eqn:ITBE}
|\lambda({v})| \le c c_{\sigma-\epsilon} \int_{|x|<\rho}
 \frac{|{v}|}{|{v}-x||x|} |x|^{\sigma-\epsilon}
 |d\lambda| +
 c\rho^{-1}|{v}|.
\end{equation}

To bound the integral in \eqref{eqn:ITBE}, first consider the
integral over an annulus $2^n|{v}| < |x| \le 2^{n+1}|{v}|$, where $n$ is a
positive integer with $2^n|{v}|<\rho$.  The contribution from this
annulus is at most
\[
c |{v}|(2^n|{v}|)^{-2+\sigma-\epsilon}\int_{2^n|{v}|<|x|
 \le 2^{n+1}|{v}|}
 |d\lambda|,
\]
which by Holder's inequality is
\begin{equation}
\label{eqn:BHI}
\le
c 
 |{v}|(2^n|{v}|)^{-1+\sigma-\epsilon}\left(\int_{2^n|{v}|<|x| \le
 2^{n+1}|{v}|} |d\lambda|^2\right)^{1/2}.
\end{equation}
It follows from the definition of the norm on $\mc{H}_\Sigma$ that the
integral in \eqref{eqn:BHI} is at most $(2^{n+1}|{v}|)^\morrey
\|\psi_\Sigma\|_*$.  As a consequence, \eqref{eqn:BHI} is no greater
than $c
(2^n)^{\morreyovertwominusone+\sigma-\epsilon}
|{v}|^{\sigma-\epsilon+\morreyovertwo}$.
Summing up these annular contributions, we find that the contribution
to the integral in \eqref{eqn:ITBE} from the region where $|x|>2|{v}|$
is at most
\[
c |{v}|^{\sigma-\epsilon+\morreyovertwo}\sum_{n\ge 1, \;\;
2^n|{v}|<\rho} (2^\kappa)^n,\quad\quad\mbox{where $\kappa \eqdef
\morreyovertwominusone + \sigma - \epsilon$.}
\]

To bound the integral in \eqref{eqn:ITBE} over the region where
$|x|<|{v}|/2$, divide this region into annuli of the form
$2^{-n-1}|{v}|\le|x|\le 2^{-n}|{v}|$ where $n$ is a positive integer.
By the same trick as before, the contribution from the $n^{th}$ such
annulus is at most $c(2^{-n}|{v}|)^{\sigma-\epsilon+\morreyovertwo}$,
so the sum of these contributions is at most
$c|{v}|^{\sigma-\epsilon+\morreyovertwo}$.  Finally, the integral over
the region where $|{v}|/2 < |x| < 2|{v}|$ satisfies a bound of the
same form, as we can see by considering annuli centered where $x={v}$.

The conclusion from the above calculations is that if $|{v}|<\rho/2$, then
\begin{equation}
\label{eqn:PTTC}
|\lambda({v})| \le c c_{\sigma-\epsilon}
 |{v}|^{\sigma-\epsilon+\morreyovertwo}
\left( 1 + \sum_{n\ge 1, \; \; 2^n|{v}| < \rho} (2^\kappa)^n\right) +
c\rho^{-1}|{v}|
\end{equation}
where $\kappa=\morreyovertwominusone+\sigma-\epsilon$.  It follows
from \eqref{eqn:PTTC} that $\sigma$ must equal $1$.  Indeed, if
$\sigma$ were less than $1$, then one could choose $\epsilon$ so that
$\sigma'\eqdef\sigma-\epsilon+\morreyovertwo$ is greater than $\sigma$
and less than $1$.  Then $\kappa$ would be negative, so
\eqref{eqn:PTTC} would give $|\lambda({v})| \le \text{const}\cdot
|{v}|^{\sigma'}$, contradicting the maximality of $\sigma$.

Granted that $\sigma=1$, pick any $\epsilon\in(0,\morreyovertwo)$, and
do the sum in
\eqref{eqn:PTTC} to see that $|\lambda({v})|\le c_0|{v}|$, where $c_0$
has an a priori upper bound in terms of $\eta$ and the parameters used
in the gluing.  This proves that $\lambda$ is Lipschitz at ${v}=0$,
and thus $\psi_\Sigma$ is Lipschitz at $u=0$.

To prove that $\lambda$ is Lipschitz at ${v}=w\neq0$, introduce a new
coordinate function
\[
{v}' \eqdef {v} - w - \frak{f}_1(w)\cdot(\overline{{v}}-\overline{w}).
\]
Viewed as a function of ${v}'$, the function
\[
\lambda' \eqdef \lambda - \lambda(w) +
\frak{f}_0(w)\cdot(\overline{{v}}-\overline{w})
\]
obeys an equation of the form
\[
\frac{\partial\lambda'}{\partial\overline{{v}'}} +
\frak{f}_1'\frac{\partial\lambda'}{\partial{v}'} + \frak{f}_0' = 0,
\]
where $\frak{f}_1'$ and $\frak{f}_0'$ have the same properties as do
their unprimed counterparts in \eqref{eqn:ACR}.  Given this, a repeat
of the arguments just given to prove that $\lambda$ is Lipschitz at ${v}=0$
proves that $\lambda'$ is Lipschitz at ${v}'=0$.  As the Lipschitz
constant that appears in the latter argument has a uniform bound, this
proves that $\lambda$ is Lipschitz on a neighborhood of ${v}=0$, and
thus that $\psi_\Sigma$ is Lipschitz near $u=0$.
\end{proof}

\subsection{Continuity of the obstruction section}
\label{sec:sCont}

\begin{proposition}
\label{prop:sCont}
The section $\frak{s}:\times_2[5r,\infty)\times\mc{M}\to\mc{O}$ is
continuous.
\end{proposition}

\begin{proof}
Consider a smooth map from a neighborhood of $0$ in $\R$ to
$\times_2[5r,\infty)\times\mc{M}$.  We will show that the restriction
of the section to this path is continuous at $0$.  The proof has four
steps.

{\em Step 1.\/} We first recall the precise meaning of continuity in
this context.  For $\tau\in\R$, denote the domain of the associated
branched cover in $\mc{M}$ by $\Sigma_\tau$, and set
$\Sigma\eqdef\Sigma_0$.  Identify each domain $\Sigma_\tau$ with
$\Sigma$ via a path of diffeomorphisms $\tau\mapsto\varphi_\tau :
\Sigma\to\Sigma_\tau$.  The diffeomorphism $\varphi_\tau$ identifies
the projection $\Sigma_\tau\to\R\times S^1$ with a map
$\pi_\tau:\Sigma\to\R\times S^1$.  We can choose the diffeomorphisms
$\varphi_\tau$ so that $\varphi_0$ is the identity map, and so that
$\pi_\tau$ agrees with $\pi\eqdef\pi_0$ on the ends of $\Sigma$.

Now identify $T^{0,1}\Sigma_\tau$ with $T^{0,1}\Sigma$ as follows.
First, use the diffeomorphism $\varphi_\tau$ to pull back
$T^{0,1}\Sigma_\tau$ to a subbundle of $T^*\Sigma\tensor_\R\C$.  Then
use orthogonal projection with respect to the metric on $\Sigma$ to
identify the latter subbundle with $T^{0,1}\Sigma$.

Under the above identifications, the kernel of $D_{\Sigma_\tau}^*$
defines a subspace $W_\tau$ of the space of $L^2$ sections of
$T^{0,1}\Sigma$.  Standard perturbation theory for linear operators
shows that $W_\tau$ varies smoothly with $\tau$.  In particular, for
$\tau$ close to $0$ in $\R$, orthogonal projection identifies $W_\tau$
with $W$.  Indeed, this is how the vector bundle structure on $\mc{O}$
is defined.

Now for $|\tau|$ small, the section $\frak{s}$ defines a vector in
$W_\tau$, which is identified with a vector $\frak{s}_\tau\in W$.
More explicitly,
\begin{equation}
\label{eqn:stau}
\frak{s}_\tau =
\Pi^\tau
\mc{F}^\tau(\psi_\tau),
\end{equation}
where $\Pi^\tau$ denotes the composition of the projection to $W_\tau$
with the identification $W_\tau\to W$; $\mc{F}^\tau$ is shorthand for
the map $\mc{F}_{\Sigma_\tau}$ in \eqref{eqn:7SF}; and
$\psi_\tau\eqdef \psi_{\Sigma_\tau}$ is given by
Proposition~\ref{prop:TRE} applied to the triple
$(T_-(\tau),T_+(\tau),\Sigma_\tau)$.  We want to show that the map
$\tau\mapsto\frak{s}_\tau$ is continuous at $0$.

{\em Step 2.\/} We now study the $\tau$ dependence of the various
parts of \eqref{eqn:stau}.  Since the subspace $W_\tau$ varies
smoothly with $\tau$, so does the projection $\Pi^\tau$. 
Next, for any $\psi\in\mc{B}_{\Sigma_\tau}$, not
necessarily the one given by Proposition~\ref{prop:TRE}, write
\[
\mc{F}^\tau(\psi) = F^\tau(\psi)
+ \mc{G}^\tau(\psi),
\]
where $F^\tau(\psi)$ is shorthand for $F'_{\Sigma_\tau}(\psi)$.  Thus
$\mc{G}^\tau(\psi)$ is the sum of all but the first two terms on the
right hand side of the $\tau$ version of equation
\eqref{eqn:ThetaSigma}, with $\psi_\Sigma=\psi$ and with $\psi_\pm$
given by Proposition~\ref{prop:CMT}.  By equation
\eqref{eqn:ThetaSigma} and Proposition~\ref{prop:CMT}, the assignment
$(\tau,\psi)\mapsto\mc{G}^\tau(\psi)$ defines a smooth function from a
neighborhood of $0$ in $\R$ cross $\mc{H}_\Sigma$ to
$\mc{H}_0(\Sigma)$.  The function $(\tau,\psi)\mapsto F^\tau(\psi)$ is
not necessarily smooth, but we have the following weaker statement:
\begin{lemma}
\label{lem:a1Cont}
For $\psi_0=\psi_{\Sigma_0}$ given by Proposition~\ref{prop:TRE} at
$\tau=0$, we have
\[
\lim_{\tau\to 0} \left\|F^\tau(\psi_0) -
F^0(\psi_0)\right\| = 0.
\]
\end{lemma}

{\em Step 3.\/} Assuming Lemma~\ref{lem:a1Cont}, we now complete the
proof of Proposition~\ref{prop:sCont}.  By Lemma~\ref{lem:a1Cont} and the
other conclusions of Step 2, it is enough to show that
$\psi_{\tau}\in\mc{H}_\Sigma$ is a continuous function of
$\tau$ at $\tau=0$.

To prove the latter statement, recall that $\psi_{\tau}$ is the
fixed point of a contraction mapping $\mc{I}^\tau$, which is defined
as in \eqref{eqn:FCM} but with all the terms depending on $\tau$.
Thus
\begin{equation}
\label{eqn:FCM1}
\begin{split}
\|\psi_{\tau} - \psi_{0}\|_* &=
\|\mc{I}^\tau(\psi_{\tau}) - \mc{I}^0(\psi_{0})\|_* \\
& \le 
\|\mc{I}^\tau(\psi_{\tau}) - \mc{I}^\tau(\psi_{0})\|_*
+
\|\mc{I}^\tau(\psi_{0}) - \mc{I}^0(\psi_{0})\|_*.
\end{split}
\end{equation}
The contraction property of $\mc{I}^\tau$ asserts that
\begin{equation}
\label{eqn:FCM2}
\|\mc{I}^\tau(\psi_{\tau}) - \mc{I}^\tau(\psi_{0})\|_*
\le
\frac{1}{2}
\|\psi_{\tau} - \psi_{0}\|_*
\end{equation}
for all $\tau$.  Meanwhile, Lemma~\ref{lem:a1Cont} and the other
conclusions of Step 2 imply that
\begin{equation}
\label{eqn:FCM3}
\lim_{\tau\to 0} \|\mc{I}^\tau(\psi_{0}) -
\mc{I}^0(\psi_{0})\|_*
= 0.
\end{equation}
It follows from \eqref{eqn:FCM1}--\eqref{eqn:FCM3} that
$\lim_{\tau\to 0} \|\psi_{\tau} - \psi_{0}\|_*=0$.

{\em Step 4.\/} We now prove Lemma~\ref{lem:a1Cont}.  Let
$\psi=\psi_{\Sigma_0}$ be the function given by
Proposition~\ref{prop:TRE} at $\tau=0$.  Away from the ramification
points of $\pi$, the function $F^\tau(\psi)$ varies smoothly with
$\tau$.  The only difficulty arises from the variation of
$F^\tau(\psi)$ near the ramification points.

To understand the latter, it proves convenient to choose the
diffeomorphisms $\varphi_\tau$ in Step 1 to have two additional
properties that concern each critical point $p$ of $\pi$.  First,
there is a neighborhood of $p$ in $\Sigma$ on which $\varphi_\tau$
sends the complex structure on $\Sigma$ to the complex structure on
$\Sigma_\tau$.  Second, there is a holomorphic coordinate $u$
identifying a smaller neighborhood of $p$ with the disk of radius
$R>0$ in $\C$, such that the projection $\pi_\tau$ in this
neighborhood is given by
\begin{equation}
\label{eqn:pitau}
\pi_\tau(u) = z_p + u^{q+1} + b_{q-1}u^{q-1} + \cdots +b_0,
\end{equation}
where each $b_j$ varies smoothly with $\tau$ and vanishes at
$\tau=0$.

In this neighborhood, as in equation \eqref{eqn:GNC}, we have
\begin{equation}
\label{eqn:Ftau}
F^\tau(\psi)
=
\frac{\partial_{\overline{u}}\overline{\pi_\tau}}{\partial_u\pi_\tau}
a(\pi_\tau^*t,\psi)
\partial_u\psi d\overline{u} + P(\pi_\tau^*t,\psi),
\end{equation}
where $P(t,w)$ is a smooth function of its arguments.  So to prove
Lemma~\ref{lem:a1Cont}, it suffices to show that for all
$\varepsilon_1>0$, there exists $\varepsilon_2>0$, such that if
$|\tau|<\varepsilon_2$, if $\rho\in(0,R/4)$, and if $u_0\in\C$ with
$|u_0|\le R/4$, then
\[
\rho^{-v} \int_{|u-u_0|<\rho}
\left|
\frac{\partial_{\overline{u}}\overline{\pi_\tau}}{\partial_u\pi_\tau}
a(\pi_\tau^*t,\psi)
-
\left(\frac{\overline{u}}{u}\right)^q
a(\pi_0^*t,\psi)\right|^2 |\partial_u\psi|^2 <
\varepsilon_1.
\]
Recall from Proposition~\ref{prop:TRE}(c) that $\psi$ is Lipschitz.
Also, the function $a$ and its derivatives are uniformly bounded.
Hence it is enough to show that
\begin{equation}
\label{eqn:mcP}
\rho^{-v} \int_{|u-u_0|<\rho}
\left|
\frac{\partial_{\overline{u}}\overline{\pi_\tau}}{\partial_u\pi_\tau}
-
\left(\frac{\overline{u}}{u}\right)^q
\right|^2
\end{equation}
can be made as small as desired by taking $|\tau|>0$ sufficiently small.

To prove this, note that for $|u|<R$ and for $|\tau|$ small,
$|\partial_u\pi_\tau - (q+1)u^q|\le \text{const}\cdot |\tau|$.  It follows
that for any $\varepsilon_3>0$, the integrand in
\eqref{eqn:mcP} is greater than $\varepsilon_3$ only where $|u|\le
\text{const}\cdot|\tau|^{1/q}\varepsilon_3^{-1/2q}$.  The contribution to
\eqref{eqn:mcP} outside of this region is bounded by a constant
multiple of $\varepsilon_3$.  Since the integrand in \eqref{eqn:mcP}
is uniformly bounded, the remaining contribution to
\eqref{eqn:mcP} is at most a constant multiple of
$\left(|\tau|^{1/q}\varepsilon_3^{-1/2q}\right)^{2-v}$, which can be
made arbitrary small by taking $|\tau|$ sufficiently small with
respect to $\varepsilon_3$.  This completes the proof of
Lemma~\ref{lem:a1Cont} and Proposition~\ref{prop:sCont}.
\end{proof}

\subsection{Smoothness of the obstruction section}
\label{sec:sos}

The moduli space $\mc{M}$ of branched covers has a natural
stratification defined as follows.  For $k=0,\ldots,N_++N_--2$, let
$\mc{M}_{(k)}$ denote the set of $\Sigma\in\mc{M}$ for which the set
of ramification points in $\Sigma$ has cardinality $N_++N_--2-k$.  In
particular, $\mc{M}_{(0)}$ is an open dense subset of $\mc{M}$,
consisting of branched covers in which every ramification point $p$ is
simple, meaning that the projection $\pi:\Sigma\to\R\times S^1$ is
described in local coordinates near $p$ by $\pi(u)=z_p+u^2$.  The set
$\mc{M}_{(k)}$ is a complex manifold of complex dimension
$N_++N_--2-k$.

\begin{lemma}
\label{lem:sos}
For each $k$, the restriction of $\frak{s}$ to
$\times_2[5r,\infty)\times\mc{M}_{(k)}$ is smooth.
\end{lemma}

\begin{proof}
This follows from a slight upgrading of the proof of
Proposition~\ref{prop:sCont}, so we will carry over the notation from
that proof.  Consider a smooth map from a neighborhood of $0$ in $\R$
to $\times_2[5r,\infty)\times\mc{M}_{(k)}$.  We want to prove that the
expression in \eqref{eqn:stau} varies smoothly with $\tau$, where
$\psi_{\tau}$ is the fixed point of the contraction mapping
$\mc{I}^\tau$.  For this purpose it is enough to show that
$\mc{I}^\tau$ varies smoothly with $\tau$.  The only missing step is
to show that the function $(\tau,\psi)\mapsto F^\tau(\psi)$ is smooth.
Since our path stays in a fixed stratum $\mc{M}_{(k)}$, the
polynomials $\pi_\tau$ in \eqref{eqn:pitau} must have the form
\[
\pi_\tau(u)=z_p+u^{q+1} + b_0(\tau).
\]
Then $\partial_u\pi_\tau$ is independent of $\tau$, so the ratio
$\partial_{\ubar}\overline{\pi_\tau}/\partial_u\pi_\tau$ that appears
in \eqref{eqn:Ftau} does not depend on $\tau$.  It follows from
\eqref{eqn:Ftau} that $F$ is smooth as required.
\end{proof}

\subsection{Zeroes of $\frak{s}$ have simple ramification points when
  $J$ is generic}
\label{sec:zss}

\begin{lemma}
\label{lem:zss}
If $J$ is generic, then all zeroes of $\frak{s}$ are contained in the
open stratum $\times_2[5r,\infty)\times\mc{M}_{(0)}$.
\end{lemma}

\begin{proof}
The proof has four steps.

\medskip

{\em Step 1.\/} Here is the setup: Fix
$(T_-,T_+,\Sigma)\in\frak{s}^{-1}(0)$, and let $C$ denote the
corresponding $J$-holomorphic curve produced by the gluing
construction.  By Theorem~\ref{thm:immersed}, we can assume that $J$
is generic so that $C$ is unobstructed and immersed.  Fix $\delta>0$
very small and $l>>2$, and let $\mc{U}$ denote a small neighborhood of
$J$ in the space of $C^l$ admissible almost complex structures $J'$
that agree with $J$ within distance $\delta$ of $u_-$ and $u_+$.  We
then have a universal moduli space $\mc{C}$ consisting of pairs
$(J',C')$ where $J'\in\mc{U}$ and $C'$ is a $J'$-holomorphic curve
that is a deformation of $C$.  Let $\mc{Z}\subset\mc{C}$ denote the
set of pairs $(J',C')$ such that $C'$ is obtained by the $J'$ version
of the gluing construction from a zero of $\frak{s}$ on
$\times_2[5r,\infty)\times\mc{M}_{(k)}$ with $k>0$.  Note that
$\mc{Z}$ is invariant under the $\R$ action on $\mc{C}$.  As in our
previous genericity arguments, it is then enough to show that if
$\delta>0$ is sufficiently small, then $\mc{Z}$ is a codimension $2$
subvariety of $\mc{C}$.

\medskip

{\em Step 2.\/} Fix a point $p\in C$ arising from a ramification point
of $\Sigma$ under the gluing construction.  Given $(J',C')$ in a small
neighborhood $\mc{N}$ of $(J,C)$ in $\mc{C}$, we now describe the
local structure of $C'$ near $p$.

Let $B$ denote a disc containing the origin in $\C$ with coordinate
$v$.  For each $(J',C')$, fix a smooth embedding
$\varphi_{(J',C')}:B\to C'$, such that $\varphi_{(J,C)}$ maps $B$ into
a neighborhood of $p$, and such that $\varphi_{(J',C')}$ depends
smoothly on $(J',C')$.

Recall our local coordinates $z$ and $w$ in a neighborhood of
$\R\times\alpha$ from \S\ref{sec:NRIC}.  It follows from
\eqref{eqn:t10ry} that given $(J',C')$, the function $x\eqdef
\varphi_{(J',C')}^*z$ on $B$ obeys the equation
\[
\frac{\partial x}{\partial\vbar} - a
\frac{\partial\overline{x}}{\partial\vbar} = 0.
\]
As a consequence, for $(J',C')$ in a small neighborhood $\mc{N}$ of
$(J,C)$ in $\mc{C}$, the function $x$ near
$v=0$ has the form
\[
x = x_0 + (1-|a_0|^2)^{-1} (\mc{P} + a_0\overline{\mc{P}}) +
O(|v|^{q+2}),
\]
where $a_0$ denotes the value of the function $a$ at $v=0$, while $q$
denotes the ramification index of $p$ in $C$, and $\mc{P}$ is a
holomorphic polynomial of degree $q+1$ whose coefficients depend
smoothly on $(J',C')$.  Moreover, the maps $\varphi_{(J',C')}$ can be
chosen so that
\begin{equation}
\label{eqn:Pv}
\mc{P}(v) = v^{q+1} + b_{q-1}v^{q-1} + \cdots + b_0.
\end{equation}
Ramification points in $C'$ near $p$ correspond to
roots of $\partial_v\mc{P}$.

The coefficients $b_0,\ldots,b_{q-1}$ identify $\mc{P}$ with an
element of $\C^q$, so we have defined a smooth map
$\mc{P}:\mc{N}\to\C^q$.  There is a complex codimension 1 subvariety
$\Delta\subset\C^q$ such that $\partial_v\mc{P}$ has $q$ distinct
roots if and only if the coefficients of $\mc{P}$ correspond to a
point in $\C^q\setminus\Delta$.  So by Step 1, to prove
Lemma~\ref{lem:zss}, it suffices to show that the differential
\begin{equation}
\label{eqn:dPJC}
d\mc{P}_{(J,C)}:T_{(J,C)}\mc{C}\longrightarrow \C^q
\end{equation}
is surjective.

\medskip
{\em Step 3.\/} To prepare for the proof that \eqref{eqn:dPJC} is
surjective, we now construct some useful tangent vectors in
$T_{(J,C)}\mc{C}$.

As in \S\ref{sec:G1}, a tangent vector in $T_{(J,C)}\mc{C}$ consists
of a pair $(j,\zeta)$, where $j$ is a $(0,1)$ bundle automorphism of
$T(\R\times Y)$, and $\zeta$ is an $L^2_1$ section of the normal
bundle $N_C$, such that
\[
D_C\zeta = j_C.
\]
Here $D_C:C^\infty(N_C)\to C^\infty(N_C\tensor T^{0,1}C)$ denotes the
linear deformation operator associated to $C$, and
$j_C\in\Hom^{0,1}(TC,NC)$ is defined in \eqref{eqn:jC}.

The operator $D_C$ can be described more explicitly near $p$ as
follows.  By choosing $B$ sufficiently small, we can assume that the
vector field $\partial_s$ is not tangent to $C$ on the image of $B$.
Over $B$, we can then use $\partial_s$ to trivialize $N_C$, and
$d\vbar$ to trivialize $T^{0,1}C$.  In these trivializations, if $f$
is a complex function on $B$, then
\[
D_Cf = \partial_{\vbar} f + \nu_C f + \mu_C \overline{f}
\]
on $B$, where $\nu_C$ and $\mu_C$ are complex functions on $B$.

Now to construct some useful tangent vectors, let $\rho$ denote the
diameter of $B$.  Fix a smooth function
$\beta:[0,\infty)\to[0,1]$ which equals $1$ on $[0,\rho/2]$ and
$0$ on $[\rho,\infty)$.  Let $\mc{H}_B$ denote the space of
$L^2_1$ functions on $B$ whose restriction to $\partial B$
is in the span of $\{\rho^{-p}v^p\}_{-\infty < p < q}$.  Define an
operator
\begin{equation}
\label{eqn:DCB}
\begin{split}
D_{C,B}f:\mc{H}_B & \longrightarrow L^2(B;\C),\\
f & \longmapsto \partial_{\vbar}f + \beta(\nu_C f + \mu_C
\overline{f}).
\end{split}
\end{equation}

\begin{lemma}
\label{lem:DCB}
The operator $D_{C,B}$ in \eqref{eqn:DCB} is Fredholm.  Its index
is $2q$ and its cokernel is trivial.  Its kernel has a basis
$\{f_{k,A}\}_{0\le k < q,\, A\in\{0,1\}}$ such that
\begin{equation}
\label{eqn:fkA}
f_{k,A} = i^A v^k + O(|v|^{k+1})
\end{equation}
as $v\to 0$.
\end{lemma}

\begin{proof}
The operator $D_{C,B}$ differs from $\partial_{\vbar}$ by a zeroth
order term.  Since the latter is a Fredholm, index $2q$ operator from
$\mc{H}_B$ to $L^2(B;\C)$, so is $D_{C,B}$.

Note that each $f\in\Ker(D_{C,B})$ extends from $B$ to the
whole of $\C$ as a function that is holomorphic on $\C\setminus
B$.  Moreover, if $f$ is not identically zero, then $f$ behaves
at large $|v|$ as $c v^k + O(|v|^{k-1})$ with $c\neq 0$ and $k<q$.
Finally, all zeroes of $f$ have positive multiplicity.  It follows
that $f$ has at most $q-1$ zeroes.  This implies that
$\dim\Ker(D_{C,B})\le 2q$, and hence $\Coker(D_{C,B})=\{0\}$.
It also follows that a zero of a kernel element has multiplicity at
most $q-1$, and this implies that $\Ker(D_{C,B})$ has a basis of
the desired form.
\end{proof}

As in \S\ref{sec:G1}, there is a codimension 1 subvariety $B'\subset
B$, such that any $C^k$ function $f$ with support on the interior of
$B$ can be realized as $j_C$ for some $j\in T_J\mc{U}$ for some
$\delta>0$, provided that $f$ vanishes on a neighborhood of $B'$.  For
$\varepsilon>0$ small, let $\chi_\varepsilon:B\to[0,1]$ be a smooth
function which is $0$ within distance $\varepsilon$ of $B'$ and which is
$1$ where the distance to $B'$ is $\ge 2\varepsilon$.  Fix a basis
$\{f_{k,A}\}$ for $\ker(D_{C,B})$ as in Lemma~\ref{lem:DCB}, and
choose $j_{\varepsilon,k,A}\in T_J\mc{U}$ such that
\[
(j_{\varepsilon,k,A})_C = \chi_{\varepsilon} D_C(\beta f_{k,A}).
\]

To complete this to a tangent vector to $\mc{C}$, let $D_C^{-1}$
denote the unique right inverse of $D_C$ that maps to the $L^2$
orthogonal complement of $\Ker(D_C)$.  Define
\begin{equation}
\label{eqn:zetae}
\zeta_{\varepsilon,k,A} \eqdef D_C^{-1}(\chi_{\varepsilon}
D_C(\beta f_{k,A}))
+ (\beta f_{k,A})_0,
\end{equation}
where $(\cdot)_0$ denotes $L^2$ orthogonal projection onto
$\Ker(D_C)$.  Then
\begin{equation}
\label{eqn:tv}
(j_{\varepsilon,k,A}, \zeta_{\varepsilon,k,A}) \in T_{(J,C)}\mc{C}
\end{equation}
Define $\mc{T}_\varepsilon$ to be the span of the
tangent vectors \eqref{eqn:tv} for $0\le k < q$ and $A\in\{0,1\}$.

\medskip

{\em Step 4.\/} We now complete the proof of Lemma~\ref{lem:zss}.  By
Step 2, it suffices to show that if $\varepsilon>0$ is sufficiently
small, then $d\mc{P}_{(J,C)}$ restricts to an isomorphism from
$\mc{T}_\varepsilon$ to the space of polynomials of the form
\eqref{eqn:Pv}.

Note that for $(j,\zeta)\in T_{(J,C)}\mc{C}$, if
$\zeta=cv^k+O(|v|^{k+1})$ as $v\to 0$ with $c\neq 0$ and $k<q$, then
\[
d\mc{P}_{(J,C)}(j,\zeta) = c v^k + O(|v|^{k+1}).
\]
If we could take $\varepsilon=0$ in \eqref{eqn:zetae}, then we would
be done by \eqref{eqn:fkA}, since $\zeta_{\varepsilon,k,A} = \beta
f_{k,A}$ when $\varepsilon=0$.  The claim still holds when
$\varepsilon>0$ is small, because $\zeta_{\varepsilon,k,A}$ converges
in the $C^\infty$ topology on compact sets to $\beta f_{k,A}$ as
$\varepsilon\to 0$.
\end{proof}

%% file: obg-map.tex
Continue with the setup and notation from \S\ref{sec:gluing}.  The
goal of this section is to prove Theorem~\ref{thm:GT} below, which
asserts roughly that the gluing map \eqref{eqn:gluingMap}, applied to
triples $(T_-,T_+,\Sigma)$ with $T_-,T_+$ large, describes all curves
in $\mc{M}_J(\alpha_+,\alpha_-)$ that are ``close to breaking'' into
$U_+$ and $U_-$ along branched covers of $\R\times\alpha$.

\subsection{Statement of the gluing theorem}
\label{sec:gluingStatement}

The set of curves that are ``close to breaking'' in the above sense is
denoted by $\mc{G}_\delta(U_+,U_-)$.  The precise definition of
$\mc{G}_\delta(U_+,U_-)$ in general was given in Definition~{I.1.10}.
We now recall this definition for convenience, using our standing
assumptions
\eqref{eqn:i} and \eqref{eqn:ii} to recast it slightly.

We will use the following notation: If $\psi_-$ is a section of the
normal bundle to $u_-$ with $|\psi_-|<\delta_0$, then $e_-\circ\psi_-$
denotes the immersed surface in $\R\times Y$ whose domain is that of
$u_-$, given by composing the section $\psi_-$ with the exponential
map $e_-$.  If $\psi_+$ is a section of the normal bundle to $u_+$
with $|\psi_+|<\delta_0$, define $e_+\circ\psi_+$ likewise.

\begin{definition}
\label{def:CTB}
For $\delta>0$, define $\widetilde{\mc{G}}_\delta(U_+,U_-)$ to be the set of
immersed (except possibly for finitely many singular points) surfaces
in $\R\times Y$ that can be decomposed as $C_-\cup C_0
\cup C_+$, such that the following hold:
\begin{itemize}
\item
There is a real number $R_-$, and a section $\psi_-$ of the normal
bundle to $u_-$ with $|\psi_-|<\delta$, such that $C_-$ is the
$s\mapsto s+R_-$ translate of the $s\le 1/\delta$ part of $e_-\circ\psi_-$.
\item
Likewise, there is a real number $R_+$, and a section $\psi_+$ of the
normal bundle to $u_+$ with $|\psi_+|<\delta$, such that $C_+$ is the
$s\mapsto s+R_+$ translate of the $s\ge -1/\delta$ part of
$e_+\circ\psi_+$.
\item
$R_+-R_->2/\delta$.
\item
$C_0$ is a connected genus zero surface with boundary which is
contained in the radius $\delta$ tubular neighborhood of
$\R\times\alpha$, such that the tubular neighborhood projection
$C_0\to\R\times\alpha$ is a branched covering.  Moreover $C_0$ has
positive ends of multiplicities $a_{\Nbar_++1},\ldots,a_{N_+}$, and
 negative ends of multiplicities $a_{-\Nbar_--1},\ldots,a_{-N_-}$.
\item
$\partial C_0=\partial C_- \sqcup \partial C_+$, where the positive
boundary circles of $C_-$ agree with the negative boundary circles of
$C_0$, and the positive boundary circles of $C_0$ agree with the
negative boundary circles of $C_+$.
\end{itemize}
Let $\mc{G}_\delta(U_+,U_-)$ denote the set of surfaces
$C\in\widetilde{\mc{G}}_\delta(U_+,U_-)$ such that $C$ is
$J$-holomorphic.  Note that the definition implies that any element of
$\mc{G}_\delta(U_+,U_-)$ is in $\mc{M}^J(\alpha_+,\alpha_-)$ and has
index $2$.
\end{definition}

\begin{definition}
Given $\delta>0$, define
$\mc{U}_\delta\subset\times_2[5r,\infty)\times\mc{M}$ to be
the set of $(T_-,T_+,\Sigma)$ such that
$u(T_-,T_+,\Sigma)\in\widetilde{\mc{G}}_\delta(U_+,U_-)$.
\end{definition}

\begin{theorem}
\label{thm:GT}
Fix $h\in(0,1)$, and let $r_0,\varepsilon$ be as in
Proposition~\ref{prop:TRE}.  Then:
\begin{description}
\item{(a)}
If $R$ is sufficiently large with respect to $\delta$, then
\[
\times_2[R,\infty)\times\mc{M}\subset\mc{U}_\delta.
\]
\item{(b)}
If $r>r_0$ is chosen sufficiently large and if $\delta>0$ is
sufficiently small with respect to $r$, then the gluing map
\eqref{eqn:gluingMap} restricts to a homeomorphism
\begin{equation}
\label{eqn:GM}
G:\frak{s}^{-1}(0)\cap\mc{U}_\delta\stackrel{\simeq}{\longrightarrow}
\mc{G}_\delta(U_+,U_-).
\end{equation}
\end{description}
\end{theorem}


\begin{proof}
Part (a) follows from Propositions \ref{prop:TRE}(b) and
\ref{prop:CMT}(c).
To prove part (b), we will show in Lemmas~\ref{lem:IGM} and
\ref{lem:SGM} below that if $r$ is sufficiently large and if $\delta$ is
sufficiently small with respect to $r$, then the map \eqref{eqn:GM} is
a bijection.  Continuity of the map \eqref{eqn:GM} follows from
Proposition~\ref{prop:CMT}(b) together with the proof of
Proposition~\ref{prop:sCont}.
\end{proof}

\subsection{Injectivity of the gluing map}
\label{sec:IGM}

\begin{lemma}
\label{lem:IGM}
Fix $h\in(0,1)$.  If $r>r_0$ is sufficiently large, and if
$\delta>0$ is sufficiently small, then the restricted gluing map
\eqref{eqn:GM} is injective.
\end{lemma}

\begin{proof} The proof has two steps.

{\em Step 1.\/} Fix $r>r_0$ and $\delta>0$, and let
$(T_-,T_+,\Sigma)\in\mc{U}_\delta$.  We now show that if
$u(T_-,T_+,\Sigma)$ is $J$-holomorphic, then $u(T_-,T_+,\Sigma)$
determines $\Sigma$.

Choose a decomposition $C_-\cup C_0\cup C_+$ of $u(T_-,T_+,\Sigma)$ as
in Definition~\ref{def:CTB}.  Recall the coordinates $(z,w)$ on a
tubular neighborhood of $\R\times \alpha$.  Let $\frak{p}: C_0 \to
\R\times S^1$ denote the tubular neighborhood projection sending
$(z,w)\mapsto (z,0)$.  Since $C_0$ and the $z=\text{constant}$ disks
are $J$-holomorphic, it follows that the map $\frak{p}$ is a branched
cover on (the domain of) $C_0$.  As such, it pulls back the complex
structure on $\R\times S^1$ to a complex structure $j$ on $C_0$ (which
generally does not agree with the restriction of the almost complex
structure $J$ on $\R\times Y$).  Let $\widetilde{\frak{p}}:
(\widetilde{C}_0,\widetilde{j}) \to \R\times S^1$ denote the element
of $\mc{M}$ obtained by attaching half-infinite cylinders to the
$\Nbar_+$ positive boundary circles and the $\Nbar_-$ negative
boundary circles of $(C_0,j)$; the orderings and asymptotic markings
of the resulting ends are induced from those of the negative ends of
$u_+$ and the positive ends of $u_-$ respectively via the
identification $\partial C_0 = \partial C_+
\sqcup \partial C_-$.

We claim that $\widetilde{C}_0$ and $\Sigma$ (with their maps to
$\R\times S^1$ and orderings and asymptotic markings of their ends)
define the same element of $\mc{M}$.  To see this, let $\Sigma_0$ be
obtained from $\Sigma$ by removing the $s>R_+$ part of the first
$\Nbar_+$ positive ends and the $s<R_-$ part of the first
$\Nbar_-$ negative ends.  The gluing construction defines a
parametrization $f:\Sigma_0\stackrel{\simeq}{\longrightarrow} C_0$
with $\frak{p}\circ f = \pi$.  It follows from the definition of the
complex structure $j$ on $C_0$ that the map $f$ is holomorphic with
respect to $j$.  Then $f$ extends to a biholomorphic map
$\widetilde{f}:\Sigma\to (\widetilde{C}_0,\widetilde{j})$, which
satisfies $\widetilde{\frak{p}}\circ\widetilde{f}=\pi$ and preserves
the orderings and asymptotic markings of the ends.

{\em Step 2.\/} We now show that if $r>r_0$ is sufficiently large, if
$\delta$ is sufficiently small, and if
$u(T_-,T_+,\Sigma)\in\widetilde{\mc{G}}_\delta(U_+,U_-)$, then $T_-$
and $T_+$ are determined by $u(T_-,T_+,\Sigma)$.  It suffices to prove
the following two claims:

\begin{description}
\item{(i)}
If $r>r_0$ is sufficiently large, then for any given $R$, if $\delta$
is sufficiently small with respect to $R$, then
$u(T_-,T_+,\Sigma)\in\widetilde{\mc{G}}_\delta(U_+,U_-)$ implies
$T_-,T_+ > R$.
\item{(ii)}
For any $r>r_0$, if $R$ is sufficiently large, if
$T_-,T_+,T_-',T_+'>R$, and if $u(T_-,T_+,\Sigma) =
u(T_-',T_+',\Sigma)$, then $(T_-,T_+)=(T_-',T_+')$.
\end{description}

Proof of (i): Given $p\in Y$ and $\rho>0$, let $B(p,\rho)\subset Y$
denote the ball of radius $\rho$ around $p$ in $Y$.  If $\delta_1>0$,
is sufficiently small, then there exist points $p_-,p_+\in Y$ with the
following two properties: First, $p_\pm$ is contained in the
projection of $u_\pm$ to $Y$.  Second, $\R\times B(p_-,2\delta_1)$
does not intersect $u_+$, and $\R\times B(p_+,2\delta_1)$ does not
intersect $u_-$.  Fix $\delta_1$ and $p_\pm$ as above.

If $r>r_0$ is sufficiently large, then the estimates in
\S\ref{sec:gluing} imply that for any $(T_-,T_+,\Sigma) \in
\times_2(5r,\infty)\times\mc{M}$, the sections $\psi_\pm$ produced by
the gluing construction satisfy $|\psi_\pm(p_\pm)|<\delta_1$.  Fix $r$
with this property.

Next, fix $\delta<\delta_1$, and suppose that $C\eqdef
u(T_-,T_+,\Sigma) \in\widetilde{\mc{G}}_\delta(U_-,U_+)$.  Choose a
decomposition $C=C_-\cup C_0\cup C_+$ as in Definition~\ref{def:CTB}.
Let $a$ denote the supremum of $s$ on the intersection of $u_-$ with
$\R\times B(p_-,2\delta_1)$, and let $b\eqdef s(p_-)$.  It follows
from the conditions on $C_0$ and $C_+$ that any point in $C\cap
(\R\times B(p_-,\delta_1))$ must be in $C_-$, and hence must have
$s\le a + R_- + \delta_1$.  On the other hand, since
$C=u(T_-,T_+,\Sigma)$, it follows from our choice of $r$ that under
the gluing construction, $p_-$ gives rise to a point in $C\cap
(\R\times B(p_-,\delta_1))$ with $s\ge b + s_- - T_- - \delta_1$.
Finally, the conditions in Definition~\ref{def:CTB} imply that $s_-\ge
R_-+1/\delta$.  Combining the above inequalities, we find that $T_-$
is greater than $1/\delta$ plus a constant depending only on $u_-$.
Similarly, $T_+$ is greater than $1/\delta$ plus a constant depending
only on $u_+$.

Proof of (ii): Given $x\in\R$, let $\Phi_x$ denote the automorphism of
$\R\times Y$ sending $(s,y)\mapsto(s+x,y)$.  We can find a point $p_-$
in the $s\le 0$ part of $u_-$, and a real number $0 < \rho <
\delta_0$, such that $\R\times B(p_-,\rho)$ does not intersect $u_+$,
and such that the intersection of $u_-$ with $\R\times B(p_-,\rho)$ is
a single disc $B_-$ on which the projection to $Y$ is an embedding.
It follows from this last condition that there exist constants
$c_-,\epsilon_->0$ with the following property:
\begin{description}
\item{(*)}
Let $\psi_-$ be a section of the normal bundle to
$B_-$ with $|\psi_-|,|\nabla\psi_-|<\epsilon_-$.  Then for any
$x_-\in\R$ and for any $p_-'\in B_-$, we have
\begin{equation}
\label{eqn:NBS}
\text{dist}(e_-(\psi_-(p_-)),\Phi_{x_-}(e_-(\psi_-(p_-')))) \ge c_-|x_-|.
\end{equation}
\end{description}

Now fix $r>r_0$ and $R$. Let $T_-,T_+,T_-',T_+'>R$ and suppose that
$u(T_-,T_+,\Sigma)=u(T_-',T_+',\Sigma)$.  Let $\psi_-$ and $\psi_-'$
denote the sections of the normal bundles to $u_{-T}$ and $u_{-T'}$
respectively coming from the gluing construction.  Use the
translations $\Phi_{T_-}$ and $\Phi_{T_-'}$ to regard both $\psi_-$
and $\psi_-'$ as sections of the normal bundle to $u_-$.  By
Propositions~\ref{prop:CMT}(c) and \ref{prop:TRE}(b), there are
constants $c,\lambda>0$ depending only on $u_+$ and $u_-$ such that
$|\psi_-|, |\psi_-'| < c\exp(-\lambda R)$.  In particular, if $R$ is
sufficiently large then $|\psi_-|,|\psi_-'|<\rho/2$.

The point $p_-$ in $u_-$ gives rise to the point
$\Phi_{s_--T_-}(e_-(\psi_-(p_-)))$ in the gluing $u(T_-,T_+,\Sigma)$.
Since this point is also in $u(T_-',T_+',\Sigma)$, there must exist
$p_-'$ in $B_-$ with
\begin{equation}
\label{eqn:MSW}
e_-(\psi_-(p_-)) = \Phi_{x_-}(e_-(\psi_-'(p_-'))),
\end{equation}
where $x_-\eqdef T_- - T_-'$.

Now the bound on $|\psi_-|$, together with the elliptic regularity in
Lemma~\ref{lem:ER}, leads to a bound of the same form on
$|\nabla\psi_-|$.  Hence if $R$ is sufficiently large, then (*) is
applicable so that the inequality \eqref{eqn:NBS} holds.

On the other hand, by bounding the derivatives of the contraction
mappings used to define $\psi_-$, one can show that $\psi_-$ depends
smoothly on $T_+$ and $T_-$, with
$\left\|\frac{\partial\psi_-}{\partial T_\pm}\right\|\le
c\exp(-\lambda R)$, where again $c,\lambda>0$ depend only on $u_+$ and
$u_-$.  Therefore
\[
\text{dist}(e_-(\psi_-(p_-')),
e_-(\psi_-'(p_-'))) \le c\exp(-\lambda R)(|x_-| + |x_+|)
\]
where $x_+ \eqdef T_+ - T_+'$.  Combining this with \eqref{eqn:NBS}
and \eqref{eqn:MSW}, we obtain
\[
c_-|x_-| \le c\exp(-\lambda R)(|x_-| + |x_+|).
\]

By a symmetric argument, there is a constant $c_+$ depending only on
$u_+$ such that
\[
c_+|x_+| \le c\exp(-\lambda R)(|x_-| + |x_+|).
\]
If $R$ is sufficiently large, then the above two inequalities together
imply that $x_-=x_+=0$, so that $(T_-,T_+)=(T_-',T_+')$.
\end{proof}

\subsection{Surjectivity of the gluing map}

\begin{lemma}
\label{lem:SGM}
For fixed $h\in(0,1)$, if $r>r_0$ is chosen sufficiently large and if
$\delta>0$ is sufficiently small with respect to $r$, then the
restricted gluing map \eqref{eqn:GM} is surjective.
\end{lemma}

\begin{proof} The proof has three steps.

{\em Step 1.\/} Here is the setup: Let $C\in\mc{G}_\delta(U_+,U_-)$,
and decompose $C=C_-\cup C_0 \cup C_+$ as in Definition~\ref{def:CTB}.
Let $T_-$ denote the real number for which the smallest critical value
of $s|_{C_0}$ is $R_-+T_-+1$.  Let $T_+$ denote the real number for
which the largest critical value of $s|_{C_0}$ is $R_+-T_+-1$.  It
follows from the conditions in Definition~\ref{def:CTB} that if
$1/\delta \ge 5r+5$ (which we assume for the rest of this proof), then
$T_-,T_+ \ge 5r$.  Also, as in the proof of Lemma~\ref{lem:IGM}, the
decomposition of $C$ determines a branched cover $\Sigma$ in $\mc{M}$,
with $s_-$ close to $R_-+T_-$ and $s_+$ close to $R_+-T_+$.

The section $\psi_-$ of the normal bundle to $u_-$ given by
Definition~\ref{def:CTB} determines a section of the normal bundle to
$u_{-T}$, which we also denote by $\psi_-$, and which satisfies equation
\eqref{eqn:7-} on $u_{-T}'$.  Likewise, we have a
section $\psi_+$ of the normal bundle to $u_{+T}$ satisfying equation
\eqref{eqn:7+} on $u_{+T}'$.  Part of the curve $C$ consists of the
exponential map images of the sections $\psi_\pm$ over $u_{\pm T}'$.
The rest of $C$ is described, in our coordinates $((s,t),w)$ on a
tubular neighborhood of $\R\times\alpha$, by a map on $\Sigma'$
sending $x\mapsto (\pi(x),\psi_0(x))$ where $\psi_0$ is a
complex-valued function on $\Sigma'$.  Let $\Sigma''$ be obtained from
$\Sigma'$ by removing the cylinders $\Sigma_i$.  Let $\psi_\Sigma$
denote the restriction of $\psi_0$ to $\Sigma''$.  Then $\psi_\Sigma$
satisfies equation \eqref{eqn:7Sigma} on $\Sigma''$.

To show that $C$ is obtained from the gluing construction, we want to
extend $\psi_-$ over the rest of $u_{-T}$, extend $\psi_+$ over the rest of
$u_{+T}$, and extend $\psi_\Sigma$ over the rest of $\Sigma$, so that:
\begin{description}
\item{(i)}
Equation \eqref{eqn:7-} holds on all of $u_{-T}$, equation
\eqref{eqn:7+} holds on all of $u_{+T}$, and equation
\eqref{eqn:7Sigma} holds on all of $\Sigma$.
\item{(ii)}
On each cylinder $\Sigma_i$ with $i<0$, we have
\begin{equation}
\label{eqn:ii-}
\beta_-(\eta_{-T} + \psi_-) + \beta_\Sigma\psi_\Sigma = \psi_0.
\end{equation}
Likewise, on each cylinder $\Sigma_i$ with $i>0$, we have
\begin{equation}
\label{eqn:ii+}
\beta_+(\eta_{+T} + \psi_+) + \beta_\Sigma\psi_\Sigma = \psi_0.
\end{equation}
\item{(iii)}
$\|\psi_\Sigma\|_* < \varepsilon$, where $\varepsilon$ is given by
Proposition~\ref{prop:TRE}.
\item{(iv)}
$\psi_-$ is orthogonal to the kernel of $D_-$, and $\psi_+$ is
orthogonal to the kernel of $D_+$.
\end{description}

{\em Step 2.\/} We now show that there exist $r_1>1$ and $\delta_1>0$
such that if $r>r_1$ and $\delta\le\delta_1$, then $\psi_\pm$ and
$\psi_\Sigma$ can be extended to satisfy conditions (i)--(iii) above.
This step has two substeps.

{\em Step 2.1.\/} Consider one of the cylinders $\Sigma_i$, identified
with $[A,B]\times\widetilde{S^1}$ with coordinates $(s,\tau)$.  Here
$\widetilde{S^1}$ denotes the $m$-fold cover of $S^1$, where $m$ is
the degree of the restriction of the covering $\Sigma\to\R\times S^1$
to $\Sigma_i$.  On this cylinder, the function $\psi\eqdef\psi_0$
satisfies an equation of the form
\begin{equation}
\label{eqn:42}
(\partial_s + L_m)\psi + F(\psi)=0
\end{equation}
where $F$ is type 1 quadratic in the sense of
Definition~\ref{def:quadratic}.  The purpose of this substep is to
establish some properties of equation \eqref{eqn:42}.

Recall that $L^2_{3/2}(\widetilde{S^1};\R^2)$ denotes the completion
of the space of smooth $\R^2$-valued functions on $\widetilde{S^1}$
using the norm defined by
\[
\|\eta\|_{L^2_{3/2}}^2 \eqdef \int_{\widetilde{S^1}} \left\langle \eta,
|L_m|^{3}\eta\right\rangle.
\]
Let $\Pi_+$ (resp.\ $\Pi_-$) denote the $L^2$ orthogonal projection
from $L^2_{3/2}(\widetilde{S^1};\R^2)$ to the span of the eigenvectors
of $L_m$ with positive (resp.\ negative) eigenvalues.

\begin{lemma}
\label{lem:IFT1}
There exist constants $\epsilon_0>0$ and $c$ with the
following property.  Suppose $B-A\ge 1$ and $\epsilon<\epsilon_0$.
Let $\lambda_A\in\Pi_+L^2_{3/2}(\widetilde{S^1};\R^2)$ and
$\lambda_B\in\Pi_- L^2_{3/2}(\widetilde{S^1};\R^2)$ be given with
$L^2_{3/2}$ norm less than $\epsilon$.  Then there exists a unique
solution $\psi$ to equation \eqref{eqn:42} on $[A,B] \times
\widetilde{S^1}$ with $\|\psi\|_{L^2_2} < c\epsilon$ that satisfies
the boundary conditions $\Pi_+\psi(A,\cdot)=\lambda_A$ and
$\Pi_-\psi(B,\cdot) =
\lambda_B$.
\end{lemma}

\begin{proof}
Define a map
\[
\begin{split}
\mc{F}: L^2_2([A,B]\times\widetilde{S^1}) & \to
\Pi_+L^2_{3/2}(\widetilde{S^1};\R^2) \times \Pi_-
L^2_{3/2}(\widetilde{S^1};\R^2) \times
L^2_1([A,B]\times\widetilde{S^1}),\\
\psi & \mapsto (\Pi_+\psi(A,\cdot), \Pi_-\psi(B,\cdot),
(\partial_s + L_m)\psi + F(\psi)).
\end{split}
\]
Calculations as in the proof of Lemma~\ref{lem:CC} show that the
derivative of $\mc{F}$ at $0$ is invertible, and the operator norm of
its inverse has an upper bound independent of $B-A$. In addition, 
since $F$ is type 1 quadratic, it follows that
\[
\|(d\mc{F}_\psi - d\mc{F}_0)\xi\|_{L^2_1} \le
c\|\psi\|_{L^2_2}\|\xi\|_{L^2_2}
\]
where $c$ is independent of $B-A\ge 1$.  The lemma now follows from
the inverse function theorem.
\end{proof}

{\em Proof of Lemma~\ref{lem:SGM}, Step 2.2.\/} Fix
$i\in\{-1,\ldots,-\Nbar_-\}$, and consider the problem of extending $\psi_-$
over the $s\ge s_--T_-$ portion of the $i^{th}$ positive end of
$u_{-T}$, and extending $\psi_\Sigma$ over the $s \le s_i$
portion of the $i^{th}$ negative end of $\Sigma$.  Both of these
cylinders are identified with subcylinders of
$\R\times\widetilde{S^1}$.  Thus we need to find extensions of $\psi_-$
over $[s_--T_-,\infty)\times\widetilde{S^1}$, and of $\psi_\Sigma$ over
$(-\infty,s_i]\times\widetilde{S^1}$, that satisfy the equations
\begin{equation}
\label{eqn:7-Sigma}
\begin{split}
\Theta_-(\psi_-,\psi_\Sigma) &= 0, \quad \quad s\ge s_--T_-,\\
\Theta_\Sigma(\psi_-,\psi_\Sigma,0) &= 0, \quad \quad s\le s_i,
\end{split}
\end{equation}
and that also satisfy equation \eqref{eqn:ii-} when $s_--T_-\le s \le
s_i$.

The following lemma provides solutions to the equations
\eqref{eqn:7-Sigma}.

\begin{lemma}
\label{lem:IFT2}
There exist constants $\epsilon_0>0$, $r_1>1$, and $c$ with the
following property.  Suppose $r>r_1$ and $\epsilon < \epsilon_0$.  Let
$\lambda_-\in \Pi_+ L^2_{3/2}(\widetilde{S^1};\R^2)$ and
$\lambda_\Sigma\in\Pi_- L^2_{3/2}(\widetilde{S^1};\R^2)$ be given with
$L^2_{3/2}$ norm less than $\epsilon$.  Then there exists a unique
solution $(\psi_-,\psi_\Sigma)$ to \eqref{eqn:7-Sigma} with
$\|\psi_-\|_{L^2_2}, \|\psi_\Sigma\|_{L^2_2} < c\epsilon$ that
satisfies the boundary conditions
$\Pi_+\psi_-(s_--T_-,\cdot)=\lambda_-$ and
$\Pi_-\psi_\Sigma(s_i,\cdot) =
\lambda_\Sigma$.
\end{lemma}

\begin{proof}
This is an application of the inverse function theorem similar to the
proof of Lemma~\ref{lem:IFT1}.  In more detail, write $A\eqdef
s_--T_-$ and $B\eqdef s_i$.  Define a map
\begin{gather*}
\mc{F}: L^2_2(s\ge A) \times L^2_2(s\le B)  \rightarrow
\Pi_+L^2_{3/2} \times \Pi_-L^2_{3/2} \times L^2_1(s\ge A) \times
L^2_1(s\le B),\\
(\psi_-,\psi_\Sigma) \longmapsto \big(\Pi_+\psi_-(A,\cdot),
\Pi_-\psi_\Sigma(B,\cdot), \Theta_-(\psi_-,\psi_\Sigma),
\Theta_\Sigma(\psi_-,\psi_\Sigma,0)\big).
\end{gather*}
The derivative of $\mc{F}$ at $(0,0)$ has the schematic form
\begin{equation}
\label{eqn:DFS}
\begin{split}
(\psi_-,\psi_\Sigma) \longmapsto & \big(\Pi_+\psi_-(A,\cdot),
\Pi_-\psi_\Sigma(B,\cdot),\\
& \;\;(\partial_s + L_m)\psi_- +
a_-(\psi_-,\psi_\Sigma), (\partial_s+L_m)\psi_\Sigma +
a_\Sigma(\psi_-,\psi_\Sigma)\big).
\end{split}
\end{equation}
As in the proof of Lemma~\ref{lem:IFT1}, the map \eqref{eqn:DFS}, with
the $a_-$ and $a_\Sigma$ terms removed, is invertible, and the
operator norm of its inverse is less than some constant $c_0$ which
does not depend on $A$ and $B$.  To prove that the map \eqref{eqn:DFS}
itself is invertible, with its inverse bounded independently of $A$
and $B$, it is enough to show that
\begin{equation}
\label{eqn:AMS}
\|a_-(\psi_-,\psi_\Sigma)\|_{L^2_1},
\|a_\Sigma(\psi_-,\psi_\Sigma)\|_{L^2_1} \le
\frac{1}{4c_0}\left(\|\psi_-\|_{L^2_2} + \|\psi_\Sigma\|_{L^2_2}\right).
\end{equation}
The bound \eqref{eqn:AMS} follows directly from equations
\eqref{eqn:Theta-} and \eqref{eqn:ThetaSigma} and the decay estimates
on $\eta_{-T}$, provided that $r$ is sufficiently large.  It also
follows from \eqref{eqn:Theta-} and \eqref{eqn:ThetaSigma} that
\[
\|(d\mc{F}_{(\psi_-,\psi_\Sigma)} -
d\mc{F}_{(0,0)})(\xi_-,\xi_\Sigma)\|_{L^2_1} \le
c\big(\|\psi_-\|_{L^2_2} + \|\psi_\Sigma\|_{L^2_2}\big)
\big(\|\xi_-\|_{L^2_2} + \|\xi_\Sigma\|_{L^2_2}\big).
\]
where $c$ is independent of $A$ and $B$ provided that $r\ge 1$.  The
lemma now follows from the inverse function theorem.
\end{proof}

If $\delta$ is sufficiently small, then we can apply
Lemma~\ref{lem:IFT2} with $\lambda_- =
\Pi_+(\psi_0-\eta_{-T})(s_--T_-,\cdot)$ and $\lambda_\Sigma =
\Pi_-\psi_0(s_i,\cdot)$, to obtain a solution $(\psi_-,\psi_\Sigma)$
to the equations \eqref{eqn:7-Sigma} satisfying the above boundary
conditions and with $\|\psi_-\|_{L^2_2},
\|\psi_\Sigma\|_{L^2_2}<c\delta$.  By the uniqueness assertion in
Lemma~\ref{lem:IFT1}, equation \eqref{eqn:ii-} also holds, because by
equation \eqref{eqn:HWN}, both $\psi_0$ and $\beta_-(\eta_{-T}+\psi_-) +
\beta_\Sigma\psi_\Sigma$ satisfy equation \eqref{eqn:42}.  

By an analogous process, if $r$ is sufficiently large and $\delta$ is
sufficiently small, then for $i=\{1,\ldots,\Nbar_+\}$ we can extend $\psi_+$
over the $s\le s_++T_+$ portion of the $i^{th}$ negative end of
$u_{+T}$ and extend $\psi_\Sigma$ over the $s\ge s_i$ portion of the
$i^{th}$ positive end of $\Sigma$.  In this way we find
$(\psi_-,\psi_\Sigma,\psi_+)$ satisfying conditions (i) and (ii) from
Step 1.  By the $L^2_2$ bounds on the $\psi$'s from
Lemma~\ref{lem:IFT2}, condition (iii) will also hold if $\delta$ is
sufficiently small.

\medskip

{\em Proof of Lemma~\ref{lem:SGM}, Step 3.\/} We now show that if
$r>r_1$ is sufficiently large, then there exists
$\delta_2\in(0,\delta_1)$ such that if $\delta\le\delta_2$, then the
inputs coming from Definition~\ref{def:CTB} can be modified so that
the extensions $\psi_-,\psi_\Sigma,\psi_+$ produced by Step 2 also
satisfy condition (iv) in Step 1.

To measure the failure of condition (iv), let $\nu_-$ denote the
section of the normal bundle to $u_{-}$ (or $u_{-T}$) determined by
infinitesimal translation of $u_{-}$ in the $\R$ direction in
$\R\times Y$.  Then $|\nu_-|\le 1$; and since $u_-$ has index $1$ and is
unobstructed, $\nu_-$ spans $\Ker(D_-)$.  Define $\nu_+$ analogously
for $u_{+}$ (or $u_{+T}$).  Then condition (iv) holds if and only if
the $L^2$ inner products $\langle
\psi_-,\nu_-\rangle, \langle\psi_+,\nu_+\rangle$ both vanish.

To understand the inner product $\langle
\psi_-,\nu_-\rangle$, it proves convenient to write the
$L^2$ inner product on $u_{-T}$ as $\langle \psi_1,\psi_2\rangle =
\langle \psi_1,\psi_2\rangle_- + \langle \psi_1,\psi_2\rangle_+$, where
$\langle \psi_1,\psi_2\rangle_+$ denotes the contribution from the
$s\ge s_- -T_-+r/2$ portion of $u_{-T}$.  By the asymptotic analysis
in \S\ref{sec:decay}, there are $r$-independent constants
$c_1,\lambda>0$ such that
\begin{equation}
\label{eqn:c1}
|\langle \psi_-,\nu_-\rangle_+| \le c_1\delta\exp(-\lambda r).
\end{equation}
By Lemma~\ref{lem:IFT2}, there is a constant $c_2$ with
\begin{equation}
\label{eqn:c2}
|\langle \psi_-,\nu_-\rangle_-| \le
 c_2\delta.
\end{equation}

Now suppose $\delta<\delta_1$.  Given real numbers $x_-,x_+$ with
$|x_\pm|\le\delta_1 - \delta$, consider modifying the data from
Definition~\ref{def:CTB} by replacing $(R_-,R_+)$ with
$(R_--x_-,R_+-x_+)$, while adjusting the sections $\psi_\pm$
accordingly so as to describe the same curve $C$.  This has the effect
of replacing $(T_-,T_+)$ by $(T_-+x_-,T_+-x_+)$ in Step 1.  The
conditions in Definition~\ref{def:CTB} still hold with $\delta$
replaced by $\delta_1$, so we can repeat the procedure in Step 2, to
obtain a new triple $(\psi_-^x,\psi_\Sigma^x,\psi_+^x)$ obeying
conditions (i)--(iii) in Step 1.

To compare $\psi_-^x$ with $\psi_-^0$, use translation of $s$ to
regard both as sections of $u_-$.  Observe that where $s\le r/2$ on
$u_-$, the image of $e_-\circ\psi_-^x$ is the $s\mapsto s+x_-$
translate of the image of $e_-\circ\psi_-^0$.  Now the $s\mapsto
s+x_-$ translate of $u_-$ is the image of $e_-(x_-\nu_-+\zeta_-)$
where $|\zeta_-|\le c|x_-|^2$.  It follows that where $s\le r/2$ on
$u_-$, we can write
\begin{equation}
\label{eqn:psi-x}
\psi_-^x = \psi_-^0 + x_-\nu_- + \gamma_-, \quad\quad |\gamma_-| \le c
|x_-|(|x_-| + \delta).
\end{equation}

This representation of $\psi_-^x$ 
is not valid on all of $u_-$.
Nonetheless, $\psi_-^x$ and $\psi_+^x$ vary smoothly on the whole of
$u_-$ and $u_+$ as $x=(x_-,x_+)$ varies in the square
$\Omega\subset\R^2$ where both coordinates are less than
$\delta_1-\delta$.  This follows from the inverse function theorem,
since the map $\mc{F}$ that appears in the proof of
Lemma~\ref{lem:IFT2} varies smoothly as $x$ varies.

Now define $f:\Omega\to\R^2$ by sending $x=(x_1,x_2)\in\Omega$
to the pair $(\langle\psi_-^x,\nu_-\rangle,
\langle\psi_+^x,\nu_+\rangle)$.  By the previous paragraph, $f$
is a smooth map.  Expanding $\psi_-^x$ as in \eqref{eqn:psi-x}, using
the estimate \eqref{eqn:c2} for $\psi_-^0$, and using the estimate
\eqref{eqn:c1} for $\psi_-^x$, we find that
\begin{equation}
\label{eqn:psix-}
\langle\psi_-^x,\nu_-\rangle = c_-x_- + \frak{r}_-,
\end{equation}
where $c_-$ depends only on $u_-$ and
\begin{equation}
\label{eqn:r-}
|\frak{r}_-| \le
c(\delta+\exp(-\lambda r)|x_-|+|x_-|^2).
\end{equation}
Studying $\psi^x_+$ in the same way, we obtain analogues of
\eqref{eqn:psix-} and \eqref{eqn:r-} with the `$-$' subscripts replaced
by `$+$' subscripts.  Since $f$ is continuous, a standard topological
argument finds a point $x_0\in\Omega$ with $f(x_0)=0$, provided that
$r$ is sufficiently large and $\delta$ is sufficiently small.
\end{proof}

%% file: obg-deform.tex
Continue with the gluing setup from \S\ref{sec:gluing}.  Recall that
Theorem~\ref{thm:GT} relates gluings of $U_+$ and $U_-$ along a
branched covered cylinder to zeroes of a section
$\frak{s}:\times_2[5r,\infty)\times\mc{M}\to\mc{O}$.  To count the
ends of the corresponding index 2 moduli space, we need to count the
zeroes of $\frak{s}$ over an appropriate slice of the quotient of the
domain by an $\R$-action, as explained below.  It is difficult to
count these zeroes directly because the section $\frak{s}$ involves
functions that are defined only implicitly through
Proposition~\ref{prop:TRE}.  To facilitate this count, we now explain
how to deform $\frak{s}$ to a simpler section, the ``linearized
section'' $\frak{s}_0$, without any zeroes crossing the boundary of
the relevant slice of the domain, so that the count of zeroes does not
change.  After defining the linearized section in \S\ref{sec:DLS}, we
state the deformation result in
\S\ref{sec:SDR} and prove it in
\S\ref{sec:NLE}-\S\ref{sec:PDR}.

\subsection{The linearized section $\frak{s}_0$}
\label{sec:DLS}

We now define the linearized section $\frak{s}_0$.  Note that we
previously defined a version of the linearized section in \S{I.3},
over a slightly different domain; the definition given here is
essentially equivalent to the definition given there, as we will
explain in Remark~\ref{rem:cf}.

Recall the notation $\lambda_i$ from \S\ref{sec:solve-}.  For
$i=1,\ldots,\Nbar_+$ or $i=-1,\ldots,-\Nbar_-$, let $\mc{B}_i$ denote
the $\lambda_i$ eigenspace of $L_{a_i}$.  Recall from the asymptotic
analysis in Proposition~\ref{prop:AF} that there is a constant
$\kappa>0$, and for each $i$ as above there is an eigenfunction
$\gamma_i\in\mc{B}_i$, such that for $i=1,\ldots,\Nbar_+$, the function
$\eta_i$ describing the $i^{th}$ negative end of $u_+$ satisfies
\begin{equation}
\label{eqn:gammai+}
\left|\eta_i(s,\tau) - e^{|\lambda_i| s}\gamma_i(\tau)\right|
\le c e^{(|\lambda_i|+\kappa)s}.
\end{equation}
Likewise, for $i=-1,\ldots,-\Nbar_-$, the function $\eta_i$ describing
the $i^{th}$ positive end of $u_-$ satisfies
\begin{equation}
\label{eqn:gammai-}
\left|\eta_i(s,\tau) - e^{-\lambda_i s}\gamma_i(\tau)\right|
 \le c e^{(-\lambda_i-\kappa)s}.
\end{equation}
Here, as usual, $c$ denotes a constant that depends only on $u_+$ and
$u_-$, but which may change from one appearance to the next. 

We will need to assume that the collection of eigenfunctions
$\gamma=\{\gamma_i\}$ given by \eqref{eqn:gammai+} and
\eqref{eqn:gammai-} is admissible in the sense of Definition~{I.3.2}.
This means that the following two conditions hold:
\begin{description}
\item{(1)} All negative ends of $u_+$ and all positive ends of $u_-$
are nondegenerate.  That is, $\gamma_i\neq 0$ for each
$i\in\{1,\ldots,\Nbar_+\}\cup\{-1,\ldots,-\Nbar_-\}$.
\item{(2)} $u_+$ does not have a pair of overlapping negative ends,
and $u_-$ does not have a pair of overlapping positive ends.  That is,
suppose $i,j\in\{1,\ldots,\Nbar_+\}$ satisfy
$\ceil{a_i\theta}/a_i=\ceil{a_j\theta}/a_j$, or
$i,j\in\{-1,\ldots,-\Nbar_-\}$ satisfy
$\floor{a_i\theta}/a_i=\floor{a_j\theta}/a_j$, so that that the
eigenspaces $\mc{B}_i$ and $\mc{B}_j$ are identified via coverings.
Then for all $g_i\in\Z/a_i$ and $g_j\in\Z/a_j$, the action by deck
transformations satifies $g_i\cdot\gamma_i\neq g_j\cdot\gamma_j$.
\end{description}
Propositions~\ref{prop:G1} and \ref{prop:G2} guarantee that $\gamma$
is admissible if $J$ is generic.

Now fix a branched cover $\Sigma\in\mc{M}$, and let
$\sigma\in\Coker(D_\Sigma)$.  Recall from \S{I.2.3} that the metric on
each $\Sigma\in\mc{M}$ is chosen to agree with the pullback of the
metric on $\R\times S^1$, except on neighborhoods of the ramification
points that project to balls of radius $1$ in $\R\times S^1$.  Let
$D_\Sigma^*$ denote the formal adjoint of $D_\Sigma$ with respect to
this metric, and identify $\sigma$ with an element of
$\Ker(D_\Sigma^*$).  On the complement of the ramification points in
$\Sigma$, use $d\zbar$ to trivialize $T^{0,1}\Sigma$, and thereby
regard $\sigma$ as a complex valued function.  On the complement of
the aforementioned neighborhoods of the ramification points, $\sigma$
satisfies the equation
\begin{equation}
\label{eqn:afrp}
(\partial_s - i\partial_t - S(t))\sigma=0,
\end{equation}
where $S(t)$ is a symmetric $2\times 2$ real matrix, see
\S{I.2.2}.  For $i=1,\ldots,\Nbar_+$, the
restriction of $\sigma$ to the $i^{th}$ positive end of $\Sigma$
determines a function $\sigma_i:[s_i,\infty) \times
\widetilde{S^1}\to\C$, where $\widetilde{S^1}$ denotes the $a_i$-fold
cover of $S^1$.  Likewise, for $i=-1,\ldots,-\Nbar_-$, the restriction
of $\sigma$ to the $i^{th}$ negative end of $\Sigma$ determines a
function $\sigma_i:(-\infty,s_i]\times\widetilde{S^1}\to\C$.  In
either case, equation \eqref{eqn:afrp} on the end becomes
\begin{equation}
\label{eqn:boe}
(\partial_s - L_{a_i})\sigma_i(s,\cdot) = 0.
\end{equation}
Let $\Pi_{\mc{B}}$ denote the projection in
$L^2(\widetilde{S^1};\R^2)$ to $\mc{B}_i$. Then it follows from
\eqref{eqn:boe} that there is an eigenfunction $\beta_i\in\mc{B}_i$ such that
\begin{equation}
\label{eqn:betai}
\Pi_{\mc{B}}\sigma_i(s,\cdot) = e^{\lambda_i(s-s_i)}\beta_i.
\end{equation}

\begin{definition}
Define the {\em linearized section\/}
$\frak{s}_0:\times_2(5r,\infty)\times\mc{M}\to\mc{O}$ as follows: If
$\sigma\in\Coker(D_\Sigma)$ has associated eigenfunctions
$\beta_i\in\mc{B}_i$, then
\begin{equation}
\label{eqn:defsL}
\frak{s}_0(T_-,T_+,\Sigma)(\sigma) \eqdef \sum_{i=1}^{\Nbar_+}
e^{-\nu_i}\langle\gamma_i,\beta_i\rangle -
\sum_{i=-1}^{-\Nbar_-}e^{-\nu_i}\langle\gamma_i,\beta_i\rangle,
\end{equation}
where the brackets denote the inner product on
$L^2(\widetilde{S^1},\R^2)$, and
\begin{equation}
\label{eqn:nui}
\nu_i\eqdef \left\{\begin{array}{cl} |\lambda_i|(s_+-s_i+T_+), &
i=1,\ldots,\Nbar_+,\\
\lambda_i(s_i-s_-+T_-), & i=-1,\ldots,-\Nbar_-.
\end{array}\right.
\end{equation}
\end{definition}

The linearized section $\frak{s}_0$ appears as part of the original
section $\frak{s}$, as follows.  By equation \eqref{eqn:defs}, we can write
\begin{equation}
\label{eqn:WOS}
\frak{s}(T_-,T_+,\Sigma)(\sigma) =
\left\langle\sigma,\eta'+\mc{R}(\psi_\Sigma)\right\rangle,
\end{equation}
where
\[
\eta' \eqdef \frac{1}{2}\left(\frac{\partial\beta_-}{\partial
  s}\eta_{-T} + \frac{\partial\beta_+}{\partial
  s}\eta_{+T}\right)d\overline{z},
\]
while $\mc{R}(\psi_\Sigma)$ denotes the sum of all the other terms in
\eqref{eqn:ThetaSigma} that enter into $\mc{F}_\Sigma(\psi_\Sigma)$.
Recall that $\eta'$ is supported on the ends of $\Sigma$ labeled by
$1,\ldots,\Nbar_+$ and $-1,\ldots,-\Nbar_-$.  Let $\Pi_{\mc{B}}\eta'$
denote the $(0,1)$-form on $\Sigma$ obtained from $\eta'$ by
projecting, for each $i$, the part of $\eta'$ on the $i^{th}$ end onto
the eigenspace $\mc{B}_i$.  Then equation
\eqref{eqn:defsL} can be rewritten as
\begin{equation}
\label{eqn:WSLCF}
\frak{s}_0(T_-,T_+,\Sigma)(\sigma) = 
\sqrt{2} \langle\sigma,\Pi_{\mc{B}}\eta'\rangle.
\end{equation}

\subsection{Counting zeroes of the obstruction section and its linearization}
\label{sec:SDR}

Given $R \ge 10r$,
let $\mc{V}_R$ denote the set of triples
$(T_-,T_+,\Sigma)\in\times_2[5r,\infty)\times\mc{M}$ such that
\begin{equation}
\label{eqn:R}
T_++s_+-s_-+T_-=R.
\end{equation}
This means that the curves $U_+$ and $U_-$ are translated away from
each other by distance $R$ in the pregluing.  We will see in
\S\ref{sec:signs} that the signed count of gluings $\#G(u_+,u_-)$ is
determined by a count of zeroes of $\frak{s}$ on $\mc{V}_R$, modulo a
certain $\R$ action, where $R$ is fixed and large.  We now want to
show that counting zeroes of $\frak{s}_0$ will give the same result.

For this purpose we will linearly interpolate from $\frak{s}$ to
$\frak{s}_0$.  For each $t\in[0,1]$ define a section
\[
\frak{s}_t\eqdef t\frak{s} + (1-t)\frak{s}_0.
\]
The following proposition implies that $\frak{s}_t$ has no zeroes on
the boundary of $\mc{V}_R$ when $R$ is fixed and large.  To state it,
let $\lambda$ denote the smallest of the numbers $|\lambda_i|$ for
$i=1,\ldots,\Nbar_+$ and $i=-1,\ldots,-\Nbar_-$, and let $\Lambda$
denote the largest of these numbers.

\begin{proposition}
\label{prop:deform}
Assume that $\gamma$ is admissible as in \S\ref{sec:DLS}.  In the
gluing construction, if we choose $r$ sufficiently large and
$h<\lambda/4\Lambda$, then for all $t\in[0,1]$, every triple
$(T_-,T_+,\Sigma)\in\times_2[5r,\infty)\times\mc{M}$ with $
\frak{s}_t(T_-,T_+,\Sigma)=0
$
satisfies
\[
T_+, T_- > \frac{\lambda R}{3\Lambda},
\]
where $R$ is defined by \eqref{eqn:R}.
\end{proposition}

To relate this to counting zeroes, first recall that $\R$ acts on the
moduli space of branched covers $\mc{M}$ by translating the $s$
coordinate.  We extend this to an action on
$\times_2[5r,\infty)\times\mc{M}$ fixing the $[5r,\infty)$ factors.
This $\R$ action extends to the obstruction bundle.  That is, if
$\Sigma_1,\Sigma_2\in\mc{M}$ are in the same orbit under the $\R$
action, then there is a canonical isomorphism
$\Coker(D_{\Sigma_1})=\Coker(D_{\Sigma_2})$.  It follows directly from
the definitions that under the above identification,
\[
\frak{s}_t(T_-,T_+,\Sigma_1) = \frak{s}_t(T_-,T_+,\Sigma_2)
\]
for each $t\in[0,1]$.  Thus $\frak{s}_t$ is well defined on
$\times_2[5r,\infty)\times\mc{M}/\R$.

We now want to count zeroes of $\frak{s}_t$ over $\mc{V}_R/\R$, where
$R$ is fixed and large.  For this purpose, note that there is a
natural identification
\begin{equation}
\label{eqn:identifyV}
\mc{V}_R/\R \simeq \{\Sigma\in\mc{M} \mid  
-R/2+5r\le s_-,
s_+\le R/2-5r\}.
\end{equation}
Given a branched cover $\Sigma$ for which $-R/2+5r\le s_-$ and $s_+\le
R/2-5r$, this identification sends
\[
\Sigma\longmapsto(s_-+R/2,R/2-s_+,[\Sigma]) \in \mc{V}_R/\R.
\]
Since $\mc{M}$ is a smooth manifold, it follows that the interior of
$\mc{V}_R/\R$ is a smooth manifold, and the boundary
$\partial(\mc{V}_R/\R)$ is identified with the set of branched covers
$\Sigma$ for which $s_-=-R/2+5r$ or $s_+=R/2-5r$.  Such branched
covers correspond to equivalence classes of triples
$(T_-,T_+,\Sigma)\in\mc{V}_R$ with $T_-=5r$ or $T_+=5r$.  Moreover,
since $\mc{M}$ has a canonical orientation as a complex manifold, the
identification
\eqref{eqn:identifyV} defines an orientation of $\op{int}(\mc{V}_R/\R)$.

\begin{definition}
Assume that $\gamma$ is admissible.  Fix $h<\lambda/4\Lambda$ and
$r>>0$ in the gluing construction.  Given $R>15r\Lambda/\lambda$ and
$t\in[0,1]$, define the {\em relative Euler class\/}
\[
e(\mc{O}\to\mc{V}_R/\R,\frak{s}_t)\in\Z
\]
as follows: Let $\frak{s}_t'$ be a section of $\mc{O}$ over
$\mc{V}_R/\R$ such that $\frak{s}_t'=\frak{s}_t$ on
$\partial(\mc{V}_R/\R)$, and such that all zeroes of $\frak{s}_t'$ are
nondegenerate.  Define $e(\mc{O}\to\mc{V}_R/\R,\frak{s}_t)$ to be the
signed count of zeroes of $\frak{s}_t'$, using the orientation of
$\mc{V}_R/\R$ determined by \eqref{eqn:identifyV} and the orientation of
$\mc{O}$ defined in \S{I.2.6}.  We usually denote this count by
$\#(\frak{s}_t^{-1}(0)\cap\mc{V}_R/\R)$, even though the zeroes of
$\frak{s}_t$ itself on $\mc{V}_R/\R$ may be degenerate.
\end{definition}

\begin{lemma}
The relative Euler class $\#(\frak{s}_t^{-1}(0)\cap\mc{V}_R/\R)$ is
well defined and does not depend on the choice of
$R>15r\Lambda/\lambda$ or $t\in[0,1]$.
\end{lemma}

\begin{proof}
We know from Proposition~\ref{prop:sCont} that the family of sections
$\frak{s}_t$ is continuous.  Furthermore,
Proposition~\ref{prop:deform} guarantees that $\frak{s}_t$ is
nonvanishing near $\partial(\mc{V}_R/\R)$.  Hence the only issue is to
check that $\mc{V}_R/\R$ is compact.

For this purpose, the key is to show that
\begin{description}
\item{(*)}
For any $\Sigma\in\mc{M}$, all ramification points have
$\pi^*s\in[s_-,s_+]$.
\end{description}
To prove (*), recall that since $(U_+,U_-)$ is a gluing pair, it is
required that under the partial order $\ge_\theta$ in
Definition~{I.1.8}, the partition $(a_{\Nbar_++1},\ldots,a_{N_+})$ is
minimal and the partition $(a_{-\Nbar_--1},\ldots,a_{-N_-})$ is
maximal.  Now suppose that $\Sigma\in\mc{M}$ has a ramification point
with $\pi^*s>s_+$.  Then we can decompose
$\Sigma=\Sigma_1\sqcup\Sigma_2$, where $\Sigma_1$ contains this
ramification point and has positive ends indexed by
$\Nbar_++1,\ldots,N_+$.  Our standing assumption
\eqref{eqn:ii} implies that $\Sigma$ has index zero, and hence so do
$\Sigma_1$ and $\Sigma_2$, see \S{I.1.2}.  The existence of $\Sigma_1$
directly contradicts the minimality of the partition
$(a_{\Nbar_++1},\ldots,a_{N_+})$.  Likewise, maximality of the
partition $(a_{-\Nbar_--1},\ldots,a_{-N_-})$ forbids the existence of
a ramification point with $\pi^*s<s_-$.

It follows from (*) that \eqref{eqn:identifyV} identifies
$\mc{V}_R/\R$ with a subset of $\mc{M}_{R/2-5r}$.  Now
$\mc{M}_{R/2-5r}$ is compact by the assumption \eqref{eqn:ii} and
Lemma~{I.2.8}.  Hence any sequence in $\mc{V}_R/\R$ has a subsequence
whose corresponding branched covers converge to some element of
$\mc{M}_{R/2-5r}$.  By continuity of the functions $s_+$ and $s_-$,
this limiting branched cover corresponds to an element of
$\mc{V}_R/\R$.
\end{proof}

\begin{remark}
\label{rem:cf}
To coordinate the counting here with that in Part I, we need to
compare the formalism here with that of \S{I.3}.  In \S{I.3.1}, given
$R,r>0$ and given admissible eigenfunctions $\gamma$, we defined
$\frak{s}_0$ as a section over $\mc{M}_R$.  By (*) above,
\eqref{eqn:identifyV} identifies $\mc{V}_R/\R$ with a subset of
$\mc{M}_{R/2-5r}$ that contains all of $\mc{M}_{R/2-5r-1}$.  Under
this identification, the definition of $\frak{s}_0$ over $\mc{V}_R/\R$
given here is a special case of the definition in
\S{I.3.1}, where $(R,r)$ here corresponds to $(R/2-5r,5r)$ there, and
where we take $\gamma$ to be the eigenfunctions determined by the
negative ends of $u_+$ and the positive ends of $u_-$.  In particular,
the $R$-independent count $\#(\frak{s}_0^{-1}(0)\cap\mc{V}_R)/\R)$
defined above agrees with the corresponding count
$\#\frak{s}_0^{-1}(0)$ defined in \S{I.3.2}.  The reason is that
Proposition~\ref{prop:deform} implies that $\frak{s}_0$ has no zeroes
on $\mc{M}_{R/2-5r}\setminus\mc{M}_{R/2-5r-1}$ if $R>>r>>0$.
\end{remark}

In conclusion, we have:

\begin{corollary}
\label{cor:ecd}
Assume that $J$ is generic so that $\gamma$ is admissible.  If $r$ is
chosen sufficiently large and if $h$ is chosen sufficiently small in
the gluing construction, then for $R$ sufficiently large, the relative
Euler class $\#(\frak{s}^{-1}(0)\cap\mc{V}_R/\R)$ is well-defined,
does not depend on $R$, and satisfies
\begin{equation}
\label{eqn:thecount2}
\#(\frak{s}^{-1}(0)\cap\mc{V}_R/\R) =
\#(\frak{s}_0^{-1}(0)\cap\mc{V}_R/\R).
\end{equation}
\end{corollary}

In \S\ref{sec:signs} we will relate the the left hand side of
\eqref{eqn:thecount2} to the signed count of gluings
$\#G(u_+,u_-)$.  The rest of \S\ref{sec:deform} is devoted to the
proof of Proposition~\ref{prop:deform}, beginning with some
preliminary lemmas.

\subsection{Nonlinear estimate}
\label{sec:NLE}

The proof of Proposition~\ref{prop:deform} will use an upper bound
on the term $\langle\sigma,\mc{R}(\psi_\Sigma)\rangle$ in equation
\eqref{eqn:WOS}.  To state this bound, 
let $\nu$ denote the smallest of the numbers $\nu_i$ defined in
\eqref{eqn:nui}.  Also recall the norms $\|\cdot\|$ and $\|\cdot\|_*$ from
\S\ref{sec:BSS}.

\begin{lemma}
\label{lem:NLE}
Suppose that in the gluing construction, $h$ is chosen so
that $4h\Lambda<\lambda$.  Then
\begin{equation}
\label{eqn:NLE}
\|\mc{R}(\psi_\Sigma)\| \le c e^{-\nu -
\lambda r/2}.
\end{equation}
\end{lemma}

\begin{proof}
Estimating $\|\mc{R}(\psi_\Sigma)\|$ as in \eqref{eqn:FSEE}, we
find that
\begin{equation}
\label{eqn:NLE1}
\|\mc{R}(\psi_\Sigma)\| \le c\left(\|\psi_\Sigma\|_*^2 + e^{-\lambda
r}\|\psi_\Sigma\|_*\right).
\end{equation}
Note that the term $e^{-\lambda r}\|\psi_\Sigma\|_*$ appears here
because of the $\frak{q}_0\cdot\psi_\Sigma$ and
$\frak{q}_0'\cdot\nabla\psi_\Sigma$ terms in $\mc{R}(\psi_\Sigma)$.

Using Proposition~\ref{prop:TRE}(b) and our assumption that
$4h\Lambda<\lambda$, we estimate
\begin{equation}
\label{eqn:NLE2}
\begin{split}
\|\psi_\Sigma\|_* & \le 
c \left(\sum_{i=1}^{\Nbar_+}e^{-|\lambda_i|(s_+-s_i + T_+ - 2rh)}
+
\sum_{i=-1}^{-\Nbar_-} e^{-\lambda_i(s_i - s_- + T_- - 2rh)}\right)\\
&= 
c \left(\sum_{i=1}^{\Nbar_+}e^{-\nu_i + 2rh|\lambda_i|}
+
\sum_{i=-1}^{-\Nbar_-} e^{-\nu_i + 2rh\lambda_i}\right) \\
& \le c e^{-\nu + \lambda r/2}.
\end{split}
\end{equation}
Putting \eqref{eqn:NLE2} into \eqref{eqn:NLE1} gives
\begin{equation}
\label{eqn:NLE3}
\|\mc{R}(\psi_\Sigma)\| \le c\left(e^{-2\nu + \lambda r} + e^{-\nu -
\lambda r/2}\right).
\end{equation}
It follows from the definitions that $\nu \ge 5 \lambda r$.  Hence
\eqref{eqn:NLE3} implies \eqref{eqn:NLE}.
\end{proof}

\subsection{Ends with the same eigenvalue}
\label{sec:keylemma}

The proof of Proposition~\ref{prop:deform} will also need
Lemma~\ref{lem:keylemma} below, regarding the structure of the
cokernel in the case when $N_-=1$ and the eigenvalues
$\lambda_1,\ldots,\lambda_{N_+}$ are all equal, say to $\lambda$.  By
Remark~{I.2.12}, this last condition is equivalent to
\begin{equation}
\label{eqn:sev}
\frac{\ceil{a_1\theta}}{a_1} =\cdots=\frac{\ceil{a_{N_+}\theta}}{a_{N_+}}.
\end{equation}

Recall from \S{I.3.1} that the assumption
\eqref{eqn:sev} allows us to identify all the eigenspaces
$\mc{B}_1,\ldots,\mc{B}_{N_+}$ with each other via coverings as
follows.  Write $\ceil{a_1\theta}/a_1 = \eta_0/m_0$ where $\eta_0$ and
$m_0$ are integers and $m_0>0$ is as small as possible.  Then $a_i$ is
divisible by $m_0$ for each $i=1,\ldots,N_+$.  Fix an eigenfunction
$\varphi_{m_0}$ of $L_{m_0}$ with eigenvalue $\lambda$.  Then for
every positive integer $d$, the eigenfunction $\varphi_{m_0}$ pulls
back to an eigenfunction $\varphi_{dm_0}$ of $L_{dm_0}$ with the same
eigenvalue $\lambda$.  There is now a canonical isomorphism from
$\mc{B}_i$ to $\mc{B}_j$ sending $\varphi_{a_i}$ to $\varphi_{a_j}$.
This identification is made implicitly below.

Also note that the product of cyclic groups
\[
G \eqdef \Z/a_1\times\cdots\times \Z/{a_{N_+}}
\]
acts on $\bigoplus_{i=1}^{N_+}\mc{B}_i$.  Here the $i^{th}$ factor
$\Z/a_i$ acts on $\mc{B}_i$ by deck transformations of the
corresponding eigenfunctions, and trivially on $\mc{B}_j$ for $j\neq
i$.

Now given $\Sigma\in\mc{M}$, define $\Pi^+:\Coker(D_\Sigma)\to
\bigoplus_{i=1}^{N_+}\mc{B}_i$ as follows.  Given
$\sigma\in\Coker(D_\Sigma)$, for $i=1,\ldots,N_+$ write
$\Pi_{\mc{B}}\sigma_i(s,\cdot)=e^{\lambda s}\zeta_i$; then
\[
\Pi^+(\sigma) \eqdef (\zeta_1,\ldots,\zeta_{N_+}).
\]

\begin{lemma}
\label{lem:keylemma}
Suppose that $N_-=1$ and that \eqref{eqn:sev} holds.  Then for any
$\Sigma\in\mc{M}$, there exists $g\in G$ such that
\begin{equation}
\label{eqn:key}
g\cdot \Pi^+(\Coker(D_\Sigma)) \subset \left\{(\zeta_1,\ldots,\zeta_{N_+})
\;\bigg|\;
\sum_{i=1}^{N_+}\zeta_i=0\right\}.
\end{equation}
\end{lemma}

\begin{proof}
Assume below that the eigenvalue $\lambda$ of
$L_{a_1},\ldots,L_{a_{N_+}}$ is not repeated; the proof when $\lambda$
is repeated is similar.

Let $\Sigma\in\mc{M}$.  Recall from \S{I.2.1} that $\Sigma$ determines
an oriented weighted tree $\tau(\Sigma)$ whose edges correspond to
cylinders in $\Sigma$ between ramification points.  By downward
induction, for each edge $e$ of the tree $\tau(\Sigma)$, the covering
multiplicity $m(e)$ of the corresponding cylinder in $\Sigma$ is
divisible by $m_0$. Hence we can lift the branched covering $\pi$ on
$\Sigma$ to a continuous map
\[
\widetilde{\pi}: \Sigma \longrightarrow \R\times\R/2\pi m_0\Z.
\]
For each $i=1,\ldots,N_+$, choose an identification of the $i^{th}$
positive end of $\Sigma$ with $[s_i,\infty)\times\R/2\pi a_i\Z$, such
that the projection to $[s_i,\infty)\times\R/2\pi m_0\Z$ agrees with
$\widetilde{\pi}$.  This identification will differ from the
asymptotic marking by the action of some $g_i\in\Z/a_i$.  Let $g\eqdef
(g_1,\ldots,g_{N_+})$.

Next, let $\sigma\in\Coker(D_\Sigma)$, and write $g\cdot\Pi^+\sigma =
(\zeta_1,\ldots,\zeta_{N_+})$.  Fix a smooth function
$\chi:\R\to[0,1]$ such that $\chi(s)=1$ when $s\le s_++2$ and
$\chi(s)=0$ when $s\ge s_++3$.  Write $\widetilde{\pi} =
(s,\widetilde{t})$.  Define a function $f:\Sigma\to\C$ by
\[
f \eqdef \chi(s) e^{-\lambda s}
\varphi_{m_0}(\widetilde{t}).
\]
Recall that on any cylinder in $\Sigma$ corresponding to an edge $e$
of $\tau(\Sigma)$, the operator $D_\Sigma$ has the form
$\frac{1}{2}d\zbar\tensor(\partial_s+L_{m(e)})$.  It follows that
\[
D_\Sigma f = \frac{1}{2}
(\partial_s\chi) e^{-\lambda s}\varphi_{m_0}(\widetilde{t})\, d\zbar.
\]
In particular, $D_\Sigma f$ is supported only where $s_++2 \le s \le
s_++3$, and here the metric on $\Sigma$ agrees with the pullback of
the metric on $\R\times S^1$.  Also, $f$ is $L^2_1$ since $\lambda <
0$.  We then have
\[
0 = \langle D_\Sigma^*\sigma, f\rangle = \langle \sigma, D_\Sigma f
\rangle = \sum_{i=1}^{N_+}\langle
\zeta_i,\varphi_{a_i}\rangle,
\]
where the brackets denote the relevant $L^2$ inner products.  Under
our identifications of the $\mc{B}_i$'s, this means that
$\sum_{i=1}^{N_+}\zeta_i=0$.
\end{proof}

It is not hard to further show, using Lemma~{I.2.18}(a), that the inclusion in
\eqref{eqn:key} is actually an equality.  However we will not need this.

\subsection{Proof of the deformation result}
\label{sec:PDR}

With the preliminaries in place, we now prove
Proposition~\ref{prop:deform}.  Fix $h<\lambda/4\Lambda$ and suppose
that Proposition~\ref{prop:deform} is false for this $h$.  Then:
\begin{description}
\item{(*)}
For each $n=1,2,\ldots$ there exist real numbers $r_n\ge n$ and
$t_n\in[0,1]$, and a triple $(T_{-n},T_{+n},\Sigma_n)\in\times_2
[5r_n,\infty)\times\mc{M}$, such that if we fix $r=r_n$ in the gluing
construction, then $\frak{s}_{t_n}(T_{-n},T_{+n},\Sigma_n)=0$ and
$\min\{T_{-n},T_{+n}\}
\le \lambda R_n/3\Lambda$.
\end{description}
Here $R_n\eqdef T_{+n}+s_{+n} - s_{-n}+T_{-n}$ where $s_{\pm n}$ denotes the
value of $s_\pm$ for $\Sigma_n$.  We will use (*) to deduce a
contradiction, in four steps.

\medskip

{\em Step 1.\/} We begin with some setup.  Recall that associated to
each $\Sigma_n$ is a tree $\tau(\Sigma_n)$, with a projection
$p:\Sigma_n\to\tau(\Sigma_n)$ and a metric coming from the $s$
coordinate.  By passing to a subsequence, we may assume that the
sequence of branched covers $\{\Sigma_n\}$ in $\mc{M}/\R$ converges in
the sense of Definition~{I.2.27} to a tree $\tau_*$ together with a
branched cover $\Sigma_{*j}$ for each internal vertex $j$ of $\tau_*$.
Note that conditions (a)--(d) in Definition~{I.2.27} imply that:
\begin{itemize}
\item
For each $\Sigma_n$ in the sequence and for each internal vertex $j$
of $\tau_*$, there is a corresponding set $\Lambda_{nj}$ of
ramification points in $\Sigma_n$.  The tree $\tau_*$ is obtained from
the tree $\tau(\Sigma_n)$ by, for each $j$, collapsing all the
vertices in $\tau(\Sigma_n)$ corresponding to ramification points in
$\Lambda_{nj}$ and all the edges between them to the $j^{th}$ vertex of
$\tau_*$.
\item
There is an $n$-independent constant $\Delta_*$ such that any two
ramification points in the same $\Lambda_{nj}$ project to points in
the tree $\tau(\Sigma_n)$ with distance $\le
\Delta_*$.
\item
If $j$ and $j'$ are distinct internal vertices of $\tau_*$, then for
each $n$, in the tree $\tau(\Sigma_n)$ we have
$\lim_{n\to\infty} \text{dist}(p(\Lambda_{nj}),p(\Lambda_{nj'}))=\infty$.
\end{itemize}
By passing to a further subsequence, we may improve this last
condition to
\[
\text{dist}(p(\Lambda_{nj}),p(\Lambda_{nj'}))\ge n.
\]

Now fix $n$ large and drop the `$n$' subscripts below.  Choose $i_1$
such that $\nu_{i_1}=\nu$ for $\Sigma=\Sigma_n$.  Without loss of
generality, $i_1\in\{1,\ldots,\Nbar_+\}$.  Let $j$ denote the internal
vertex of $\tau_*$ that is adjacent to the leaf $i_1$.  If $e$ is an
edge of $\tau_*$ incident to $j$, call $e$ ``essential'' if $e$ is
incident to a leaf $i$ with $\lambda_i=\lambda_{i_1}$.  In particular,
this requires that $i$ is positive.  If $e$ is an edge of $\tau_*$
incident to $j$ which is either internal or incident to a leaf $i$
with $\lambda_i\neq\lambda_{i_1}$, call $e$ ``inessential''.

\medskip

{\em Step 2.\/}
We claim that there is an $n$-independent constant $\kappa>0$ such
that if there is an inessential edge from the vertex $j$ to the leaf
$i$, then $\nu_i \ge \nu + \kappa n$ when $n$ is large.

To prove this when $i$ is positive, we compute that
\begin{equation}
\label{eqn:inessential}
\nu_i - \nu  = (|\lambda_i| - |\lambda_{i_1}|)(s_+-s_{i_1}+T_+) +
|\lambda_i|(s_{i_1}-s_i),
\end{equation}
and observe that $s_+-s_{i_1}+T_+ \ge 5n$ and
$|s_{i_1}-s_i|\le
\Delta_*$.
Note that $|\lambda_i|<|\lambda_{i_1}|$ is impossible when $n$ is
large, because then \eqref{eqn:inessential} would imply that
$\nu_i<\nu$, contradicting the definition of $\nu$.  Since the edge
from $j$ to $i$ is essential, the only remaining possibility is that
$|\lambda_i|>|\lambda_{i_1}|$.  The claim now follows immediately from
\eqref{eqn:inessential}.

Suppose next that $i$ is negative.  It follows from the definitions
that
\[
\begin{split}
R 
&= \frac{\nu}{|\lambda_{i_1}|} + \frac{\nu_i}{\lambda_i} +
(s_{i_1} - s_i)
\le \frac{\nu+\nu_i}{\lambda} + \Delta_*,
\end{split}
\]
and so
\begin{equation}
\label{eqn:RD}
\nu_i \ge \lambda R - \nu - \lambda\Delta_*.
\end{equation}
Now $T_+\ge \nu/\Lambda$, because there exists a positive end $i'$
with $s_{i'}=s_+$.  Likewise $T_-\ge \nu/\Lambda$.  The assumption in
(*) that $\min\{T_-,T_+\}\le \lambda R/3\Lambda$ then implies that
$\lambda R \ge 3 \nu$.  Putting this into \eqref{eqn:RD} and using the
fact that $\nu\ge 5\lambda n$ proves the claim.

\medskip

{\em Step 3.\/} We now complete the proof of
Proposition~\ref{prop:deform} in the case when there are at least two
inessential edges incident to $j$.

Choose paths in $\tau_*$ starting along these edges to leaves $i_2$,
$i_3$.  Let $w$ denote the central vertex in the tree
$\tau(\Sigma)$ for $i_1$, $i_2$, and $i_3$.  The vertex $w$ in
$\tau(\Sigma)$ projects to the vertex $j$ in $\tau_*$.

By Lemma~{I.2.18}, there is a unique, nonvanishing
$\sigma\in\Coker(D_\Sigma)$ with
\[
\Pi_\mc{A}\sigma_{i_1}(s,\cdot)= e^{\lambda_{i_1}(s-s_{i_1})}\gamma_{i_1}
\]
and $\Pi_\mc{A}\sigma_i=0$ for $i\notin\{i_1,i_2,i_3\}$.  Here
$\Pi_{\mc{A}}\sigma_i$ denotes the projection of $\sigma_i$ onto the
subspace $\mc{A}_i\supset\mc{B}_i$ consisting of the (two-dimensional)
span of the eigenfunctions of $L_{a_i}$ that have the same winding
number as the eigenfunctions with eigenvalue $\lambda_i$.

For each leaf $i$, let $\beta_i$ denote the corresponding
eigenfunction associated to $\sigma$ via \eqref{eqn:betai}.  In
particular $\beta_{i_1}=\gamma_{i_1}$.  The plan is to show that if
$n$ is large, then $\frak{s}_{t}(T_-,T_+,\Sigma)(\sigma)$ is dominated
by the term $e^{-\nu_{i_1}}\langle\gamma_{i_1},\beta_{i_1}\rangle$,
and in particular nonzero.  This will give the desired contradiction
to (*).

To start, we claim that the special cokernel element $\sigma$ decays
away from the central vertex $w$, in the following sense: There are
$n$-independent constants $c,\kappa>0$ such that for any point
$x\in\Sigma=\Sigma_n$,
\begin{equation}
\label{eqn:DACV}
|\sigma(x)| < c e^{-\kappa\cdot \text{dist}(p(x),w)}.
\end{equation}
Indeed, Corollary {I.2.23} gives \eqref{eqn:DACV} with the right hand
side multiplied by $|\sigma(\widetilde{w})|$, where
$\widetilde{w}\in\Sigma$ projects to $w$;  and Propositions {I.2.21} and
{I.2.25} imply that $|\sigma(\widetilde{w})|<c$.

The inequality \eqref{eqn:DACV} has two important consequences.
First,
\begin{equation}
\label{eqn:s2c}
\|\sigma\|_{L^2} < c.
\end{equation}
Second, if a leaf $i\in\{i_2,i_3\}$ is not adjacent to the internal vertex
$j$ of $\tau_*$, then
\begin{equation}
\label{eqn:bkn}
\|\beta_i\| \le \|\Pi_\mc{A}\sigma_i(s_i,\cdot)\| < c e^{-\kappa n}.
\end{equation}
Likewise, if $i\in\{i_2,i_3\}$ is adjacent to $j$, then
\begin{equation}
\label{eqn:bka}
\|\beta_i\| < c.
\end{equation}

Now to show that $\frak{s}_t(T_-,T_+,\Sigma)(\sigma)$ is nonzero, use
\eqref{eqn:WOS} and \eqref{eqn:WSLCF} to write the latter
 as a sum of three terms:
\begin{equation}
\label{eqn:so3t}
\begin{split} 
\frak{s}_t(T_-,T_+,\Sigma)(\sigma) = & 
\big(1+t\big(\sqrt{2}-1\big)\big)\frak{s}_0(T_-,T_+,\Sigma)(\sigma)\\
&
+
t\langle\sigma,\mc{R}(\psi_\Sigma)\rangle
+
t\langle\sigma,\eta'-\Pi_{\mc{B}}\eta'\rangle.
\end{split}
\end{equation}
Our choice of $\sigma$ implies that the first term in \eqref{eqn:so3t}
is given by
\begin{equation}
\label{eqn:sb1}
\frak{s}_0(T_-,T_+,\Sigma)(\sigma) = e^{-\nu}\|\gamma_{i_1}\|^2 \pm
e^{-\nu_{i_2}}\langle\gamma_{i_2},\beta_{i_2}\rangle \pm
e^{-\nu_{i_3}}\langle\gamma_{i_3},\beta_{i_3}\rangle.
\end{equation}
Also observe that if $i\in\{i_2,i_3\}$, then
\begin{equation}
\label{eqn:sb2}
\left|e^{-\nu_{i}}\left\langle\gamma_{i},
\beta_{i}\right\rangle\right|
 < c e^{-\nu - \kappa n}.
\end{equation}
If the leaf $i$ is adjacent to the vertex $j$ in $\tau_*$, then this
follows from \eqref{eqn:bka} and Step 2; otherwise this follows from
\eqref{eqn:bkn}.  Next, the inequality \eqref{eqn:s2c} and
Lemma~\ref{lem:NLE}, together with the fact that $r=r_n\ge n$, imply
that
\begin{equation}
\label{eqn:sb3}
|\langle\sigma,\mc{R}(\psi_\Sigma)\rangle| \le c e^{-\nu-\lambda n/2}.
\end{equation}
Finally, Proposition~{I.2.25} and the inequality
\eqref{eqn:DACV}, together with the fact that $r\ge n$, imply that if
$\kappa>0$ is chosen sufficiently small, then
\begin{equation}
\label{eqn:sb4}
\left|\langle\sigma,\eta'-\Pi_{\mc{B}}\eta'\rangle\right| \le
 c e^{-\nu - \kappa n}.
\end{equation}
By the nondegenerate ends assumption, $\|\gamma_{i_1}\|>0$.  Hence
\eqref{eqn:so3t}--\eqref{eqn:sb4} imply that
$\frak{s}_t(T_-,T_+,\Sigma)(\sigma)\neq 0$ if $n$ is sufficiently
large.

\medskip

{\em Step 4.\/} To complete the proof of
Proposition~\ref{prop:deform}, we need to handle the case where there
is at most one inessential edge incident to $j$.  In this case $j$ has
one incoming, inessential edge, and $k\ge 2$ outgoing edges, all of
which are essential.  Denote the positive leaves corresponding to the
essential edges by $i_1,\ldots,i_k$, where as before, $\nu_{i_1}=\nu$.

Pick a leaf $i_0\notin\{i_1,\ldots,i_k\}$.  By
Lemma~{I.2.18}, there is a unique, nonvanishing
$\sigma\in\Coker(D_\Sigma)$ with
\[
\Pi_{\mc{A}}\sigma_{i_1}(s,\cdot) = e^{\lambda(s-s_{i_1})}\gamma_{i_1}
\]
and $\Pi_{\mc{A}}\sigma_i=0$ for $i\notin\{i_0,i_1,i_2\}$.  For each
leaf $i$, let $\beta_i$ denote the corresponding eigenfunction
associated to $\sigma$ via \eqref{eqn:betai}.  In particular
$\beta_{i_1}=\gamma_{i_1}$.

Proposition~{I.6.3} tells us (roughly) that if $n$ is large, then
near the part of $\Sigma$ that gets collapsed to the vertex $j$ of
$\tau_*$, the $(0,1)$-form $\sigma$ is well approximated by a
nonvanishing cokernel element $\sigma_Z$ for a branched cover
\[
\Sigma_Z\in\mc{M}(a_{i_1},\ldots,a_{i_k} \mid a_{i_1}+\cdots+a_{i_k}).
\]
Note that $\sigma_Z$ will be an honest cokernel element, and not one
of the more general elements of $\widetilde{\Coker}(D_{\Sigma_Z})$
allowed by Proposition~{I.6.3}, as a consequence of Lemma~{I.2.20}.
Combining the precise result of Proposition~{I.6.3} with
Lemma~\ref{lem:keylemma}, we find that for any $\epsilon>0$, if $n$ is
sufficiently large then there exists a deck transformation $g$ of the
covering $\widetilde{S^1}\to S^1$ such that in the notation from
\S\ref{sec:keylemma},
\[
\zeta_{i_2} = -(1+O(\epsilon))g\cdot\zeta_{i_1},
\]
where `$O(\epsilon)'$ here denotes a number with absolute value less
than $\epsilon$.  Hence the part of
$\frak{s}_0(T_-,T_+,\Sigma)(\sigma)$ coming from the leaves $i_1$ and
$i_2$ is given by
\begin{equation}
\label{eqn:part}
e^{-\nu_{i_1}}\langle\gamma_{i_1},\beta_{i_1}\rangle 
+
e^{-\nu_{i_2}}\langle\gamma_{i_2},\beta_{i_2}\rangle
=
e^{-\nu}\langle\gamma_{i_1}-(1+O(\epsilon))g\cdot
\gamma_{i_2},\gamma_{i_1}\rangle.
\end{equation}

On the other hand, by the nonoverlapping ends assumption,
$\gamma_{i_1}\neq g\cdot\gamma_{i_2}$.  Since $\gamma_{i_1}\neq 0$, it
follows that there is a constant $c>0$ such that if $n$ is
sufficiently large, then the expression in \eqref{eqn:part} is greater
than $c e^{-\nu}$.  A virtual repeat of the arguments in Step 3
concludes that $\frak{s}_t(T_-,T_+,\Sigma)(\sigma)\neq 0$ if $n$ is
sufficiently large.  This completes the proof of
Proposition~\ref{prop:deform}.
\qed

%% file: obg-coherent.tex

This section consists of a lengthy digression on how to ``coherently''
orient all moduli spaces of unobstructed, immersed, $J$-holomorphic
curves in $\R\times Y$, so that the orientations behave well under
gluing of the usual kind where there is no obstruction bundle.  This
is, up to some choices, an established procedure (and one can also
allow non-immersed curves). However we will need to rework it from a
special perspective in order to set up the discussion of signs in
obstruction bundle gluing in \S\ref{sec:signs}.

\subsection{Algebraic preliminaries}
\label{sec:AP}

We begin by reviewing some very basic material about orientations, in
order to fix notation.

If $V$ is a finite dimensional vector space over $\R$, let $\mc{O}(V)$
denote the set of orientations of $V$.  If $\frak{o}\in\mc{O}(V)$, we
denote the opposite orientation by $-\frak{o}$.  If $W$ is another
finite dimensional vector space over $\R$, define
$\mc{O}(V)\tensor\mc{O}(W)$ to be the set of pairs
$(\frak{o}_V,\frak{o}_W)\in\mc{O}(V)\times\mc{O}(W)$, modulo the
relation $(\frak{o}_V,\frak{o}_W)\sim(-\frak{o}_V,-\frak{o}_W)$.
There is a canonical isomorphism
\begin{equation}
\label{eqn:conbas}
\mc{O}(V)\tensor\mc{O}(W)\stackrel{\simeq}{\longrightarrow}\mc{O}(V\oplus W)
\end{equation}
obtained by concatenating bases.  More generally, an exact sequence of
finite dimensional vector spaces over $\R$,
\begin{equation}
\label{eqn:lesfvs}
0 \longrightarrow V_1 \stackrel{f_1}{\longrightarrow} V_2
\stackrel{f_2}{\longrightarrow} \cdots
\stackrel{f_{k-1}}{\longrightarrow} V_k \longrightarrow 0,
\end{equation}
induces an element
\[
\frak{o}(f_1,\ldots,f_{k-1})\in
\mc{O}(V_1)\tensor\cdots\tensor\mc{O}(V_k)
\]
defined as follows.  Choose a basis $(v_{1,1},\ldots,v_{1,n_1})$ of
$V_1$, and let $\frak{o}_1\in\mc{O}(V_1)$ denote the corresponding
orientation.  For $i=2,\ldots,k-1$, choose elements
$v_{i,1},\ldots,v_{i,n_i}\in V_i$ such that
$(f_{i-1}(v_{i-1,1}),\ldots,
f_{i-1}(v_{i-1,n_{i-1}}),v_{i,1},\ldots,v_{i,n_i})$ is a basis of
$V_i$, and let $\frak{o}_i\in\mc{O}(V_i)$ denote the corresponding
orientation.  Then
\begin{equation}
\label{eqn:oes}
\frak{o}(f_1,\ldots,f_{k-1})\eqdef
\frak{o}_1\tensor\cdots\tensor\frak{o}_k.
\end{equation}
This orientation does not depend on the choice of the elements
$v_{i,j}$; in fact it is induced by an isomorphism of tensor
products of determinant lines
\[
\bigotimes_{\mbox{\scriptsize $i$ even}}\det(V_i) \simeq
\bigotimes_{\mbox{\scriptsize $i$ odd}}\det(V_i) 
\]
which depends only on the long exact sequence \eqref{eqn:lesfvs}, see
e.g.\ \cite{fh}.  In addition, \eqref{eqn:oes} is invariant under
homotopy of exact sequences.  That is, if $\{f_i(t)\}_{t\in[0,1]}$ are
homotopies of maps $V_i\to V_{i+1}$ for $i=1,\ldots,k-1$, such that
the sequence given by $f_1(t),\ldots,f_{k-1}(t)$ is exact for each
$t\in[0,1]$, then
\[
\frak{o}(f_1(0),\ldots,f_{k-1}(0)) =
\frak{o}(f_1(1),\ldots,f_{k-1}(1)).
\]

If $D$ is a Fredholm operator, define
\[
\mc{O}(D)\eqdef\mc{O}(\Ker(D))\tensor\mc{O}(\Coker(D)).
\]
If $\{E_t\}_{t\in[0,1]}$ and $\{F_t\}_{t\in[0,1]}$ are Banach space
bundles over $[0,1]$, then a continuous path $\{D_t:E_t\to
F_t\}_{t\in[0,1]}$ of Fredholm operators induces an isomorphism
\begin{equation}
\label{eqn:continuation}
\Phi_{\{D_t\}}:\mc{O}(D_0)\stackrel{\simeq}{\longrightarrow}\mc{O}(D_1),
\end{equation}
defined as follows.  One can choose decompositions $F_t\simeq
V_t\oplus W_t$, depending continuously on $t$, such that $V_t$ is
finite dimensional, and if $\Pi_{W_t}:F_t\to W_t$ denotes the
projection to $W_t$, then $\Pi_{W_t}D_t$ is surjective.  There is
then, for each $t$, an exact sequence
\begin{equation}
\label{eqn:OES}
0 \longrightarrow \Ker(D_t) \longrightarrow \Ker(\Pi_{W_t}D_t)
\stackrel{D_t}{\longrightarrow} V_t \longrightarrow \Coker(D_t)
\longrightarrow 0.
\end{equation}
The exact sequence \eqref{eqn:OES} induces an isomorphism
\begin{equation}
\label{eqn:sesii}
\mc{O}(D_t) \simeq \mc{O}(\Pi_{W_t}D_t)\tensor\mc{O}(V_t).
\end{equation}
Now the family of subspaces $\{\Ker(\Pi_{W_t}D_t)\}_{t\in[0,1]}$ defines a
vector bundle over $[0,1]$, as does the family of subspaces
$\{V_t\}_{t\in[0,1]}$.  These vector bundles induce
isomorphisms $\mc{O}(\Pi_{W_0}D_0)\simeq\mc{O}(\Pi_{W_1}D_1)$ and
$\mc{O}(V_0)\simeq\mc{O}(V_1)$.  Combining these isomorphisms with
\eqref{eqn:sesii} gives the isomorphism \eqref{eqn:continuation}.
The latter does not depend on $V_t$ and $W_t$, and is invariant under
homotopy of the path $\{D_t\}$ rel endpoints.

\subsection{A linear gluing exact sequence}
\label{sec:LGES}

We now present a variant of the ``linear gluing'' construction of
\cite{bm,fh}, designed to fit well in the obstruction bundle context.

We first introduce a class of Fredholm operators that one needs to
orient, in order to orient moduli spaces of pseudoholomorphic curves. Fix a
positive integer $n$. (The main concern of this paper is the case
$n=1$.)

\begin{definition}
An {\em orientation triple\/} is a triple $\widetilde{C}=(C,E,\{S_k\})$, where:
\begin{itemize}
\item
$C$ is a Riemann surface with cylindrical ends, such that each end is
designated ``positive'' or ``negative''; each positive end is
identified with $[0,\infty)\times S^1$; each negative end is
identified with $(-\infty,0]\times S^1$; the positive ends are labeled
$1,\ldots,N_+$; and the negative ends are labeled $-1,\ldots,-N_-$.
On each end, denote the $[0,\infty)$ or $(-\infty,0]$ coordinate by
$s$ and the $S^1$ coordinate by $t$.
\item
$E$ is a rank $n$ Hermitian vector bundle over $C$, with a fixed
trivialization on each end.
\item
Associated to the $k^{th}$ end of $C$ is a smooth family $S_k(t)$ of
symmetric $2n\times 2n$ matrices parametrized by $t\in S^1$, such that
zero is not an eigenvalue of the operator $i\partial_t + S_k$.
\end{itemize}
\end{definition}

\begin{definition}
For $\widetilde{C}$ as above, define $\mc{D}(\widetilde{C})$ to be the
set of differential operators $D:C^\infty(E)\to
C^\infty(T^{0,1}C\tensor E)$ with the following properties:
\begin{itemize}
\item
There is a complex structure $j$ on $C$, agreeing with the standard
one on the ends, such that in local coordinates and trivializations,
$D$ equals
$\overline{\partial}$ plus a zeroth order term.
\item
On the $k^{th}$ end
of $C$, write
\[
D\psi = \frac{1}{2}(\partial_s + i\partial_t + M_k(s,t))\psi \tensor (ds -
idt),
\]
where $M_k(s,t)$ is a $2n\times 2n$ matrix.  Then $\lim_{|s|\to\infty}
M_k(s,\cdot) = S_k(\cdot)$ in the sense of \cite[\S2]{fh}.
\end{itemize}
\end{definition}

It is a standard fact that any such $D$ extends to a Fredholm operator
$L^2_1(E)\to L^2(T^{0,1}C\tensor E)$.  Moreover, a homotopy of such
differential operators defines a continuous path of Fredholm
operators \cite[Prop. 7]{fh}.  Finally, the space $\mc{D}$ is
contractible, so for any two operators $D,D'\in\mc{D}(\widetilde{C})$,
there is a canonical bijection $\mc{O}(D)=\mc{O}(D')$.  We denote this
set of orientations by $\mc{O}(\widetilde{C})$.

We now consider gluing two orientation triples
$\widetilde{C}_-=(C_-,E_-,\{S_k^-\})$ and $\widetilde{C}_+=
(C_+,E_+,\{S_k^+\})$.  Assume that the first $l$ positive ends of
$\widetilde{C}_-$ agree with the first $l$ negative ends of
$\widetilde{C}_+$, in the sense that $S_k^-=S_{-k}^+$ for
$k=1,\ldots,l$.  Fix a large $R>0$.  Define a new surface $C$ by
identifying, for each $k=1,\ldots,l$, the $s=2R$ circle in the
$k^{th}$ end of $C_-$ with the $s=-2R$ circle in the $-k^{th}$ end of
$C_+$.  For each $k=1,\ldots,l$, the part of $C$ coming from the
$k^{th}$ end of $C_-$ and the $-k^{th}$ end of $C_+$ is a cylinder
$Z_k$.  We identify $Z_k\simeq [-2R,2R]\times S^1$, so that
translation of $s$ by $\pm 2R$ identifies $Z_k$ with the $0\le s\le
4R$ portion of the $k^{th}$ end of $C_-$, or with the $-4R\le s\le 0$
of the $-k^{th}$ end of $C_+$.  Use the fixed trivializations of $E_-$
and $E_+$ over the ends to glue them to a bundle $E_C$ over $C$.
Denote the glued orientation triple by
\[
\widetilde{C}_-\#_{l}\widetilde{C}_+ \eqdef
\left(C,E_C,\{S_k^-\}_{k\notin\{1,\ldots,l\}}
\cup\{S_k^+\}_{k\notin\{-1,\ldots,-l\}}\right). 
\]
Order the positive ends of $C$ so that the positive ends of $C_+$ come
first (in their given order), followed by the unglued positive ends of
$C_-$ (in their given order).  Likewise, order the negative ends of
$C$ so that the negative ends of $C_-$ come first (in the order
$-1,-2,\ldots$), followed by the unglued negative ends of $C_+$.

We will use ``linear gluing'' to define a canonical isomorphism
\[
\mc{O}(\widetilde{C}_-) \tensor \mc{O}(\widetilde{C}_+)
= \mc{O}(\widetilde{C}_- \#_l \widetilde{C}_+).
\]
Choose operators $D_-\in\mc{D}(\widetilde{C}_-)$ and
$D_+\in\mc{D}(\widetilde{C}_+)$.  Let
$D_C\in\mc{D}(\widetilde{C_-}\#_l\widetilde{C}_+)$ be an operator that
agrees with $D_-$ on $C_-$ off of the $s\ge R$ part of the first $l$
positive ends, and that agrees with $D_+$ on $C_+$ off of the $s\le
-R$ part of the first $l$ negative ends.  Note that $D_C-D_-$ and
$D_C-D_+$ are zeroth order operators on the cylinders $Z_k$.  Choose a
finite dimensional subspace $V_\pm$ of $L^2(T^{0,1}C_\pm\tensor E_\pm)$
such that if $W_\pm$ denotes the orthogonal complement of $V_\pm$, and
if $\Pi_{W_\pm}:L^2(T^{0,1}C_\pm\tensor E_\pm)\to W_\pm$ denotes the
orthogonal projection, then $\Pi_{W_\pm}\circ D_\pm$ is surjective.
Here is the version of linear gluing that we will need.

\begin{proposition}
\label{prop:LGES}
If $R$ is sufficiently large, and if $|D_C-D_-|$ and $|D_C-D_+|$ are
sufficiently small on the cylinders $Z_k$ for $k=1,\ldots,l$, then
there is an exact sequence
\begin{equation}
\label{eqn:LGES}
0 \to \Ker(D_C) \stackrel{f}{\to} \Ker(\Pi_{W_-}D_-) \oplus
\Ker(\Pi_{W_+}D_+) \stackrel{g}{\to} V_-\oplus V_+ 
\stackrel{h}{\to} \Coker(D_C) \to 0
\end{equation}
\end{proposition}

\begin{proof}
The construction has three steps.

{\em Step 1.\/} We first introduce some notation.  Fix a smooth
function $\beta:\R\to[0,1]$ such that $\beta(s)=1$ for $s\le 0$ and
$\beta(s)=0$ for $s\ge 1$.  (This notation differs from that of
\S\ref{sec:pregluing}.)  Define functions
$\beta_-,\beta_+:C\to[0,1]$ as follows.  On the cylinder $Z_k$,
\[
\beta_-(s,t) \eqdef \beta(s/R), \quad\quad \beta_+(s,t) \eqdef \beta(-s/R).
\]
On $C_-$ off of the first $l$ positive ends, define $\beta_-\eqdef 1$ and
$\beta_+\eqdef 0$.  On $C_+$ off of the first $l$ negative ends, define
$\beta_+\eqdef 1$ and $\beta_-\eqdef 0$.

Now consider $\psi_-\in L^2_1(E_-)$ and $\psi_+\in L^2_1(E_+)$, and
define
\begin{equation}
\label{eqn:definepsi}
\psi \eqdef \beta_-\psi_- + \beta_+\psi_+ \in L^2_1(E_C).
\end{equation}
We can then express
\begin{equation}
\label{eqn:dcp}
D_C\psi = \beta_-\Theta_-(\psi_-,\psi_+) + \beta_+\Theta_+(\psi_-,\psi_+),
\end{equation}
where
\begin{equation}
\label{eqn:definethetas}
\begin{split}
\Theta_-(\psi_-,\psi_+)
&\eqdef
D_-\psi_- + (D_C-D_-)\psi_- + (\dbar\beta_+)\psi_+ \in
L^2(T^{0,1}C_-\tensor E_-),\\
\Theta_+(\psi_-,\psi_+) &\eqdef D_+\psi_+ + (D_C-D_+)\psi_+ +
(\dbar\beta_-)\psi_- \in L^2(T^{0,1}C_+\tensor E_+).
\end{split}
\end{equation}
Here we interpret $D_C-D_-\eqdef 0$ off of the support of $\beta_-$,
and $D_C-D_+\eqdef 0$ off of the support of $\beta_+$.  Note that
\eqref{eqn:dcp} follows from \eqref{eqn:definethetas} because
$\beta_-=1$ on the support of $d\beta_+$ and $\beta_+=1$ on the
support of $d\beta_-$.

{\em Step 2.\/}
We now prove a key lemma:

\begin{lemma}
\label{lem:uniquepsi}
Suppose $\psi\in L^2_1(E_C)$ satisfies
$D_C\psi=\beta_-\theta_-+\beta_+\theta_+$ where $\theta_\pm\in
L^2(T^{0,1}C_\pm\tensor E_\pm)$.  Then there exist unique $\psi_\pm\in
L^2_1(E_\pm)$ such that $\psi=\beta_-\psi_-+\beta_+\psi_+$ and
\begin{equation}
\label{eqn:uniquepsi}
\Theta_-(\psi_-,\psi_+) = \theta_-, \quad \quad
\Theta_+(\psi_-,\psi_+) = \theta_+.
\end{equation}
\end{lemma}

\begin{proof}
We first introduce a linear version of Lemma~\ref{lem:IFT2}.  Fix
$k\in\{1,\ldots,l\}$, let $S_k(t)\eqdef S_k^-(t)=S_{-k}^+(t)$, and
consider the asymptotic operator
\[
L_k\eqdef i\partial_t+S_k(t):C^\infty(S^1;\R^{2n})\to C^\infty(S^1;\R^{2n}).
\]
Let $\Pi_{k+}$ and $\Pi_{k-}$ denote the projections from $L^2(S^1;\R^{2n})$ to
the sums of the positive and negative eigenspaces of $L_k$
respectively.  Consider the following operator (where all functions in the
various function spaces take values in $\R^{2n}$):
\begin{gather*}
\mc{F}_k:L^2_1([-R,\infty)\times S^1) \oplus L^2_1((-\infty,R]\times
S^1) \longrightarrow
\quad\quad\quad\quad\quad\quad\quad\quad\quad\quad\quad\quad
\\
\quad\quad\; \Pi_{k+}L^2_{1/2}(S^1) \oplus \Pi_{k-}L^2_{1/2}(S^1)
\oplus L^2([-R,\infty)\times S^1) \oplus L^2((-\infty,R]\times
S^1),\\
(\psi_-,\psi_+) \longmapsto (\Pi_{k+}\psi_-(-R,\cdot),
\Pi_{k-}\psi_+(R,\cdot), \Theta_-(\psi_-,\psi_+), \Theta_+(\psi_-,\psi_+)).
\end{gather*}
Here $\Theta_-$ and $\Theta_+$ are defined by identifying
$[-R,R]\times S^1$ with the $R\le s \le 3R$ portion of the $k^{th}$
positive end of $C_-$ and with the $-3R\le s \le -R$ portion of the $k^{th}$
negative end of $C_+$.  A linear version of Lemma~\ref{lem:IFT2} shows that
under the hypotheses of Proposition~\ref{prop:LGES}, the map $\mc{F}_k$
is an isomorphism.

Proceeding with the proof of Lemma~\ref{lem:uniquepsi}, choose an
arbitrary decomposition $\psi=\beta_-\psi_-'+\beta_+\psi_+'$ with
$\psi_\pm'\in L^2_1(E_\pm)$.  Then by \eqref{eqn:dcp} we have
\begin{equation}
\label{eqn:uh}
\beta_-(\theta_--\Theta_-(\psi_-',\psi_+')) + \beta_+(\theta_+ -
\Theta_+(\psi_-',\psi_+')) = 0.
\end{equation}
Now write $\psi_-=\psi_-'+\psi_-''$ and $\psi_+ = \psi_+'+\psi_+''$.
Since $\Theta_\pm$ is a linear function on $L^2_1(E_-)\oplus
L^2_1(E_+)$, the desired equations \eqref{eqn:uniquepsi} are
equivalent to the equations
\begin{align}
\label{eqn:Th-}
\Theta_-(\psi_-'',\psi_+'') &= \theta_- - \Theta_-(\psi_-',\psi_+'),\\
\label{eqn:Th+}
\Theta_+(\psi_-'',\psi_+'') &= \theta_+ - \Theta_+(\psi_-',\psi_+').
\end{align}
By \eqref{eqn:uh}, the right hand side of \eqref{eqn:Th-} is supported
on the $s\ge R$ portion of the first $l$ positive ends of $C_-$, while
the right hand side of \eqref{eqn:Th+} is supported on the $s\le -R$
portion of the first $l$ negative ends of $C_+$.  The required
$\psi_-''$ and $\psi_+''$ are now given on the $s \ge R$ part of the
$k^{th}$ positive end of $C_-$ and on the $s\le -\R$ part of the
$k^{th}$ negative end of $C_+$ for $k=1,\ldots,l$ by
\[
(\psi_-'',\psi_+'') = \mc{F}_k^{-1}(0,0,\theta_- -
\Theta_-(\psi_-',\psi_+'), \theta_+ - \Theta_+(\psi_-',\psi_+')).
\]
We can, and must, take $\psi_\pm''=0$ on the rest of $C_\pm$.
\end{proof}

{\em Step 3.\/} We now define the maps in the sequence
\eqref{eqn:LGES} and prove exactness.

{\em Definition of $f$.\/}
Let $\psi\in\Ker(D_C)$ be given.  By Lemma~\ref{lem:uniquepsi},
there exist unique $\psi_-\in L^2_1(E_-)$ and $\psi_+\in L^2_1(E_+)$
such that $\psi=\beta_-\psi_-+\beta_+\psi_+$ and
\begin{equation}
\label{eqn:awte}
\Theta_-(\psi_-,\psi_+)=0,\quad\quad\Theta_+(\psi_-,\psi_+)=0.
\end{equation}
Let $\phi_\pm$ denote the $L^2$ orthogonal projection of $\psi_\pm$
onto $\Ker(\Pi_{W_\pm}D_\pm)$.  Define
\[
f(\psi)\eqdef (\phi_-,\phi_+).
\]

{\em Definition of $g$.\/}
Let $\phi_-\in\Ker(\Pi_{W_-}D_-)$ and
$\phi_+\in\Ker(\Pi_{W_+}D_+)$ be given.
We claim that there are unique $\zeta_-\in L^2_1(E_-)$ and $\zeta_+\in
L^2_1(E_+)$ such that $\zeta_-$ is $L^2$ orthogonal to
$\Ker(\Pi_{W_-}D_-)$, $\zeta_+$ is $L^2$ orthogonal to $\Ker(\Pi_{W_+}D_+)$,
and the pair
\begin{equation}
\label{eqn:defpair}
(\psi_-,\psi_+)\eqdef(\phi_-+\zeta_-,\phi_++\zeta_+)
\end{equation}
solves the equations
\begin{equation}
\label{eqn:tts}
\Pi_{W_-}\Theta_-(\psi_-,\psi_+)=0, \quad\quad
\Pi_{W_+}\Theta_+(\psi_-,\psi_+)=0.
\end{equation}
To see this, let $F_-:W_-\to \Ker(\Pi_{W_-}D_-)^\perp$ denote the
inverse of $\Pi_{W_-}D_-$, and let
$F_+:W_+\to\Ker(\Pi_{W_+}D_+)^\perp$ denote the inverse of
$\Pi_{W_+}D_+$.  Then applying $F_-$ and $F_+$ to the first and second
equations in \eqref{eqn:tts} respectively, we obtain a pair of
equations which can be written as
\begin{gather*}
\begin{pmatrix}
1 + F_-\Pi_{W_-}(D_C-D_-) & F_-\Pi_{W_-}(\dbar\beta_+) \\
F_+\Pi_{W_+}(\dbar\beta_-) & 1+F_+\Pi_{W_+}(D_C-D_+)
\end{pmatrix}
\begin{pmatrix}
\zeta_- \\ \zeta_+ \end{pmatrix}\quad\quad\quad\quad\quad\quad\quad\\
\quad\quad\quad\quad\quad\quad\quad\quad\quad\quad\quad\quad =
\begin{pmatrix}
-F_-\Pi_{W_-}((D_C-D_-)\phi_- + (\dbar\beta_+)\phi_+) \\
-F_+\Pi_{W_+}((D_C-D_+)\phi_+ + (\dbar\beta_-)\phi_-)
\end{pmatrix}.
\end{gather*}
If $|D_C-D_-|$ and $|D_C-D_+|$ are sufficiently small, and $R$ is
sufficiently large, with respect to the operator norms of $F_+$ and
$F_-$, then these equations have a unique solution.  In terms of this
unique solution, define
\[
g(\phi_-,\phi_+) \eqdef
(\Theta_-(\psi_-,\psi_+),\Theta_+(\psi_-,\psi_+)).
\]

{\em Definition of $h$.\/} Given $(\theta_-,\theta_+)\in V_-\oplus
V_+$, define $h(\theta_-,\theta_+)$ to be the equivalence class of
$\beta_-\theta_-+\beta_+\theta_+$ in $\Coker(D_C)$.

{\em $f$ is injective:\/} Let $\psi\in\Ker(D_C)$ and suppose that
$f(\psi)=0$.  This means that $\psi=\beta_-\zeta_-+\beta_+\zeta_-$,
where $\zeta_\pm$ is $L^2$-orthogonal to $\Ker(\Pi_{W_\pm} D_\pm)$ and
$\Theta_\pm(\zeta_-,\zeta_+)=0$.  By uniqueness of the solution to the
equations \eqref{eqn:tts} for $\phi_-=\phi_+=0$, it follows that
$\zeta_-=\zeta_+=0$, and so $\psi=0$.

{\em $\op{Im}(f)\subset\Ker(g)$:\/}  Immediate from the definitions.

{\em $\Ker(g)\subset \op{Im}(f)$:\/} Suppose $g(\phi_-,\phi_+)=0$.
This means that there exist $\zeta_\pm$ orthogonal to $\Ker(\Pi_{W_\pm}
D_\pm)$ such that the pair $(\psi_-,\psi_+)$ defined in
\eqref{eqn:defpair} satisfies the equations
\eqref{eqn:awte}.  Then $\psi\eqdef\beta_-\psi_-+\beta_+\psi_+$
 is in $\Ker(D_C)$ by equation \eqref{eqn:dcp}, and by definition
$f(\psi)=(\phi_-,\phi_+)$.

{\em $\op{Im}(g)\subset\Ker(h)$:\/}  Immediate from \eqref{eqn:dcp}.

{\em $\Ker(h) \subset \op{Im}(g)$:\/} Let $(\theta_-,\theta_+)\in
V_-\oplus V_+$ be given, and suppose that $h(\theta_-,\theta_+)=0$.
This means that there exists $\psi\in L^2_1(E_C)$ with
\[
D_C\psi = \beta_-\theta_-+\beta_+\theta_+.
\]
By Lemma~\ref{lem:uniquepsi}, there exist $\psi_\pm\in
 L^2_1(E_\pm)$ such that
\[
(\Theta_-(\psi_-,\psi_+),\Theta_+(\psi_-,\psi_+)) = (\theta_-,\theta_+).
\]
Let $\phi_\pm$ denote the $L^2$ orthogonal projection of $\psi_\pm$
onto $\Ker(\Pi_{W_\pm}D_\pm)$.  Then by definition,
$g(\phi_-,\phi_+) = (\theta_-,\theta_+)$.

{\em $h$ is surjective:\/} Given $\xi\in L^2(E_C)$, we need to find
$(\theta_-,\theta_+)\in V_-\oplus V_+$ and $\psi\in L^2_1(E)$ such
that
\begin{equation}
\label{eqn:ldi}
D_C\psi + \beta_-\theta_- + \beta_+\theta_+ = \xi.
\end{equation}
Choose any decomposition $\xi=\beta_-\xi_-+\beta_+\xi_+$ with
$\xi_\pm\in L^2(T^{0,1}C_\pm\tensor E_\pm)$.  Since $\Pi_{W_\pm} D_\pm$ is
surjective, there exist $\psi_\pm\in L^2_1(E_\pm)$ and $\theta_\pm\in
V_\pm$ such that
\begin{equation}
\label{eqn:upds}
D_-\psi_-+\theta_-=\xi_-,\quad\quad
D_+\psi_++\theta_+=\xi_+.
\end{equation}
Now write $\psi=\beta_-(\psi_-+\psi_-')+\beta_+(\psi_++\psi_+')$.
Then to solve the desired equation \eqref{eqn:ldi}, it is enough to
find $\psi_\pm'$ such that
\begin{align}
\label{eqn:foo1}
\Theta_-(\psi_-',\psi_+') & \eqdef
-\Theta_-(\psi_-,\psi_+)-\theta_-+\xi_-,\\
\label{eqn:foo2}
\Theta_+(\psi_-',\psi_+') & \eqdef -\Theta_+(\psi_-,\psi_+) - \theta_+
+ \xi_+.
\end{align}
By \eqref{eqn:definethetas} and \eqref{eqn:upds}, the right hand side
of \eqref{eqn:foo1} is supported in the $s\ge R$ portion of the first
$l$ positive ends of $C_-$, while the right hand side of
\eqref{eqn:foo2} is supported in the $s\le -R$ portion of the first
$l$ negative ends of $C_+$.  It follows that as in the proof of
Lemma~\ref{lem:uniquepsi}, we can use the maps $\mc{F}_k^{-1}$ to find the
required $\psi_-'$ and $\psi_+'$.
\end{proof}

\begin{remark}
Counting dimensions in the exact sequence \eqref{eqn:LGES} recovers
the standard fact that $\op{ind}(D_C)=\op{ind}(D_-)+\op{ind}(D_+)$.
\end{remark}

\subsection{Gluing orientations}
\label{sec:GO}

Proposition~\ref{prop:LGES} allows us to glue orientations as follows.
The exact sequence \eqref{eqn:LGES} induces an isomorphism
\begin{equation}
\label{eqn:or0}
\mc{O}(D_C)\simeq \mc{O}(\Pi_{W_-}D_-)\tensor
\mc{O}(\Pi_{W_+}D_+)\tensor \mc{O}(V_-)\tensor \mc{O}(V_+).
\end{equation}
Combining this with the $D_-$ and $D_+$ versions of \eqref{eqn:sesii},
we obtain an isomorphism
\begin{equation}
\label{eqn:LGI}
\mc{O}(D_-)\tensor\mc{O}(D_+)
\simeq
\mc{O}(D_C).
\end{equation}

\begin{lemma}
\label{lem:HI}
The isomorphism \eqref{eqn:LGI} does not depend on $V_-$ and $V_+$,
and is invariant under homotopy of the data $(D_-,D_+;R,D_C)$, so that
it induces a well-defined isomorphism
\[
\mc{O}(\widetilde{C}_-) \tensor \mc{O}(\widetilde{C}_+)
\simeq \mc{O}(\widetilde{C}_- \#_l \widetilde{C}_+).
\]
\end{lemma}

\begin{proof}
The proof has three steps.

{\em Step 1.\/} To prove that \eqref{eqn:LGI} is homotopy invariant,
the main difficulty is that the dimensions of $\Ker(D_C)$ and
$\Coker(D_C)$ may jump during a homotopy.  To deal with this issue, we
first give an alternate description of the isomorphism \eqref{eqn:LGI}
which does not directly refer to the kernel or cokernel of $D_C$.

Continuing with the notation from the proof of
Proposition~\ref{prop:LGES}, define $V_C\subset L^2(T^{0,1}C\tensor
E_C)$ to be the subspace consisting of sections
$\beta_-\theta_-+\beta_+\theta_+$ where $\theta_\pm\in V_\pm$.  Assume
that no nonzero element of $V_-$ or $V_+$ is supported entirely on the
$|s|\ge R$ part of the first $l$ positive or negative ends, so that
the map $V_-\oplus V_+\to V_C$ sending
$(\theta_-,\theta_+)\mapsto\beta_-\theta_-+\beta_+\theta_+$ is an
isomorphism.  (In general one can arrange this by a slight
perturbation of $V_\pm$.)  Let $W_C$ denote the orthogonal complement
of $V_C,$ and let $\Pi_{W_C}:L^2(T^{0,1}C\tensor E_C)\to W_C$ denote
the orthogonal projection.  Note that $\Pi_{W_C}D_C$ is surjective,
because the map $h$ in \eqref{eqn:LGES} is surjective.  So as in
\eqref{eqn:OES}, there is an exact sequence
\begin{equation}
\label{eqn:OES2}
0 \longrightarrow \Ker(D_C) \longrightarrow \Ker(\Pi_{W_C}D_C)
\stackrel{D_C}{\longrightarrow} V_C \longrightarrow
\Coker(D_C) \longrightarrow 0.
\end{equation}

We now define a map
\begin{equation}
\label{eqn:altphi}
\Phi: \Ker(\Pi_{W_C}D_C) \longrightarrow \Ker(\Pi_{W_-}D_-)\oplus
\Ker(\Pi_{W_+}D_+)
\end{equation}
as follows.  Suppose $\psi\in\Ker(\Pi_{W_C}D_C)$.  This means that
there exist unique $\theta_\pm\in V_\pm$ such that
\[
D_C\psi = \beta_-\theta_- + \beta_+\theta_+.
\]
By Lemma~\ref{lem:uniquepsi}, there are unique $\psi_\pm\in
L^2_1(E_\pm)$ satisfying equations \eqref{eqn:definepsi} and
\eqref{eqn:uniquepsi}.  Now let $\phi_\pm$ denote the $L^2$ orthogonal
projection of $\psi_\pm$ onto $\Ker(\Pi_{W_\pm}D_\pm)$, and define
\[
\Phi(\psi) \eqdef (\phi_-,\phi_+).
\]

It follows directly from the definitions that the exact sequences
\eqref{eqn:LGES} and \eqref{eqn:OES2} fit into a commutative diagram

\begin{footnotesize}
\[
\begin{CD}
\Ker(D_C) @>>> \Ker(\Pi_{W_C}D_C) @>{D_C}>> V_C @>>> \Coker(D_C)
\\
@| @VV{\Phi}V @AA{\simeq}A @| \\
\Ker(D_C) @>f>> \Ker(\Pi_{W_-}D_-)\oplus\Ker(\Pi_{W_+}D_+) @>g>>
V_-\oplus V_+ @>h>> \Coker(D_C).
\end{CD}
\]
\end{footnotesize}

\noindent
By the five lemma, $\Phi$ is an isomorphism, and hence induces an
isomorphism
\begin{equation}
\label{eqn:or4}
\mc{O}(\Pi_{W_C}D_C) \simeq \mc{O}(\Pi_{W_-}D_-)\tensor
\mc{O}(\Pi_{W_+}D_+).
\end{equation}
Moreover, it follows from the above commutative diagram that under the
canonical isomorphisms \eqref{eqn:sesii} for $D_\pm$ and $D_C$, the
isomorphism \eqref{eqn:or4} agrees with \eqref{eqn:LGI}.

{\em Step 2.\/} We now show that the isomorphism \eqref{eqn:LGI} does
not depend on $V_-$ and $V_+$.  For this purpose it is enough to show
that the isomorphism \eqref{eqn:LGI} is unchanged under replacing
$V_-$ and $V_+$ by larger subspaces $V_-'\supset V_-$ and $V_+'\supset
V_+$.  It follows directly from the definitions that the primed and
unprimed versions of $\Phi$ fit into a commutative diagram
\[
\begin{CD}
\Ker(\Pi_{W_C}D_C) @>{\Phi}>> \Ker(\Pi_{W_-}D_-)\oplus
\Ker(\Pi_{W_+}D_+) \\
@VVV @VVV \\
\Ker(\Pi_{W_C'}D_C) @>{\Phi'}>> \Ker(\Pi_{W_-'}D_-) \oplus
\Ker(\Pi_{W_+'}D_+)
\end{CD}
\]
where the vertical maps are the inclusions.  It follows readily that
the primed and unprimed versions of the isomorphism
\eqref{eqn:LGI} agree.

{\em Step 3.\/} We now prove homotopy invariance.  Given a homotopy
\[
\left\{(D_-(t),D_+(t),D_C(t))\right\}_{t\in[0,1]}
\]
of triples of operators, we need to prove that the diagram
\begin{equation}
\label{eqn:HCD}
\begin{CD}
\mc{O}(D_-(0))\tensor\mc{O}(D_+(0)) @>{\simeq}>> \mc{O}(D_C(0)) \\
@VV{\simeq}V @VV{\simeq}V \\
\mc{O}(D_-(1)) \tensor \mc{O}(D_+(1)) @>{\simeq}>> \mc{O}(D_C(1))
\end{CD}
\end{equation}
commutes, where the horizontal isomorphisms are the $t=0$ and $t=1$
versions of \eqref{eqn:LGI}, while the vertical isomorphisms are
defined in \eqref{eqn:continuation}.  To do so, choose subspaces $V_\pm(t)$
with the required properties that depend continuously on $t$.  Then the
family of maps
\[
\left\{\Phi_t:\Ker(\Pi_{W_C(t)}D_C(t)) \longrightarrow
\Ker(\Pi_{W_-(t)}D_-(t)) \oplus \Ker(\Pi_{W_+(t)}D_+(t))\right\}_{t\in[0,1]}
\]
defines an isomorphism of vector bundles over $[0,1]$.  It now follows
after unraveling the definitions that the diagram
\eqref{eqn:HCD} commutes.
\end{proof}

If $\frak{o}_-\in\mc{O}(\widetilde{C}_-)$ and
$\frak{o}_+\in\mc{O}(\widetilde{C}_+)$, denote the corresponding glued
orientation by
\[
\frak{o}_-\#_l\frak{o}_+ \in
\mc{O}(\widetilde{C}_-\#_l\widetilde{C}_+).
\]

\subsection{Associativity of linear gluing}

We now show that the operation of gluing orientations is associative.
More precisely, consider three orientation triples $\widetilde{C}_-,
\widetilde{C}_0, \widetilde{C_+}$, such that the first $l_-$ negative ends of
$\widetilde{C}_0$ (numbered $-1,\ldots,-l_-$) agree with the first
$l_-$ positive ends of $\widetilde{C}_-$ in that order, while the
first $l_+$ positive ends of $\widetilde{C}_0$ likewise agree with the
first $l_+$ negative ends of $\widetilde{C}_+$.

\begin{lemma}
\label{lem:associative}
If $\frak{o}_\pm\in\mc{O}(\widetilde{C}_\pm)$ and
$\frak{o}_0\in\mc{O}(\widetilde{C}_0)$, then
\begin{equation}
\label{eqn:associative}
(\frak{o}_-\#_{l_-}\frak{o}_0)\#_{l_+}\frak{o}_+ =
\frak{o}_-\#_{l_-}(\frak{o}_0\#_{l_+}\frak{o}_+)
\end{equation}
in $\mc{O}(\widetilde{C}_-\#_{l_-}\widetilde{C}_0\#_{l_+}\widetilde{C}_+)$.
\end{lemma}

\begin{proof}
The plan is to relate both sides of \eqref{eqn:associative} to an
analogue of linear gluing which glues the three orientation triples
$\widetilde{C}_-,\widetilde{C}_0,\widetilde{C}_+$ together
simultaneously.

{\em Step 1.\/} We begin by explaining the triple linear gluing
operation, in somewhat more detail than is necessary here, because it
will play an important role in the proof of Theorem~\ref{thm:count} in
\S\ref{sec:HES}.

Let $C\eqdef C_-\#_{l_-}C_0\#_{l_+}C_+$ denote the glued curve, and
define functions $\beta_-,\beta_0,\beta_+:C\to[0,1]$ as follows.  For
$k\in\{-1,\ldots,-l_-\}$, let $Z_k$ denote the cylinder where the
$-k^{th}$ end of $C_-$ is glued to the $k^{th}$ end of $C_0$.  For
$k\in\{1,\ldots,l_+\}$, let $Z_k$ denote the cylinder where the
$k^{th}$ end of $C_0$ is glued to the $-k^{th}$ end of $C_+$.  For
each $k$ identify
\[
Z_k \simeq [-2R, 2R] \times S^1.
\]
On the cylinder $Z_k$ for
$k\in\{-1,\ldots,-l_-\}$, define
\[
\beta_-(s,t)\eqdef\beta(s/R), \quad\quad \beta_0(s,t)\eqdef
\beta(-s/R), \quad\quad \beta_+(s,t)\eqdef 0,
\]
where $\beta$ is the function defined in \S\ref{sec:LGES}.  On the
cylinder $Z_k$ for $k\in\{1,\ldots,l_+\}$, define
\[
\beta_-(s,t)\eqdef 0, \quad\quad \beta_0(s,t) \eqdef \beta(s/R),
\quad\quad \beta_+(s,t)\eqdef \beta(-s/R).
\]
On $C_-$, off of the first $l_-$ positive ends, define $\beta_-\eqdef
1$ and $\beta_0,\beta_+\eqdef 0$.  On $C_0$, off the first $l_-$
negative ends and the first $l_+$ positive ends, define $\beta_0\eqdef
1$ and $\beta_-,\beta_+\eqdef 0$.  On $C_+$, off of the first $l_+$
negative ends, define $\beta_+\eqdef 1$ and $\beta_-,\beta_0\eqdef 0$.

Choose operators $D_\pm\in\mc{D}(\widetilde{C}_\pm)$ and
$D_0\in\mc{D}(\widetilde{C}_0)$.  Let
\[
D_C\in\mc{D}(\widetilde{C}_-\#_{l_-}\widetilde{C}_0\#_{l_+}\widetilde{C}_+)
\]
be an operator that agrees with $D_-$ on $C_-$ off of the first $l_-$
positive ends, agrees with $D_0$ on $C_0$ off of the first $l_-$
negative ends and the first $l_+$ positive ends, and agrees with $D_+$
on $C_+$ off of the first $l_+$ negative ends.  Assume that $R$ is
large and that $|D_C-D_\pm|$ and $|D_C-D_0|$ are small on the cylinders $Z_k$
for $k\in\{\pm 1,\ldots,\pm l_\pm\}$.

Now consider $\psi_\pm\in L^2_1(E_\pm)$ and $\psi_0\in L^2_1(E_0)$,
and define
\begin{equation}
\label{eqn:psi3}
\psi\eqdef \beta_-\psi_- + \beta_0\psi_0 + \beta_+\psi_+ \in L^2_1(E_C).
\end{equation}
We can then express
\[
D_C\psi = \beta_-\Theta_-(\psi_-,\psi_0) +
\beta_0\Theta_0(\psi_-,\psi_0,\psi_+) + \beta_+\Theta_+(\psi_0,\psi_+),
\]
where the $\Theta$'s are defined by the obvious analogue of
\eqref{eqn:definethetas}.

Now choose finite dimensional subspaces $V_\pm\subset
L^2(T^{0,1}C_\pm\tensor E_\pm)$ and $V_0\subset L^2(T^{0,1}C_0\tensor
E_0)$ such that if $W_\pm$ and $W_0$ denote their orthogonal
complements, then $\Pi_{W_{\pm}}D_\pm$ and $\Pi_{W_0}D_0$ are
surjective.  Then the exact sequence \eqref{eqn:LGES} has a
straightforward generalization to an exact sequence
\begin{gather}
\nonumber
0 \longrightarrow \Ker(D_C) \longrightarrow \Ker(\Pi_{W_-}D_-) \oplus
\Ker(\Pi_{W_0}D_0) \oplus \Ker(\Pi_{W_+}D_+) \longrightarrow
\quad\quad\;\\
\label{eqn:TLG}
\quad\quad\quad\quad\quad\longrightarrow V_-\oplus V_0\oplus
V_+ \longrightarrow \Coker(D_C) \longrightarrow 0.
\end{gather}
As before, this exact sequence induces a homotopy invariant isomorphism
\begin{equation}
\label{eqn:hii}
\mc{O}(D_-)\tensor\mc{O}(D_0)\tensor\mc{O}(D_+) \simeq \mc{O}(D_C).
\end{equation}

As in \S\ref{sec:GO}, one can give an alternate description of the
isomorphism \eqref{eqn:hii} as follows.  Let
\[
V_C \eqdef \left\{\beta_-\theta_-+\beta_0\theta_0+\beta_+\theta_+\mid
\theta_-\in V_-,\theta_0\in V_0,\theta_+\in V_+\right\}.
\]
Choose $V_\pm,V_0$ such that the map $V_-\oplus V_0\oplus V_+\to V_C$
is an isomorphism, and let $W_C$ denote the orthogonal complement of
$V_C$.  The map \eqref{eqn:altphi} then has an obvious analogue
\[
\widetilde{\Phi}: \Ker(\Pi_{W_C}D_C) \longrightarrow
\Ker(\Pi_{W_-}D_-)\oplus \Ker(\Pi_{W_0}D_0) \oplus
\Ker(\Pi_{W_+}D_+),
\]
which is an isomorphism and induces the map \eqref{eqn:hii} on
orientations.

{\em Step 2.\/} We now relate triple linear gluing to the composition
of two ordinary linear gluings.  Let $D_{0-}$ denote the operator on
$C_-\#_{l_-} C_0$ that agrees with $D_0$ on the ends $1,\ldots,l$ of
$C_0$, and that agrees with $D_C$ on the rest of $C_-\#_{l_-} C_0$.
Let $\beta_-,\beta_{0-}:C_-\#_{l_-} C_0\to[0,1]$ denote the cutoff
functions for this gluing, and let
\[
V_{0-} \eqdef \left\{\beta_-\theta_- + \beta_{0-}\theta_0 \mid
\theta_-\in V_-,\theta_0\in V_0\right\}.
\]
Define $D_{0+}, \beta_{0+}, V_{0+}$ likewise for $C_0\#_{l_+}C_+$.
Tracing through the definitions shows that the diagram

\begin{footnotesize}
\[
\begin{CD}
\Ker(\Pi_{W_C}D_C) @>{\Phi}>> \Ker(\Pi_{W_{0-}}D_{0-}) \oplus
\Ker(\Pi_{W_+}D_+) \\
@VV{\Phi}V @VV{\Phi\times\text{id}}V \\
\Ker(\Pi_{W_-}D_-)\oplus \Ker(\Pi_{W_{0+}}D_{0+}) @>{\text{id}\times\Phi}>> 
\Ker(\Pi_{W_-}D_-)\oplus \Ker(\Pi_{W_0}D_0) \oplus
\Ker(\Pi_{W_+}D_+)
\end{CD}
\]
\end{footnotesize}

\noindent
commutes, because both compositions are equal to the map
 $\widetilde{\Phi}$.  The lemma follows.
\end{proof}

\subsection{Orienting the moduli spaces}
\label{sec:orienting}

We now use the linear gluing operation to orient all moduli spaces
of immersed, unobstructed pseudoholomorphic curves in $\R\times Y$, so
that the orientations behave well under gluing.  We follow the
approach of \cite{bm}, but with some slightly different choices.

\paragraph{Orienting the Fredholm operators.}
For each orientation triple $\widetilde{C}=(C,E,\{S_k\})$, we will choose an
orientation $\frak{o}_{\widetilde{C}}\in\mc{O}(\widetilde{C})$.  We
want these choices to satisfy four axioms.  The first axiom concerns
the complex linear case:
\begin{description}
\item{(OR1)}
If $C$ has no ends, and if $D\in\mc{D}(\widetilde{C})$ is $\C$-linear,
then $\frak{o}_{\widetilde{C}}$ corresponds to the canonical
orientation of $\mc{O}(D)$ coming from the complex vector space
structure on $\Ker(D)$ and $\Coker(D)$.
\end{description}
The second axiom describes the behavior of the orientations under
``complete'' gluing, where we glue all the positive ends of one curve
to all the negative ends of another:
\begin{description}
\item{(OR2)}
If $\widetilde{C}_-$ has exactly $l$ positive ends, if
$\widetilde{C}_+$ has exactly $l$ negative ends, and if the $k^{th}$
positive end of $\widetilde{C}_-$ agrees with the $k^{th}$ negative end
of $\widetilde{C}_+$ for each $k=1,\ldots,l$, then
\[
\frak{o}_{\left(\widetilde{C}_-\#_l\widetilde{C}_+\right)} =
\frak{o}_{\widetilde{C}_-} \#_l
\frak{o}_{\widetilde{C}_+}.
\]
\end{description}

Before stating the third axiom, we need to introduce some mod 2
indices.  Consider a loop of symmetric $2n\times 2n$ matrices
$\{S(t)\}_{t\in S^1}$ such that zero is not an eigenvalue of the
operator $i\partial_t+S$.  Define
\[
\varepsilon(s) \eqdef n + \mu\left(\{\Psi(t)\}_{t\in [0,2\pi]}\right)
\mod 2,
\]
where $\Psi(t)$ is the path of symplectic matrices with $\Phi(0)=1$
generated by $S(t)$ as in equation (I.2.7), and $\mu$ denotes the
Maslov index.  Given an orientation triple
$\widetilde{C}=(C,E,\{S_k\})$, define
\[
\varepsilon_-(\widetilde{C}) \eqdef
\sum_{k=-1}^{-N_-}\varepsilon(S_k), \quad\quad
\varepsilon_+(\widetilde{C}) \eqdef \sum_{k=1}^{N_+}\varepsilon(S_k),
\quad\quad \varepsilon(\widetilde{C}) \eqdef
\varepsilon_+(\widetilde{C}) - \varepsilon_-(\widetilde{C}).
\]

The third axiom concerns the disjoint union
$\widetilde{C}_1\sqcup\widetilde{C}_2$ of two orientation triples
$\widetilde{C}_1$ and $\widetilde{C}_2$.  Here the positive or
negative ends of $\widetilde{C}_1\sqcup\widetilde{C}_2$ are ordered so
that the ends of $\widetilde{C}_1$ come first, followed by the ends of
$\widetilde{C}_2$, in their given order.  If
$D_1\in\mc{D}(\widetilde{C}_1)$ and $D_2\in\mc{D}(\widetilde{C}_2)$,
then \eqref{eqn:conbas} defines a canonical isomorphism
$\mc{O}(D_1\oplus D_2)=\mc{O}(D_1)\tensor\mc{O}(D_2)$, and hence
$\mc{O}(\widetilde{C}_1\sqcup\widetilde{C}_2)=\mc{O}(\widetilde{C}_1)
\tensor\mc{O}(\widetilde{C}_2)$. The axiom is now:

\begin{description}
\item{(OR3)}
\[
\frak{o}_{\widetilde{C}_1 \sqcup \widetilde{C}_2} =
(-1)^{\varepsilon_-(\widetilde{C}_1)\varepsilon(\widetilde{C}_2)}
\frak{o}_{\widetilde{C}_1}\tensor\frak{o}_{\widetilde{C}_2}.
\]
\end{description}

To find orientations satisfying the above three axioms, we first
choose arbitrary orientations for certain special orientation triples.
For each path of symmetric matrices $\{S(t)\}_{t\in S^1}$ as above,
consider the orientation triple
\[
\widetilde{C}_S\eqdef(\C,\C\times\C^n,S(\cdot)),
\]
where the end of $\C$ is identified with $[0,\infty)\times S^1$ via
the exponential function.

\begin{lemma}
Given orientation choices $\frak{o}_{\widetilde{C}_S}$ for each $S$ as
above, there is a unique way to extend these to choose orientations
$\frak{o}_{\widetilde{C}}$ for all orientation triples
$\widetilde{C}=(C,E,\{S_k\})$, such that axioms (OR1)--(OR3) hold.
\end{lemma}

\begin{proof}
This follows from the argument in \cite[\S3]{bm}, using the
associativity property \eqref{eqn:associative}.  Note that to
translate between the conventions in \cite{bm} and those here, one
needs to reverse the ordering of the negative ends of each $C$.
\end{proof}

\paragraph{Orienting the moduli spaces.}
We now explain how a system of orientations as above with $n=1$
orients the moduli spaces of immersed, unobstructed $J$-holomorphic
curves in $\R\times Y$.  Fix a parametrization of each Reeb orbit by a
map $\alpha:S^1\to\R\times Y$ such that $\partial_t\alpha$ is a
constant positive multiple of the Reeb vector field.  Also fix a
trivialization of the contact plane field $\xi$ over each Reeb orbit
in $Y$.  In this trivialization, the linearized Reeb flow on the
contact planes along $\alpha$ is given by $\partial_t - J_0
S_\alpha(t)$ where $S_\alpha(t)$ is a $2\times 2$ symmetric matrix and
$J_0\eqdef
\begin{pmatrix} 0 & -1 \\ 1 & 0 \end{pmatrix}$.


Now let $\alpha_+=(\alpha_1,\ldots,\alpha_{N_+})$ and
$\alpha_-=(\alpha_{-1},\ldots,\alpha_{-N_+})$ be ordered lists of Reeb
orbits, let $C\in\mc{M}^J(\alpha_+,\alpha_-)$ be immersed, and let
$N_C$ denote the normal bundle to $C$.  Over the $k^{th}$ end of $C$,
trivialize $N_C$ by using a choice of coordinates $(z,w)$ in a
neighborhood of the corresponding Reeb orbit provided by
Lemma~\ref{lem:EM}, and identifying $N_C$ with the tangent spaces to
the $z=\text{constant}$ disks.  Then the deformation operator $D_C$ is
an element of the space $\mc{D}(\widetilde{C})$ where
$\widetilde{C}=(C,N_C,\{S_{\alpha_k}\})$.  Thus the chosen orientation
$\frak{o}_{\widetilde{C}}$ determines an orientation in $\mc{O}(D_C)$.
If $C$ is unobstructed so that $\Coker(D_C)=\{0\}$, then this orients
$\Ker(D_C)=T_C\mc{M}^J(\alpha_+,\alpha_-)$.

\begin{definition}
A system of {\em coherent orientations\/} of the moduli spaces of
immersed unobstructed $J$-holomorphic curves in $\R\times Y$ is a
system of orientations determined as above from orientations
$\frak{o}_{\widetilde{C}}$ for all $\widetilde{C}$ with $n=1$
satisfying axioms (OR1)--(OR3) above, together with axiom (OR4) below.
\end{definition}

To state axiom (OR4), consider a family $\{S(t)\}_{t\in S^1}$ of
$2\times 2$ symmetric matrices such that the associated symplectic
matrix $\Phi(2\pi)$ is elliptic, i.e.\ has eigenvalues $e^{\pm2\pi
i\theta}$ for some $\theta\in\R\setminus\Z$.  We now describe a
canonical orientation in $\mc{O}(\widetilde{C}_S)$.

First consider the case where $S(t)=\theta$ for all $t$.  Then there
is a canonical orientation in $\mc{O}(\widetilde{C}_\theta)$, because
the operator
\[
D\eqdef\overline{\partial} + \theta d\overline{z} \in
\mc{D}(\widetilde{C}_\theta)
\]
is complex linear.  In general, if $\Phi(2\pi)$ is elliptic, then
there is some $\theta\in\R\setminus\Z$ such that
\begin{description}
\item{(*)} the path $\{\Phi(t)\}_{t\in[0,2\pi]}$ is homotopic rel
  endpoints to $\left\{
e^{2\pi i\theta t}
\right\}_{t\in[0,2\pi]}$.
\end{description}
Here $e^{2\pi i \theta t}$ is regarded as an element of $\op{Sp}(2)$
via the inclusion $\op{U}(1) = \op{SO}(2) \subset \op{Sp}(2)$.  A
homotopy as just described can be used to define a continuous path of
Fredholm operators, and hence an isomorphism
\begin{equation}
\label{eqn:hdi}
\mc{O}(\widetilde{C}_S)\simeq
\mc{O}(\widetilde{C}_\theta)
\end{equation}
via \eqref{eqn:continuation}.  Moreover, for any given
$\theta\in\R\setminus\Z$, the set of families $\{S(t)\}$ satisfying
(*) is contractible, so the isomorphism \eqref{eqn:hdi} is canonical.
Thus the canonical orientation in $\mc{O}(\widetilde{C}_\theta)$
induces a canonical orientation in $\mc{O}(\widetilde{C}_S)$.
Our last axiom is now:

\begin{description}
\item{(OR4)}
If $\{S(t)\}_{t\in S^1}$ is a family of $2\times 2$ symmetric matrices
such that the associated $2\times 2$ symplectic matrix $\Phi(2\pi)$ is
elliptic, then $\frak{o}_{\widetilde{C}_S}$ is the canonical choice
described above.
\end{description}

\begin{remark}
\label{rem:morgen}
Axioms (OR1)--(OR4) imply the following generalization of (OR4).
Consider an orientation triple $\widetilde{C}=(C,E,\{S_k\})$ such that
each end is elliptic, i.e.\ for each $k$, if
$\{\Phi_k(t)\}_{t\in[0,2\pi]}$ is the path of symplectic matrices
determined by $S_k$, then $\Phi_k(2\pi)$ has eigenvalues on the unit
circle.  Then there is a canonical orientation in
$\mc{O}(\widetilde{C})$, obtained by deforming to the complex linear
case as in \eqref{eqn:hdi}.  (For example this is how we orient the
operator $D_\Sigma$, and with it the obstruction bundle $\mc{O}$, in
\S{I.2.6}.)  Axioms (OR1)--(OR4) imply that $\frak{o}_{\widetilde{C}}$
agrees with this canonical orientation.
\end{remark}

%% file: obg-signs.tex
Let $(U_+,U_-)$ be a gluing pair satisfying \eqref{eqn:i} and
\eqref{eqn:ii}.  Throughout this section, fix $h<\lambda/4\Lambda$ in
the gluing construction as in \S\ref{sec:SDR}, and fix a system of
coherent orientations.  Also assume that $J$ is generic so that all
non-multiply-covered $J$-holomorphic curves are unobstructed, all
non-multiply-covered $J$-holomorphic curves of index $\le 2$ are
immersed (see Theorem~\ref{thm:immersed}), the obstruction section
$\frak{s}$ vanishes only for simple branched covers (see
Lemma~\ref{lem:zss}) where it is smooth (by Lemma~\ref{lem:sos}), and
the collection of eigenfunctions $\gamma$ in
\S\ref{sec:DLS} is admissible (see Propositions~\ref{prop:G1} and
\ref{prop:G2}).

Recall that $\alpha_+$ denotes the list of Reeb orbits corresponding
to the positive ends of $U_+$, and $\alpha_-$ denotes the list of Reeb
orbits corresponding to the negative ends of $U_-$.  This section will
relate the count of zeroes of $\frak{s}$ defined in
\S\ref{sec:SDR}, to a count of those ends
of the index $2$ part of the moduli space
$\mc{M}^J(\alpha_+,\alpha_-)/\R$ that are ``close to breaking'' into
$U_+$ and $U_-$ along branched covers of $\R\times\alpha$.  This
entails putting together the previous results, and then comparing
signs of zeroes of $\frak{s}$ with signs associated to these ends via
the coherent orientations.

\subsection{Statement of the result}

Recall the definition of $\mc{G}_\delta(U_+,U_-)$ from
\S\ref{sec:gluingStatement};  this describes curves that are ``close
to breaking'' in the above sense.  In \S{I.1.3} we defined an integer
$\#G(U_+,U_-)$ which counts ends of the index $2$ moduli space in
$\mc{G}_\delta(U_+,U_-)/\R$.  We recall the definition here for convenience:

\begin{definition}
\label{def:count}
Let $0<\delta'<\delta$ be small, and let
$\mc{U}\subset\mc{M}^J(\alpha_+.\alpha_-)/\R$ be an open set such
that:
\begin{itemize}
\item
$\mc{G}_{\delta'}(U_+,U_-)/\R \subset \mc{U} \subset
 \mc{G}_\delta(U_+,U_-)/\R$.
\item
The closure $\overline{\mc{U}}$ has only finitely many boundary points.
\end{itemize}
Define $\#G(U_+,U_-)\in\Z$ to be minus the signed count of boundary
points of $\overline{\mc{U}}$, where $\overline{\mc{U}}$ is oriented
via the coherent orientations.  (The orientation on $\mc{M}^J/\R$ is
induced from that of $\mc{M}^J$ via the ``$\R$-direction first''
convention, see \S{I.1.1}.)  Lemma~{I.1.11} implies that if $\delta>0$
is sufficiently small, then this count is well-defined and independent
of choices.
\end{definition}

The main result of this section is the following theorem, which
relates $\#G(U_+,U_-)$ to the count of zeroes of $\frak{s}$ that was
defined in \S\ref{sec:SDR}.
 
\begin{theorem}
\label{thm:count}
In the gluing construction, if we choose $r$ sufficiently large, then
for $R$ sufficiently large,
\[
\#G(U_+,U_-) = \epsilon(U_+)\cdot\epsilon(U_-)
\cdot\#(\frak{s}^{-1}(0)\cap\mc{V}_R/\R).
\]
\end{theorem}

Here $\epsilon(U_\pm)$ denotes the sign associated to $U_\pm$ by the
system of coherent orientations.  Theorem~\ref{thm:count}, together
with Corollary~\ref{cor:ecd} (see Remark~\ref{rem:cf}), implies the main
Theorem~\ref{thm:main}.

\subsection{Reducing to a local statement}

We now use the gluing theorem~\ref{thm:GT} to reduce
Theorem~\ref{thm:count} to a ``local'' statement involving comparing
orientations.

Recall from Theorem~\ref{thm:GT}(b) that if $\delta>0$ is sufficiently
small with respect to $r$, then the gluing map identifies
$\mc{G}_\delta(U_+,U_-)$ with a subset of $\frak{s}^{-1}(0)$.
Moreover, by Theorem~\ref{thm:GT}(a) and
Proposition~\ref{prop:deform}, if $R$ is sufficiently large then the
gluing map sends $\frak{s}^{-1}(0)\cap\mc{V}_R$ into
$\mc{G}_\delta(U_+,U_-)$.  For such $R$, our fixed coherent
orientations determine an orientation of the $1$-manifold
$\frak{s}^{-1}(0)/\R$ in a neighborhood of $\mc{V}_R/\R$.

We will see in \S\ref{sec:ECO} and \S\ref{sec:HES} that our
assumptions on $J$ imply that $\frak{s}$ is transverse to the zero
section, so that $\frak{s}^{-1}(0)$ is smooth.  Choose $R$ large as
above, and generic so that $\frak{s}^{-1}(0)/\R$ intersects
$\mc{V}_R/\R$ transversely in a finite set of points.  For each point
\[
(T_-,T_+,[\Sigma])\in\frak{s}^{-1}(0)\cap\mc{V}_R/\R,
\]
define a sign
\[
\epsilon_{\mc{M}}(T_-,T_+,[\Sigma])\in\{\pm1\}
\]
as follows: $\epsilon_{\mc{M}}(T_-,T_+,[\Sigma])\eqdef +1$ if and only
if near $(T_-,T_+,[\Sigma])$, the orientation on $\frak{s}^{-1}(0)/\R$
determined by the coherent orientations points in the increasing $R$
direction.  Also, define
\[
\epsilon_{\frak{s}}(T_-,T_+,[\Sigma])\in\{\pm1\}
\]
to be the sign of $(T_-,T_+,[\Sigma])$ as a zero of $\frak{s}$, see
\S\ref{sec:SDR}.

\begin{lemma}
\label{lem:count0}
If $r$ is chosen sufficiently large in the gluing construction, and if
$R$ is sufficiently large with respect to $r$, and generic so that
$\frak{s}^{-1}(0)$ intersects $\mc{V}_R$ transversely, then
\begin{equation}
\label{eqn:epsilonlemma}
\#G(U_+,U_-) = \sum_{(T_-,T_+,[\Sigma]) \in
\frak{s}^{-1}(0)\cap\mc{V}_R/\R} \epsilon_{\mc{M}}(T_-,T_+,[\Sigma]).
\end{equation}
\end{lemma}

\begin{proof}
Fix $r>r_0$ sufficiently large, and $\delta>0$ sufficiently small with
respect to $r$, as in Theorem~\ref{thm:GT}(b).  By
Theorem~\ref{thm:GT}(a) and Proposition~\ref{prop:deform}, if $R$ is
sufficiently large then the gluing map sends
$\frak{s}^{-1}(0)\cap\bigcup_{R'\ge R}\mc{V}_{R'}$ into
$\mc{G}_\delta(U_+,U_-)$.  Fix a generic such $R$ so that
$\frak{s}^{-1}(0)$ intersects $\mc{V}_R$ transversely.  Define
$\mc{U}\subset\mc{M}^J(\alpha_+,\alpha_-)/\R$ to be the image of the
gluing map on $\frak{s}^{-1}(0)\cap\bigcup_{R'>R}\mc{V}_{R'}$.  By
Theorem~\ref{thm:GT}(b) and the compactness of the set $\bigcup_{R'\le
R}\mc{V}_R/\R$, there exists $\delta'\in(0,\delta)$ such that $\mc{U}$
contains all of $\mc{G}_{\delta'}(U_+,U_-)/\R$.  As in
Definition~\ref{def:count}, if $\delta$ is sufficiently small (the
$\delta$ chosen above is already small enough), then
$\#G(U_+,U_-)=-\#\partial\overline{\mc{U}}$.  But
$-\#\partial\overline{\mc{U}}$ is clearly the same as the count on the
right hand side of \eqref{eqn:epsilonlemma}.
\end{proof}

As a consequence, to prove Theorem~\ref{thm:count} it sufficies to
prove the following:

\begin{lemma}
\label{lem:count1}
If $r$ is chosen sufficiently large in the gluing construction, then with
$R$ as in Lemma~\ref{lem:count0}, for each
$(T_-,T_+,[\Sigma])\in\frak{s}^{-1}(0)\cap\mc{V}_R/\R$, we have
\[
\epsilon_{\mc{M}}(T_-,T_+,[\Sigma]) =
\epsilon(U_-)\cdot\epsilon(U_+)\cdot\epsilon_{\frak{s}}(T_-,T_+,[\Sigma]).
\]
\end{lemma}

\subsection{Eliminating the coherent orientations}
\label{sec:ECO}

We now reduce Lemma~\ref{lem:count1} to a more explicit statement
which does not refer to coherent orientations.

Recall from \S\ref{sec:OSGM} that the obstruction section $\frak{s}$
is defined on the set of triples $(T_-,T_+,\Sigma)$ with $T_+\ge 5r$
and $\Sigma\in\mc{M}$.  It proves convenient henceforth to replace the
coordinates $(T_-,T_+)$ by
\[
R_-\eqdef s_- - T_-, \quad\quad R_+ \eqdef s_+ + T_+.
\]
That is, $R_\pm$ is the amount by which the curve $u_\pm$ is
translated in the $\R$ direction in the pregluing.  In these
coordinates, $\frak{s}$ is defined on triples $(R_-,R_+,\Sigma)$ with
$R_+ \ge s_+ + 5r$ and $R_- \ge s_- - 5r$.  The set $\mc{V}_R$ in
\S\ref{sec:SDR} corresponds to triples $(R_-,R_+,\Sigma)$ as above
such that $R_+-R_-=R$.  Any element of $\mc{V}_R/\R$ has a
distinguished representative with $R_\pm=\pm R/2$, and this determines
the identification of $\mc{V}_R/\R$ with a subset of $\mc{M}$ from
\S\ref{sec:SDR}.

In these new coordinates, consider a point
$(R_-,R_+,\Sigma)\in\frak{s}^{-1}(0)\cap\mc{V}_R$, with $R$ large as
in Lemma~\ref{lem:count0}, and let $C$ denote the corresponding
$J$-holomorphic curve given by the gluing theorem.  The key to the
proof of Lemma~\ref{lem:count1} is to compare signs associated to (i)
the derivative of the gluing construction and (ii) linear gluing.

\paragraph{(i)}
Use the $L^2$ inner product to identify $\Coker(D_\Sigma)$ with its
dual, and thereby regard $\frak{s}$ as taking values in
$\Coker(D_\Sigma)$.  We then have a sequence of maps
\begin{equation}
\label{eqn:DSE}
0 \longrightarrow \Ker(D_C) \stackrel{\mc{I}}{\longrightarrow}
 \R^2\oplus T_\Sigma\mc{M}
\stackrel{\nabla\frak{s}}{\longrightarrow} \Coker(D_\Sigma)
\longrightarrow 0.
\end{equation}
Here $\nabla\frak{s}$ denotes the differential of $\frak{s}$ at
$(R_-,R_+,\Sigma)$.  Meanwhile,
\[
\mc{I}:\Ker(D_C)\stackrel{\simeq}{\longrightarrow}\Ker(\nabla\frak{s})
\]
is the inverse of the derivative of the gluing map.  (The discussion
in \S\ref{sec:sos} shows that the gluing map is smooth
here, and we will see in \S\ref{sec:HES} that its derivative
$\Ker(\nabla\frak{s})\to\Ker(D_C)$ is an isomorphism.).  Since $C$ is
unobstructed, dimension counting shows that the sequence
\eqref{eqn:DSE} is exact.  Since $T_\Sigma\mc{M}$ and
$\Coker(D_\Sigma)$ have canonical orientations, the exact sequence
\eqref{eqn:DSE} determines an isomorphism
\begin{equation}
\label{eqn:Phi1}
\Phi_1:\mc{O}(\Ker(D_C))\simeq \mc{O}(\R^2).
\end{equation}

\paragraph{(ii)}
The linear gluing construction from \S\ref{sec:coherent} defines an
isomorphism
\begin{equation}
\label{eqn:LGD0}
\mc{O}(D_C)\simeq \mc{O}(D_-)\tensor \mc{O}(D_0) \tensor \mc{O}(D_+),
\end{equation}
where $D_0$ is an appropriate index $0$ operator on $\Sigma$, compare
\S\ref{sec:tsnkm}.  By Remark~\ref{rem:morgen}, there is a canonical
orientation in $\mc{O}(D_0)$.  Thus, since $D_\pm$ and $D_C$ are
unobstructed, the isomorphism \eqref{eqn:LGD0} determines an
isomorphism
\begin{equation}
\label{eqn:LGD}
\mc{O}(\Ker(D_C)) \simeq \mc{O}(\Ker(D_-))\tensor\mc{O}(\Ker(D_+)).
\end{equation}
On the other hand, the $\R$ action on the moduli spaces of
$J$-holomorphic curves determines isomorphisms $\Ker(D_\pm)\simeq \R$.
Thus the isomorphism \eqref{eqn:LGD} determines an isomorphism
\begin{equation}
\label{eqn:Phi2}
\Phi_2:\mc{O}(\Ker(D_C))\simeq \mc{O}(\R^2).
\end{equation}

\begin{lemma}
\label{lem:count2}
If $r$ is sufficiently large in the gluing construction, and if $R$ is
sufficiently large as in Lemma~\ref{lem:count0}, then for each
$(R_-,R_+,\Sigma)\in\frak{s}^{-1}(0)\cap\mc{V}_R$, the isomorphisms
$\Phi_1$ and $\Phi_2$ defined in \eqref{eqn:Phi1} and \eqref{eqn:Phi2}
agree.
\end{lemma}

This lemma will be proved in \S\ref{sec:HES}.  Granted this, we can
now give:

\begin{proof}[Proof of Lemma~\ref{lem:count1}.]
Assume that $r$ is sufficiently large as in Lemma~\ref{lem:count2} and
that $R$ is sufficiently large as in Lemma~\ref{lem:count0}, and let
$(R_-,R_+,\Sigma)$ be a transverse intersection of $\frak{s}^{-1}(0)$
with $\mc{V}_R$.  Transversality here means that the projection
\[
\Pi:
T_{(R_-,R_+,\Sigma)}\frak{s}^{-1}(0) \longrightarrow
T_{(R_-,R_+)}\R^2
\]
is an isomorphism.  Moreover,
\[
\epsilon_{\mc{M}}(R_-,R_+,\Sigma) = \text{sign}(\det(\Pi)),
\]
where $\text{sign}(\det(\Pi))$ is computed using the orientation of
$T_{(R_-,R_+,\Sigma)}\frak{s}^{-1}(0)$ determined by the coherent
orientations, together with the standard orientation of $\R^2$.  Using
these same orientations, we also see from the exact sequence
\eqref{eqn:DSE} that
\[
\text{sign}(\Phi_1) = \text{sign}(\det(\Pi))\cdot
\epsilon_{\frak{s}}(R_-,R_+,\Sigma).
\]
On the other hand, we have
\[
\text{sign}(\Phi_2) = \epsilon(U_-) \cdot \epsilon(U_+),
\]
because by definition the isomorphism $\mc{O}(D_\pm)\simeq
\R$ is orientation-preserving if and only if $\epsilon(U_\pm)=\pm1$.
Combining the above three equations with Lemma~\ref{lem:count2} proves
Lemma~\ref{lem:count1}.
\end{proof}

\subsection{Setting up the linear gluing exact sequence}
\label{sec:tsnkm}

Fix $(R_-,R_+,\Sigma)\in\frak{s}^{-1}(0)$, and let $C$ denote the
associated glued curve.  To prepare for the proof of
Lemma~\ref{lem:count2}, we now show that if $r,T_-,T_+$ are
sufficiently large, then a version of the linear gluing exact sequence
\eqref{eqn:TLG} is applicable, with $C_\pm = u_\pm$ and $C_0\approx\Sigma$.

Here is the precise setup.  Let $E_\pm$ denote the normal bundle to $u_\pm$,
and let $E_C$ denote the normal bundle to $C$.  Recall that we have
linear deformation operators $D_\pm:C^\infty(E_\pm)\to
C^\infty(T^{0,1}C_\pm\tensor E_\pm)$ and $D_C:C^\infty(E_C)\to
C^\infty(T^{0,1}C\tensor E_C)$.  Use the coordinates $z,w$ in a
neighborhood of the Reeb orbit $\alpha$ as usual to trivialize $E_-$
over the positive ends of $C_-$ and to trivialize $E_+$ over the
negative ends of $C_+$.

Let $C'$ denote the surface obtained from $C$ by removing the
$s<s_--T_-$ portion of the first $N_-$ negative ends and the
$s>s_++T_+$ portion of the first $N_+$ positive ends.  Let $C_0$
denote the surface obtained from $C'$ by attaching infinite
cylindrical ends to the boundary circles.  Note that $C_0$ is
naturally identified with $\Sigma$, because $\Sigma'$ parametrizes
$C'$ by a map sending $u\mapsto(z,w)=(\pi(u),\psi_\Sigma(u))$.  The
identification $\imath:C_0\to\Sigma$ is bi-Lipschitz, and off of
the ramification points it is smooth.

For future reference, here is a more explicit description of $C_0$
near a ramification point $p$.  Let $v$ be a holomorphic local
coordinate on $C_0$ that vanishes at $p$.  It follows from
\eqref{eqn:t10ry}, as in \S\ref{sec:zss}, that the holomorphic
coordinate $v$ can be rescaled so that near $v=0$,
\begin{equation}
\label{eqn:vwx}
\begin{split}
z & = z_0 + (1-|a_0|^2)^{-1}(v^2 + a_0\overline{v}^2) +
O(|v|^3),\\
w & = w_0 + c v + O(|v|^2),
\end{split}
\end{equation}
where $z_0, w_0, a_0$ denote the values of $z,w,a$ at $p$, and
$c$ is a nonzero constant.

Now define a bundle $E_0$ over $C_0$, and a differential operator
\[
D_0:C^\infty(E_0)\to C^\infty(T^{0,1}C_0\tensor E_0),
\]
as follows.  Over $C'$, define $E_0$ to be the pullback of the normal
bundle to $C$ in $\R\times Y$.  The coordinate $w$ trivializes this
bundle on the complement of the ramification points.  Use this
trivialization to extend $E_0$, with trivialization, over the ends of
$C_0$.

Next, define a continuous bundle map
\begin{equation}
\label{eqn:cbm}
\imath^{0,1}:T^{0,1}\Sigma \longrightarrow T^{0,1}C_0
\end{equation}
as follows.  On the complement of the ramification points,
$\imath^{0,1}$ is a smooth bundle map defined by pulling back from
$T^{0,1}\Sigma$ to $T_\C^*C_0$ via the map $\imath:C_0\to\Sigma$, and
then projecting along $T^{1,0}C_0$ to $T^{0,1}C_0$.  Here the complex
structure on $C_0$ is chosen to agree with that of $C$ over the
support of $\beta_\Sigma$, and to agree with the standard complex
structure on the cylinder over the ends of $C_0$.  The map
\eqref{eqn:cbm} extends continuously over the ramification points,
where it is zero.  We can choose the complex structure on $C_0$ so
that $\imath^{0,1}$ is an isomorphism on the complement of the
ramification points; this is because the $(0,1)$ part of the $1$-form
$d\overline{z}$ is
\begin{equation}
\label{eqn:dz01}
(d\overline{z})^{0,1} = (1-|a|^2)^{-1}(d\overline{z}
- \overline{a}dz).
\end{equation}
Finally, define
\begin{equation}
\label{eqn:D0}
D_0 \eqdef \beta_\Sigma D_C + (1-\beta_\Sigma)\imath^{0,1}D_\Sigma,
\end{equation}
where $\beta_\Sigma$ is the cutoff function defined in \S\ref{sec:defpre}.

In order to obtain a version of the linear gluing exact sequence, we
want to choose a finite dimensional subspace $V_0\subset
L^2(T^{0,1}C_0\tensor E_0)$ such that if $W_0$ denotes the orthogonal
complement of $V_0$ and if $\Pi_{W_0}$ denotes the orthogonal projection
onto $W_0$, then $\Pi_{W_0} D_0$ is surjective.  For this purpose,
define a continuous bundle map
\begin{equation}
\label{eqn:cbt}
\frak{n}: \Sigma\times\C \longrightarrow E_0,
\end{equation}
covering the Lipschitz map $\imath^{-1}:\Sigma\to C_0$, as follows.
Over $\Sigma\setminus \Sigma'$, the map $\frak{n}$ is just the
trivialization of $E_0$ over the ends of $C_0$.  Given $u\in \Sigma'$
and $w\in\C$, let $p$ denote the point in $C$ corresponding to $u$,
let $W\in T_p(\R\times Y)$ denote the tangent vector to the
$z=\text{constant}$ disc corresponding to $w$, and define
$\frak{n}(u,w)$ to be the projection of $W$ onto the normal bundle to
$C$ at $p$.  Note that $\frak{n}$ is zero at each ramification point,
and an isomorphism at every other point of $\Sigma$.  Tensoring the
bundle maps \eqref{eqn:cbm} and \eqref{eqn:cbt} defines a Lipschitz
bundle map
\[
\frak{n}^{0,1}\eqdef \imath^{0,1}\tensor\frak{n}: T^{0,1}\Sigma
\longrightarrow T^{0,1}C_0\tensor E_0
\]
which vanishes at the ramification points and is a smooth isomorphism
elsewhere.  In terms of this last bundle map, define
\[
V_0 \eqdef \frak{n}^{0,1} (\Coker(D_\Sigma)) \subset
L^2(T^{0,1}C_0\tensor E_0).
\]

We will see below that $\Pi_{W_0}D_0$ is surjective for this choice of
$V_0$.  We also want to understand the kernel of $\Pi_{W_0}D_0$.  For
this purpose define a linear map
\[
\rho : C^0(E_0) \longrightarrow T_\Sigma\mc{M}
\]
as follows.  Let $\mc{R}$ denote the set of ramification points of
$\Sigma$.  Due to our choice of almost complex structure $J$, the
ramification points are simple, so branched covers in $\mc{M}$ near
$\Sigma$ are determined by the projections of their ramification
points to $\R\times S^1$.  Consequently there is a natural identification
\[
T_\Sigma\mc{M} = \bigoplus_{u\in\mc{R}}T_{\pi(u)}(\R\times S^1).
\]
Under this identification, the map $\rho$ sends a continuous section
$\psi$ of $E_0$ to the collection of tangent vectors
$(z_u)_{u\in\mc{R}}$, where $z_u$ denotes the pushforward of $\psi(u)$ by
the projection $(z,w)\mapsto (z,0)$.

\begin{lemma}
\label{lem:tsnkm}
If $r$ is sufficiently large in the gluing construction, then:
\begin{description}
\item{(a)}
$\Pi_{W_0}D_0:L^2_1(E_0)\to W_0$ is surjective.
\item{(b)}
The map $\rho$ restricts to an isomorphism
\begin{equation}
\label{eqn:rhoker}
\rho:\Ker(\Pi_{W_0}D_0)\stackrel{\simeq}{\longrightarrow} T_\Sigma\mc{M}.
\end{equation}
\end{description}
\end{lemma}

\begin{proof}
First note that if $\psi\in\Ker(\Pi_{W_0}D_0)$, then by definition
there exists $\eta\in\Coker(D_\Sigma)$ with $D_0\psi = \frak{n}^{0,1}\eta$.
Since $\eta$ is a smooth $(0,1)$-form on $\Sigma$, and since
$\frak{n}^{0,1}$ is Lipschitz, it follows that $D_0\psi$ has bounded first
derivatives near the ramification points, and so by elliptic
regularity $\psi$ is $C^1$.  In particular, $\rho$ is well defined on
$\Ker(\Pi_{W_0}D_0)$.

Now since the index of the operator $\Pi_{W_0}D_0$ equals the
dimension of $T_\Sigma\mc{M}$, to prove both (a) and (b) it suffices
to show that the map \eqref{eqn:rhoker} is injective.  Suppose $\psi\in
\Ker(\Pi_{W_0}D_0)$ satisfies $\rho(\psi)=0$; we will
show that $\psi=0$.

{\em Step 1.\/}
We first show that there exists $\zeta\in L^2_1(\Sigma,\C)$ with
$\frak{n}\zeta=\psi$.

The only issue is to check that $\frak{n}^{-1}\psi$ is $L^2_1$ in a
neighborhood of each ramification point $p$.  Near $p$, in terms of
the local description
\eqref{eqn:vwx}, the normal bundle to $C_0$ near $v=0$ is trivialized
by a $(1,0)$ form that annihilates $TC_0$ and has the form
\begin{equation}
\label{eqn:nhat}
\widehat{n}=\frac{c}{2}(dz - a_0 d \overline{z}) - v dw +
O(|v|)dz+O(|v|)d\overline{z} + O(|v|^2)dw+O(|v|^2)d\overline{w}.
\end{equation}
It follows that in this local trivialization, using the Lipschitz
identification $\imath:\Sigma\to C_0$ to regard $\frak{n}$ as a map
between bundles over $C_0$, we have
\begin{equation}
\label{eqn:ltt1}
\frak{n} = -v + O(|v|^2).
\end{equation}
Since $D_0\psi\in V_0$, we know that $\psi$ is $C^1$ and
$D_0\psi=O(|v|)$, and since $\rho(\psi)=0$ we know that $\psi=O(|v|)$.
It follows that in the local trivialization \eqref{eqn:nhat},
\begin{equation}
\label{eqn:ltt2}
\frac{\partial\psi}{\partial\overline{v}}=O(|v|).
\end{equation}
Now \eqref{eqn:ltt1} and \eqref{eqn:ltt2} imply that
$\partial(\frak{n}^{-1}\psi)/\partial\overline{v}$ is bounded, and so by
elliptic regularity again, $\frak{n}^{-1}\psi$ is $L^2_1$.

{\em Step 2.\/} We now show that if $\zeta\in L^2_1(\Sigma,\C)$
satisfies $\Pi_{W_0}D_0(\frak{n}\zeta)=0$ then $\zeta=0$.

We begin by deriving a useful formula for $D_0(\frak{n}\zeta)$.  First
restrict attention to $\Sigma'$.  Here, off of the ramification
points, regard $\psi_\Sigma$ locally as a function of $z\in\R\times
S^1$.  Recall from \S\ref{sec:G1} that the graph, $C$, of
$\psi_\Sigma$ is $J$-holomorphic if and only if $\dbar(C)=0$, where
$\dbar(C)$ denotes the $1$-form on $C$ with values in the normal bundle
$N_C$ that inputs a tangent vector $v$ and outputs the projection of
$Jv$ to $N_C$.  To describe $\dbar(C)$ more explicitly in the present
case, note that the projection from $T(\R\times Y)|_C$ to $N_C$ is
given by the composition of $\frak{n}$ with the $1$-form
$dw-d\psi_\Sigma$.  Hence $\frak{n}^{-1}\dbar(C)$ is the restriction
to $C$ of $-2i$ times the $(0,1)$ part of $dw-d\psi_\Sigma$.  By
\eqref{eqn:t10ry} and \eqref{eqn:dz01}, this gives
\begin{equation}
\label{eqn:dbarc}
\dbar(C)=
2i
\frak{n}^{0,1}\left(\left(\frac{\partial\psi_\Sigma}{\partial\overline{z}} 
+ a\frac{\partial\psi_\Sigma}{\partial z} +
  b\right)d\overline{z}\right).
\end{equation}
By definition,
\[
D_C(\frak{n}\zeta) = \frac{1}{2i}
\frac{d}{d\epsilon}\bigg|_{\epsilon=0}\dbar(C_\epsilon),
\]
where $C_\epsilon$ denotes the graph of $\psi_\Sigma+\epsilon\zeta$.
Therefore
\begin{equation}
\label{eqn:dcnpsi}
\begin{split}
D_C(\frak{n}\zeta) &= 
\frak{n}^{0,1}\left(\left(\frac{\partial\zeta}{\partial\overline{z}}
+ a\frac{\partial\zeta}{\partial z} +
  (\nabla_\zeta a)\frac{\partial\psi_\Sigma}{\partial z} +
  \nabla_\zeta b\right)d\overline{z}\right)\\
&= \frak{n}^{0,1}\left(D_\Sigma\zeta + 
\mc{R}(\zeta)
\right)
\end{split}
\end{equation}
where $\nabla_\zeta$ denotes the derivative along the
$z=\text{constant}$ disks in the direction determined by $\zeta$, and
\begin{equation}
\label{eqn:Rpsi}
\mc{R}(\zeta)\eqdef
\left(
a\frac{\partial\zeta}{\partial z} +
 (\nabla_\zeta a)\frac{\partial\psi_\Sigma}{\partial z} +
 \left(\nabla_\zeta b - \nabla_\zeta b|_{w=0}\right)\right)
d\overline{z}.
\end{equation}
It then follows from \eqref{eqn:D0} that on all of $\Sigma$,
\begin{equation}
\label{eqn:d0taupsi}
D_0(\frak{n}\zeta) = 
\frak{n}^{0,1}\left(D_\Sigma\zeta + 
\beta_\Sigma
\mc{R}(\zeta)
\right).
\end{equation}

To use this formula, note that our assumption that
$\Pi_{W_0}D_0(\frak{n}\zeta)=0$ means that
$D_0(\frak{n}\zeta)\in\frak{n}^{0,1}\Coker(D_\Sigma)$.  Since
$\frak{n}^{0,1}$ is an isomorphism except at the ramification points,
it follows from this and \eqref{eqn:d0taupsi} that
\begin{equation}
\label{eqn:ate1}
D_\Sigma\zeta + 
\beta_\Sigma
\mc{R}(\zeta)
\in\Coker(D_\Sigma).
\end{equation}
Now recall from the proof of Lemma~\ref{lem:SB} that there is a
constant $c>0$, not depending on $\Sigma,T_-,T_+$, such that
\begin{equation}
\label{eqn:ate2}
\|D_\Sigma\zeta\|_{L^2} \ge c\|\zeta\|_{L^2_1}.
\end{equation}
On the other hand, inspection of \eqref{eqn:Rpsi} shows that there is
a constant $c'$ with
\begin{equation}
\label{eqn:ate3}
\left|
\mc{R}(\zeta)
\right|
\le c' |w| \left(|\zeta| + |\nabla\zeta|\right).
\end{equation}
It follows from \eqref{eqn:ate1}, \eqref{eqn:ate2}, and
 \eqref{eqn:ate3} that if $r$ is sufficiently large, so that $|w|$ is
 always sufficiently small on the support of $\beta_\Sigma$, then
 $\zeta=0$.
\end{proof}

\subsection{A homotopy of exact sequences}
\label{sec:HES}

With all the setup in place, we come now to the heart of the proof of
Theorem~\ref{thm:count}.  If $r$ is sufficiently large, then since
$D_\pm$, $D_C$, and $\Pi_{W_0}D_0$ are surjective, as in
\eqref{eqn:TLG} we obtain an exact sequence
\begin{equation}
\label{eqn:TLG2}
0 \longrightarrow \Ker(D_C) \stackrel{f}{\longrightarrow} \Ker(D_-)\oplus
\Ker(\Pi_{W_0}D_0)\oplus\Ker(D_+) 
\stackrel{g}{\longrightarrow} V_0 \longrightarrow 0.
\end{equation}
Here we use the cutoff functions $\beta_-$,
$\beta_0\eqdef\beta_\Sigma$, and $\beta_+$ from \S\ref{sec:gluing};
these are slightly different from the cutoff functions in
\eqref{eqn:TLG}, but the resulting exact sequence will be the same up
to homotopy, as in \S\ref{sec:GO}.  Moreover, the isomorphisms
\begin{equation}
\label{eqn:identifications}
\Coker(D_\Sigma)\simeq V_0, \quad\quad \Ker(\Pi_{W_0}D_0)\simeq
T_\Sigma\mc{M},\quad\quad\Ker(D_\pm)\simeq\R,
\end{equation}
determined by $\frak{n}^{0,1}$, $\rho$, and the $\R$-action
respectively, identify the terms in the exact sequence
\eqref{eqn:TLG2} with those in the exact sequence \eqref{eqn:DSE},
although the maps may be different.  (In these identifications we are
commuting $\Ker(D_+)$ with $\Ker(\Pi_{W_0}D_0)$, which has no effect
on orientations since the latter is even dimensional.)  On the other
hand, the exact sequence \eqref{eqn:TLG2}, together with the exact
sequence
\begin{equation}
\label{eqn:CES}
0 \longrightarrow \Ker(D_0) \longrightarrow \Ker(\Pi_{W_0}D_0)
\stackrel{D_0}{\longrightarrow}
V_0 \longrightarrow
 \Coker(D_0) \longrightarrow 0,
\end{equation}
determines the isomorphism \eqref{eqn:LGD0} on orientations.  So to prove
Lemma~\ref{lem:count2}, and thus Theorem~\ref{thm:count}, it is enough
to prove the following:

\begin{lemma}
\label{lem:last}
Let $(R_-,R_+,\Sigma)\in\frak{s}^{-1}(0)$.  If $r$ is
sufficiently large in the gluing construction, then under the identifications
\eqref{eqn:identifications}:
\begin{description}
\item{(a)}
The isomorphism
\[
\mc{O}(T_\Sigma\mc{M})\tensor\mc{O}(\Coker(D_\Sigma))
\stackrel{\simeq}{\longrightarrow} \mc{O}(D_0)
\]
determined by \eqref{eqn:CES} sends the tensor product of the
canonical orientations of $T_\Sigma\mc{M}$ and $\Coker(D_\Sigma)$ to the
canonical orientation of $D_0$.
\item{(b)}
The exact sequences \eqref{eqn:DSE} and \eqref{eqn:TLG2} are homotopic
through exact sequences, and so induce the same isomorphism on
orientations.
\end{description}
\end{lemma}

\begin{proof}
Assertion (a) follows by deforming to the complex linear case.

The proof of assertion (b) has the following outline:

{\em Part 1.\/}
We will first construct a homotopy of exact sequences
\begin{equation}
\label{eqn:HES}
0 \longrightarrow \Ker(D_C) \stackrel{f_\tau}{\longrightarrow} \R^2
\oplus T_\Sigma\mc{M} \stackrel{g_\tau}{\longrightarrow} \Coker(D_\Sigma)
 \longrightarrow 0,
\end{equation}
parametrized by $\tau\in[0,1]$, such that when $\tau=1$, the exact
sequence \eqref{eqn:HES} agrees with \eqref{eqn:DSE}.

{\em Part 2.\/}
We will then relate the exact sequence \eqref{eqn:HES} for $\tau=0$ to
the exact sequence \eqref{eqn:TLG2} by defining a map
\begin{equation}
\label{eqn:rhotilde}
\widetilde{\rho}:\Ker(D_-) \oplus \Ker(\Pi_{W_0}D_0)\oplus\Ker(D_+)
\stackrel{\simeq}{\longrightarrow} \R^2\oplus T_\Sigma\mc{M}
\end{equation}
such that the following diagram commutes:
\begin{equation}
\label{eqn:rhotildediagram}
\begin{CD}
\Ker(D_C) @>{f}>> \Ker(D_-) \oplus \Ker(\Pi_{W_0}D_0)\oplus\Ker(D_+)
@>{g}>> V_0 \\
@| @VV{\widetilde{\rho}}V @AA{\frak{n}^{0,1}}A \\
\Ker(D_C) @>{f_0}>> \R^2\oplus T_\Sigma\mc{M} @>{g_0}>>
\Coker(D_\Sigma)
\end{CD}
\end{equation}

{\em Part 3.\/}
Lastly, we will show that $\widetilde{\rho}$ is an isomorphism which is
homotopic through isomorphisms to the isomorphism given by $\rho$ and
the identifications $\Ker(D_\pm)\simeq\R$.

The details follow.

\medskip

{\em Part 1.\/} Fix $(r_-,r_+,X)\in \R^2\oplus T_\Sigma\mc{M}$, and
fix $\tau\in[0,1]$.  We begin with a somewhat lengthy definition of
$g_\tau(r_-,r_+,X)$.

Let $\phi_\pm\in\Ker(D_\pm)$ correspond to $r_\pm\in\R$.  Let $\phi_0$
denote the unique element of $\Ker(\Pi_{W_0}D_0)$ for which
$\rho(\phi_0)=X$; this is provided by Lemma~\ref{lem:tsnkm}.  Given
$\zeta_\pm\in L^2_1(E_\pm)$ orthogonal to $\Ker(D_\pm)$ and given
$\zeta_\Sigma\in L^2_1(\Sigma,\C)$, consider
\begin{gather}
\label{eqn:nzs}
\psi_-\eqdef \phi_-+\zeta_-, \quad\quad
\psi_0\eqdef\phi_0+\frak{n}\zeta_\Sigma, \quad\quad
\psi_+\eqdef\phi_++\zeta_+,\\
\label{eqn:pnd}
\psi \eqdef \beta_-\psi_-+\beta_0\psi_0+\beta_+\psi_+ \in L^2_1(N_C).
\end{gather}
Recall that the construction of the linear gluing exact sequence
\eqref{eqn:TLG2}
writes
\[
D_C(\psi) = \beta_-\Theta_-(\psi_-,\psi_0) +
\beta_0\Theta_0(\psi_-,\psi_0,\psi_+) +
\beta_+\Theta_+(\psi_0,\psi_+),
\]
where $\Theta_\pm$ and $\Theta_0$ are defined by analogy with
\eqref{eqn:definethetas}.  Near the ramification points,
$\Theta_0=D_C\psi_0$, and so by \eqref{eqn:dcnpsi} and
\eqref{eqn:nzs}, the above equation can be rewritten in the form
\begin{equation}
\label{eqn:dcpsi0}
D_C(\psi) = \beta_-\Theta_-(\psi_-,\psi_0) +
\beta_0\frak{n}^{0,1}\Theta_\Sigma(\psi_-,\psi_0,\psi_+) +
\beta_+\Theta_+(\psi_0,\psi_+)
\end{equation}
where $\Theta_\Sigma\in L^2(T^{0,1}\Sigma,\C)$.

On the other hand, the derivative of the gluing construction in the
direction $(r_-,r_+,X)$ defines
an alternate expression
\begin{equation}
\label{eqn:dcpsi1}
D_C(\psi) = \beta_-\Theta_-'(\psi_-,\psi_0) +
\beta_0\frak{n}^{0,1}\Theta_\Sigma'(\psi_-,\psi_0,\psi_+) +
\beta_+\Theta_+'(\psi_0,\psi_+)
\end{equation}
as follows.

The $\Theta'$'s are first-order differential operators, so to define
the expression \eqref{eqn:dcpsi1} we can assume that $\psi_\pm$ and
$\psi_0$ are smooth.  Consider a smooth one-parameter family of
triples $(R_-(\epsilon),R_+(\epsilon),\Sigma(\epsilon))$, parametrized
by $\epsilon$ in a neighborhood of $0$ in $\R$, such that
$(R_-(0),R_+(0),\Sigma(0))$ agrees with our given element
$(R_-,R_+,\Sigma)\in\frak{s}^{-1}(0)$, while
$\frac{d}{d\epsilon}\big|_{\epsilon=0}R_\pm(\epsilon) = r_\pm$ and
$\frac{d}{d\epsilon}\big|_{\epsilon=0}\Sigma(\epsilon)=X$.  Let
$\widetilde{\zeta}_\pm(\epsilon)$ be a smooth one-parameter family of
sections of the normal bundle to $u_\pm$ such that
$\widetilde{\zeta}_\pm(0)$ is the section produced by the gluing
construction (in Propositions~\ref{prop:CMT} and \ref{prop:TRE} and
denoted there by $\psi_\pm$) applied to $(R_-,R_+,\Sigma)$, while
$\frac{d}{d\epsilon}\big|_{\epsilon=0}\widetilde{\zeta}_\pm(\epsilon)
= \zeta_\pm$.

Fix a small neighborhood $U$ of the ramification points in $\Sigma$.
Note that when $\epsilon$ is sufficiently small, there is a canonical
diffeomorphism of $\Sigma\setminus U$ with a subset of
$\Sigma(\epsilon)$, respecting the projections to $\R\times S^1$.  Let
$\widetilde{\zeta}_\Sigma(\epsilon)$ be a one-parameter family of
complex-valued functions on $\Sigma\setminus U$ such that
$\widetilde{\zeta}_{\Sigma}(0)$ is the restriction to $\Sigma\setminus
U$ of the function on $\Sigma$ produced by the gluing construction
(Proposition~\ref{prop:TRE}) applied to $(R_-,R_+,\Sigma)$, and
\begin{equation}
\label{eqn:confusing}
\frac{d}{d\epsilon}\bigg|_{\epsilon=0}
 \widetilde{\zeta}_\Sigma(\epsilon) = \frak{n}^{-1}\psi_0.
\end{equation}

To continue, let $C(\epsilon)$ denote the partially defined surface
obtained from the pregluing construction applied to
$(R_-(\epsilon),R_+(\epsilon),
\Sigma(\epsilon))$ using the sections
$(\widetilde{\zeta}_-(\epsilon),\widetilde{\zeta}_\Sigma(\epsilon),
\widetilde{\zeta}_+(\epsilon))$;  this is defined over the complement
of $U$.  Here equation \eqref{eqn:7} writes $\dbar(C(\epsilon))$ in
 the form
\begin{equation}
\label{eqn:dbartilde}
\begin{split}
2i \dbar(C(\epsilon)) =
& \beta_-(\epsilon)\widetilde{\Theta}_-(\widetilde{\zeta}_-(\epsilon),
\widetilde{\zeta}_{\Sigma(\epsilon)})
+
\beta_+(\epsilon)\widetilde{\Theta}_+(\widetilde{\zeta}_{\Sigma(\epsilon)},
\widetilde{\zeta}_+) \\
&+ 
\beta_0(\epsilon)\frak{n}^{0,1}\widetilde{\Theta}_{\Sigma(\epsilon)}
(\widetilde{\zeta}_-(\epsilon),
\widetilde{\zeta}_{\Sigma(\epsilon)},\widetilde{\zeta}_+(\epsilon))
\end{split}
\end{equation}
(where $\widetilde{\Theta}$ here corresponds to $\Theta$ in equation
\eqref{eqn:7}).  Note that on the part of $C$ corresponding to the
complement of $U$, the normal derivative of the family of surfaces
$\{C_\epsilon\}$ at $\epsilon=0$ is given by the section $\psi$
defined in \eqref{eqn:pnd}.  Thus differentiating equation
\eqref{eqn:dbartilde} at $\epsilon=0$, and using the fact that the
$\widetilde{\Theta}$'s vanish at $\epsilon=0$, gives
\[
\begin{split}
D_C(\psi) 
&= \beta_-\frac{\partial}{\partial\epsilon}\Big|_{\epsilon=0}
\widetilde{\Theta}_- + 
\beta_0\frak{n}^{0,1}
\frac{\partial}{\partial\epsilon}\Big|_{\epsilon=0}
\widetilde{\Theta}_{\Sigma(\epsilon)} +
\beta_+\frac{\partial}{\partial\epsilon}\Big|_{\epsilon=0}
\widetilde{\Theta}_+.
\end{split}
\]
The expansion \eqref{eqn:dcpsi1} is now defined over
the complement of $U$ by setting $\Theta_\pm'\eqdef
\frac{\partial}{\partial\epsilon}\big|_{\epsilon=0}\widetilde{\Theta}_\pm$
and
$\Theta_\Sigma'\eqdef\frac{\partial}{\partial\epsilon}\big|_{\epsilon=0}
\widetilde{\Theta}_{\Sigma(\epsilon)}$.  Near the ramification points, this
expansion agrees with \eqref{eqn:dcpsi0}, and as such has a canonical
extension over $U$.

With the preceding understood, for $\tau\in[0,1]$ define
\[
\Theta_-^\tau \eqdef (1-\tau)\Theta_- + \tau \Theta_-', \quad\;
\Theta_\Sigma^\tau \eqdef (1-\tau)\Theta_\Sigma + \tau \Theta_\Sigma',
\quad\;
\Theta_+^\tau \eqdef (1-\tau)\Theta_+ + \tau \Theta_+'.
\]
Keep in mind that the $\Theta^\tau$'s depend implicitly on the triple
$(r_-,r_+,X)\in \R^2\oplus T_\Sigma\mc{M}$.  
Note that one can write
\begin{gather*}
\Theta_-^\tau = D_-\zeta_- + \mc{R}_-^\tau(\zeta_-,\zeta_\Sigma),
\quad\quad \Theta_+^\tau = D_+\zeta_+ +
\mc{R}_+^\tau(\zeta_\Sigma,\zeta_+),\\
\Theta_\Sigma^\tau = D_\Sigma\zeta_\Sigma +
(\frak{n}^{0,1})^{-1}D_0\phi_0 +
\mc{R}_\Sigma^\tau(\zeta_-,\zeta_\Sigma,\zeta_+),
\end{gather*}
where each term in $\mc{R}^\tau_\pm$ or $\mc{R}_\Sigma^\tau$ that is
linear in $(\zeta_-,\zeta_\Sigma,\zeta_+)$ maps from $L^2_1$ to $L^2$
with small operator norm when $r$ is large.  It follows by the usual
arguments that if $r$ is sufficiently large, then there exist unique
$\zeta_\pm \in L^2_1(E_\pm)$ that are $L^2$-orthogonal to
$\Ker(D_\pm)$, and $\zeta_\Sigma\in L^2_1(\Sigma,\C)$, such that
\begin{equation}
\label{eqn:thetatau}
\Theta_-^\tau(\psi_-,\psi_0) = 0, \quad \quad
\Theta_\Sigma^\tau(\psi_-,\psi_0,\psi_+)\in \Coker(D_\Sigma),
\quad\quad \Theta_+^\tau(\psi_0,\psi_+) = 0.
\end{equation}
Moreover, $\zeta_\pm$ and $\zeta_\Sigma$ vary continuously with
$\tau$.  For this distinguished $\zeta_\pm$ and $\zeta_\Sigma$, we define
\[
g_\tau(r_-,r_+,X) \eqdef
\Theta_\Sigma^\tau(\psi_-,\psi_0,\psi_+).
\]

To complete the construction of the exact sequence \eqref{eqn:HES},
note that by \eqref{eqn:dcpsi0} and \eqref{eqn:dcpsi1}, we have a map
\begin{equation}
\label{eqn:obviousmap}
\begin{split}
\Ker(g_\tau) & \longrightarrow \Ker(D_C),\\
(r_-,r_+,X) & \longmapsto \beta_-\psi_- + \beta_0\psi_0 +
\beta_+\psi_+.
\end{split}
\end{equation}
A linear version of Lemmas~\ref{lem:IGM} and
\ref{lem:SGM} shows that if $r$ is sufficiently large, then the map
\eqref{eqn:obviousmap} is an isomorphism. We now define $f_\tau$ to be
the inverse of the map \eqref{eqn:obviousmap}.  Thus $f_\tau$ is
injective and $\op{Im}(f_\tau)=\Ker(g_\tau)$.  Since $C$ is
unobstructed, $\dim\Ker(D_C)=2$, and so by dimension counting,
the sequence \eqref{eqn:HES} is exact.

We now show that when $\tau=1$, the exact sequence \eqref{eqn:HES}
agrees with \eqref{eqn:DSE}.  We first show that
$g_1=\nabla\frak{s}$.  Let $(r_-,r_+,X)\in\R^2\oplus
T_\Sigma\mc{M}$ be given, and let $\phi_\pm$, $\phi_0$, and
$(R_-(\epsilon),R_+(\epsilon),\Sigma(\epsilon))$ be as before.  For
each $\epsilon$, the gluing construction finds a unique triple
$(\widetilde{\zeta}_-(\epsilon),
\widetilde{\zeta}_\Sigma(\epsilon),\widetilde{\zeta}_+(\epsilon))$,
where $\widetilde{\zeta}_\pm(\epsilon)$ is an $L^2_1$ section of
$u_\pm$ orthogonal to $\Ker(D_\pm)$, and
$\widetilde{\zeta}_\Sigma(\epsilon)\in L^2_1(\Sigma(\epsilon),\C)$,
such that $\widetilde{\Theta}_\pm=0$ and
$\widetilde{\Theta}_{\Sigma(\epsilon)}\in\Coker(D_{\Sigma(\epsilon)})$.
These depend smoothly on $\epsilon$ (see \S\ref{sec:sos}), and by
definition,
\begin{equation}
\label{eqn:defnablas}
\nabla\frak{s}(r_-,r_+,X) =
\frac{d}{d\epsilon}\Big|_{\epsilon=0}\widetilde{\Theta}_{\Sigma(\epsilon)}
\in\Coker(D_\Sigma).
\end{equation}
Now define $\zeta_\pm\eqdef
\frac{d}{d\epsilon}\big|_{\epsilon=0}\widetilde{\zeta}_\pm(\epsilon)\in
L^2_1(E_\pm)$ and $\psi_\pm\eqdef\phi_\pm+\zeta_\pm$.  Also, define
$\psi_0\in L^2_1(E_0)$ as follows.  Off of the ramification points,
$\psi_0$ is given by equation
\eqref{eqn:confusing}.  In a neighborhood of the ramification
points, $\psi_0$ is the normal derivative of the family of
surfaces $C(\epsilon)$.  Note that $\rho(\psi_0)=X$.  Moreover, near
the ramification points, $D_C\psi_0=\frak{n}^{0,1}\nabla\frak{s}$.  It
follows as in the proof of Lemma~\ref{lem:tsnkm} that
$\psi_0-\phi_0=\frak{n}\zeta_\Sigma$ for some $\zeta_\Sigma\in
L^2_1(\Sigma,\C)$.  The triple $(\zeta_-,\zeta_\Sigma,\zeta_+)$ is
then the unique solution to the equations \eqref{eqn:thetatau}, so by
the definition of $g_1$ and equation \eqref{eqn:defnablas} we
conclude that
$g_1(r_-,r_+,X)=\nabla\frak{s}(r_-,r_+,X)$.
Similarly, $f_1=\mc{I}$.

\medskip

{\em Part 2.\/} We now define the map $\widetilde{\rho}$ in
\eqref{eqn:rhotilde}.  Let $(\phi_-,\phi_0',\phi_+)$ be given, where
$\phi_\pm\in\Ker(D_\pm)$ and $\phi_0'\in\Ker(\Pi_{W_0}D_0)$.  As in
the definition of the map $g$ in
\eqref{eqn:TLG2}, there are unique $\zeta_\pm\in L^2_1(E_\pm)$ orthogonal to
$\Ker(D_\pm)$ and $\zeta_0\in L^2_1(E_0)$ orthogonal to
$\Ker(\Pi_{W_0}D_0)$, such that $\psi_\pm\eqdef\phi_\pm + \zeta_\pm$
and $\psi_0\eqdef \phi_0'+\zeta_0$ satisfy
\[
\Theta_-(\psi_-,\psi_0)=0, \quad\quad \Theta_+(\psi_0,\psi_+)=0,
\quad\quad \Theta_0(\psi_-,\psi_0,\psi_+)\in V_0.
\]
Near the ramification points, $\Theta_0=D_0\psi_0$, so as in the proof
of Lemma~\ref{lem:tsnkm}, $\rho(\psi_0)$ is defined.  Let $r_\pm$
correspond to $\phi_\pm$ under our usual identification
$\Ker(D_\pm)\simeq\R$, and define
\[
\widetilde{\rho}(\phi_-,\phi_0',\phi_+) \eqdef (r_-,r_+,\rho(\psi_0)).
\]

We now show that the diagram \eqref{eqn:rhotildediagram} commutes.  To
see that the right square commutes, continue with the notation from
the definition of $\widetilde{\rho}$, and let $\phi_0$ denote the
unique element of $\Ker(\Pi_{W_0}D_0)$ for which
$\rho(\phi_0)=\rho(\psi_0)$.  As in the proof of
Lemma~\ref{lem:tsnkm},
\begin{equation}
\label{eqn:zetaSigma}
\psi_0=\phi_0 + \frak{n}\zeta_\Sigma
\end{equation}
for some $\zeta_\Sigma\in L^2_1(\Sigma,\C)$.  Then
$(\psi_-,\psi_0,\psi_+)$ is the unique solution to the equations
\eqref{eqn:thetatau} for $\tau=0$ and
$(r_-,r_+,X)=\widetilde{\rho}(\phi_-,\phi_0',\phi_+)$.  So by
definition,
\[
\frak{n}^{0,1}g_0\widetilde{\rho}(\phi_-,\phi_0',\phi_+) =
\frak{n}^{0,1}\Theta_\Sigma(\psi_-,\psi_0,\psi_+) =
\Theta_0(\psi_-,\psi_0,\psi_+) = g(\phi_-,\phi_0',\phi_+).
\]
Similarly, the left square in \eqref{eqn:rhotildediagram} commutes.

\medskip

{\em Part 3.\/}
We now show that $\widetilde{\rho}$ is an isomorphism, which is
homotopic through isomorphisms to the map sending
$(\phi_-,\phi_0',\phi_+)\mapsto (r_-,r_+,\rho(\phi_0'))$.  To see this,
for $\tau\in[0,1]$ consider the linear interpolation
\[
\widetilde{\rho}_\tau(\phi_-,\phi_0',\phi_+) \eqdef
(r_-,r_+,\tau\rho(\psi_0)+(1-\tau)\rho(\phi_0')).
\]
It is enough to show that $\widetilde{\rho}_\tau$ is injective for
each $\tau$.  Suppose to the contrary that
$\widetilde{\rho}_\tau(\phi_-,\phi_0',\phi_+)=0$.  Then by
Lemma~\ref{lem:tsnkm},
\begin{align}
\label{eqn:phipm0}
\phi_\pm & =0,\\
\label{eqn:tphi0'}
\tau\phi_0 + (1-\tau)\phi_0' &= 0.
\end{align}
It follows from \eqref{eqn:phipm0}, as in the definition of $g$ in
Proposition~\ref{prop:LGES}, that for any $\varepsilon>0$, if
$r,T_\pm$ are large, then
\begin{align}
\label{eqn:psipmi}
\|\psi_\pm\|_{L^2_1} & \le
\varepsilon\|\phi_0'\|_{L^2_1},\\
\label{eqn:psi0i}
\|\psi_0-\phi_0'\|_{L^2_1} & \le
\varepsilon\|\phi_0'\|_{L^2_1}.
\end{align}
The inequality \eqref{eqn:psipmi}, together with the equation
$\Theta_0(\psi_-,\psi_0,\psi_+)\in V_0$, implies as in the proof of
Lemma~\ref{lem:tsnkm} that the function $\zeta_\Sigma$ defined in
\eqref{eqn:zetaSigma} satisfies
\begin{equation}
\label{eqn:zeta0'i}
\|\zeta_\Sigma\|_{L^2_1} \le \varepsilon
\|\phi_0'\|_{L^2_1}
\end{equation}
if $r$ is sufficiently large.  On the other hand, equations
\eqref{eqn:zetaSigma} and \eqref{eqn:tphi0'} imply that
\[
\phi_0' = \tau\left((\phi_0'-\psi_0) + \frak{n}\zeta_\Sigma\right).
\]
This contradicts \eqref{eqn:psi0i} and \eqref{eqn:zeta0'i} if
$\varepsilon$ is chosen sufficiently small.

This completes the proof of Lemma~\ref{lem:last}, and thus
Theorem~\ref{thm:count} is proved.
\end{proof}